\documentclass[12pt, a4paper]{amsart}

\usepackage{amssymb, amsmath}
\usepackage[all]{xy}
\usepackage[inline]{enumitem}
\usepackage{hyphenat}
\usepackage[colorlinks,linkcolor=blue,citecolor=blue,urlcolor=red]{hyperref}
\usepackage{xcolor}
\usepackage{mathtools}
\usepackage{tikz-cd}
\usepackage[labelformat=empty]{caption}
 \usepackage[ left=3cm, right=3cm, top = 2.8cm, bottom = 2.8cm ]{geometry}

\makeatletter
\def\@tocline#1#2#3#4#5#6#7{\relax
  \ifnum #1>\c@tocdepth 
  \else
    \par \addpenalty\@secpenalty\addvspace{#2}%
    \begingroup \hyphenpenalty\@M
    \@ifempty{#4}{%
      \@tempdima\csname r@tocindent\number#1\endcsname\relax
    }{%
      \@tempdima#4\relax
    }%
    \parindent\z@ \leftskip#3\relax \advance\leftskip\@tempdima\relax
    \rightskip\@pnumwidth plus4em \parfillskip-\@pnumwidth
    #5\leavevmode\hskip-\@tempdima
      \ifcase #1
      \or\or \hskip 2em \or \hskip 2em \else \hskip 3em \fi%
      #6\nobreak\relax
    \dotfill\hbox to\@pnumwidth{\@tocpagenum{#7}}\par
    \nobreak
    \endgroup
  \fi}
\makeatother

\newcommand{\eq}[2]{\begin{equation}\label{#1}#2 \end{equation}}
\newcommand{\ml}[2]{\begin{multline}\label{#1}#2 \end{multline}}
\newcommand{\mlnl}[1]{\begin{multline*}#1 \end{multline*}}

\newcommand{\arir}{\ar@{^{(}->}}
\newcommand{\aril}{\ar@{_{(}->}}
\newcommand{\are}{\ar@{>>}}

\newcommand{\xr}[1] {\xrightarrow{#1}}

\newtheorem{lemma}{Lemma}[section]
\newtheorem{thm}[lemma]{Theorem}

\newtheorem{prop}[lemma]{Proposition}
\newtheorem{proposition}[lemma]{Proposition}
\newtheorem{cor}[lemma]{Corollary}

\theoremstyle{definition}
\newtheorem{defn}[lemma]{Definition}
\newtheorem{defn-prop}[lemma]{Definition-Proposition}

\newtheorem{nota}[lemma]{Notation}
\newtheorem{para}[lemma]{}

\newtheorem*{warning*}{Warning}

\theoremstyle{remark}
\newtheorem{qn}[lemma]{Question}

\newtheorem{remark}[lemma]{Remark}
\newtheorem{rmk}[lemma]{Remark}

\newtheorem{claim}{Claim}[lemma]
\newtheorem*{claim*}{Claim}

\newcounter{zaehler} 
\numberwithin{equation}{lemma}

\newcommand{\N}{\mathbb{N}}

\newcommand{\Q}{\mathbb{Q}}
\newcommand{\Z}{\mathbb{Z}}
\newcommand{\W}{\mathbb{W}}

\renewcommand{\P}{\mathbf{P}}
\newcommand{\A}{\mathbf{A}}
\newcommand{\F}{\mathbf{F}}
\newcommand{\G}{\mathbf{G}}

\newcommand{\sC}{\mathcal{C}}

\newcommand{\sE}{\mathcal{E}}
\newcommand{\sF}{\mathcal{F}}

\newcommand{\sH}{\mathcal{H}}
\newcommand{\sI}{\mathcal{I}}
\newcommand{\sJ}{\mathcal{J}}

\newcommand{\sO}{\mathcal{O}}
\newcommand{\sP}{\mathcal{P}}

\newcommand{\sS}{\mathcal{S}}

\newcommand{\sK}{\mathcal{K}}

\newcommand{\sN}{\mathcal{N}}
\newcommand{\sW}{\mathcal{W}}
\newcommand{\sU}{\mathcal{U}}
\newcommand{\sV}{\mathcal{V}}
\newcommand{\sX}{\mathcal{X}}
\newcommand{\sY}{\mathcal{Y}}
\newcommand{\sZ}{\mathcal{Z}}

\newcommand{\tF}{{\widetilde{F}}}
\newcommand{\tG}{{\widetilde{G}}}

\newcommand{\Cor}{\operatorname{\mathbf{Cor}}}

\newcommand{\HI}{\operatorname{\mathbf{HI}}}

\newcommand{\RSC}{{\operatorname{\mathbf{RSC}}}}

\newcommand{\cone}{\operatorname{cone}}
\newcommand{\Ext}{\operatorname{Ext}}
\newcommand{\ul}[1]{{\underline{#1}}}

\newcommand{\PST}{{\operatorname{\mathbf{PST}}}}

\newcommand{\NST}{\operatorname{\mathbf{NST}}}

\newcommand{\Hom}{\operatorname{Hom}}
\newcommand{\uHom}{\operatorname{\underline{Hom}}}

\newcommand{\Ker}{\operatorname{Ker}}

\renewcommand{\Im}{\operatorname{Im}}
\newcommand{\Coker}{\operatorname{Coker}}
\newcommand{\Tr}{\operatorname{Tr}}
\newcommand{\Nm}{\operatorname{Nm}}
\newcommand{\Div}{\operatorname{Div}}
\newcommand{\Pic}{\operatorname{Pic}}
\newcommand{\Br}{\operatorname{Br}}
\newcommand{\Spec}{\operatorname{Spec}}

\newcommand{\td}{\operatorname{trdeg}}

\newcommand{\dlog}{\operatorname{dlog}}

\newcommand{\Proj}{\operatorname{Proj}}
\newcommand{\Sm}{\operatorname{\mathbf{Sm}}}

\newcommand{\Ab}{\operatorname{\mathbf{Ab}}}

\newcommand{\tr}{{\operatorname{tr}}}
\newcommand{\Ztr}{{\operatorname{\mathbb{Z}_{\tr}}}}

\newcommand{\op}{{\operatorname{op}}}

\newcommand{\red}{{\operatorname{red}}}

\newcommand{\Zar}{{\operatorname{Zar}}}
\newcommand{\Nis}{{\operatorname{Nis}}}

\newcommand{\et}{{\operatorname{\acute{e}t}}}

\newcommand{\inj}{\hookrightarrow}

\newcommand{\surj}{\rightarrow\!\!\!\!\!\rightarrow}

\newcommand{\id}{{\operatorname{id}}}

\newcommand{\codim}{{\operatorname{codim}}}

\newcommand{\Sym}{{\operatorname{Sym}}}

\newcommand{\CH}{{\operatorname{CH}}}
\newcommand{\Frac}{{\operatorname{Frac}}}

\newcommand{\Ij}{{\operatorname{Inj}}}

\newcommand{\ol}{\overline}

\renewcommand{\epsilon}{\varepsilon}
\renewcommand{\div}{\operatorname{div}}

\renewcommand{\sp}{{\rm sp}}
\newcommand{\la}{\langle}
\newcommand{\ra}{\rangle}

\newcommand{\MNST}{\operatorname{\mathbf{MNST}}}

\newcommand{\MCor}{\operatorname{\mathbf{MCor}}}

\newcommand{\MPST}{\operatorname{\mathbf{MPST}}}
\newcommand{\CI}{\operatorname{\mathbf{CI}}}

\newcommand{\CItsp}{\CI^{\tau,sp}}
\newcommand{\CItspNis}{\CI^{\tau,sp}_{\Nis}}

\newcommand{\bcube}{{\ol{\square}}}
\newcommand{\bcubee}{\ol{\square}^{(1)}}
\newcommand{\cube}{\square}

\newcommand{\ulMPST}{\operatorname{\mathbf{\underline{M}PST}}}
\newcommand{\ulMNST}{\operatorname{\mathbf{\underline{M}NST}}}

\newcommand{\ulMCor}{\operatorname{\mathbf{\underline{M}Cor}}}
\newcommand{\ulMCorls}{\ulMCor_{ls}}

\newcommand{\uMPST}{\operatorname{\mathbf{\underline{M}PST}}}
\newcommand{\uMNST}{\operatorname{\mathbf{\underline{M}NST}}}

\newcommand{\uMCor}{\operatorname{\mathbf{\underline{M}Cor}}}
\newcommand{\uMCorls}{\ulMCor_{ls}}
\newcommand{\tensor}{\otimes}

\newcommand{\MCorls}{\MCor_{ls}}

\newcommand{\ulomega}{\underline{\omega}}
\newcommand{\uomega}{\underline{\omega}}

\def\rmapo#1{\overset{#1}{\longrightarrow}}
\def\Sd#1{S_{(#1)}}

\title{On the cohomology of reciprocity sheaves}
\author{Federico Binda, Kay R\"ulling, \and Shuji Saito}

\address{Dipartimento di Matematica ``Federigo Enriques'',  Universit\`a degli Studi di Milano\\ Via Cesare Saldini 50, 20133 Milano, Italy}
\email{federico.binda@unimi.it}

\address{Bergische Universit\"at Wuppertal\\ Gau\ss str. 20, D-42119 Wuppertal, Germany}
\email{ruelling@uni-wuppertal.de}

\address{Graduate School of Mathematical Sciences, University of Tokyo, 3-8-1 Komaba, Tokyo 153-8941, Japan}
\email{sshuji@msb.biglobe.ne.jp}

\thanks{F.B.\ is supported by the PRIN ``Geometric, Algebraic and Analytic Methods in Arithmetic'' and is a member of the INdAM group GNSAGA. K.R.\ was supported by the DFG Heisenberg Grant RU 1412/2-2. S.S.\ is supported by the JSPS KAKENHI Grant (20H01791).}
\begin{document}

\begin{abstract}
    In this paper we show the existence of an action of Chow correspondences on the cohomology of reciprocity sheaves. In order to do so, we prove a number of structural results, such as a projective bundle formula, a blow-up formula, a Gysin sequence, and the existence of proper pushforward. In this way we recover and generalize analogous statements for the cohomology of Hodge sheaves and Hodge-Witt sheaves.

We give several applications of the general theory to problems which have been classically studied. Among these applications, we construct new birational invariants of smooth projective varieties and obstructions to the existence of zero-cycles of degree one from the cohomology of reciprocity sheaves. 
\end{abstract}
\maketitle
\tableofcontents
\section*{Introduction}\label{sec:Intro}
\addtocontents{toc}{\protect\setcounter{tocdepth}{1}}
\subsection{Overview}
It is a well-known fact that a large class of cohomology theories for algebraic varieties can be equipped with an exceptional, covariant functoriality, despite the fact that they are naturally contravariant. 
The existence of this kind of ``trace'', or ``Gysin'' morphism associated to projective (or even proper) maps of smooth schemes is usually manifesting the existence of some Poincar\'e  duality theory for the cohomology one is interested in; if one replaces cohomology with homology, which is  naturally covariant, the exceptional functoriality is conversely represented by the existence of a pull-back along a certain class of maps. The construction of cohomological Gysin morphisms has occupied vast literature, stemming from   Grothendieck's trace formalism for coherent cohomology \cite{Ha66}. 

A classical instance in the homological setting is represented by the Chow groups. If $X$ is a smooth quasi-projective variety over a field $k$, the Chow groups $\CH_*(X)$ are naturally covariant for proper maps  and admit contravariant Gysin maps for quasi-projective local complete intersection morphisms \cite{Fu}. 
Fulton's construction of the Gysin morphism was  later promoted by Voevodsky in the context of his triangulated category of mixed motives ${\bf DM}^{\rm eff}_{\Nis}(k)$ over a perfect field $k$. Associated to a codimension $n$ closed immersion of smooth $k$-schemes, $i\colon Z\to X$,  Voevodsky \cite{V-TCM} constructed a distinguished triangle 
\[M(X-Z)\to M(X) \xrightarrow{i^*} M(Z)(n)[2n]\xrightarrow{\partial_{X,Z}} M(X-Z)[1] \]
where $i^*$ is the Gysin morphism and $\partial_{X,Z}$ is a residue map. Combining it with a projective bundle formula for motives, also provided by Voevodsky, the classical method of Grothendieck allows one to define exceptional functoriality along an arbitrary projective morphism between smooth $k$-varieties, factoring it as a closed immersion followed by a projection of a projective bundle. This as well as the naturality properties of Voevodsky's Gysin maps have been studied in detail by  D\'eglise \cite{Deg08}, \cite{Deg12}.

In  more recent times, Gysin morphisms for generalized cohomology theories have been constructed in the context of $\A^1$-homotopy theory, making use of the full six functor formalism as developed by \cite{AyoubThesis1},  \cite{AyoubThesis2} and \cite{CDBook}. See \cite{DJK} for more history and updated developments in that direction.

From the Gysin sequence, the projective bundle formula and the blow-up formula (the latter being also an ingredient in the construction of the first one) in the triangulated category of Voevodsky's motives it is possible to get corresponding formulas for every cohomology theory which is representable in ${\bf DM}^{\rm eff}_{\Nis}(k)$. This is the case of the sheaf cohomology of any complex of (strictly) $\A^1$-invariant Nisnevich sheaves with transfers. 

However, $\A^1$-invariant Nisnevich sheaves do not encompass all of the phenomena that one would like to study. Interesting examples of  sheaves which fail to satisfy this property are given by the sheaves of (absolute and relative) differential forms, $\Omega^i_{-/\Z}, \Omega^i_{-/k}$, the $p$-typical de Rham-Witt sheaves of Bloch-Deligne-Illusie, 
$W_m\Omega^i$, smooth commutative $k$-groups schemes with a unipotent part (seen as sheaves with transfer),  or the complexes $R\epsilon_* \Z/p^r(n)$, where $\Z/p^r(n)$ is the \'etale motivic complex of weight $n$ with $\Z/p^r$ coefficients, 
$\epsilon$ is the change of site functor from the \'etale to the Nisnevich topology, and $p>0$ is the characteristic of $k$.

For some of the above examples, instances of an exceptional functoriality have been studied before, with results scattered in the literature. In the case of the sheaves of differential forms, the existence of the pushforward is of course a consequence of general Grothendieck duality, e.g.,\ \cite{Ha66}, \cite{NeemanGroth}. 
In this paper, we offer a unified approach to treat the cohomology of arbitrary \emph{reciprocity sheaves}, a notion that includes all of the above examples: this is a particular abelian\footnote{The fact that the category of reciprocity sheaves is abelian is a non-trivial result, see \cite{S-purity}.} subcategory $\RSC_{\Nis}$ of the category of Nisnevich sheaves with transfers on $\Sm_k$. Its objects satisfy, roughly speaking, the property that for any $X\in \Sm_k$, each section $a\in F(X)$ ``has bounded ramification'', i.e., that the corresponding map $a\colon \Z_{tr}(X)\to F$ factors through a quotient $h_0(\sX)$ of $\Z_{tr}(X)$, associated to a pair $\sX = (\ol{X}, X_\infty)$  where $\ol{X}$  is a proper scheme over $k$ and $X_\infty$ is an effective Cartier divisor on $\ol{X}$ such that $X=\ol{X} - |X_\infty|$. See \ref{para:RSC} for more details. The category of reciprocity sheaves has been introduced by Kahn-Saito-Yamazaki in \cite{KSY2} (see also its precursor \cite{KSY1}), and is based on a generalization of the idea of Rosenlicht and Serre of the modulus of a rational map from a curve to a commutative algebraic group \cite[III]{Serre-GACC}.

Voevodsky's category of homotopy invariant Nisnevich sheaves,  $\HI_{\Nis}$ is an abelian subcategory of $\RSC_{\Nis}$. Heuristically, $\A^1$-invariant sheaves are special reciprocity sheaves with the property that every section $a\in F(X)$ has ``tame'' ramification at infinity.
Slightly more exotic examples of reciprocity sheaves are given by the sheaves ${\rm Conn}^1$ (in characteristic zero), whose sections over $X$ are rank $1$-connections, or ${\rm Lisse}^1_\ell$ (in characteristic $p>0$), whose sections on $X$ are the lisse $\overline{\mathbb{Q}}_\ell$-sheaves of rank $1$. Since $\RSC_{\Nis}$ is abelian and it is equipped with a lax\footnote{In the sense that only a weak form of associativity is known to hold, cf.  \cite[Thm. 1.5]{RSY}} symmetric monoidal structure \cite{RSY}, many more interesting examples can be manufactured by taking kernels, quotients and tensor products.  See \ref{subsec:exaRSC} for even more examples.

\subsection{Cohomology of cube invariant sheaves}In order to formulate our main results, we need a bit of extra notation. In \cite{KMSY1}, the authors introduced the category $\uMCor$ of \emph{modulus correspondences}, whose objects are pairs $\sX = (\ol{X}, X_\infty)$, called \emph{modulus pairs}, where $\ol{X}$ is a separated scheme of finite type over $k$ equipped with an effective Cartier divisor $X_\infty$ (the case $X_\infty=\emptyset$ is allowed) such that the \emph{interior} $\ol{X}-|X_\infty| = X$ is smooth. The morphisms are finite correspondences on the interiors satisfying some admissibility and properness conditions, see \ref{para:MCor}. The category $\uMCor$ admits a symmetric monoidal structure, denoted $\otimes$. Let $\uMPST$ be the category of additive presheaves of abelian groups on $\uMCor$.  Given $\sX\in \uMCor$ and $F\in \uMPST$, we write $F_{\sX}$ for the presheaf on the small \'etale site $\ol{X}_{\et}$ given by $U\mapsto F(U, U\times_{\ol{X}} X_\infty)$. We say that $F$ is a Nisnevich sheaf if, for every $\sX\in \uMCor$, the restriction $F_{\sX}$ is a Nisnevich sheaf; the full subcategory of Nisnevich sheaves of $\uMPST$ is denoted $\uMNST$. Thanks to \cite{KMSY1}, the inclusion $\uMNST \subset \uMPST$ has an exact left adjoint (the sheafification functor).

Among the objects of $\uMPST$, we are interested in a special class, namely those which satisfy the properties of being \emph{cube invariant}, \emph{semipure} and with \emph{$M$-reciprocity}, see  {\ref{para:CI}}.
The first two properties are easy to explain. Let $\bcube=(\P^1, \infty)\in \uMCor$. Then $F\in \uMPST$ is cube-invariant if for any $\sX\in \uMCor$ the natural map 
\[F(\sX) \to F(\sX\otimes \bcube)\]
induced by the projection $\ol{X}\times \P^1\to \ol{X}$ is an isomorphism. We have that $F$ is semipure if the natural map
\[F(\sX) \to F(X, \emptyset), \quad (X = \ol{X}-|X_\infty|)\]
is injective. The last condition of $M$-reciprocity is slightly more technical, and we refer the reader to 
the body of the paper. We write $\CI^{\tau, sp}$ for the category of cube invariant, semipure presheaves with $M$-reciprocity and $\CItspNis$ for  $\CI^{\tau, sp}\cap \uMNST$. 

It is possible to show (see \cite[1.6]{MS}, \cite[2.3.7]{KSY2}) that there is a fully faithful functor 
\[\ul{\omega}^{\CI}\colon \RSC_{\Nis} \to \CItspNis\]
admitting an exact\footnote{i.e. commuting with finite limits and colimits.} left adjoint, so that one can in particular specialize Theorem \ref{thm:main-intro} below on cube invariant sheaves to the case of reciprocity sheaves. If $G\in \RSC_{\Nis}$, we write $\widetilde{G}$ for $\ul{\omega}^{\CI} (G)$, and for $F\in \CItspNis$, $n\geq 1$ let us write 
\[\gamma^n F = \uHom_{\uMPST}(\widetilde{\mathbf{G}_m^{\otimes_{\HI_{\Nis}} n}}, F) \cong \uHom_{\uMPST}(\widetilde{\sK^M_n}, F).\]
This is a form of (negative) twist, see \ref{def:gtwist}, called \emph{contraction} in Voevodsky's theory (\cite[23]{MVW}). The tensor product with subscript $\HI$ is the tensor product for homotopy invariant Nisnevich sheaves with transfers from \cite[Chapt. 8]{MVW},  $\sK^M_n$ is the sheaf of improved Milnor $K$-theory introduced in \cite{Kerz} and the isomorphism follows from a result of Voevodsky \cite[5.5]{RSY}. See Theorem \ref{lem:dR0mod} and  Theorem \ref{thm:gdRW} for some computations of the twists. The  Bloch formula implies that for any family of supports $\Phi$ and any cycle $\alpha\in CH_{\Phi}^i(X)$
(see \ref{para:supports}) there is a natural cupping map   
\[c_\alpha\colon (\gamma^i F)_{\sX}[-i]\to R\ul{\Gamma}_{\Phi}F_{\sX} \quad  \text{in } D(X_\Nis),\]
which is compatible with refined intersection and pullback, see \ref{para:cyclecup}.

The following theorem summarizes parts of our results. Write $\uMCor_{ls}$ for the subcategory of $\uMCor$ whose objects $\sX = (X,D)$ satisfy the additional condition that $X\in \Sm$ and $|D|$ is a simple normal crossing divisor.
\begin{thm}\label{thm:main-intro}Let $F\in \CItspNis$ and let $\sX = (X, D)\in \uMCor_{ls}$. 
    \begin{enumerate}
        \item (Projective bundle formula, Theorem \ref{thm:pbf}) Let $V$ be a locally free $\sO_{X}$-module of rank $n+1$ and let $P={\bf P}(V)\xrightarrow{\pi} X$ be the corresponding projective bundle.  Let $\sP = (P, \pi^*D)$. Then there is a natural isomorphism in $D(X_\Nis)$
        \[ \sum_{i=0}^n \lambda^i_V\colon \bigoplus_{i=0}^n (\gamma^i F)_{\sX}[-i] \to R\pi_* F_{\sP}\]
       where $\lambda^i_V$ is induced by $c_{\xi^i}$ for the $i$-fold power $\xi^i\in \CH^i(X)$ of the first Chern class $\xi$ of $V$.
        \item (Gysin sequence, Theorem \ref{thm:gysin-tri})
Let $i\colon Z\inj X$ be a smooth closed subscheme of codimension $j$ intersecting $D$ transversally (Def.\ \ref{defn:ti})
and set $\sZ=(Z, D_{|Z})$. 
Then there is a canonical distinguished triangle in $D(X_{\Nis})$

\begin{equation}\label{eq:gysin-intro}i_*\gamma^j F_{\sZ}[-j]\xr{g_{\sZ/\sX}} F_{\sX}\xr{\rho^*} 
                             R\rho_* F_{(\tilde{X}, D_{|\tilde{X}}+ E)}\xr{\partial} i_*\gamma^j F_{\sZ}[-j+1],\end{equation}
where $\rho:\tilde{X}\to X$ is the blow-up of $X$ along $Z$ and $E=\rho^{-1}(Z)$. The Gysin map $g_{\sZ/\sX}$ satisfies an excess intersection formula \eqref{eq:exc-inters-formulaGysin}, it is compatible with smooth base change (Prop. \ref{prop:gysin-bc}), and with the cup product with Chow classes (Prop.\ \ref{prop:gysin-cup}). 
    \end{enumerate}
\end{thm}
We stress the fact that, in constrast to the $\A^1$-invariant setting, our Gysin sequence does not involve the cohomology of the open complement of $Z\subset X$ but, rather, the cohomology of a modulus pair constructed by taking the blow-up of $X$ along $Z$. 
When $F= \widetilde{G}$ and $G\in \HI_{\Nis}$, one can in fact verify that \eqref{eq:gysin-intro} gives back the classical Gysin sequence of D\'eglise and Voevodsky. For non-$\A^1$-invariant sheaves, the existence of the Gysin map is new  essentially in all of the above-mentioned examples: for instance, it does not follow from the work of Gros \cite{gros} for the de Rham-Witt sheaves. Other interesting cases are given by $F = \widetilde{{\rm Conn}^1}$ or $\widetilde{{\rm Lisse}^1}$, 
see Corollary \ref{cor:CL1}. We may also apply \eqref{eq:gysin-intro} for  $D=\emptyset$ and 
$F$  the whole de Rham-Witt complex complex and obtain in this way a Gysin sequence for the crystalline cohomology
$Ru_{X*}\sO_{X/W_n}$, where  $u_X: (X/W_n)_{\rm crys}\to X_{\Nis}$ is the natural map of sites, 
which generalizes to higher codimension the classical sequence induced by the residue map along a smooth closed divisor,
see Corollary \ref{cor:dRW-Gysin} and the following remark.

The key computation leading to the above results is the vanishing 
$H^i(Y, F_{(Y,\rho^*L)})=0$, for $i\ge 1$, where $\rho: Y\to \A^n$ is the blow-up in the origin and 
$L\subset \A^n$ a hyperplane passing through the origin, see Theorem \ref{thm:blow-upAn}.
The proof of this theorem occupies almost all of section $2$ and relies deeply on the theory of modulus sheaves with
 transfers. 

By factoring any projective morphism as a closed embedding followed by a projection from a projective bundle, we can use Theorem \ref{thm:main-intro} to construct pushforward maps (in fact, we construct, the pushforward with proper support along a quasi-projective morphism). See Definition \ref{defn:pfs} and Proposition \ref{prop:pfs} for the main properties. 
 Note that the pushforward is compatible with composition, smooth base-change and cup product with Chow classes. See \ref{para:RSC-pf} and Theorem \ref{prop:bc-pf}.
 
For $F = \ul{\omega}^{\CI} W_m\Omega^i$, the construction gives even a refinement of the pushforward map for cohomology of Hodge-Witt differentials
constructed by Gros \cite{gros}. See Cor.\ \ref{cor:refined-trace-intro} below. 

\subsection{Chow correspondences} When a cohomology theory is equipped with pushforward with proper support and a cup-product with cycles, it is possible, with a bit of extra work, to produce an action of Chow correspondences. Let $S$ be a separated $k$-scheme of finite type, and let $C_S$ be the category whose objects are maps $(f\colon X\to S)$ with the property that the induced map $X\to \Spec(k)$ is smooth and quasi-projective.  As for morphisms, we set (if $Y$ is connected)
\[ C_S(X,Y) =  \CH^{\dim Y}_{\Phi^{\rm prop}_{X\times_S Y}}(X\times Y)\]
where  $\Phi^{\rm prop}_{X\times_S Y}$ is the family of supports on $X\times Y$ consisting of closed subsets which are contained in $X\times_S Y$ and that are proper over $X$. Composition is given by the usual composition of correspondences  using the refined intersection product \cite[16]{Fu}.  If $F^\bullet$ is a bounded below complex of reciprocity sheaves and $(f\colon X\to S)$ and $(g\colon Y\to S)$ are objects of $C_S$, we can define for $\alpha\in C_S(X,Y)$ a morphism

\[\alpha^* \colon Rg_* F^\bullet_Y \to Rf_* F^\bullet_X  \text{ in } D^+(S_{\Nis})\] 
that is compatible with the composition of correspondences, satisfies a projection formula, and gives back the pushforward for reciprocity sheaves when $\alpha = [\Gamma_h^t]$ is the transpose of the graph of a proper $S$-morphism $h\colon X\to Y$. See Proposition \ref{prop:cycle-action}.  

For homotopy invariant sheaves, the existence of the action of Chow correspondences follows from work of Rost \cite{Rost} and  D\'eglise \cite{Deg12} (although, to our knowledge, this has not beet spelled out explicitly in the literature).

Previous instances of constructions of an action of Chow correspondences on the cohomology of Hodge and Hodge-Witt differentials can be found in \cite{CR11} and \cite{CR12}. However, we remark that the approach followed in this paper is conceptually different: in \cite{CR11} and \cite{CR12}, the existence of the whole de Rham and de Rham-Witt complex, with its structure of graded algebra, was used. In contrast, here the projective push-forward is directly constructed starting from a single reciprocity sheaf $F$ (and its twists). Our statements are also finer, since we get morphisms defined at the level of derived categories, rather than just between the cohomology groups.

\subsection{Applications} Let us now discuss how we can apply the formulas established so far to get new interesting invariants. 
\subsubsection{Obstructions to the existence of zero cycles of degree $1$}
In Section \ref{ssec:BMtype-obstructions} we explain how to use the proper correspondence action on the cohomology of an arbitrary reciprocity sheaf to construct very general obstructions of Brauer-Manin type to the  existence of  zero cycles on smooth projective varieties over function fields, recovering the classical obstruction as a special case. 

Here is the main result (see Theorem \ref{thm:dds}):
\begin{thm}\label{thm:BM-intro}
Let $f\colon Y\to X$ be  a dominant quasi-projective morphism between connected smooth $k$-schemes.
Assume that there are integral subschemes $V_i\subset Y$ which are proper, surjective, and generically finite
over $X$ of degree $n_i$, $i=1,\ldots, s$. Set $N={\rm gcd}(n_1,\ldots, n_s)$.
Let $F^\bullet\in {\rm Comp}^+(\RSC_\Nis)$ be a bounded below complex of reciprocity sheaves. 
Then there exists a morphism $\sigma: Rf_* F_Y^\bullet\to F_X^\bullet$ in $D(X_\Nis)$
such that the composition
\[F_X^\bullet\xr{f^*} Rf_*F_Y^\bullet\xr{\sigma} F_X^\bullet \]
is multiplication with $N$.
\end{thm}
In particular, if  $f$ is proper and $f^*\colon H^i(X, F^\bullet_X)\to H^i(Y, F^\bullet_Y)$  is not split injective, then the generic fiber of $f$ cannot have index $1$, i.e., there cannot exist a zero cycle of degree $1$. It is then possible to assemble the morphisms $\sigma$ in order to produce a generalization of the classical Brauer-Manin obstruction in the case of the function field of a curve (see \eqref{eq:BM5} and the references there for more details). 
This is explained in Corollary \ref{cor2:dds}.

See also the end of Section  \ref{ssec:BMtype-obstructions}  for a comprehensive list of references to previous works where unramified cohomology groups have been used to study obstructions to the local-global principle for rational points, rather than for 0-cycles,  over special types of global fields. 

\subsubsection{Birational invariants} Once we have established an action of Chow correspondences on the cohomology of reciprocity sheaves, this can be used to find birational invariants. 

Let us fix again a separated $k$-scheme of finite type $S$. 
We say that $(f\colon X\to S)$ and  $(g\colon Y\to S)\in C_S$, with $X$ and $Y$ integral, 
are \emph{properly birational over $S$} if there exists an integral scheme $Z$ (that we call \emph{proper birational correspondence}) over $S$ 
and two proper birational  $S$-morphisms $Z\to X$, $Z\to Y$ (note that we don't assume that $f$ or $g$ is proper). If we let $Z_0\subset X\times Y$ be the image of $Z\to X\times Y$, we can then look at the composition $[Z_0]^*\circ [Z_0^t]^*$, and get for example the following result.
\begin{thm}[see Theorem \ref{thm:BItop}]\label{thm:intro-BItop} 
Let $F\in \RSC_\Nis$ and assume that $F(\xi)=0$, for all points $\xi$ 
which are finite and separable over a  point of $X$ or $Y$ of codimension $\ge 1$.
Then any proper birational correspondence between $X$ and $Y$ induces an isomorphism
\[Rg_*F_Y\xr{\simeq} Rf_* F_X.\]
\end{thm}

If $Y=S$ in the statement of Theorem \ref{thm:intro-BItop}, we get a vanishing $R^i f_* F_X = 0$ for $i\geq 1$ and for any projective birational morphism $f\colon X\to Y$ and $F$ as in the theorem. The prototype example of a sheaf satisfying the condition $F(\xi)=0$ is the sheaf of top differential forms, $\Omega^{\dim X}_{/k}$. For this, the birational invariance is classical in characteristic zero, and follows from Hironaka's resolution of singularities. In positive characteristic it was proven in \cite{CR11} by using a similar action of Chow correspondences (although the statements in \emph{loc.cit.\ }were for the cohomology groups, not for the whole complexes in the derived category, see also \cite{Kovacs}). On the other hand, Theorem \ref{thm:intro-BItop} provides a very general class of birational invariants, many of which are new to us: for example, using results of Geisser-Levine,
we can consider the cohomology of the \'etale motivic complexes $R^i\epsilon_*( \Z/p^n(d))$ (for all $i$ and $n$ if ${\rm char}(k)=p>0$), where $d=\dim X = \dim Y$.  
See Corollary \ref{cor:BI-exa} for a more extensive list.

Among the other applications, we can use Theorem \ref{thm:intro-BItop}  
 to generalize parts of \cite[Thm 3.3]{Pirutka} (which generalizes \cite[Prop 3.4]{CV}). See Corollary \ref{cor:BI-exa20} for more details. 

We remark that the global sections of reciprocity sheaves enjoy a general invariance under proper (stable) birational correspondences, without assuming $F(\xi)=0$ for $\xi$ as above. See Theorem \ref{lem:BIH0} and the notations there. 

As a   byproduct of \ref{thm:BM-intro}, we also get (stably) proper birational invariance (see Definition \ref{sssec:BI}) for the $n$-torsion of the relative Picard scheme, $\Pic_{X/S}[n]$, for all $n$ and any  flat, geometrically integral, and projective morphism $X\to S$ between smooth connected $k$-schemes such that the generic fiber has index $1$.
This is classical and known to the experts if $S$ is the spectrum of an algebraically closed field, but it is new for general $S$. See Corollary \ref{cor-pictor}.

%

\subsubsection{Decomposition of the diagonal} In section \ref{para:dec-diag} we investigate the implications of the cycle action in case we have a decomposition of the diagonal, a method which was first employed in  \cite{BlSr}. For example we obtain:
\begin{thm}[see Theorem \ref{thm:comp-diag}]\label{intro-thm:comp-diag}
Let $f\colon X\to S$ be a smooth projective morphism, where  $S$ is the henselization of a smooth 
$k$-scheme  in a 1-codimensional point or
a regular connected affine scheme of dimension $\leq 1$ and of finite type over a function field $K$ over $k$. Assume that the diagonal cycle $[\Delta_{X_\eta}]$ of the generic fiber $X_\eta$ of $f$ has an integral decomposition. Then, for any $F\in \RSC_{\Nis}$, the pullback along $f$ induces an isomorphism 
\[H^0(X, F)\cong H^0(S, F).\]
\end{thm}
See Remark \ref{rem:dec_diagonal} for some conditions under which the diagonal decomposes.
Note that in the case $F = R^{i} \epsilon_* \Z/p^n(j)$, with $(i,j)\neq (0,0)$, 
and $X$ is defined over an algebraically closed field of characteristic $p>0$ and admits an integral decomposition of the diagonal,
 we obtain $H^0(X, R^{i} \epsilon_* \Z/p^n(j))=H^0(\Spec k,R^{i} \epsilon_* \Z/p^n(j))=0$. 
(The vanishing follows from \cite{GL}.)
This immediately implies a positive answer to Problem 1.2 of \cite{ABBvB}, and reproves Theorem 1 in \emph{loc.\ cit.}
See Corollary \ref{cor:dd0}. See also the recent work \cite{Otabe}, for a different approach.

In case $S=\Spec k$ and $F$ is $\A^1$-invariant Theorem \ref{intro-thm:comp-diag} is classical; Totaro proved  that it also holds 
for $F=\Omega^i_{-/k}$  (see \cite[Lem 2.2]{Totaro}) and  - building on ideas of Voisin and Colliot-Th\'el\`ene-Pirutka - 
used this to find many new examples of hypersurfaces that are not stably rational. It is an interesting question, whether the flexibility in the choice of the sheaf
$F$ coming from Theorem \ref{intro-thm:comp-diag}~---~e.g., $F$ can be any quotient of $\Omega^i_{-/k}$, say $F=\Omega^N_{-/k}/\dlog K_N^M$ from Corollary \ref{cor:BI-exa}~---~can be used  to find new examples of non-stably rational varieties.


Results for higher cohomology groups are also obtained if $F$ satisfies certain extra assumptions, see 
Theorems \ref{thm:comp-diag-u}, \ref{thm:comp-diag-top} and see Corollary \ref{cor:ddtop} for examples.

\subsubsection{Cohomology of ordinary varieties} Following Bloch-Kato \cite{BK} and Illusie-Raynaud \cite{IR}, we say that a variety $X$ over a perfect field $k$ of characteristic $p>0$ is  \emph{ordinary} if $H^m(X, B^r_X)=0$ for all $m$ and $r$, where $B^r_X = {\rm Im}(d\colon \Omega^{r-1}_X\to \Omega^{r}_X)$. It is equivalent to ask that the Frobenius $F\colon H^q(X, W\Omega^r_X)\to H^q(X, W\Omega^r_X)$ is bijective for all $q$ and $r$. If $X$ is an abelian variety $A$, this recovers the property that  the $p$-rank of $A$ is the maximum possible, namely equal to its dimension. For them, we have the following result.

\begin{cor}[see Corollary \ref{cor:ordinary}]Let $f\colon X\to S$ be a surjective morphism between smooth projective connected $k$-schemes.
Assume that the generic fiber has index prime to $p$.
Then
\[ X \text{ is ordinary } \Longrightarrow S \text{ is ordinary.}\]
\end{cor}
Note that the assumption on the generic fiber is of course guaranteed if $X_{k(S)}$ has a zero cycle of degree prime to $p$
 (for example, when $X_{k(S)}$ is an abelian variety). Similar implications hold for the properties ``$X$ is Hodge-Witt'' or ``the crystalline cohomology of $X$ is torsion-free''. See Remark \ref{rmk:HWcrystorsionfree}.

In connection to ordinary varieties, let us also mention the following result (see Corollary \ref{cor:refined-trace}):  
\begin{cor}\label{cor:refined-trace-intro}
Let $f\colon Y\to X$ be a morphism of relative dimension $r\ge 0$
between smooth projective $k$-schemes. Assume that $X$ is ordinary.
Then the Ekedahl-Grothendieck pushfoward (see \cite[II, 1.]{gros})  factors via
\eq{eqcor:refined-trace-intro}{R\Gamma(Y, W_n\Omega^q_Y)[r]\to R\Gamma(Y,  W_n\Omega^q_Y/B_{n,\infty}^q)[r]
\xr{f_*} R\Gamma(X, W_n\Omega^{q-r}_X),}
where $B_{n,\infty}^q= \bigcup_s F^{s-1}d W_{n+s-1}\Omega^{q-1}$ (see \cite[IV, (4.11.2)]{IR}) and  
$f_*$ is induced by the pushforward from \ref{para:RSC-pf}.
\end{cor}
Note that this is an essentially immediate consequence of the fact that the sheaves $B^q_{n, \infty}$ are reciprocity sheaves,  our general formalism and the computation of the twists of Theorem \ref{thm:gdRW}. In fact, even when $X$ is not ordinary, we always obtain a factorization in top degree
    \[R\Gamma(Y, W_n\Omega^r_Y)[r]\to R\Gamma(Y,  W_n\Omega^r_Y/B_{n,\infty}^r)[r]
\xr{f_*} R\Gamma(X, W_n\sO_X)
    \]
as a byproduct of the proof of Corollary \ref{cor:refined-trace}.

\subsubsection{Relationship with logarithmic motives} In \cite{logmot}, Park, {\O}stv{\ae}r, and the first author recently introduced a triangulated category of \emph{logarithmic motives} over a field $k$. Similar in spirit to Voevodsky's construction, the starting point is the category $lSm/k$ of log smooth (fs)-log schemes over $k$, promoted then to a category of correspondences. The localization with respect to a new Grothendieck topology, called the dividing-Nisnevich topology, and with respect to the log scheme $\bcube$, the log compactification of $\A^1_k$, produces the category denoted by $\mathbf{logDM}^{\rm eff}_{dNis}(k)$.

A theorem of Saito, see \cite{SaitologRSC}, shows that there exists a fully faithful exact functor
\[\mathcal{L}og\colon \RSC_{\Nis} \to \mathbf{Shv}^{\rm ltr}_{dNis}(k, \Z)\]
such that $\mathcal{L}og(F)$ is strictly $\bcube$-invariant in the sense of \cite[Def. 5.2.2]{logmot}, 
where the target is the category of dividing Nisnevich sheaves with log transfers on $lSm/k$. See \cite[2,4]{logmot}. This shows that Nisnevich cohomology of reciprocity sheaves is representable in $\mathbf{logDM}^{\rm eff}_{dNis}(k)$. Formulas like the projective bundle formula, the blow-up formula, the existence of the Gysin sequence and so on in $\mathbf{logDM}^{\rm eff}_{dNis}(k)$ can then be used to re-deduce \emph{a posteriori} some of the results in the present paper, under some auxiliary assumptions.

We warn the reader that in the proof of the main result of \cite{SaitologRSC} one needs in an essential way the formalism of push forward maps along projective morphisms, that we show in the present work. 

Moreover, note that the motivic formulas given in \cite{logmot}  cannot be used to deduce results involving higher modulus, that we do instead systematically in the present paper, and that the projective bundle formula, the blow-up formula and the Gysin triangle (using the identification of the log Thom space) in \cite{logmot} are  only proved under the assumptions of resolution of singularities, which we don't need. Finally, a general theory of log motives over a base (not just over a field) would be necessary to get the full strength of the sheaf-theoretic version of the results in this work.

\begin{warning*} The content of Theorem \ref{thm:main-intro} and of other main results in this paper (namely Corollary \ref{cor:PnInv} and Theorem \ref{thm:bus}) are a sheaf theoretic analogue to some of the results  on motives with modulus in \cite{KMSY3}, more precisely to \cite[Thm. 7.3.2]{KMSY3}, \cite[Thm. 7.4.3]{KMSY3} and \cite[Thm. 7.4.4]{KMSY3} (the latter being in fact a Theorem of K.~Matsumoto, proved only for the inclusion of a smooth divisor $Z$ in $X$, whereas we consider the case of $Z$ being a smooth closed subscheme of any codimension). 
    
    We warn the reader that our results \emph{cannot be recovered} from the existing  literature: for this to be the case it would be necessary to show that the cohomology of $\bcube$-invariant sheaves is representable in the category of motives with modulus $\mathbf{\underline{M}DM}^{\rm eff}(k)$ constructed in \cite{KMSY3}. In view of \cite[Thm. 5.2.4]{KMSY3}, one would require a positive answer to the following two questions.
\begin{qn}(1) Is the Nisnevich cohomology of $\bcube$-invariant sheaves invariant under blow-up with center contained in the support of the modulus?
    
    (2) Is a $\bcube$-invariant sheaf $F$ equivalent (in the derived category of sheaves) to its derived Suslin complex $RC_*^{\bcube}(F)$ defined in \cite[Def. 5.2.3]{KMSY3}? 
\end{qn}
Both questions seem out of reach  for general $\bcube$-invariant sheaves: note that (1) would amount to answer affirmatively to \cite[Question 1, p.4]{KMSY1}, and  that a (weaker) version of it is the content of Theorem \ref{thm:blow-upAn}, which is one of the crucial technical results of this paper.

    Question (2) is equivalent to ask whether the cohomology of a $\bcube$-invariant sheaf with transfers is again $\bcube$-invariant. For $\mathbb{A}^1$-invariant sheaves with transfers this is a deep theorem of Voevodsky, and boils down to studying a non-trivial interaction between the Nisnevich sheafification functor and the localization functor $L_{\mathbb{A}^1}(-)$. For semi-pure sheaves ({cf. \ref{para:CI} below}), this is shown in \cite{S-purity}, but the general case is wide open  (the first and third author once claimed the general case in characteristic 0 but a gap was found in its proof).  We hope that the main results of this paper are useful in attempts to answer the above open questions.

Moreover, even if both questions are answered positively, in order to get the full statement of Theorem \ref{thm:main-intro} from the motivic point of view it would be necessary to develop the whole theory of motives with modulus over a base, which is not available at the moment.
\end{warning*}

\subsection{Organization of the paper} We conclude this introduction with a quick presentation of the structure of the paper.

In \S \ref{sec:pre} we discuss some preliminaries and fix the notation. Nothing in this Section is new, and it can be found in \cite{KMSY1}, \cite{KMSY2}, \cite{MS}. In \S 2 we prove a key ``descent'' property for $\bcube$-invariant sheaves, namely Proposition \ref{prop:mu}. This is a crucial technical result that allows us to prove the invariance of the cohomology of cube invariant sheaves along a certain class of blow-ups, see Theorem  \ref{thm:blow-upAn}. Once this is established, we proceed to prove that the cohomology of cube invariant sheaves is also invariant with respect to the product with the modulus pair $(\P^n, \P^{n-1})$, Theorem \ref{thm:PnInv}. In \S \ref{sec:smblowup} we prove a smooth blow-up formula; in \S \ref{sec:twists} we introduce the twist and prove some of its basic properties. In \S \ref{sec:cupProduct} we use Rost' theory of cycle modules together with a formula for the tensor product of reciprocity sheaves to construct the cup product with Chow classes. In \S \ref{sec:pbf} we prove the projective bundle formula, and in \S \ref{sec:gysin} we construct the Gysin sequence: for this we essentially follow the steps of Voevodsky's construction in \cite{V-TCM}, but we also get a finer theory with supports (the local Gysin map).  In \S \ref{sec:pf} we assemble the Gysin maps and the morphisms induced by the projective bundle formula to construct general pushforwards. In this section we make use also of the cancellation theorems of \cite{MS}. In \S \ref{sec:prop-cor-action} we explain the construction of the action of Chow correspondences on reciprocity sheaves (and complexes of sheaves). Finally, in \S  \ref{sec:gen-appl} and \S \ref{sec:examples} we collect the main applications and a list of examples of reciprocity sheaves. The reader who is mostly interested in examples and applications may read the last two sections without having precise knowledge of modulus sheaves with transfers.
\medskip

In the paper, we use frequently the results from \cite{S-purity}, which plays a fundamental role for us.

\subsection*{Acknowledgements}The authors are grateful to Jean-Louis Colliot-Th\'el\`ene for remarks on a preliminary version of this manuscript, for providing with an extensive list of references and for pointing to Corollary \ref{cor:BI-exa20}. 
F.B.\ wishes to thank Alberto Merici for useful conversations and K.R.\ thanks Stefan Sch\"oer for useful discussions  on the 
Brauer group related to this  work and Christian Liedtke for a comment on a preliminary version. Part of this work has been done while the first author worked at the University of Regensburg, supported by the DFG SFB/CRC 1085 ``Higher Invariants'', and the third author was visiting the same institution also supported by the DFG SFB/CRC 1085, in the Spring of 2018 and of 2019. Yet another part of this work has been done during a visit of the second and the third author at the University of Milan in Spring 2020 thanks to an invitation by Luca Barbieri-Viale. We wish to thank all the institutions for their support. 
Finally, we would like to thank the referee for carefully reading our manuscript and for giving constructive comments which substantially helped improving the exposition.

\addtocontents{toc}{\protect\setcounter{tocdepth}{2}}
\section{Preliminaries}\label{sec:pre}

\subsection{Notations and conventions}
In the whole paper we fix a perfect base field $k$.
We denote by $\Sm$ the category of smooth separated $k$-schemes.
We write $\P^1= \P^1_k$ etc. and $X\times Y= X\times_k Y$, for $k$-schemes $X$, $Y$.
For a function field  $K/k$ we denote by 
$K\{x_1, \ldots, x_n\}$ the henselization of $K[x_1,\ldots, x_n]_{(x_1,\ldots. x_n)}$.
Let $R$ be a regular noetherian $k$-algebra. By \cite[(1.8) Thm]{Popescu86} and \cite[Exp I, Prop 8.1.6]{SGA4I}
we can write $R=\varinjlim_i R_i$, where $(R_i)_i$ is a directed system of smooth $k$-algebras, and we use the notation
$F(R)=\varinjlim_i F(\Spec R_i)$, for any presheaf $F$ on $\Sm$.

If $X$ is a scheme and $F$ is a Nisnevich sheaf on $X$, we will denote by $H^i(X, F)=H^i(X_\Nis, F)$ 
the $i$th cohomology group of $F$ on the small Nisnevich site of $X$,
similar with higher direct images. We denote by $X_{(n)}$ (resp. $X^{(n)}$) the set
of $n$ (resp. co-) dimensional points in $X$. 

\subsection{A recollection on modulus sheaves with transfers}
We recall some terminology and notations from the theory of modulus sheaves with transfers,
see \cite{KMSY1}, \cite{KMSY2}, \cite{KSY2}, and \cite{S-purity} for details.

\begin{para}\label{para:MCor}
A modulus pair $\sX=(\ol{X}, X_\infty)$ consists of 
a separated $k$-scheme of finite type  $\ol{X}$ and an effective (or empty) Cartier divisor $X_\infty$
such that $X:= \ol{X}\setminus |X_\infty|$ is smooth; it is called {\em proper} if $\ol{X}$ is proper over $k$.
Given  two modulus pairs $\sX=(\ol{X}, X_\infty)$ and $\sY=(\ol{Y}, Y_\infty)$, with opens 
$X:=\ol{X}\setminus |X_\infty|$ and $Y:=\ol{Y}\setminus |Y_\infty|$,  an
{admissible left proper prime correspondence} from $\sX$ to $\sY$ is given by
an integral closed subscheme $Z\subset X\times Y$ which is finite and surjective over a connected component of $X$,
such that the normalization of its closure $\ol{Z}^N\to \ol{X}\times \ol{Y}$ is proper over $\ol{X}$ and satisfies
\eq{para:MCor0}{X_{\infty|\ol{Z}^N}\ge Y_{\infty|\ol{Z}^N},}
as Weil divisors on $\ol{Z}^N$, where $X_{\infty|\ol{Z}^N}$ (resp. $Y_{\infty|\ol{Z}^N}$) denotes the pullback of $X_\infty$ (resp. $Y_\infty$) to $\ol{Z}^N$. The free abelian group generated by such correspondences is denoted by $\ulMCor(\sX, \sY)$.
By \cite[Propositions 1.2.3, 1.2.6]{KMSY1}, modulus pairs and left proper admissible correspondences define an additive category that
we denote by $\uMCor$. We write $\MCor$ for the full subcategory of  $\uMCor$  whose objects are proper modulus pairs. We denote by $\tau$ the inclusion functor $\tau\colon \MCor \to \uMCor$.
The induced category of {additive} presheaves of abelian groups is denoted by $\uMPST$ (resp. $\MPST$). 
We have functors 
\[\omega\colon \MCor \to \Cor, \quad \ul{\omega}\colon \uMCor \to \Cor\]
given by $(\ul{X}, X_\infty) \mapsto \ul{X} \setminus |X_\infty|$, where $\Cor$ is the category of finite correspondences introduced by Suslin-Voevodsky (see e.g. \cite{MVW}).
Note that there  is also a fully faithful functor
\[\Cor\to \uMCor, \quad X\mapsto (X,\emptyset).\]
We will abuse notation by writing
\eq{para:MCor01}{X=(X,\emptyset)\in \uMCor, \quad \text{for } X\in \Sm.}
Write $\tau^*$ for the restriction functor along $\tau$ and write $\tau_!$ for its left Kan extension. Similarly, write $\omega^*$ (resp. $\ul{\omega}^*$) for the restriction functor along $\omega$ (resp. $\ul{\omega}$) and $\omega_!$ (resp. $\ul{\omega}_!$) for its left Kan extension. 
We have the following  {commutative} diagrams at our disposal
\eq{para:MCor1}{
\xymatrix{\uMPST\ar[dr]_{\ulomega_!}  &  & \MPST\ar[ll]_{\tau_!}\ar[dl]^{\omega_!}\\
                                    &   \PST &   
                                    }
                                    \quad
\xymatrix{\uMPST\ar@/^1pc/[rr]^{\tau^*}   &  & \ar[ll]^{\tau_!}\MPST\\
                                    &   \PST.\ar[ur]_{\omega^*}\ar[ul]^{\ulomega^*} &   
                                    }}

Here $\PST$ is the category of presheaves of abelian groups on $\Cor$,
the functors in the left triangle are left adjoint to the functors in the right triangle,
all the functors are exact, the diagrams commute,
and we have $\tau^*F(\sX)=F(\sX)$, $\ulomega^*F(\sX)=F(X)$ and 
\eq{para:MCor2}{\ulomega_!F(X)=F(X,\emptyset)\overset{\eqref{para:MCor01}}{=:}F(X)}
for $\sX = (\ol{X}, X_\infty)$ and $X = \ol{X}\setminus |X_\infty|$. 

 We denote by $\Ztr(\sX)$ the presheaf on $\uMCor$ (resp. $\MCor$) 
represented by $\sX$ in $\uMCor$ (resp. in $\MCor$). We have $\tau_!\Ztr(\sX)=\Ztr(\sX)$ and $\ulomega_!\Ztr(\sX)=\Ztr(X)$.
 
Let $\sX=(\ol{X}, X_\infty)$, $\sY=(\ol{Y}, Y_\infty)\in \uMCor$. We set
\[\sX\otimes \sY := (\ol{X}\times \ol{Y}, p^*X_\infty+ q^*Y_\infty),\]
where $p$ and $q$ are the projections from $\ol{X}\times \ol{Y}$ to $\ol{X}$ and $\ol{Y}$, respectively.
In fact this defines a symmetric monoidal structure on  $\uMCor$ (resp. $\MCor$)
which extends (via Yoneda) uniquely to a right exact monoidal structure $\otimes$ on $\uMPST$ (resp. $\MPST$).
Similarly, there is a monoidal structure on $\PST$. The functors $\ulomega_!$, $\omega_!$, $\tau_!$ are monoidal, since they are all defined as left Kan extensions of the functors $\uomega, \omega$ and $\tau$, which are clearly monoidal.  
For $F\in \uMPST$  the functor $(-)\otimes F:\uMPST\to \uMPST$ admits a right adjoint
denoted by $\uHom_{\uMPST}(F,-)$; similar with $F\in \MPST$ (see, e.g., \cite[8]{MVW}). 
\end{para}


{   

\begin{para}\label{para:sheaf}
For $F\in \uMPST$ and $\sX=(\ol{X}, X_\infty)\in \uMCor$ denote by $F_{\sX}$ the presheaf
\eq{para:sheaf1}{(\et/\ol{X})^{\op}\ni U\mapsto F_{\sX}(U):= F(U, X_{\infty|U}), }
where $(\et/\ol{X})$ denotes the category of all \'etale maps $U\to \ol{X}$.
We say $F$ is a Nisnevich sheaf if $F_{\sX}$ is a Nisnevich sheaf, for all $\sX\in \uMCor$.
We denote by $\uMNST$ the full subcategory of $\uMPST$ consisting of Nisnevich sheaves.

We say $F\in \MPST$ is a Nisnevich sheaf if $\tau_!F$ is and denote the corresponding full subcategory 
by $\MNST$. The functors in \eqref{para:MCor1} restrict to Nisnevich sheaves and have the same
adjointness and exactness properties, see \cite[4.2.5, 5.1.1, 6.2.1]{KMSY2}.
Furthermore, there are Nisnevich sheafification functors 
\[\ul{a}_{\Nis}: \uMPST\to \uMNST, \quad a_\Nis: \MPST\to \MNST,\]
\[a_{\Nis}^V: \PST\to \NST,\]
which are left adjoint to the forgetful functors, restrict to the identity on  Nisnevich sheaves and satisfy
\eq{para:sheaf2}{
\ulomega_! \ul{a}_\Nis= a^V_\Nis\ulomega_!, \quad \omega_! a_\Nis= a^V_\Nis\omega_!, \quad    \tau_! a_{\Nis}=  \ul{a}_\Nis \tau_!}
   and 
\eq{para:sheaf3}{a_{\Nis}\omega^*=\omega^* a_{\Nis}^V, \quad \ul{a}_{\Nis}\ulomega^*=\ulomega^* a^V_\Nis,}
   see \cite[Thm 2]{KMSY1}, \cite[Thm 4.2.4, 4.2.5, 6.2.1]{KMSY2} ($a_\Nis^V$ was constructed by Voevodsky).
It follows that $\NST$, $\uMNST$, and $\MNST$ are  Grothendieck  abelian categories and that  the sheafification functors are exact. For $F\in \uMPST$ and  $\sX=(\ol{X}, X_{\infty})\in\uMCor$ we have
\eq{para:sheaf0}{\ul{a}_\Nis(F)(\sX)= \varinjlim_{f:\ol{Y}\to \ol{X}}F_{(\ol{Y}, f^*X_\infty), \Nis}(\ol{Y}),}
where the limit is over all proper morphisms $f:\ol{Y}\to \ol{X}$ which restrict to an isomorphism
over $X=\ol{X}\setminus|X_\infty|$ and $F_{(\ol{Y}, f^*X_\infty), \Nis}$ denotes the Nisnevich sheafification
of the presheaf $F_{(\ol{Y}, f^*X_\infty)}$ on the  site $\ol{Y}_\Nis$, see \cite[Thm.2(1)]{KMSY1}.
In the following we will use the notation 
\[F_\Nis:= \ul{a}_\Nis(F), \quad H_\Nis:= a_\Nis^V(H), \quad F\in \ulMPST, \,H\in \PST.\]
\end{para}

\begin{lemma}\label{lem:crit-surj}
A morphism $\varphi:F\to G$ in $\ulMNST$ is surjective (i.e., has vanishing cokernel), if
for all $\sX=(\ol{X}, X_\infty)\in \uMCor$, with $\ol{X}$ normal, and all $x\in \ol{X}$
the morphism
\[F(\sX_{(x)})\to G(\sX_{(x)})\]
is surjective, where $\sX_{(x)}=(\ol{X}_{(x)}, X_{\infty|\ol{X}_{(x)}})$ and $\ol{X}_{(x)}= \Spec \sO_{\ol{X},x}^h$ is the spectrum of the henselization of the local ring $\sO_{\ol{X},x}$. 
\end{lemma}
\begin{proof}
Let $C$ be the cokernel of $\varphi$ in $\ulMPST$. We want to show $\ul{a}_\Nis(C)=0$.
For $\sX\in \ulMCor$, set $C_\sX= \Coker(\varphi_{\sX}: F_{\sX}\to G_{\sX})$ in 
the category of presheaves on $(\et/\ol{X})$; denote by $C_{\sX,\Nis}$ its  Nisnevich sheafification. 
By \eqref{para:sheaf0} it suffices to show $C_{\sX,\Nis}=0$, if $\ol{X}$ is normal.
The latter is equivalent to the surjectivity of $\varphi_{\sX}$ in the category of Nisnevich sheaves  on $\ol{X}$,
which is equivalent to the statement.
\end{proof}

\begin{para}\label{para:CI}
Set $\bcube:=(\P^1, \infty ) \in \MCor$. For $F\in \ulMPST$ we say that
    \begin{enumerate}
        \item $F$ is {\em cube-invariant} if the
map $F(\sX)\to F(\sX\otimes \bcube)$ induced by the pullback along the projection is an isomorphism.
\item $F$ has {\em $M$-reciprocity} if the counit map $\tau_!\tau^*F\to F$ is an isomorphism.
\item $F$ is {\em semipure} if the unit map $F\to \ulomega^*\ulomega_! F$ is injective. 
    \end{enumerate}
We denote by $\uMPST^\tau$ the full subcategory of $\uMPST$ consisting of the objects with $M$-reciprocity.
Note that for $\sX$ a proper modulus pair we have $\Ztr(\sX)\in \uMPST^\tau$.
We denote by $\CItsp$ the full subcategory of $\uMPST$ consisting of the cube-invariant semipure objects 
with $M$-reciprocity. We set
\[\uMNST^\tau:= \uMPST^\tau\cap \uMNST\quad \text{and}\quad \CItspNis:=\CItsp\cap \uMNST.\]
By \cite[Thm 10.1]{S-purity}, the sheafification functor $\ul{a}_\Nis$ restricts to
\eq{para:spNis1}{\ul{a}_\Nis: \CItsp\to \CItspNis.}
The natural inclusion $\CItspNis\inj \uMPST^\tau$ has a left adjoint
\eq{para:CI1}{h^{\bcube,\sp}_{0,\Nis}: \uMPST^\tau\to \CItspNis}
given by
\[h^{\bcube,\sp}_{0,\Nis}(F)= \ul{a}_{\Nis}(\ul{h}_0^\bcube(F)^{\sp}),\]
where for $G\in \uMPST$ 
\begin{enumerate}
\item
$\ul{h}_0^\bcube(G)\in \uMPST$ is the maximal cube invariant quotient of $G$
defined by 
 {
\eq{para:spNis2}{
\ul{h}_0^\bcube(G)(\sX)=\Coker(G(\sX\otimes \bcube)\xr{i_0^*-i_1^*} G(\sX)),}  }
where $i_\epsilon: \{\epsilon\}\to \bcube$, $\epsilon\in \{0,1\}$, are induced by the natural closed immersions,
\item $G^{\sp}=\Im(G\to \ulomega^*\ulomega_! G)$ denotes the \emph{semipurification} of $F$.
\end{enumerate}
The left adjointness of \eqref{para:CI1} to the natural inclusion follows from \cite[Lem 1.14(i)]{MS} and the adjunction $\tau_!\dashv \tau^*$. 
We note that for any $F\in\uMPST$ the presheaf  $h^{\bcube,\sp}_{0,\Nis}(F)$ is defined and is in fact a cube-invariant, semipure Nisnevich sheaf on $\uMCor$.

For $\sX$ a proper modulus pair we set
\eq{para:CI2}{h^{\bcube,\sp}_{0,\Nis}(\sX):=h^{\bcube,\sp}_{0,\Nis}(\Ztr(\sX))\in \CItspNis.}
\end{para}

\begin{lemma}\label{lem:Nis}
Let $F\in \CItspNis$ and $G,H\in \ulMPST^\tau$. 
Assume there is a surjection $\Ztr(\sX)\surj G$, for some $\sX\in\uMCor$.
We have
\begin{enumerate}[label=(\arabic*)]
\item\label{lem:Nis1}   $\uHom_{\ulMPST}(G, F)\in \CItspNis$;
\item\label{lem:Nis3} 
       $\Hom_{\ulMPST}(H\otimes G, F)= 
               \Hom_{\ulMPST}(h_{0,\Nis}^{\bcube,\sp}(H), \uHom_{\uMPST}( G, F))$.
\end{enumerate}
\end{lemma}
\begin{proof}
\ref{lem:Nis1}. First assume $G=\Ztr(\sX)$, for some $\sX\in \uMCor$. 
In this case $\uHom(G, F)(\sY)=F(\sX\otimes\sY)$.
Clearly this defines a cube-invariant Nisnevich sheaf. It has $M$-reciprocity by \cite[Lem 1.27(2)]{S-purity},
and has semipurity by \cite[Lem 1.29(2)]{S-purity}.
Hence $\uHom(G,F)\in \CItspNis$ in this case. In the general case 
consider a resolution 
\[\bigoplus_j \Ztr(\sY_j)\to \bigoplus_i \Ztr(\sX_i)\to G\to 0.\]
We obtain an exact sequence
\eq{lem:sp4}{0\to \uHom(G, F)\to \prod_i \uHom(\Ztr(\sX_i), F)\to \prod_j \uHom(\Ztr(\sY_j), F).}
This directly implies cube-invariance and semipurity. 
The sheaf property holds since $\ul{i}_\Nis\ul{a}_\Nis: \uMPST\to \uMPST$ is left exact,
 where $\ul{i}_\Nis$ is the forgetful functor.
In general $M$-reciprocity won't hold since $\tau_!\tau^*$ does not commute with infinite products;
however it clearly holds if the first product in \eqref{lem:sp4} is finite and by assumption we find such a 
resolution.
\ref{lem:Nis3} follows from  \ref{lem:Nis1} and adjunction.
\end{proof}

\begin{para}\label{para:RSC}
The full subcategory of $\PST$  given by $\RSC:= \ulomega_!\CI^{\tau,\sp}$ is
called the category of {\em reciprocity presheaves}. The full subcategory of $\NST$ given by 
$\RSC_{\Nis}:=\ulomega_!\CItspNis$ is called the category of {\em reciprocity sheaves}.
It is direct to see that $\RSC$ is an abelian category, closed under sub-objects and quotients in $\PST$. 
On the other hand, it is a theorem \cite[Thm. 0.1]{S-purity} that $\RSC_{\Nis}$ is also abelian. 
We use the following notation for a proper modulus pair $\sX$
\[h_0(\sX):=\ulomega_!(h_0^\bcube(\sX))= \ulomega_!(h_0^{\bcube, \sp}(\sX))\in \RSC\]
and 
\[h_{0,\Nis}(\sX):=\ulomega_!(h_{0,\Nis}^\bcube(\sX))=\ulomega_!(h_{0,\Nis}^{\bcube,\sp}(\sX))\in \RSC_{\Nis}.\]
Note that $h_{0,\Nis}(\sX)=h_0(\sX)_\Nis$.
By  \cite[(1.13)]{MS} (see also \cite[Prop 2.3.7]{KSY2}) there is an adjunction
\eq{para:tut2}{ 
\xymatrix{
\CItspNis\ar@<-1ex>[r]_-{\uomega_!} &   \RSC_\Nis\ar@<-1ex>[l]_{\uomega^{\CI}},
}}
where $\uomega^{\CI}$ is right adjoint to $\uomega_!$ and is given by
\[\uomega^{\CI}(F)=\tau_!\Hom_{\MPST}(h_0^{\bcube}(-), \omega^*F).\]
In the notation of \cite{KSY2} we have $\uomega^{\CI}=\tau_!\omega^{\CI}$. 

 {Recall that Voevodsky's category of homotopy invariant Nisnevich sheaves,  $\HI_{\Nis}$ is an abelian subcategory of $\RSC_{\Nis}$, and thanks to \cite[Th. 5.6]{VoPST}, the natural inclusion $\HI_{\Nis}\to \NST$ has a left adjoint  
\eq{para:tut3}{ 
h_{0,\Nis}^{\A^1} : \NST \to \HI_\Nis.}
By \cite[Prop. 2.3.2]{KSY2}, we have
\eq{para:tut4}{ 
h_{0,\Nis}^{\A^1}(h_{0,\Nis}(\sX)) =h_{0,\Nis}^{\A^1}(\Ztr(\omega\sX)).}
} 
\end{para}

}

\section{Cohomology of blow-ups and invariance properties}

\subsection{A lemma on modulus descent}

\begin{nota}\label{bcube1}
For $m,n\ge 1$ we use the following notation
\[\bcube^{(m,n)}:= (\P^1, m\cdot 0 +n\cdot \infty),\quad \bcube^{(n)}:= \bcube^{(n,n)}.\]
In particular,
\[\bcubee=(\P^1,0+\infty).\]
\end{nota}

\begin{lemma}\label{lem:h0P1}
Let $R$ be an integral regular  $k$-algebra.
For all $m,n\ge 1$, there is an isomorphism
\[\theta_{m,n}: 
h_0(\bcube^{(m,n)})(R)\xr{\simeq} ((R[t]/t^m)^\times\oplus (R[z]/z^n)^\times)/R^\times\oplus \Z,\] 
where $R^\times$ acts diagonally on the direct sum.
If $Z\in \ulomega_!\Ztr(\bcube^{(m,n)})(R)$ is a prime correspondence which
we can write as $Z=V(g)$, for an irreducible polynomial 
$g=a_rt^r+ \ldots + a_1 t+a_0\in R[t]$ with $a_r, a_0\in R^\times$, and $r\ge 1$, then 
\[\theta_{m,n}(Z)=(g(t)/(t-1)^r, g_{\infty}(z)/(1-z)^r, r),\]
where $g_{\infty}(z)= a_0 z^r+\ldots+ a_{r-1} z+ a_r$.
Furthermore, if $m'\le m$ and $n'\le n$,
then we obtain a commutative diagram
\[\xymatrix{
h_0(\bcube^{(m',n')})(R)\ar[r]^-{\theta_{m',n'}}\ar[d] & 
                                 ((R[t]/t^{m'})^\times\oplus (R[z]/z^{n'})^\times)/R^\times\oplus \Z\ar[d]\\
h_0(\bcube^{(m,n)})(R)\ar[r]^-{\theta_{m,n}} &
                                ((R[t]/t^m)^\times\oplus (R[z]/z^n)^\times)/R^\times\oplus \Z,
}\]
where the vertical map on the left hand-side is induced by 
$\bcube^{(m',n')}\to \bcube^{(m,n)}$ in $\uMCor$ and
 the vertical map on the right is the natural quotient map.
 \end{lemma}
\begin{proof}
The map $\theta_{m,n}$ is the composition of the two isomorphisms
\begin{align*}
h_0(\bcube^{(m,n)})(R)& \xr{\simeq\, (*)} \Pic(\P^1_R, m\cdot 0+ n\cdot \infty)\\
                &    \xr{\simeq \,  (**)}  ((R[t]/t^m)^\times\oplus (R[z]/z^n)^\times)/R^\times\oplus \Z,
\end{align*}
which are defined as follows. We denote by $F_R:=m\cdot 0_R+ n\cdot \infty_R\subset \P^1_R$ 
the closed subscheme:
(*) is induced by the classical map from Weil to Cartier divisors 
\[\Ztr(\bcube^{(m,n)})(R)\ni D\mapsto 
(\sO(D), \id_{\sO_{F_R}})\in \Pic(\P^1_R, F_R),\]
where $\sO(D)$ is the line bundle on $\P^1_R$ given by 
$\sO(D)(U)=\{f\in R(t)^\times\mid \div_U(f)\ge D \}$; it is an isomorphism by  \cite[Thm 1.1]{RY}.
For (**) consider the  exact sequence
\[H^0(\P^1_R, \sO^\times)\to H^0(F_R, \sO^\times)\to \Pic(\P^1_R, F_R)
\to \Pic(\P_R^1)\to \Pic(F_R).\]
The last map decomposes as $\Pic(R)\oplus\Pic(\P^1) \to \Pic(F_R)$ given by
\[(M,\sO(\{1\})^{\otimes n})\mapsto (M_{|\P^1_R}\otimes (\sO(\{1\})^{\otimes n})_{|F_R}=M_{|F_R}.\]
Since $F_R\to \Spec R$ has a section the map $\Pic(R)\to \Pic(F_R)$ is injective. Hence
the above sequence yields an exact sequence
\[H^0(\P^1_R, \sO^\times)\to H^0(F_R, \sO^\times)\to \Pic(\P^1_R, F_R)
\xr{d} \Z \to 0,\]
where  $d(L,\alpha):=d(L):=\deg(L_{|\P^1_{\Frac(R)}})$;
 we can choose a splitting of $d$ by
$r\mapsto (\sO_{\P^1_R}(\{1\})^{\otimes r}, \id_{\sO_{F_R}})$;
the map in the middle sends  $u\in H^0(F_R, \sO^\times)$ to 
$(\sO_{\P^1_R}, u\cdot: \sO_{F_R}\xr{\simeq} \sO_{F_R})$, where $u\cdot$ is the isomorphism given by multiplication by $u$.
Let $(L,\alpha)$ be a pair with $L$  a line bundle on $\P^1_R$ with $d(L)=r$ and 
$\alpha: \sO_{F_R}\xr{\simeq} L_{|F_R}$ an isomorphism;  we find an isomorphism
$\varphi :L\otimes\sO(\{1\})^{\otimes -r}\xr{\simeq} \sO_{\P_R^1}$ and define the isomorphism 
$\alpha'$ as the composition
\[\alpha'=(\alpha'_{m\cdot 0},\alpha'_{n\cdot \infty}):
\sO_{F_R}\xr{\simeq\, \alpha}L_{|F_R}
    = (L\otimes\sO(\{1\})^{\otimes -r})_{|F_R} 
\xr{\varphi_{|F_R}} (\sO_{\P^1})_{|F_R}, \]
where the equality follows from the fact that we have a canonical identification $\sO(\{1\})_{|F_R} = \sO_{|F_R}$.
Hence $\varphi$ induces an isomorphism $(L\otimes\sO(\{1\})^{\otimes -r}, \alpha)\cong (\sO_{\P^1_R}, \alpha')$;
the isomorphism (**) is given by
\[(L,\alpha)\mapsto (\alpha'_{m\cdot0}(1), \alpha'_{n\cdot \infty}(1), d(L)).\]

Let $Z=V(g)\in \Ztr(\bcube^{(m,n)})(R)$ be a prime correspondence as in the statement. 
Write $t=T_0/T_1$ and let $G\in R[T_0, T_1]$ be the homogenization of $g$.
We have an isomorphism
\[\sO(Z)\otimes \sO(\{1\})^{\otimes -r}= \sO_{\P^1_R}\cdot \tfrac{(T_0-T_1)^r}{G}
 \xr{\simeq}\sO_{\P^1_R},\]
where the second isomorphism is given by multiplication with $G/(T_0-T_1)^r$.
Thus $\theta_{m,n}$ admits the description from the statement, where $z=1/t$.
The commutativity of the diagram follows directly from this.
\end{proof}

\begin{remark}\label{rmk:h0P1}
Denote by $\W_m$ the ring-scheme of big Witt vectors of length $m$.
If $A$ is a ring we can identify the $A$-rational points of the underlying group scheme
with
\[\W_m(A)= (1+tA[t])^\times/(1+t^{m+1} A[t])^\times.\]
Then the maps $\theta_{m,n}$ from Lemma \ref{lem:h0P1}, $m,n\ge 1$, induce isomorphisms in $\NST$
\[\theta_{m,n}: h_{0,\Nis}(\bcube^{(m,n)})\xr{\simeq} \W_{m-1}\oplus \W_{n-1} \oplus \G_m \oplus \Z.\]
Indeed, it follows immediately from Lemma \ref{lem:h0P1} that we have such an isomorphism of Nisnevich sheaves.
To check the compatibility with transfers it suffices to check the compatibility
with transfers of the limit $\varprojlim_{m,n}\theta^{m,n}$ (since the transition maps are surjective).
Since $\W\oplus \W\oplus \G_m\oplus \Z$ is a $\Z$-torsion-free  sheaf on $\Sm_{\Nis}$ 
for which the pull-back along dominant \'etale maps is injective,
the compatibility with transfers follows automatically from \cite[Lem 1.1]{MS}.
\end{remark}

\begin{lemma}\label{lem:Gm}
The unit map
\eq{lem:Gm1}{h_{0,\Nis}^{\bcube, \sp}(\bcubee)\xr{\simeq} 
\ulomega^*\ulomega_! h_{0,\Nis}^{\bcube, \sp}(\bcubee)
\cong \ulomega^*(\G_m\oplus \Z)}
is an isomorphism in $\CItspNis$. Furthermore, the natural maps
\eq{lem:Gm2}{h_{0,\Nis}^{\bcube, \sp}(\bcube^{(m,n)})\to  h_{0,\Nis}^{\bcube,\sp}(\bcubee)}
are surjective, for  all $m,n\ge 1$,  and there exists a splitting in $\CItspNis$
\[ s_{m,n} : \ulomega^*(\G_m\oplus \Z) \to h_{0,\Nis}^{\bcube, \sp}(\bcube^{(m,n)})\]
of \eqref{lem:Gm2} such that the following diagram is commutative for integers
$m'\geq m$ and $n'\geq n$:
\eq{lem:Gm2-1}{\xymatrix{
\ulomega^*(\G_m\oplus \Z) \ar[r]^-{s_{m',n'}} \ar[rd]_-{s_{m,n}} 
&h_{0,\Nis}^{\bcube, \sp}(\bcube^{(m',n')}) \ar[d] \\
&h_{0,\Nis}^{\bcube, \sp}(\bcube^{(m,n)}) . \\ }}
\end{lemma}
\begin{proof}
The second isomorphism in \eqref{lem:Gm1} holds by Lemma \ref{lem:h0P1} and 
Remark \ref{rmk:h0P1};  
the unit map is injective by semipurity.
We show the surjectivity of the composite map
\eq{lem:Gm3}{h_{0,\Nis}^{\bcube, \sp}(\bcube^{(m,n)})\to  h_{0,\Nis}^{\bcube,\sp}(\bcubee)
\to \ulomega^*(\G_m\oplus \Z), }
for $m,n\ge 1$. By Lemma \ref{lem:crit-surj} it suffices to show the surjectivity on $(\Spec R, (f))$, where
$R$ is an integral normal local $k$-algebra and $f\in R\setminus\{0\}$, such that $R_f$ is regular.
Denote by
\[\psi: \Ztr(\bcube^{(m,n)})(R,f)\to R_f^\times\oplus\Z\]
the precomposition of $\eqref{lem:Gm3}$ evaluated at $(R,f)$ with the 
quotient map 
\[\Ztr(\bcube^{(m,n)})(R,f)\to h_{0,\Nis}^{\bcube, \sp}(\bcube^{(m,n)})(R,f).\]
By Lemma \ref{lem:h0P1} 
\eq{lem:Gm3.5}{\psi(V(a_0+ a_1t+\ldots +a_r t^r))= ((-1)^r a_0/a_r, r),}
provided that $Z=V(a_0+ a_1t+\ldots +a_r t^r)$ is an admissible prime correspondence and $a_i\in R_f$. 
We claim that $\psi$ is surjective.
To this end, observe that for $a\in R_f^\times$ we find $N\ge 0$ and  $b\in R$ such that
\eq{lem:Gm4}{ab=f^{nN}, \quad \text{and}\quad af^{mN}\in R.}
Set $W:=V(t^{mnN} + (-1)^{mnN} a)\subset \Spec R_f[t, 1/t]$ and $K=\Frac(R)$.
Let $t^{mnN} + (-1)^{mnN} a=\prod_i h_i$ be the decomposition into monic irreducible factors in $K[t,1/t]$
and denote by $W_i\subset \Spec R_f[t, 1/t]$ the closure of $V(h_i)$. (Note that $W_i=W_j$ for $i\neq j$ is allowed.)
The $W_i$ correspond to the components of $W$ which are dominant over $R_f$; 
since $W$ is finite (the polynomial defining $W$ is monic) and surjective over $R_f$, so are the $W_i$. We claim
\eq{lem:Gm5}{W_i\in \Ztr(\bcube^{(m,n)})(R,f).}
Indeed, let $I_i$ (resp. $J_i$) be the ideal of the closure of $W_i$ in $\Spec R[t]$ (resp. $\Spec R[z]$ with $z=1/t$).
By \eqref{lem:Gm4} 
\[bt^{nmN} +  (-1)^{mnN} f^{nN}\in I_i \quad \text{and} 
\quad f^{mN} +  (-1)^{mnN} f^{mN}a z^{mnN}\in J_i.\]
Hence $(f/t^m)^{nN}\in R[t]/I_i$ and $(f/z^n)^{mN}\in R[z]/J_i$.
It follows that $f/t^m$  (resp. $f/z^n$) is integral over $R[t]/I_i$ (resp. $R[z]/J_i$);
thus \eqref{lem:Gm5} holds. Put
 \[W_a= \sum_i W_i \in\Ztr(\bcube^{(m,n)})(R,f).\]
We claim 
\eq{lem:Gm6}{\psi(W_a)=(a, mnN)\in R_f^\times\oplus \Z.}
Indeed, it suffices to show this after restriction to the generic point of $R$, 
in which case it follows directly from the definition of the $W_i$ and \eqref{lem:Gm3.5}.
This implies the surjectivity of $\psi$ and that of \eqref{lem:Gm3}.
Next, we show that \eqref{lem:Gm3} has a splitting. Let 
$\omega_a^{m,n}\in h_{0,\Nis}^{\bcube, \sp}(\bcube^{(m,n)})(R,f)$ be the class of $W_a$ and 
$\lambda_a^{m,n}=\omega_a^{m,n}-\omega_1^{m,n}$, where
$\omega_1^{m,n}$ is defined as $\omega_a^{m,n}$ replacing $a$ by $1$ (and using the same $N$).
By \eqref{lem:Gm6} the image of $\lambda_a^{m,n}$ under the map \eqref{lem:Gm3}:
\[ h_{0,\Nis}^{\bcube, \sp}(\bcube^{(m,n)})(R,f) \to R_f^\times\oplus \Z\]
is $(a,0)$. 

\begin{claim}\label{claim;lem:Gm}
$\lambda_a^{m,n}$ is independent of the choice of $N$, and we have
\eq{lem:Gm7}{\lambda_{ab}^{m,n} =\lambda_a^{m,n} + \lambda_b^{m,n}\quad\text{ for } a,b \in R_f^\times.}
Moreover,  for $m'\geq m$ and $n'\geq n$, the image of $\lambda^{m',n'}_a$ under
\[h_{0,\Nis}^{\bcube, \sp}(\bcube^{(m',n')}) \to
h_{0,\Nis}^{\bcube, \sp}(\bcube^{(m,n)})\]
coincides with $\lambda_a^{m,n}$.
\end{claim}

By the semipurity of $h_{0,\Nis}^{\bcube, \sp}(\bcube^{(m,n)})$ and \cite[Thm 3.1]{S-purity}, 
we have an injective homomorphism
\eq{lem:Gm8}{h_{0,\Nis}^{\bcube, \sp}(\bcube^{(m,n)})(R,f) \hookrightarrow 
\ulomega_! h_{0,\Nis}^{\bcube, \sp}(\bcube^{(m,n)})(K)=h_0(\bcube^{(m,n)})(K).}
By Lemma \ref{lem:h0P1} the isomorphism
\[\theta_{m,n}: h_0(\bcube^{(m,n)})(K)\xr{\simeq} ((K[t]/t^m)^\times \oplus (K[z]/z^n)^\times)/K^\times \oplus \Z\]
sends $\omega^{m,n}_a$ to 
\[\theta_{m,n}(\omega^{m,n}_a) =\left (\frac{(-1)^{mnN}a}{(t-1)^{mnN}}, \frac{1}{(1-z)^{mnN}}, mnN\right).\]
Thus $\theta_{m,n}(\lambda^{m,n}_a)= (a,1,0)$,  which is independent of $N$. 
By the injectivity of \eqref{lem:Gm8} this implies the first two assertions of the claim;
similarly the final assertion of the claim follows form the commutative diagram in Lemma \ref{lem:h0P1}.

Since $\lambda_a^{m,n}$ does not change if we replace $f$ by $uf$ with $u\in R^\times$,
the map $a\to \lambda_a^{m,n}$ glues to give a global  morphism of Nisnevich sheaves
which induces the splitting $s_{m,n}$ from the statement.
It remains to check that $s_{m,n}$ is compatible with transfers.
To this end it suffices to check that $\ul{\omega}_!(s_{m,n})$ is compatible with transfers
and since the transition maps are surjective it further suffices to show that
\[\varprojlim_{m,n} \ul{\omega}_!(s_{m,n}): \G_m\oplus \Z\to 
\varprojlim_{m,n} \ul{\omega_!}h^{\bcube, \sp}_{0,\Nis}(\bcube^{(m,n)})\]
is compatible with transfers. Since we can identify the target with
$\W\oplus\W\oplus\G_m\oplus\Z$ by Remark \ref{rmk:h0P1}, the compatibility
holds automatically by \cite[Lem 1.1]{MS}.
\end{proof}

\begin{prop}\label{prop:mu}
Denote by $\psi: \A^1_y\times \A^1_s\to \A^1_x\times \A^1_s$ the morphism induced
by the $k[s]$-algebra morphism $k[x,s]\to k[y,s]$, $x\mapsto ys$.
We denote  by the same symbol the induced morphism in $\ulMCor$
\eq{prop:mu1}{\psi: \bcubee_y\otimes \bcube^{(2)}_s\to \bcubee_x\otimes \bcubee_s.}
Let $F\in \CItspNis$ and $\sX\in \ulMCor$. Then $\psi^*$ factors as follows
\[\xymatrix{                                                 & F(\bcubee_y\otimes \bcubee_s\otimes \sX)\ar[d]\\
F(\bcubee_x\otimes \bcubee_s\otimes\sX)\ar[r]^{\psi^*}\ar@{.>}[ru] &  F(\bcubee_y\otimes \bcube^{(2)}_s\otimes\sX),
}\]
where the vertical map is induced by the natural morphism $\bcube^{(2)}_s\to \bcubee_s$.
\end{prop}
\begin{proof}
It is direct to check that $\psi$ induces a morphism \eqref{prop:mu1}.
To check the factorization statement, we may replace $F$ by $\uHom(\Ztr(\sX), F)$
to  reduce to the case $\sX=(\Spec k,\emptyset)$, see Lemma \ref{lem:Nis}\ref{lem:Nis1}.
 {By Yoneda and \eqref{para:CI1} }
we are reduced to show that we have a factorization as follows
\eq{prop:mu2}{\xymatrix{
       & h^{\bcube,\sp}_{0,\Nis}(\bcubee_y\otimes\bcubee_s)\ar@{.>}[dl]\\
h_{0,\Nis}^{\bcube, \sp}(\bcubee_x\otimes\bcubee_s) &  
h^{\bcube,\sp}_{0,\Nis}(\bcubee_y\otimes\bcube^{(2)}_s).\ar[l]_{\psi}\ar[u]_a
}}
By \cite[Lem 1.14(iii)]{MS}  and  Lemma \ref{lem:Gm}, the map $a$ is surjective.
Thus we have to show $\psi(\Ker a)=0$. 
By semipurity it suffices to show that we have a factorization as in \eqref{prop:mu2} after applying $\ulomega_!$.
By \cite[Prop 5.6]{RSY} we have 
 {
\[H:=\uomega_!(h_{0,\Nis}^{\bcube, \sp}(\bcubee_x\otimes\bcubee_s))=h_{0,\Nis}(\bcubee_x\otimes\bcubee_s)=\sK^M_2\oplus \G_m\oplus\G_m\oplus \Z,\]
where $\sK^M_2$ is the (improved) Milnor $K$-theory sheaf; in particular $H$ is $\A^1$-invariant.
Thus $\uomega_!(\psi)$ and $\uomega_!(a)$ factor via 
 {$h^{\A^1}_{0,\Nis}(h_{0,\Nis}(\bcubee_y\otimes\bcube^{(2)}_s))$ (cf. \eqref{para:tut3}). }
Thus we obtain solid arrows in $\NST$
\eq{prop:mu3}{\xymatrix{
    & H\ar@{.>}[dl]\\
H & h^{\A^1}_{0,\Nis}(h_{0,\Nis}(\bcubee_y\otimes\bcube^{(2)}_s)).\ar[l]^-{\bar{\psi}}\ar[u]_{\bar{a}}
}}
Since  $\bar{a}$  is the composition of the natural isomorphisms  { (cf. \eqref{para:tut4})}
\[ h^{\A^1}_{0,\Nis}(h_{0,\Nis}(\bcubee_y\otimes\bcube^{(2)}_s)) 
\cong h^{\A^1}_{0,\Nis}(\Ztr(\A^1_y\setminus\{0\})\otimes_{\PST} \Ztr(\A^1_s\setminus\{0\})) \cong H\]
the dotted arrow exists, which completes the proof.
}
\end{proof}

\begin{remark}\label{remark:mu}
Going through the definitions one can check that the map $H\to H$ induced by $\bar{\psi}$ in \eqref{prop:mu3}
is on a regular local ring $R$ given by 
\[ (\{a,b\}, c,d, n)\mapsto (\{a,b\}+ \{d,-1\}, cd, d, n),\]
where we use the identification $H(R)= K^M_2(R)\oplus R^\times\oplus R^\times\oplus \Z$.
\end{remark}

\subsection{Cohomology of a blow-up centered in the smooth part of the modulus}

 The goal of this Section is to prove Theorem \ref{thm:blow-upAn} below, 
 giving the invariance of the cohomology of cube invariant sheaves along a certain class of blow-ups. 
 This plays a fundamental role in what follows, and it is used in the proof of the $(\P^n, \P^{n-1})$-invariance 
 of the cohomology.

Recall the following definition from \cite[5]{S-purity}.
\begin{defn}Let $F\in \CItsp$. We define the the modulus presheaf $\sigma^{(n)}(F)$ by
    \[\sigma^{(n)}(F)(\sY) = \Coker (F(\sY)\xr{{\rm pr}^*} F(\sY\tensor (\P^1, n0+\infty))),\]
    where ${\rm pr}^*$ is the pullback along the projection ${\rm pr}\colon\sY\tensor (\P^1, n0+\infty) \to \sY$. Note that ${\rm pr}^*$ is split injective, with left inverse given by the inclusion  $i_1\colon\Spec k\inj \P^1$ of the 1-section. Hence we have an isomorphism, natural in $\sY$
    \[F(\sY\tensor (\P^1, n0+\infty)) \cong \sigma^{(n)}(F)(\sY)\oplus F(\sY). \]
   Following \cite[Def. 5.6]{S-purity}, we write $F^{(n)}_{-1}$ for $\sigma^{(n)}(F)$ when $F$ is moreover in $\uMNST$. Note that we have a natural identification
    \[F^{(n)}_{-1}=\uHom_{\uMPST}((\P^1, n\cdot 0+ \infty)/1, F) = F(-\tensor (\P^1, n0+\infty))/F(-),\]
  where $(\P^1, n\cdot 0+ \infty)/1= \Coker(\Ztr(\Spec k,\emptyset)\xr{i_1} \Ztr(\P^1, n\cdot 0+\infty))$ in $\uMPST$. 
 {By Lemma \ref{lem:Nis}\ref{lem:Nis1} we have } $F^{(n)}_{-1}\in \CItspNis$ if $F\in \CItspNis$, so that the association $F\mapsto F_{-1}^{(n)}$ gives an endofunctor of $\CItspNis$. This construction is the modulus version of Voevodsky's contraction functor, see \cite[p. 191]{MVW}.
\end{defn}
\begin{nota}\label{nota:MCorls}
We denote by $\ulMCorls$ the full subcategory of $\uMCor$ consisting of ``log smooth'' modulus pairs, i.e., objects
$\sX=(X,D)$, where $X\in \Sm$ and $|D|$ is a simple normal crossing divisor (in particular, each irreducible component of $|D|$ is a smooth divisor in $X$).
Note that $\otimes$ restricts to a monoidal structure on $\ulMCorls$.
\end{nota}
\begin{lemma}\label{lem:comp-SNCD}
Let $F\in \CItspNis$ and  $\sX=(X,D)\in\uMCorls$. 
Let $H\inj X$ be a smooth divisor, such that $|D|+H$ is SNCD, and denote by
$j: U:= X\setminus H\inj X$ the inclusion of the complement.
Then
\[R^i j_* F_{(U, D_{|U})}=0, \quad \text{for all }i\ge 1,\]
where $F_{(U, D_{|U})}$ denotes the Nisnevich sheaf on $U$ defined in \eqref{para:sheaf1}.
\end{lemma}
 {
\begin{proof}
This is an immediate consequence of \cite[Cor 8.6(3)]{S-purity}.
\end{proof}
}

\begin{lemma}\label{lem:higher-HI}
Let $F\in \CItspNis$ and $\sX=(X, D)\in\uMCorls$.
Let $E_i\subset \A^1$, $i=1,\ldots, n$, be effective (or empty) divisors and denote 
by $\pi:  \A^n_X\to X$ the projection.
Then 
\[R^i\pi_* (F_{(\A^1,E_1)\otimes\ldots \otimes(\A^1, E_n)\otimes \sX})=0,\quad \text{for all } i \ge 1, \,n\ge 0.\]
\end{lemma}
\begin{proof}
 {
First consider the case $n=1$. Set $E:=E_1$ and   $\bcube^{(E,r)}:=(\P^1, E+ r\cdot \infty)$, for $r\ge 1$. 
The natural morphism $\bcube^{(E,r)}\to \bcube$ induces a map $F_{\sX\otimes \bcube} \to F_{\sX\otimes\bcube^{(E,r)}}$.
The cohomology sheaves of the cone $C$ of this map are supported in $X\times |E+\infty|$,
whence $R^i\ol{\pi}_*C=0$, $i\ge 1$, where $\ol{\pi}:\P^1_X\to X$ is the projection.
We obtain surjections 
\[R^i\ol{\pi}_*F_{\sX\otimes\bcube}\to 
     R^i\ol{\pi}_*F_{\sX\otimes\bcube^{(E,r)}}\to 0,\quad \text{for all } i\ge 1.\]
By the cube-invariance of cohomology (see \cite[Thm 9.3]{S-purity}) the left term vanishes. 
Thus  $M$-reciprocity (see \cite[1.27(1)]{S-purity}) yields
\[0=\varinjlim_r R^i\ol{\pi}_*F_{\sX\otimes\bcube^{(E,r)}}= R^i\ol{\pi}_* j_*F_{(\A^1, E)\otimes\sX},\]
where $j:\A^1_X\inj \P^1_X$ is the open immersion.
Together with Lemma \ref{lem:comp-SNCD}  we obtain
\[R^i\ol{\pi}_* Rj^k_*F_{(\A^1,E)\otimes\sX}=0, \quad \text{for all } i\geq 1, k\geq 0.\]
Thus  the vanishing $R^i\pi_*F_{(\A^1,E)\otimes\sX}=0$  follows from the Leray spectral sequence.

The general case follows by induction, by factoring $\pi$ as $\A^n_X\xr{\pi_1} \A^{n-1}_X\xr{\pi_{n-1}} X$
and observing
\[\pi_{1*}(F_{(\A^1, E_1)\otimes \ldots\otimes(\A^1, E_n)\otimes\sX})= 
F_{1,  (\A^1, E_2)\otimes\ldots\otimes(\A^1, E_n)\otimes\sX},\]
where $F_1:=\uHom_{\uMPST}(\Ztr(\A^1,E_1), F)$ lies in $\CItspNis$ by
Lemma \ref{lem:Nis}\ref{lem:Nis1}.
}
\end{proof}

 {
\begin{para}\label{defn:ti}
We recall some standard terminology.
Let $(X, D)\in \uMCorls$, $Y\in \Sm$, and let $f: Y\to X$ be a  $k$-morphism of finite type.
We say {\em $D$ is transversal to $f$}, if  for any number of irreducible components $D_1, \ldots, D_r$
of the SNCD $|D|$, the morphism $f$ intersects the scheme-theoretic intersection
$D_1\cap\ldots \cap D_r$ transversally
(i.e., the scheme-theoretic inverse image $f^{-1}(D_1\cap \ldots \cap D_r)$ is smooth over $k$ and of codimension
$r$ in $Y$). Note that $f$ is always transversal to the empty divisor.

If $f$ is a closed immersion we also say { \em $Y$ and $D$ intersect transversally}.
Since $X$ is of finite type over a perfect field, this is equivalent to say, that  for any point $x\in Y\cap D$
we find a regular sequence of parameters $t_1,\ldots, t_n\in\sO_{X,x}$,
such that $\sO_{Y,x}=\sO_{X,x}/(t_1,\ldots, t_s)$ and the irreducible components of $|D|$ containing $x$
are in $\Spec \sO_{X,x}$ given by $V(t_{s+1}),\ldots, V(t_r)$, with $1\le s\le r\le n$. 
\end{para}
}

 {
\begin{thm}\label{thm:blow-upAn}
Let $F\in\CItspNis$ and $\sX=(X,D)\in \ulMCorls$.
Assume there  is a smooth irreducible component  $D_0$ of $|D|$ which has multiplicity 1 in $D$.
Let $Z\subset X$ be a smooth closed subscheme which is contained in $D_0$ and intersects $|D-D_0|$ transversally.
Let $\rho: Y\to X$ be the blow-up in $Z$. 
Then the natural map
\[F_{\sX}\xr{\simeq} R\rho_{*} F_{(Y, \rho^*D)}\]
is an isomorphism in the derived category of abelian Nisnevich sheaves on $X$.
\end{thm}
The proof is given in \ref{proof:thm:blow-upAn}.
The key point is to understand the case of the blow-up of $\A^2$ in the origin with $D_0$ a line,
 which is established in the next Lemma. Here, after some preliminary steps, we are reduced to prove the vanishing of the
 cohomology of the pushforward of $F$ along the projection from the blow-up to the exceptional divisor. 
This is where the modulus descent, i.e., Proposition \ref{prop:mu}, is crucially used. 

\begin{lemma}\label{lem:blow-upA2}
Let $F\in\CItspNis$ and $\sX=(X,D)\in\uMCorls$. Let $\rho: Y\to \A^2$ be the blow-up in the origin $0\in \A^2$
and let $L$ be a line containing $0$. Then
\[R^i\rho_{X*} F_{(Y,\rho^*L)\otimes \sX}=0,\quad \text{for all }i\ge 1,\]
where $\rho_X:=\rho\times\id_X: Y\times X\to \A^2\times X$ is the base change of $\rho$.
\end{lemma}
}
\begin{proof}
We can assume $X$ is henselian local 
 {and  
\[L=V(x)\subset \A^2=\Spec k[x,y].\]
}
 Set 
\[\sF:=F_{(Y, \rho^*L)\otimes \sX};\]
it is a Nisnevich sheaf on $Y\times X$. For $i\ge 1$, the higher direct images
 $R^i\rho_{X*}\sF$ are supported in $0\times X$ whence 
 \[H^j(\A^2_X, R^i\rho_{X*}\sF)=0, \quad \text{for all }i,j\ge 1,\]
 and
\[R^i\rho_{X*}\sF=0\Longleftrightarrow H^0(\A^2_X, R^i\rho_{X*}\sF)=0.\]
Furthermore, $\rho_{X*} \sF= F_{(\A^2, L)\otimes \sX}$, since $(Y, \rho^*L)\otimes \sX\cong (\A^2, L)\otimes \sX$ in
$\uMCor$ (see \cite[1]{KMSY1}). Hence by Lemma \ref{lem:higher-HI}
\[H^i(\A^2_X, \rho_{X*}\sF)=H^i(\A^2_X, F_{(\A^2, L)\otimes \sX})=0.\]
Thus the Leray spectral sequence yields
\begin{equation}\label{lem:blow-upA2-1}H^0(\A_X^2, R^i\rho_{X*}\sF)= H^i(Y\times X, \sF), \quad i\ge 0,\end{equation}
and we have to show, that this group vanishes for $i\ge 1$.
Write 
\[ Y=\Proj k[x,y][S,T]/(xT-yS)\subset \A^2\times \P^1,\]
and denote by 
\[\pi: Y\times X\inj \A^2\times \P^1_X \to \P^1_X=\Proj \sO_X[S,T]\]
the morphism induced by projection. 
In order to show that \eqref{lem:blow-upA2-1} vanishes, we can project along $\pi$ and use the Leray spectral sequence
\[ H^{i-j}(\P^1_X, R\pi^j_*\sF)\Rightarrow H^i(Y\times X, \sF)\]
to reduce the problem to showing that 
\eq{lem:blow-upA20}{H^i(\P^1_X, R^j\pi_*\sF)=0, \quad i\ge 1, j\geq 0.}
The terms $R^j\pi_* \sF$ for $j\geq 1$ are easy to handle using Lemma \ref{lem:higher-HI}. Indeed, 
set $s=S/T$  and write
\[ \P^1\setminus \{\infty\}= \A^1_s:= \Spec k[s],\quad \P^1\setminus\{0\}= \Spec k[\tfrac{1}{s}].\]
Set $U:=\A^1_s \times X$ and $V:=(\P^1\setminus\{0\})\times X$ and 
\[\sU:=(\A^1_s,0)\otimes\sX, \quad \sV:=(\P^1\setminus\{0\}) \otimes \sX.\]
We have 
\[\pi^{-1}(U)= \A^1_y\times U, 
 \quad \pi^{-1}(V)=\A^1_x\times  V,\]
and the restriction of $\pi$ to these open subsets is given by projection.
Furthermore by construction, 
\begin{equation}\label{eq;sFUV}
\sF_{|\pi^{-1}(U)}=F_{(\A^1_y,0)\otimes \sU},\quad 
\sF_{|\pi^{-1}(V)}= F_{(\A^1_x, 0)\otimes \sV}.
\end{equation}
Thus Lemma \ref{lem:higher-HI} (in the case $n=1$) yields
\[R^j\pi_*\sF=0, \quad j\ge 1.\]
It remains to show  
\eq{lem:blow-upA21}{H^i(\P^1_X, \pi_*\sF)=0, \quad i\ge 1.}
Set 
\begin{equation}\label{eq:F1semipure}F_1:=\uHom(\Ztr(\A^1_x, 0), F ).\end{equation}
    Note that $F_1\in \CItspNis$  by Lemma \ref{lem:Nis}\ref{lem:Nis1}.
Let  $j: V\inj \P^1_X$ be the open immersion. Its base change along $\pi$ induces a morphism 
\eq{lem:blow-upA22}{\iota: (\A^1_x,0)\otimes \sV\to (Y,\rho^*L)\otimes \sX \quad \text{in } \uMCor.}
This yields an exact sequence of Nisnevich sheaves on $\P^1_X$
\[0\to \pi_*\sF\xr{\pi_*(\iota^*)} j_*F_{1,\sV}\to \Gamma\to 0,\]
defining $\Gamma$; here the first map is injective by the semipurity of $F$.
Since $\Gamma$ is supported on $0\times X$ we obtain for $i\ge 2$
\begin{align*}
H^i(\P^1_X, \pi_*\sF) & = H^i(\P^1_X, j_*F_{1,\sV})\\
                                 &= H^i(V, F_{1,\sV}), & & \text{by Lem \ref{lem:comp-SNCD}},\\
                                 &=0, &&\text{by Lem \ref{lem:higher-HI}}.
\end{align*}
It remains to prove the vanishing \eqref{lem:blow-upA21} for $i=1$. This will occupy the rest of the proof.
Let
\[a\colon Y\times X\to \A^1_x\times \P^1\times X\]
be induced by the base change of the closed immersion $Y\inj \A^2\times \P^1$ followed by the 
base change of the projection $\A^2\to \A^1_x$.
The map $a$ induces a morphism
\eq{lem:blow-upA23}{\alpha: (Y,\rho^*L)\otimes \sX\to (\A^1_x, 0)\otimes \P^1_\sX\quad \text{in }\uMCor,}
where $\P^1_\sX:=\P^1\otimes \sX$, and 
which precomposed with $\iota$ from \eqref{lem:blow-upA22} yields the morphism
\eq{lem:blow-upA24}{\alpha\iota: (\A^1_x,0)\otimes \sV\to (\A^1_x, 0)\otimes \P^1_\sX,}
induced by the open immersion $\A^1_x\times V\inj \A^1_x\times \P^1_X$. This gives a factorization
\[\begin{tikzcd} F_{1, \P^1_{\sX}} \arrow[r, "\pi_*(\alpha^*)"] \arrow[rd] & \pi_*\sF \arrow[d, "\pi_*(\iota^*)"] \\
 & j_* F_{1, \sV}
\end{tikzcd}
\]
where the diagonal morphism is injective,  {by \cite[Thm 3.1(2)]{S-purity} and the semipurity of $F_1$.}
This implies that the morphism labeled $\pi_*(\alpha^*)$ is injective too. 
 {
Similarly, 
the embedding $\sV \to (\P^1, 0)\tensor\sX$, induces another injective morphism 
$F_{1, (\P^1, 0)\tensor \sX} \to j_*F_{1,\sV}$.
}
In total, we obtain the following commutative diagram
\eq{lem:blow-upA25}{\xymatrix{
0\ar[r] & F_{1, \P^1_\sX}\ar[rr]^{\pi_*(\alpha^*)}\ar@{=}[d] 
           &  &\pi_*\sF\ar[r]\ar[d]^{\pi_*(\iota^*)} 
           &  \ul{\Sigma}\ar[d]^\varphi\ar[r] & 0\\
0\ar[r] & F_{1, \P^1_\sX}\ar[rr]^{\pi_*((\alpha\iota)^*)}\ar@{=}[d] 
           & & j_*F_{1,\sV}\ar[r]
           & \ul{\Lambda}\ar[r] & 0\\
0\ar[r] & F_{1, \P^1_\sX}\ar[rr]
           & & F_{1, (\P^1, 0)\tensor \sX}\ar[r]\ar@{^{(}->}[u]
           & \ul{\Lambda}(0)\ar[r]\ar@{^{(}->}[u] & 0           ,
}}
with exact rows, defining the cokernels $\ul{\Sigma}$, $\ul{\Lambda}$ and $\ul{\Lambda}(0)$, as well as the  map $\varphi$. 
Applying $R\Gamma(\P^1_X,-)$ yields
\[\xymatrix{
\Sigma\ar[d]^\varphi \ar[r]^-{\partial_1} & H^1(\P^1_X, F_{1,\P^1_\sX})\ar[r] \ar@{=}[d]& H^1(\P^1_X, \pi_*\sF)\ar[d]\to 0\\
\Lambda  \ar[r]^-{\partial_2}  &  H^1(\P^1_X, F_{1,\P^1_\sX}) \ar[r] \ar@{=}[d]& H^1(\P^1_X, j_* F_{1, \sV})\to 0 \\
\Lambda(0) \ar@{^{(}->}[u] \ar@{->>}[r] & H^1(\P^1_X, F_{1,\P^1_\sX}) \ar[r]& H^1(\P^1_X, F_{1, (\P^1, 0)\tensor \sX}) =0 \ar[u] 
}\]
with exact rows and in which the $\partial_i$ are the connecting homomorphisms
and where 
\[\Sigma:= H^0(\P^1_X, \ul{\Sigma}), \quad \Lambda:= H^0(\P^1_X, \ul{\Lambda}), \quad \Lambda(0):=H^0(\P^1_X, \ul{\Lambda}(0)).\] 
The group $H^1(\P^1, F_{1, (\P^1, 0)\tensor \sX})$ vanishes by the cube invariance of cohomology, see \cite[Thm 9.3]{S-purity}, thus $\partial_{2|\Lambda(0)}$ is surjective,  the vanishing \eqref{lem:blow-upA21} for $i=1$ will follow, if we can show
\eq{lem:blow-upA26}{\Lambda(0)\subset \varphi(\Sigma).}
Note that $\ul{\Sigma}$, $\ul{\Lambda}$, and $\ul{\Lambda}(0)$ have support in $0\times X\subset U$, so we can compute the global sections on $U$ instead of $\P^1$ to show \eqref{lem:blow-upA26}. Now,  since $H^1(U, F_{1,\sU})=0$, by Lemma \ref{lem:higher-HI}, unravelling the definitions we obtain 
from \eqref{eq;sFUV} and \eqref{lem:blow-upA25} with $G:=F(-\otimes\sX)$ and 
$\A^1_s=\P^1\setminus\{\infty\}$ the following descriptions:
\[\Sigma= 
\frac{G((\A^1_y,0)\otimes (\A^1_s,0))}
      {\alpha^* G((\A^1_x,0)\otimes \A^1_s)},
\]
\[\Lambda= \frac{G((\A^1_x,0)\otimes (\A^1_s\setminus\{0\},\emptyset))}
      { G((\A^1_x,0)\otimes \A^1_s)},
\]
\begin{equation}\label{eq:lambda0}\Lambda(0)= \frac{G((\A^1_x,0)\otimes (\A^1_s,0))}
      { G((\A^1_x,0)\otimes \A^1_s)}.
\end{equation}
By \cite[Lem 5.9]{S-purity} we have isomorphisms  (see Notation \ref{bcube1})
\begin{equation}\label{eq:lem-blowupA2-pur1}\frac{G(\bcubee_x\otimes (\A^1_s,0))}{G((\P^1_x,\infty)\otimes(\A^1_s,0))}
\xr{\simeq}
   \frac{G((\A^1_x,0)\otimes (\A^1_s,0))}{G(\A^1_x\otimes(\A^1_s,0))},\end{equation}
\begin{equation}\label{eq:lem-blowupA2-pur2}\frac{G(\bcubee_x\otimes \bcubee_s)}{G(\bcubee_x\otimes(\P^1_s,\infty))}
\xr{\simeq}
   \frac{G(\bcubee_x\otimes (\A^1_s,0))}{G(\bcubee_x\otimes \A^1_s)}.\end{equation}
Write $j$ for the open immersion $(\A^1_x,0)  \hookrightarrow \bcubee_x$. The base change of $j^*$ induces a commutative diagram
\begin{equation}\label{eq:lem-blowupA2-factor}\begin{tikzcd}
G((\A^1_x,0)\otimes \A^1_s) \arrow[r] \arrow[rr, bend left]& G((\A^1_x,0)\otimes (\A^1_s,0)) \arrow[r, two heads] & \Lambda(0)\\
G(\bcubee_x\otimes \A^1_s) \arrow[r] \arrow[u, "j^*"] & G(\bcubee_x\otimes (\A^1_s,0)). \arrow[u, "j^*"]  \arrow[ru]
\end{tikzcd}
\end{equation}
The horizontal composite morphism is zero by \eqref{eq:lambda0}, hence the kernel of the diagonal arrow contains $G(\bcubee_x\otimes \A^1_s)$. Next, note that from \eqref{eq:lem-blowupA2-pur1} we get the surjective morphism 
\begin{equation}\label{eq:lem-blowupA2-MV}
 G(\bcubee_x\otimes (\A^1_s,0)) \oplus G(\A^1_x \otimes (\A^1_s, 0))  \to G( (\A^1_x,0) \otimes (\A^1_s, 0))\to 0.
\end{equation}
Combining \eqref{eq:lem-blowupA2-MV}, \eqref{eq:lem-blowupA2-pur2} and \eqref{eq:lem-blowupA2-factor} we get a surjection
\eq{lem:blow-upA27}{G(\A^1_x\otimes(\A^1_s,0))\oplus G(\bcubee_x\otimes \bcubee_s)\surj \Lambda(0).}
Note that 
the pullback of the open immersion $\pi^{-1}(V)\inj  Y\times X$ along $\pi^{-1}(U)\inj Y\times X$
 induces the open immersion
\[\A^1_x\times (\A^1_s\setminus\{0\})\times X\to \A^1_y\times \A^1_s \times X\]
which is induced by base change from the $k[s]$-linear map
\[k[y, s]\mapsto k[x, s, 1/s], \quad y\mapsto x/s.\]
It gives the following two morphisms in $\uMCor$
\[\iota_1: (\A^1_x,0)\otimes (\A^1_s\setminus\{0\})\otimes\sX\to (\A^1_y,0)\otimes(\A^1_s,0)\otimes \sX.\]
\[\iota_2: (\A^1_x,0)\otimes (\A^1_s\setminus \{0\})\otimes \sX\to \bcubee_y\otimes\bcube^{(2)}_s\otimes \sX.\]
Furthermore, consider the base change of the map \eqref{prop:mu1} 
\[\psi: \bcubee_y\otimes\bcube^{(2)}_s\otimes \sX\to \bcubee_x\otimes\bcubee_s\otimes\sX,\]
which is induced by $x\mapsto ys$; it restricts to
\[\psi_1: (\A^1_y,0)\otimes(\A^1_s,0)\to\A^1_x\otimes (\A^1_s,0).\]
In particular, $\psi_1\circ\iota_1$ is induced by the open immersion 
$\A^1_x\setminus\{0\}\times\A^1_s\setminus\{0\}\inj \A^1_x\times\A^1_s\setminus \{0\}$
and $\psi\circ\iota_2$ is induced by the identity on 
$\A^1_x\setminus\{0\}\times\A^1_s\setminus\{0\}$.
Consider the following diagram
\[
\xymatrix{ 
G(\A^1_x\otimes(\A^1_s,0))\ar[r]^{\psi_1^*}
         &G((\A^1_y,0)\otimes (\A^1_s,0))\ar[d]^{\iota_1^*}\ar[r]^-{r_1} &  \Sigma\ar[d]^\varphi\\
         &G((\A^1_x, 0)\otimes (\A^1_s\setminus\{0\},\emptyset))\ar[r]^-{r_2}& \Lambda\\
&G(\A^1_x\otimes(\A^1_s,0))\ar[r]^-{r_3}
      \ar@/^2pc/@{=}[uul]\ar@{^(->}[u]& \Lambda(0).\ar@{^(->}[u]
}\]
Here the maps $r_i$ are the natural maps into the quotients;
 the diagram commutes by definition of  the morphisms involved.
Hence 
\eq{lem:blow-upA28}{\Im(G(\A^1_x\otimes (\A^1_s,0))\to \Lambda(0))\subset \varphi(\Sigma).}
Consider now the following diagram
\[
\xymatrix{ 
       & G(\bcubee_y\otimes\bcubee_s)\ar[r]^-{r_1}\ar[d] & \Sigma\ar[dd]^\varphi\\
G(\bcubee_x\otimes\bcubee_s)\ar[r]^{\psi^*}\ar@{.>}[ur]
         &G(\bcubee_y\otimes\bcube^{(2)}_s)\ar[d]^{\iota_2^*} & \\
         &G((\A^1_x, 0)\otimes (\A^1_s\setminus\{0\},\emptyset))\ar[r]^-{r_2}& \Lambda\\
&G(\bcubee_x\otimes\bcubee_s)\ar[r]^-{r_3}
      \ar@/^2pc/@{=}[uul]\ar@{^(->}[u]& \Lambda(0).\ar@{^(->}[u]
}\]
Here the maps $r_1$ and $r_3$ are induced by restriction followed by the quotient map using \eqref{eq:lem-blowupA2-pur1} and \eqref{eq:lem-blowupA2-pur2}; 
the two squares and the triangle on the lower left commute by definition of the morphisms involved;
the map $\psi^*$ factors via the dotted arrow in the diagram, by Proposition \ref{prop:mu}.
This shows 
\[\Im(G(\bcubee_x\otimes \bcubee_s)\to \Lambda(0))\subset \varphi(\Sigma),\]
which together with \eqref{lem:blow-upA28} and \eqref{lem:blow-upA27} 
implies \eqref{lem:blow-upA26}. This completes the proof of the lemma.
\end{proof}

 {
\begin{lemma}\label{lem:blow-upInd}
Let the assumptions and notations be as in Theorem \ref{thm:blow-upAn}. Assume additionally $\codim(Z,X)\le 2$.
Then Theorem \ref{thm:blow-upAn} holds.
\end{lemma}
\begin{proof}
There is nothing to prove for $\codim(Z,X)=1$, we therefore consider the case $\codim(Z,X)=2$.
Since $(Y,\rho^*D)\cong (X,D)$ in $\uMCor$ we have 
$\rho_*F_{(Y,\rho^*D)}\cong F_{(X,D)}$. 
Thus it remains to show the vanishing
\eq{lem:blow-upInd1}{ R^i\rho_*F_{(Y,\rho^*D)}=0,\quad \text{for all } i\ge 1.}
The question is Nisnevich local around the points in $Z$.
Let $z\in Z$ be a point and consider the regular henselian local ring $A=\sO_{X,z}^h$.
For $V\subset X$ set $V_{(z)}:= V\times_{ X} \Spec A$.
Denote by $D'\subset X$ the closed subscheme defined by $D-D_0$.
By assumption we find a regular system of local parameters $x,y,t_1\ldots, t_s$ of $A$,
such that  $Z_{(z)}=V(x,y)$, $D_{0, (z)}=V(x)$,  and 
$D'_{(z)}=V(t_{1}^{n_{1}}\cdots t_{r}^{n_r})$, for some $r\le s$ and  $n_i\ge 1$.
Let $K\inj A$ be a coefficient field over $k$; we obtain an isomorphism
\[K\{x,y, t_1,\ldots, t_s\}\xr{\simeq} A.\]
Let $\rho_1:\widetilde{\A^2}\to \A^2$ be the blow-up in $0$.
By the above  the blow-up in $Z$
\[\rho: (Y,\rho^*D)\to (X, D)\]
is Nisnevich locally around $z$ over $k$ isomorphic to the morphism
\[(\widetilde{\A^2},\rho_1^*(x))\otimes (\A_K^{s}, (\prod_{i=1}^{r} t_i^{n_i})) 
\to (\A^2, (x))\otimes (\A_K^{s}, (\prod_{i=1}^{r} t_i^{n_i})), \]
which is induced by base change from $\rho_1$.
Hence the vanishing \eqref{lem:blow-upInd1} follows from Lemma \ref{lem:blow-upA2}.
\end{proof}
}

\begin{lemma}\label{lem:it-blowup}
Let $X$ be a finite type $k$-scheme and $Z_0\subset Z_1\subset X$ closed subschemes.
Let $\rho: X'\to X$ be the blow-up of $X$ in $Z_0$ and let $\rho': X''\to X'$ be the blow-up
of $X'$ in the strict transform $\tilde{Z}_1$ of $Z_1$.
Furthermore,  let $\sigma: Y'\to X$ be the blow-up in $Z_1$ and let $\sigma': Y''\to Y'$
be the blow-up of $Y'$ in $\sigma^{-1}(Z_0)$. Then there is an isomorphism 
\[\xymatrix{
X''\ar[rr]^{\simeq}\ar[dr]_{\rho\rho'} & & Y''\ar[dl]^{\sigma\sigma'}\\
   & X. &
}\]
\end{lemma}
\begin{proof}
Recall the following general fact: Let $\sI, \sJ\subset \sO_X$ be two coherent ideal sheaves.
Then the blow-up $\tilde{X}\to X$ of $X$ in $\sI\cdot \sJ$ is equal to the composition 
$X_2\xr{\pi_2} X_1\xr{\pi_1}X$, where $\pi_1$ is the blow-up in $\sI$ and $\pi_2$ is the blow-up in
$\pi_1^{-1}\sJ\cdot \sO_{X_1}$. This is proven using the universal property of blow-ups, see, e.g.,
\cite[\href{https://stacks.math.columbia.edu/tag/080A}{Tag 080A}]{stacks-project}.

Here denote by $\sI_i\subset \sO_X$ the ideal sheaves of $Z_i$. We have $\sI_1\subset \sI_0$.
Let $\pi:\tilde{X}\to X$ be the blow-up of $X$ in $\sI_1\cdot \sI_0$.
By the remark above, $\pi$ is isomorphic as $X$-scheme to $\sigma\sigma'$.
Furthermore, note that $\rho'$ is also equal to the blow-up of $X'$ in $\rho^{-1}(Z_1)$.
Indeed, the ideal sheaf of $\rho^{-1}(Z_1)$ is equal to 
$\rho^{-1}\sI_1\cdot \sO_{X'}= \sI_E\cdot \tilde{\sI}_1$, where $\sI_E$ is the  ideal sheaf of the exceptional divisor of
$\rho$ and $\tilde{\sI}_1$ is the ideal sheaf of $\tilde{Z}_1$; 
since $\sI_E$ is invertible, the blow-ups of $X'$ in $\tilde{\sI}_1$ and in $\rho^{-1}\sI_1\cdot\sO_{X'}$
are isomorphic. Thus by the remark above the $X$-scheme $\rho\rho'$ is isomorphic to $\pi$ as well.
\end{proof}

 {
\begin{para}\label{proof:thm:blow-upAn}{\em Proof of Theorem {\ref{thm:blow-upAn}}.}
The proof is by induction on $c=\codim(Z,X)$, the induction start for $c\le 2$ being Lemma \ref{lem:blow-upInd}.
Assume $c>2$. The question is local on $X$. Hence we can assume $X=\Spec A$ and  that there is a regular sequence 
$y_1,\ldots, y_c, t_1,\ldots, t_r\in A$ such that $Z=V(y_1,\ldots, y_c)$, $D_0=V(y_1)$ and $D'=D-D_0=V(t_1^{n_1}\cdots t_r^{n_r})$, for
some $n_i\ge 1$. Set $Z_2:=V(y_1, y_2)$. Let $\rho: Y\to X$ be the blow-up in $Z$ and denote by $\tilde{Z}_2$ the strict transform of $Z_2$. 
Then $\rho^*D$ has SNC support with the strict transform $\tilde{D}_0$ of $D_0$ being  a smooth component containing $\tilde{Z}_2$.
Furthermore,  $\tilde{Z}_2$ intersects $\rho^*D-\tilde{D}_0$ transversally and $\codim(\tilde{Z}_2, Y)=2$. 
Let  $\rho': Y'\to Y$ be the blow-up in $\tilde{Z}_2$. By Lemma \ref{lem:blow-upInd} we find
\eq{thm:blow-upAn1}{R\rho_* F_{(Y,\rho^*D)}\cong R(\rho\rho')_* F_{(Y', (\rho\rho')^* D)}.}

Let $\sigma:W\to X$ be the blow-up in $Z_2$ and set $Z_{c-1}:=\sigma^{-1}(Z)$.
Then $\sigma^* D$ has SNC support with  the exceptional divisor $E$ being a smooth component containing $Z_{c-1}$.
Furthermore $Z_{c-1}$ intersects the strict transform of $D$ transversally  and  $\codim(Z_{c-1}, W)=c-1$. 
Let $\sigma':W'\to W$ be the blow-up  in $Z_{c-1}$. By Lemma \ref{lem:blow-upInd} and induction we find
\[F_{(X, D)}\cong R\sigma_* F_{(W,\sigma^*D)}\cong R(\sigma\sigma')_*F_{W',(\sigma\sigma')^*D}.\]
Thus the statement follows from Lemma \ref{lem:it-blowup} and \eqref{thm:blow-upAn1}.\qed
\end{para}
}

\subsection{\texorpdfstring{$(\P^n,\P^{n-1})$}{Pn, Pn-1}-invariance of cohomology}

 {
We follow the basic strategy of \cite[Lem 10]{KS}. See also \cite[Prop 7.3.1]{logmot}.

\begin{lemma}\label{lem:Pn-blow-up}
Let $F\in \CItspNis$ and $\sX=(X,D)\in \uMCorls$.
Let $x\in \P^n$ be a $k$-rational point and $L\subset \P^n$ a hyperplane.
Denote by $\rho: Y\to \P^n$ the blow-up in $x$. 
Denote by $q: Y\times X\to E\times X$ the base change 
of the morphism $Y\to E$ which parametrizes the lines in $\P^n$ through $x$.
Then the pullback
\[q^*: F_{(E, L')\otimes \sX}\xr{\simeq} R q_* F_{(Y, \rho^*L)\otimes\sX}\]
is an isomorphism,  where $L'= \tilde{L}\cap E$,
with $\tilde{L}\subset Y$ the strict transform of $L$.
(Note $L'=\emptyset$, if $x\not\in L$.)
\end{lemma}
\begin{proof}
Note that the projection morphism $Y\to E$ makes $Y$ into a $\P^1$-bundle over $E$ and induces a morphism
$(Y, \rho^*L)\otimes \sX\to (E, L')\otimes \sX$. 
The latter morphism locally over $E$ has the form  of the projection $\bcube\otimes \sW\to \sW$, for some $\sW\in \uMCor$.
Indeed, over an affine neighborhood $U\subset E$ intersecting (resp. not intersecting) $L'$, 
the modulus pair $\sW$ can be taken to be $(U,L'\cap U)\otimes\sX$ (resp. $(U,\emptyset)\otimes \sX$).
In both cases the divisor $\{\infty\}\times U\times X$ on $\P^1\times U\times X$ is the restriction of the exceptional divisor to $q^{-1}(U)= \P^1\times U\times X$.
Thus the statement follows from the cube-invariance of cohomology, see \cite[Thm 9.3]{S-purity}.
\end{proof}

\begin{thm}\label{thm:PnInv}
Let $F\in \CItspNis$. Let $L\subset \P^n$ be a hyperplane and $\sX=(X,D)\in\uMCorls$.
Then the pullback 
\[ F_\sX\xr{\simeq} R\pi_* F_{(\P^n, L)\otimes \sX},\]
along the projection $\pi: \P^n_X\to X$ is an isomorphism.
\end{thm}
\begin{proof}
The case $n=1$ is \cite[Thm 9.3]{S-purity}.
Assume $n\ge 2$.  Let $x\in \P^n$ be a $k$-rational point, $L\subset \P^n$ a hyperplane with $x\in L$,
and $\rho: Y\to \P^n$ the blow-up in $x$. 
Then $R\rho_* F_{(Y,\rho^*L)\otimes \sX}= F_{(\P^n,L)\otimes\sX}$ by 
Theorem \ref{thm:blow-upAn}. 
Thus the statement follows  from  Lemma \ref{lem:Pn-blow-up} and induction.
\end{proof}
}

\begin{cor}\label{cor:PnInv}
Let $F\in \CItspNis$ and $\sX=(X,D)\in \uMCorls$.
Let $V$ be a vector bundle on $X$ and denote by 
\[\pi: \P(V):=\Proj(\Sym^\bullet_{\sO_X}(V))\to X\]
the structure map. Then $\pi^*$ induces an isomorphism
 {
\[\pi^*:  F_{\sX}\xr{\simeq} \pi_* F_{(\P(V), \pi^*D)}.\]
}
\end{cor}
\begin{proof}
The question is local on $X$, hence we can assume that $V$ is trivial of rank $n+1$.
Let $L\subset \P^n$ be a hyperplane and consider
 {
\[F_{\sX}\xr{\pi^*} \pi_*F_{\P^n\otimes \sX}\inj \pi_*F_{(\P^n,L)\otimes \sX}.\]
The second map is injective by semipurity and \cite[Thm 3.1(2)]{S-purity};}
 the composition is an isomorphism by Theorem \ref{thm:PnInv}, hence
so is the first map.
\end{proof}

\section{Smooth blow-up formula}\label{sec:smblowup}

\begin{thm}\label{thm:bus}
Let $F\in \CItspNis$ and $\sX=(X,D)\in \uMCorls$.
Let $Z\subset X$ be a smooth closed subscheme which intersects $D$ transversally.
Consider the following cartesian diagram
\[\xymatrix{
E\ar[d]_{\rho_E}\ar[r]^{i_E} & \tilde{X}\ar[d]^{\rho}\\
Z\ar[r]^{i}   & X,
}\]
in which $\rho$ is the blow-up of $X$ along $Z$. Set
\[\tilde{\sX}=(\tilde{X},  D_{|\tilde{X}}), \quad \sZ=(Z, D_{|Z}), \quad \sE=(E, D_{|E}).\]
Then there is a distinguished triangle in the bounded derived category of Nisnevich sheaves of abelian groups
$D^b(X_{\Nis})$
\[F_{\sX}\xr{\rho^*\oplus (- i^*)} R\rho_* F_{\tilde{\sX}}\oplus i_*F_{\sZ}
\xr{i_E^*+ \,\rho_E^{*}} i_*R\rho_{E*} F_{\sE}
  \to F_{\sX}[1].\]
\end{thm}
\begin{proof}  {The first part of the argument is similar to the proof of \cite[IV.1.1.]{gros}.}
 We have to show that the diagram
\[ \begin{tikzcd}
F_{\sX} \arrow[r, "\rho^*"] \arrow[d, "i^*"] & R\rho_* F_{\tilde{\sX}} \arrow[d, "i_E^*"] \\
i_*F_{\sZ} \arrow[r, " \rho_E^*"] & i_*R\rho_{E*} F_{\sE}
\end{tikzcd}
\]
is homotopy cartesian in $D^b(X_{\Nis})$. To this end it suffices to show that the following maps are isomorphisms:
\eq{thm:bus1}{\rho_E^*: F_{\sZ}\to \rho_{E*}F_{\sE},}
\eq{thm:bus2}{ \rho^*:F_{\sX}\to \rho_* F_{\tilde{\sX}},}
\eq{thm:bus3}{i_E^*: R^j\rho_* F_{\tilde{\sX}}\to i_*R^j\rho_{E*} F_{\sE}, \quad j\ge 1.}
The map \eqref{thm:bus1} is an isomorphism by Corollary \ref{cor:PnInv}, since $E$ is a projective bundle over $Z$.
The question for the other two isomorphisms is Nisnevich local.
Since Z and D intersect transversally we can assume that
 $\sX= (\A^n, \emptyset)\otimes \sZ$ with $\sZ=(Z, D_Z)\in \uMCorls$
and that $\tilde{X}$ is the blow up of $X=\A^n\times Z$ at $\{0\}\times Z$ 
(cf. the proof of Lemma \ref{lem:blow-upInd}).
Write  $\A^n=\P^n\setminus L$ and let $Y$ be the blow-up of $0\in \P^n$ and denote by
$E_0$ the exceptional divisor.  Note that $L$ is embedded isomorphically into $Y$, not intersecting $E_0$.
 {
We obtain the diagram
\[\xymatrix{
E_0\times Z\ar[r]^{\bar{\imath}_E}\ar[d]_{\bar{\rho}_E} & Y\times Z\ar[r]^q\ar[d]^{\bar{\rho}} & E_0\times Z\ar[d]^{\bar{\rho}_E}\\
Z\ar[r]^{\bar{\imath}} & \P^n\times Z\ar[r]^{\pi} & Z,
}\]
where $\bar{\imath}: Z= \{0\}\times Z\inj \P^n\times Z$ is the closed immersion, 
$\bar{\rho}$ is the base change of the blow-up, $\pi$ and $\bar{\rho}_E$ are the projections and 
$q$ is as in Lemma \ref{lem:Pn-blow-up}.
It remains to show that the following maps are isomorphisms
\eq{thm:bus5}{ \bar{\rho}^*: F_{(\P^n, L)\otimes\sZ}\to
                                             \bar{\rho}_* F_{(Y, L)\otimes \sZ},}
\eq{thm:bus6}{\bar{\imath}_E^*: R^j\bar{\rho}_*F_{(Y,L)\otimes \sZ}\to
                                         \bar{\imath}_{*}R^j \bar{\rho}_{E*} F_{E_0\otimes \sZ}, \quad j\ge 1.}
Indeed, the restriction  of these two isomorphisms to $\A^n_Z=\P^n_Z\setminus L_Z$ yields the isomorphisms 
\eqref{thm:bus2} and \eqref{thm:bus3}. 

The map \eqref{thm:bus5} is an isomorphism away from $0\times Z$.
Since source and target of $R\pi_*\eqref{thm:bus5}$ are both isomorphic to $F_{\sZ}$ 
by Theorem \ref{thm:PnInv} and Lemma \ref{lem:Pn-blow-up}, \eqref{thm:bus5} is an isomorphism everywhere.
Similarly, \eqref{thm:bus6} is an isomorphism if $R\pi_*\eqref{thm:bus6}$ is. 
To show the latter, first observe that we have 
\eq{thm:bus7}{R^a\pi_*R^b\bar{\rho}_*(F_{(Y,L)\otimes Z})=0, \quad \text{for } a\neq 0.}
Indeed, if $b\ge 1$, then $R^b\bar{\rho}_{*}F_{(Y, L)\otimes \sZ}$ has support in $0\times Z$;
if $b=0$ the cohomology for $a\ge 1$ vanishes by \eqref{thm:bus5} and Theorem \ref{thm:PnInv}.
Now $R\pi_*\eqref{thm:bus6}$ is equal to the composition 
\[\pi_*R^j\bar{\rho}_*F_{(Y,L)\otimes \sZ}\cong R^j \bar{\rho}_{E*} Rq_*F_{(Y,L)\otimes \sZ} \cong R^j \bar{\rho}_{E*} F_{E_0\otimes \sZ},\]
where the first isomorphism follows from \eqref{thm:bus7} and the Leray spectral sequence and the second isomorphism holds by Lemma \ref{lem:Pn-blow-up}.
This completes the proof.
}
\end{proof}

\section{Twists}\label{sec:twists}

\subsection{A tensor formula for homotopy invariant sheaves}

\begin{lemma}(Bloch-Gieseker)\label{lem:rootMpair}
Assume $k$ infinite of exponential characteristic $p\ge 1$.
Let $X$ be an integral quasi-projective $k$-scheme and $D$ a Cartier divisor on $X$.
Let $n\ge 1$ be an integer with $(n,p)=1$.
Then there exists a finite and surjective morphism $\pi: Y\to X$ and a Cartier divisor $E$ on $Y$
such that the following properties hold:
\begin{enumerate}[label= (\arabic*)]
\item\label{lem:rootMpair1} $Y$ is integral, normal and $\pi^{-1}(X_{\rm sm})$ is a smooth open subscheme of $Y$,
where $X_{\rm sm}$ is the smooth locus of $X$;
\item\label{lem:rootMpair2} $\pi^* D= n E$;
\item\label{lem:rootMpair3} $\deg(\pi)$ divides a power of $n$;
\item\label{lem:rootMpair4} if $D$ is effective, then so is $E$.
\end{enumerate}
\end{lemma}
\begin{proof}
The proof is a slight modification of \cite[Lem 2.1]{BG}.
First note that \ref{lem:rootMpair4} follows from \ref{lem:rootMpair2} and \ref{lem:rootMpair1}.
Also, it suffices to prove the statement for $D$ a very ample divisor. 
Let $i: X\inj \P^N:=\P$ be an immersion such that $\sO(D)=i^*\sO_{\P}(1)$.
By Bertini's theorem (see, e.g., \cite[I, Cor 6.11]{J})  we find hyperplanes
$H_0,\ldots, H_N\subset \P$ such that all the intersections $H_{i_0}\cap\ldots \cap H_{i_r}$
and $H_{i_0}\cap\ldots \cap H_{i_r}\cap X_{\rm sm}$ are transversal (or empty), for all 
$\{i_0,\ldots, i_r\}\subset \{0,\ldots, N\}$ and all $0\le r\le N$.
Let $Y_i$ be a linear polynomial defining $H_i$, so that
$\P=\Proj k[Y_0,\ldots, Y_N]$. 
Let $\Pi: \P\to \P$ be the $k$-morphism defined by $Y_i\mapsto Y_i^n$, $i=0,\ldots, N$.
Note that $\Pi$ is finite of degree $n^N$ and it is \'etale over $\P\setminus \cup_i H_i$.
Form the cartesian diagram
\[\xymatrix{ X'\ar[r]^{i'}\ar[d]_{\pi'} & \P\ar[d]^\Pi\\ X\ar[r]^i & \P.}\]
Then $X'\times_X X_{\rm sm}$ is smooth: this can be checked after base change to the algebraic closure of $k$ and
then the argument is the same as in the second and third paragraph in the proof of \cite[Lem 2.1]{BG}
(the choice of the $H_i$ is crucial here).
Let $X''\subset X'$ be an irreducible component (with reduced scheme structure) and denote by $Y$ the normalization
of $X''$ and by $\pi: Y\to X$ the composition
\[Y\to X''\inj X'\xr{\pi'} X\]
and by $E=\sO_\P(1)_{|Y}$ the pullback of $\sO_\P(1)$ along 
\[Y\to X''\inj X'\stackrel{i'}{\inj} \P.\]
Then $\pi: Y\to X$ and $E$ satisfy the conditions of the statement.
\end{proof}

\begin{lemma}\label{lem:surjT}
Let $F,G\in \PST$. 
Let 
\eq{lem:surjT1}{\uomega^*F\otimes_{\uMPST}\uomega^* G\to \uomega^*(F\otimes_{\PST} G)}
be the morphism in $\uMPST$, which is induced by adjunction  from the  isomorphism
\[\uomega_!(\uomega^*F\otimes_{\uMPST}\uomega^* G)\cong 
(\uomega_!\uomega^*F)\otimes_{\PST} (\uomega_!\uomega^*G)\cong F\otimes_{\PST} G.\]
Then we obtain a surjection in $\uMNST$
\[\ul{a}_{\Nis}(\eqref{lem:surjT1}): \ul{a}_{\Nis}(\uomega^*F\otimes_{\uMPST}\uomega^* G)
                     \surj \ul{a}_{\Nis}(\uomega^*(F\otimes_{\PST} G)).\]
\end{lemma}
\begin{proof}
Denote by $H_l$ (resp. $H_r$) the source (resp. target) of \eqref{lem:surjT1} and take $\sX=(X,D)\in \uMCor$.  
By definition, $\otimes_{\uMPST}$ (resp. $\otimes_{\PST}$) is the Day convolution of the tensor product on $\uMCor$ (resp. on $\Cor$), so that we have the following presentations
(which also hold for general $(X,D)$, cf. \cite[\S 2]{SV00})
\[H_l(X,D)= 
\left(\bigoplus_{\sY, \sZ\in \uMCor} F(\sY^o)\otimes_{\Z} G(\sZ^o)\otimes_{\Z}\uMCor(\sX, \sY\otimes \sZ)\right)/R_l,\]
where for $\sY=(\ol{Y}, Y_\infty)$ we set $\sY^o=\ol{Y}\setminus Y_\infty$, and where $R_l$
is the subgroup generated by the elements 
\[f^*a\otimes g^*b\otimes h - a\otimes b\otimes (f\otimes g)\circ h,\]
where $\sY, \sY', \sZ, \sZ'\in \uMCor$, $a\in F(\sY^o)$, $b\in G(\sZ^o)$,
$f\in \uMCor(\sY', \sY)$, $g\in \uMCor(\sZ',\sZ)$, and $h\in \uMCor(\sX, \sY'\otimes \sZ')$. Similarly, 
\[H_r(X,D)=\left(\bigoplus_{Y,Z\in \Sm} F(Y)\otimes_{\Z} G(Z)\otimes_{\Z} \Cor(X\setminus D, Y\times Z)\right)/R_r,\]
where $R_r$ is the subgroup generated by
\[f^*a\otimes g^*b\otimes h - a\otimes b\otimes (f\times g)\circ h,\]
where $Y, Y', Z, Z'\in \Sm$, $a\in F(Y)$, $b\in G(Z)$, $f\in \Cor(Y',Y)$, $g\in \Cor(Z',Z)$, and 
$h\in \Cor(X\setminus D, Y'\times Z')$. 

Let $\sum_i a_i\otimes b_i\otimes \gamma_i\in H_r(X,D)$, where $a_i\in F(Y_i)$, $b_i\in G(Y_i)$,
and $\gamma_i\in \Cor(X\setminus D, Y_i\times Z_i)$. 
By \cite[Thm 1.6.2]{KMSY1} we find a proper morphism $\rho:X'\to X$ inducing an isomorphism
$X'\setminus |\rho^*D|\xr{\simeq} X\setminus |D|$, such that the closure of any irreducible component 
of $\gamma_i$ in $X'\times Y_i\times Z_i$ is finite over $X'$, for all $i$.
By \eqref{para:sheaf0} and Lemma \ref{lem:crit-surj} we are reduced to show the following:
\begin{claim}\label{lem:surjT2}
Assume $X$ is henselian local of geometric type (i.e. $X=\Spec(\sO_{X,x}^h)$ for $X$ integral quasi-projective $k$-scheme).
Let $V\in \Cor(X\setminus D, Y\times Z)$ be a prime correspondence, such that the closure of $\ol{V} \subset X\times Y\times Z$ of $V$
is finite over $X$, and let $a\in F(Y)$, $b\in G(Z)$. Then the class of $a\otimes b\otimes V$ in $H_r(X,D)$
lies in the image of $H_l(X,D)\xr{\eqref{lem:surjT1}} H_r(X,D)$.
\end{claim}

Let $V$ be as above. Since the closure $\ol{V}\subset X\times Y\times Z$ of $V$ is integral and finite over $X$, it is local.
Denote by $v\in \ol{V}$ the closed point and by $y\in Y$, $z\in Z$ the images of $v$, respectively.
We get induced maps $\sO_{Y,y}\to \Gamma(\ol{V}, \sO_{\ol{V}})$ and $\sO_{Z,z}\to \Gamma(\ol{V}, \sO_{\ol{V}})$.
Hence 
\[V\subset (X\setminus D)\times U_1\times U_2,\]
where  $j_1:U_1\inj Y$ and $j_2:U_2\inj Z$ are open affines containing $y$ and $z$, respectively. 
Denote by $V'\in \Cor(X\setminus D, U_1\times U_2)$ the induced prime correspondence.
Then $V= (j_1\times j_2)\circ V'$ and thus 
\[a\otimes b\otimes V= j_1^*a\otimes j_2^* b\otimes V'\quad \text{in }H_r(X,D).\]
Hence Claim \ref{lem:surjT2} follows from the following:
\begin{claim}\label{lem:surjT4}
Let $(X,D)\in \uMCor$, let $Y, Z$ be smooth quasi-projective $k$-schemes,
let $V\in \Cor(X\setminus D, Y\times Z)$ be a prime correspondence and $a\in F(Y)$, $b\in G(Z)$.
Then the class of $a\otimes b\otimes V$ in $H_r(X,D)$ lies in the image of
$H_l(X,D)\xr{\eqref{lem:surjT1}} H_r(X,D)$.
\end{claim}
We prove the claim. First we reduce to $k$ infinite by a standard trick:
If $k$ is finite denote by  $k(\ell)$ a $\Z_\ell$-Galois extension of  $k$ for a prime $\ell$;
by a trace argument the (diagonal) pullback 
$H_r(X,D)\to H_r(X_{k(\ell)}, D_{k(\ell)})\times H_r(X_{k(\ell')}, D_{k(\ell')})$
is injective for $\ell\neq\ell'$.

In the following we assume $k$ infinite.
By assumption we find proper modulus pairs $\sY=(\ol{Y}, Y_\infty)$
and $\sZ=(\ol{Z}, Z_\infty)$ such that $\ol{Y}$ and $\ol{Z}$ are projective and 
$Y=\ol{Y}\setminus |Y_\infty|$ and $Z=\ol{Z}\setminus |Z_\infty|$. 
Since $V$ is closed in $X\setminus|D|\times \ol{Y}\times \ol{Z}$ we find an integer $n_0$ 
such that $V\in \uMCor((X, n_0 D), \sY\otimes \sZ)$.
Choose $n \ge n_0$ with $(n,p)=1$.
By Lemma \ref{lem:rootMpair} we find a modulus pair 
$\sY'=(\ol{Y}', Y_\infty')$ together with a finite  and surjective morphism 
$\bar{\pi}_{Y,n}: \ol{Y}'\to \ol{Y}$ such that $\deg \bar{\pi}_{Y,n}$ divides a power of $n$
and $\bar{\pi}_{Y,n}^*(Y_\infty)= n Y'_\infty$, similarly for $\ol{Z}$.
Denote by $\pi_{Y,n}: Y'\to Y$ the induced finite and surjective morphism in $\Sm$ and 
by $\pi_{Y,n}^t\in\Cor(Y, Y')$ the correspondence induced by the transpose of the graph.
In $H_r(X,D)$ we obtain
\begin{align*}
\deg(\pi_{Y,n})\deg(\pi_{Z,n})\cdot (a\otimes b\otimes V)
               &= \pi_{Y,n*}\pi_{Y,n}^*a\otimes \pi_{Z,n*}\pi_{Z,n}^*b\otimes V\\
               &=(\pi_{Y,n}^t)^*\pi_{Y,n}^*a\otimes (\pi_{Z,n}^t)^*\pi_{Z,n}^*b\otimes V\\
               &=\pi_{Y,n}^*a\otimes \pi_{Z,n}^*b\otimes (\pi_{Y,n}^t\times \pi_{Z,n}^t)\circ V.
\end{align*}
Observe that the components of  $(\pi_{Y,n}^t\times \pi_{Z,n}^t)\circ V\in \Cor(X\setminus D, Y'\times Z')$,
are the irreducible components of
\[V\times_{Y\times Z} (Y'\times Z')= (\id_{X\setminus |D|}\times \pi_{n,Y}\times\pi_{n,Z})^{-1}(V).\]
Let $W$ be such a component, it comes with a finite and surjective map $W\to V$.
Denote by $\ol{V}\subset X\times \ol{Y}\times\ol{Z}$ and $\ol{W}\subset X\times \ol{Y}'\times \ol{Z}'$ 
the closure of $V$ and $W$, respectively, and denote by $\tilde{V}\to \ol{V}$ and $\tilde{W}\to\ol{W}$
the normalizations. Since $\ol{W}$ is contained in $\ol{V}\times_{\ol{Y}\times \ol{Z}} (\ol{Y}'\times \ol{Z}')$, the 
natural maps from $\tilde{W}$ to $\ol{Y}$ and $\ol{Z}$ factor via a morphism $\tilde{W}\to \tilde{V}$.
We obtain
 \[n D_{|\tilde{W}}\ge n_0 D_{|\tilde{W}}\ge Y_{\infty|\tilde{W}} + Z_{\infty|\tilde{W}}
                                = n Y'_{\infty|\tilde{W}}+ n Z'_{\infty|\tilde{W}},\]
where the second inequality follows from $V\in \uMCor((X, n_0 D), \sY\otimes \sZ)$.
Hence
\[(\pi_{Y,n}^t\times \pi_{Z,n}^t)\circ V\in \uMCor((X,D), \sY'\otimes\sZ').\]
It follows that 
$\delta_n\cdot (a\otimes b\otimes V)$ lies in the image of
$H_l(X,D)\to H_r(X,D)$, where $\delta_n :=\deg(\pi_{Y,n})\deg(\pi_{Z,n})$.
Choose $r\ge n_0$ with $(r,p)=1=(r,n)$. 
Since $\delta_n$ divides a power of $n$ and $\delta_r$ divides a power
of $r$ we find integers $s, t$ with 
\[a\otimes b\otimes V= s \delta_n\cdot(a\otimes b\otimes V)+ 
                                t\delta_r\cdot(a\otimes b\otimes V).\]
This proves Claim \ref{lem:surjT4} and hence also the lemma.
\end{proof}

 {
\begin{prop}\label{prop:T}
For $F_1,\dots,F_n\in \HI_{\Nis}$, consider the map
\eq{lem:surjT1v}{
\uomega^*F_1\otimes_{\uMPST}  \cdots\otimes_{\uMPST} \uomega^*F_n \to 
\uomega^*(F_1\otimes_{\PST}  \cdots\otimes_{\PST} F_n) \to  
\uomega^*(F_1\otimes_{\HI_\Nis}  \cdots\otimes_{\HI_\Nis} F_n),}
where the first map is induced by \eqref{lem:surjT1} and the associativity of 
$\otimes_{\uMPST}$ and $\otimes_{\PST}$ and the second map is induced by the natural surjective map (cf. \eqref{para:tut3})
\[F_1\otimes_{\PST}  \cdots\otimes_{\PST} F_n\to 
 h_{0,\Nis}^{\A^1}(F_1\otimes_{\PST}  \cdots\otimes_{\PST} F_n):=
F_1\otimes_{\HI_\Nis}  \cdots\otimes_{\HI_\Nis} F_n,\]
where we use the notation from \ref{para:CI} and $\otimes_{\HI_{\Nis}}$ denotes the monoidal structure on $\HI_{\Nis}$ defined by Voevodsky.
Then,\eqref{lem:surjT1v}  induces an isomorphism
\eq{prop:T1}{
h^{\bcube, \sp}_{0, \Nis}(\uomega^*F_1\otimes_{\uMPST}  \cdots\otimes_{\uMPST} \uomega^*F_n) \xr{\simeq}\uomega^*(F_1\otimes_{\HI_\Nis}  \cdots\otimes_{\HI_\Nis} F_n).}
\end{prop}
\begin{proof}
 We begin by recalling from \cite[Prop. 3.2]{MS} that for $F, G\in \CItspNis$, the formula $F\otimes_{\CItspNis} G = \tau_!h_{0, \Nis}^{\bcube, \sp}(\tau^* F\otimes_{\MPST} \tau^* G)$ defines a symmetric monoidal structure on $\CItspNis$.
Next, note that $\uomega^*H\in\CItspNis$ for $H\in \HI_\Nis$ by 
\cite[Lem. 2.3.1]{KSY2} and \cite[Prop. 6.2.1b)]{KMSY2}. Moreover
\begin{align*}
	h^{\bcube, \sp}_{0, \Nis}(\uomega^*F_1\otimes_{\uMPST} \uomega^*F_2) & =  	h^{\bcube, \sp}_{0, \Nis}(\tau_! \omega^*F_1\otimes_{\uMPST} \tau_! \omega^* F_2) \\
	& = \tau_! h^{\bcube, \sp}_{0, \Nis}( \omega^*F_1\otimes_{\MPST} \omega^* F_2) \\
	&= \tau_! h^{\bcube, \sp}_{0, \Nis}( \tau^* (\uomega^*F_1)\otimes_{\MPST} \tau^* (\uomega^* F_2))= \uomega^* F_1 \otimes_{\CItspNis}\uomega^* F_2,
\end{align*}
for every $F_1, F_2 \in \HI_{\Nis}$. Here, the isomorphisms follow from \eqref{para:MCor1} and the exactness of $\tau_!$. 

We now observe that the functor $\uomega^*$ is lax monoidal from $\PST$ to $\ulMPST$ (this follows from the fact that $\uomega^*$ is right adjoint to $\uomega_!$, which is strict monoidal by construction). By applying $h^{\bcube,\sp}_{0,\Nis}$ to \eqref{lem:surjT1v}, we obtain the functorial map \eqref{prop:T1}, which we can rewrite for $n=2$ as
\begin{equation}\label{eq:monoidal} \uomega^* F_1 \otimes_{\CItspNis}\uomega^* F_2 \to \uomega^*( F_1 \otimes_{\HI_{\Nis}} F_2).
\end{equation}
In particular, $\uomega^*$ restricts to a lax symmetric monoidal functor from $\HI_{\Nis}$ to $\CItspNis$, and the statement of the Proposition is equivalent to the fact $\uomega^*$ is in fact (strictly) monoidal, i.e. that the map \eqref{eq:monoidal} is an isomorphism (note that the identity for the tensor product is simply the constant sheaf $\Z$, and that $\uomega^* \Z = \Z$). Since the tensor products in $\CItspNis$ and $\HI_{\Nis}$ are in particular associative, it is enough to prove the claim when $n=2$. 

By  Lemma \ref{lem:surjT}, the map \eqref{eq:monoidal} (or, equivalently, \eqref{prop:T1}) is surjective. 
On the other hand, we have 
\begin{align*} 
\uomega_! h^{\bcube, \sp}_{0, \Nis}( \uomega^*F_1\otimes_{\uMPST} \uomega^*F_2)  
& = a^V_\Nis\uomega_! \ul{h}^{\bcube}_0( \uomega^*F_1\otimes_{\uMPST}    \uomega^*F_2)^{\sp} \\
&=a^V_\Nis\uomega_! \ul{h}^{\bcube}_{0}\tau_!(\omega^*F_1\otimes_{\uMPST}  \omega^*F_2)\\ 
&=a^V_\Nis\uomega_!\tau_!h^{\bcube}_{0}( \omega^*F_1\otimes_{\MPST} \omega^*F_2) \\
&= a^V_\Nis\omega_!h^{\bcube}_{0}( \omega^*F_1\otimes_{\MPST}   \omega^*F_2),\end{align*}
where the first equality follows from the definition of $h^{\bcube,\sp}_{0,\Nis}$ (cf. \eqref{para:CI1}) and $\uomega_!\ul{a}_\Nis=a^V_\Nis\uomega_!$ (cf. \eqref{para:sheaf2}), and the second holds by the fact $\uomega_! A^{\sp}=\uomega_! A$ for $A\in \uMPST$ and $\uomega^*=\tau_!\omega^*$ (cf. \eqref{para:MCor1}) and the monoidality of $\tau_!$, and the third follows from $\ul{h}^{\bcube}_0(\tau_! B)=\tau_! h^\cube_0(B)$ for $B\in \MPST$, where $h^\cube_0(B)\in \MPST$ is the maximal cube invariant quotient of $B$ defined by the same way as \eqref{para:spNis2}, and the last holds by $\uomega_!\tau_!=\omega_!$ (cf.  \eqref{para:MCor1}).
Thus  $\uomega_!\eqref{prop:T1}$ is an isomorphism by \cite[Thm 5.3]{RSY} in view of $\omega^*F=\tilde{F}$ (see \cite[(3.14.5)]{RSY}) by \cite[Lem. 2.3.1]{KSY2}. 
Since both sides of \eqref{prop:T1} are semipure the map \eqref{prop:T1} is  injective as well. 
\end{proof}
}

\subsection{Definition and basic properties of twists}
\begin{defn}[see {\cite[\S 2]{MS}}]\label{def:gtwist}
Let $F\in \CItspNis$. We define $\gamma^nF$ and $F(n)$, $n\ge 0$, recursively by
\[\gamma^0F:=F, \quad \gamma^1F:=\gamma F:=\uHom_{\uMPST}(\uomega^*\G_m, F), \quad
\gamma^n F:=\gamma(\gamma^{n-1} F)\]
and 
\[F(0):= F, \quad F(1):= h^{\bcube,\sp}_{0,\Nis}(F\otimes_{\uMPST} \uomega^*\G_m), \quad 
F(n):=F(n-1)(1).\]
\end{defn}

\begin{cor}\label{cor:gtwist}
Let $F\in \CItspNis$. Then $\gamma^n F$, $F(n)\in \CItspNis$, for all $n\ge 0$.
Furthermore, 
\eq{cor:gtwist1}{\gamma^n F= \uHom_{\uMPST}((\bcubee_{\red})^{\otimes_{\uMPST} n}, F)
                                           =\uHom_{\uMPST}(\uomega^* K^M_n, F),}
and 
\eq{cor:gtwist1.0}{F(n)= 
h^{\bcube,\sp}_{0,\Nis}( F \otimes_{\uMPST} (\bcubee_{\red})^{\otimes_{\uMPST} n})
= h^{\bcube,\sp}_{0,\Nis}( F \otimes_{\uMPST} \uomega^*K^M_n),}
where 
\eq{cor:gtwist1.1}{\bcubee_{\red}:=\Coker(\Ztr(\{1\})\to \Ztr(\P^1, 0+\infty))\in \uMPST^\tau}
and $K^M_n$ is the improved Milnor $K$-theory from \cite{Kerz} (there denoted by $\hat{K}^M_n$).
\end{cor}
\begin{proof}
For a proper modulus pair $\sX$ we have $\tau_!\tau^*\Ztr(\sX)=\Ztr(\sX)$. It follows that 
$\bcubee_{\red}\in \uMPST^\tau$. 
By Lemma \ref{lem:Gm} we have $h^{\bcube, \sp}_{0,\Nis}(\bcubee_{\red})=\uomega^*\G_m$.
Thus 
\eq{cor:gtwist1.5}{\gamma F= \uHom_{\uMPST}(\bcubee_\red, F).}
Indeed (we drop the index $\uMPST$ from $\Hom$ and $\uHom$)
\begin{align*}
\uHom(\bcubee_\red, F)(\sX) &= \Hom(\Ztr(\sX)\otimes \bcubee_\red, F)\\
                                  &=\Hom(\bcubee_\red, \uHom(\Ztr(\sX), F))\\
                                  &=\Hom(h^{\bcube,\sp}_{0,\Nis}(\bcubee_\red), \uHom(\Ztr(\sX), F))\\
                                  &= \uHom(h^{\bcube,\sp}_{0,\Nis}(\bcubee_\red), F)(\sX)\\
                                  & = \gamma F(\sX),
\end{align*}
where the third equality holds by Lemma \ref{lem:Nis}\ref{lem:Nis1}, \ref{lem:Nis3}.
This implies the first equality in \eqref{cor:gtwist1} and also that
$\gamma^n F\in \CItspNis$, for all $n\ge 0$, by Lemma \ref{lem:Nis}\ref{lem:Nis1}.
For the second equality in \eqref{cor:gtwist1}, first note that it follows from \cite{Kerz}
and results by Voevodsky (see \cite[5.5]{RSY}), that we have 
\eq{cor:gtwist2}{K^M_n\cong \G_m^{\otimes_{\HI_\Nis} n}\in \HI_{\Nis}.} 
Hence by Proposition \ref{prop:T}  and \cite[Lem 1.14(iii)]{MS}, we obtain
\eq{cor:gtwist3}{\uomega^*K^M_n= h_{0,\Nis}^{\bcube,\sp}((\uomega^*\G^m)^{\otimes_{\uMPST} n})= 
h_{0,\Nis}^{\bcube,\sp}((\bcubee_{\red})^{\otimes_{\uMPST} n}).}
Thus the second equality in \eqref{cor:gtwist1} follows from the adjunction \eqref{para:CI1}.
The equalities in \eqref{cor:gtwist1.0} follow similarly.
\end{proof}

\begin{remark}\label{rmk:twistMS}
By Corollary \ref{cor:gtwist} the twist $\gamma^n F$ (resp. $F(n)$)
agrees with the definition in \cite[(2.3)]{MS} (resp. \cite[after Prop 3.2]{MS}).
\end{remark}

\begin{remark}\label{rmk:gamma}
Let $F\in \CItspNis$ and $\sX\in\uMCorls$. By \eqref{cor:gtwist1.5} and \cite[Lem 5.9]{S-purity} we have 
\[\gamma^1 F(\sX)= \frac{F((\P^1,0+\infty)\otimes \sX)}{F((\P^1,\infty)\otimes \sX)}=
                               \frac{F((\A^1,0)\otimes \sX)}{F(\A^1\otimes \sX)}.\]
\end{remark}

\begin{para}\label{para:cancel}
For later use in section \ref{sec:pf} we define a certain maps induced by adjuntion.
Let $F\in \CItspNis$. For $n\ge 0$ we have an adjunction map
\eq{para:cancel1}{F\to \uHom_{\uMPST}(\uomega^*K^M_n,   F\otimes_{\uMPST} \uomega^*K^M_n)}
which sends $a\in F(\sX)$ to (we drop the subscript $\uMPST$)
\[a\otimes \id\in
 \Hom(\Ztr(\sX)\otimes\uomega^*K^M_n,   F\otimes \uomega^*K^M_n),\]
where we identify an element $a\in F(\sX)$ with the map $a: \Ztr(\sX)\to F$.
Composing \eqref{para:cancel1} with the map induced by the natural map
$F\otimes \uomega^*K^M_n\to h_{0,\Nis}^{\bcube,\sp}(F\otimes \uomega^*K^M_n)$
Corollary \ref{cor:gtwist} yields a map
\eq{para:cancel2}{\kappa_n: F\to \gamma^n(F(n))=:\gamma^n F(n)}
which by Remark \ref{rmk:twistMS} coincides with the morphism \cite[(3.5)]{MS}.
Note that $\kappa_0$ is the identity and that for $m,n\ge 0$ the following  diagram commutes
\eq{para:cancel2.5}{\xymatrix{
F\ar[rr]^-{\kappa_{m+n}}\ar[d]_{\kappa_m} & &\gamma^{m+n} (F(m+n))\ar@{=}[d]\\
\gamma^m(F(m))\ar[rr]^-{\gamma^m\kappa_n} & & \gamma^m \gamma^n (F(m)(n)).
}}
\end{para}

\section{Cup product with Chow cycles with support}\label{sec:cupProduct}
\subsection{Milnor K-theory and intersection theory with supports}

Everything in this subsection is well-known, however we give some explanations for lack of reference.

\begin{para}\label{para:supports}
Recall that a {\em family of supports} on a scheme $X$ is a non-empty collection $\Phi$ of closed subsets of $X$
which is stable under taking finite unions and closed subsets. The main examples are the family
$\Phi_Z$, for a closed subset $Z\subset X$, which consists of all the closed subsets in $Z$, the family $\Phi^{\ge c}$ 
of all closed subsets of codimension $\ge c$, and 
the family $\Phi^{\rm prop}_{X/S}$, for a morphism $X\to S$, which consists of all closed subsets in $X$
which are proper over $S$. If $F$ is a sheaf on $X$ and $\Phi$   is a family of support, then 
\[\Gamma_\Phi(X,F)=\{s\in F(X)\mid {\rm supp}(s)\in \Phi\}=\varinjlim_{Z\in \Phi}\Gamma_Z(X,F),\]
and $\ul{\Gamma}_\Phi(F)(U)=\Gamma_{\Phi\cap U}(U, F)$, for an open $U\subset X$.
For a morphism $f:Y\to X$ we denote by $f^{-1}\Phi$ the smallest family with supports on $Y$
containing all closed subsets of the form $f^{-1}(Z)$, $Z\in \Phi$.

Let $X$ be $k$-scheme. We denote by $\CH_i(X)$ the Chow group of $i$-dimensional cycles on $X$. If $X$ is equidimensional of dimension $d$, we denote by $\CH^i(X)$ the Chow group of $i$-codimensional cycles on $X$, i.e. $\CH^i(X) = \CH_{d-i}(X)$. 
If $\Phi$ is a family of supports on $X$, we set
\eq{para:supports1}{\CH^i_\Phi(X)= \varinjlim_{Z\in \Phi} \CH_{d-i}(Z),}
where the transition maps in the directed limit are given by pushforward along closed immersions.
Note that for a closed subset $Z\subset X$ we have 
\[\CH^i_Z(X):=\CH^i_{\Phi_Z}(X)= \CH_{d-i}(Z),\]
in particular $\CH^i_X(X)=\CH^i(X)$. The notation $\CH^i_Z(X)$ is not superfluous since
if $Z$ is singular the pullback along the refined Gysin homomorphism  as in \cite[\S 6]{Fu}
relies on the embedding $Z\inj  X$.
\end{para}

\begin{para}\label{para:KIT}
We recall some facts on the relation between Milnor $K$-theory and intersection theory.
Let $K^M_i$ be the improved Milnor $K$-sheaf from \cite{Kerz}. Its restriction to $\Sm$ is homotopy invariant
and hence for $X\in \Sm$ its restriction to $(\et/X)$ is a Nisnevich sheaf denoted by $K^M_{i,X}$ and  
we have 
\eq{para:KIT1}{R\epsilon_* K^M_{i,X}= \epsilon_*K^M_{i,X},} 
where $\epsilon: X_{\Nis}\to X_{\Zar}$ denotes the canonical morphism of sites, see \cite[Thm 3.1.12]{V-TCM}.

If $Z$ is a finite-type $k$-scheme, we denote by $C_{\bullet}(i)(Z)$ the degree $i$  
(homological) Gersten complex of $K^M_{*, Z}$ (e.g. \cite[5.]{Rost}), i.e.,
\[C_n(i)(Z)= \bigoplus_{z\in Z_{(n)}} K^M_{n+i}(z),\]
and the differentials are induced by the tame symbol.
(For the tame symbol we use the sign convention from \cite[p. 328]{Rost}.)
Recall that the formation $Z\mapsto C_\bullet(i)(Z)$ is covariant functorial with respect to proper maps and
contravariant functorial with respect to quasi-finite flat maps, 
see \cite[Prop (4.6)]{Rost}.
The assignment  $U\mapsto C_\bullet(i)(U)$ defines a complex of sheaves
on $Z_{\Nis}$ which we denote by $C_{\bullet, Z}(i)$.
If $Z$ is equidimensional of dimension $e$, then we define 
\eq{para:KIT1.5}{ C^n_{Z}(i):=C_{e-n,Z}(i-e)}
and obtain the cohomological degree $i$ Gersten complex $C^\bullet_Z(i)$ the global sections of which we also
denote by $C^\bullet(i)(Z)$. 

In the following we assume $X\in \Sm$ is equidimensional.
By \cite[Prop 10(8)]{Kerz}, the Gersten complex is a resolution on the Nisnevich site for the sheaf $K^M_{i,X}$, i.e.
\[K^M_{i,X}\xr{\simeq} C^\bullet_X(i) \quad  \text{in } D(X_\Nis).\]
Note that $C^\bullet_X(i)$ sits in cohomological degree $[0,i]$.
By \eqref{para:KIT1} and since $\epsilon_*C_X^\bullet(i)$ is a flasque resolution of $\epsilon_*K^M_{i,X}$, 
we can use $C^\bullet_X(i)$ to compute Nisnevich cohomology with supports of $K^M_{i,X}$. 
If $\dim X=d$ and $\imath:Z\inj X$ is a closed immersion, then
\[\ul{\Gamma}_Z C^n_X(i)= \imath_*C_{d-n,Z}(i-d).\]
This gives rise to Bloch's formula (with support)
\eq{para:KIT2}{\CH^i_\Phi(X)= H^i_{\Phi,\Zar}(X, K^M_i)= H^i_{\Phi,\Nis}(X, K^M_i)=:H^i_{\Phi}(X, K^M_i),}
where $\Phi$ is a family of supports on $X$. 
\end{para}

\begin{lemma}\label{lem:pbKCH}
Let $f:Y\to X$ be a morphism between equidimensional smooth schemes and let $\Phi$ be a family of
supports on $X$. The following diagram commutes
\eq{lem:pbKCH1}{\xymatrix{
 H^i_\Phi(X, K^M_i) \ar[d]_{f^*}\ar[r]^-{\eqref{para:KIT2}} &\CH^i_{\Phi}(X)\ar[d]^{f^*}\\
H^i_{f^{-1}\Phi}(Y, K^M_i)\ar[r]^-{\eqref{para:KIT2}} & \CH^i_{f^{-1}\Phi}(Y) ,
}}
where the pullback on the right is induced by the refined Gysin homomorphism  in \cite[6.6]{Fu}
(see also \cite[1.1.30]{CR11})
and the pullback on the left is induced from the sheaf structure of $K^M_i$ on the category of schemes.
\end{lemma}
\begin{proof}
In \cite[12]{Rost} a morphism of complexes
\eq{lem:pbKCH2}{I(f): C^\bullet(i)(X)\to C^\bullet(i)(Y)}
is defined, depending on the choice of a coordination of the tangent bundle $TX$ of $X$ (see \cite[\S 9]{Rost} for the definition of a coordination).
It is compatible with the pullback $f^*: K^M_i(X)\to K^M_i(Y)$, by \cite[(12.3)Prop, (12.4)Cor]{Rost}.
Furthermore, if $u: U\to X$ is \'etale, a coordination of $TX$ induces by pullback a coordination of
$TU$ and hence it is direct to check that we have 
\[u_Y^*\circ I(f)= I(f_U)\circ u^*: C^\bullet(i)(X)\to C^\bullet(i)(Y_U),\]
where $u_Y: Y_U\to Y$ is the base change of $u$ along $f$ and $f_U: Y_U\to U$ is the base change of $f$ along $u$.
It follows that the choice of a coordination on $TX$ allows one to promote \eqref{lem:pbKCH2} to a morphism of complexes of sheaves on $X_\Nis$
\eq{para:KIT3.4.5}{\sI(f): C^\bullet(i)_X\to f_*C^\bullet(i)_Y}
which is compatible with the pullback $f^*: K^M_{i,X}\to f_*K^M_{i,Y}$. In view of \ref{para:KIT}, taking sections with support $\Gamma_{\Phi}(X, -)$ and then cohomology, gives a map
\[H^i(\Gamma_\Phi(X, \sI(f))) \colon H^i_{\Phi}(X, K^M_i) \to H^i_{f^{-1}\Phi}(Y, K^M_i)\] 
that we identify with the left vertical map in \eqref{lem:pbKCH1}.
Consider the following diagram of solid arrows
\eq{lem:pbKCH3}{
\xymatrix{
C^i(i)(X)\ar[d]^{I(f)}& 
C_{d-i}(i-d)(Z)\ar@{_(->}[l]\ar@{.>}[d]^{I_Z(f)}\ar@{->>}[r] &
\CH_{d-i}(Z)\ar[d]^{f^!}\\
C^i(i)(Y) &  C_{e-i}(i-e)(f^{-1}(Z))\ar@{_(->}[l]\ar@{->>}[r] &
\CH_{e-i}(f^{-1}(Z)),
}}
where $Z\in \Phi$, $d=\dim X$, $e=\dim Y$, and $f^!$ is the refined Gysin map from \cite[6.6]{Fu}.
It remains to show that there exists a dotted arrow $I_Z(f)$ making the diagram commute.
Since the pushforward on $C^\bullet(i)$ is compatible with the one on Chow groups, we can assume
that $Z$ is integral with $\dim Z=d-i$, i.e., $C_{d-i}(i-d)(Z)= \Z\cdot[Z]$.
By definition of $I(f)$ and $f^!$ it suffices to consider the case where $f=i: Y\inj X$ is a regular closed immersion
defined by a coherent ideal sheaf $J$. Denote by $N_{Y/X}=\Spec (\oplus_{n\ge 0} J^n/J^{n+1})$ the normal bundle
over $Y$ and fix a coordination $\tau$ of $N_{Y/X}$ in the sense of \cite[9., p.371]{Rost}.
Set $Z':=Z\times_X Y$ and $N:=N_{Y/X}\times_Y  Z'$; the pullback of $\tau$ along $Z'\inj Y$ induces
a coordination $\tau'$ of $N$. 
Denote by $C_{Z'/Z}$ the normal cone of $Z'\inj Z$ and by $\nu: C_{Z'/Z}\inj N$ the closed immersion induced by
$J\otimes_{\sO_X} \sO_Z\surj J\sO_Z$. 
Note that $C_{Z'/Z}$ has pure dimension $d-i$, see \cite[B 6.6]{Fu}, and thus 
\[C_{d-i}(i-d)(C_{Z'/Z})= \bigoplus_{z\in (C_{Z'/Z})^{(0)}} \Z.\]
With the notation from \cite[9., 11.]{Rost} we define $I_Z(f)$ to be the composition
\mlnl{I_Z(f): C_{d-i}(i-d)(Z)\xr{J(Z,Z')} C_{d-i}(i-d)(C_{Z'/Z})\\
                                         \xr{\nu_*} C_{d-i}(i-d)(N)\xr{r(\tau')} C_{e-i}(i-e)(Z').}
Let $D(Z,Z')\to \A^1=\Spec k[t]$ be the deformation  scheme from \cite[10.]{Rost}, so that
$D(Z,Z')_{|\A^1\setminus \{0\}}= Z\times (\A^1\setminus \{0\})$ and $D(Z,Z')_{|0}= C_{Z'/Z}$.
Then by definition  (see \cite[11.]{Rost})  
\[J(Z,Z')([Z])= \div_{D(Z,Z')}(t)_{|C_{Z'/Z}}= [C_{Z'/Z}],\]
where $[C_{Z'/Z}]$ denotes the cycle associated to the scheme $C_{Z'/Z}$, \cite[1.5]{Fu}.                                       
Thus the map $J(Z',Z)$ corresponds to the specialization map
$\sigma: \CH_{d-i}(Z)\to \CH_{d-i}(C_{Z'/Z})$ from \cite[5.2]{Fu}. Therefore, the above definition of $I_Z(f)$ makes 
the square on the right in \eqref{lem:pbKCH3} commutative, by the alternative description of $i^!$ on the Chow side
in \cite[6.2, 2nd paragraph on p. 98]{Fu}. We subdivide the left square  as follows:
\eq{lem:pbKCH4}{
\xymatrix{
C_{d-i}(i-d)(Z)\ar[rr]^{\nu_*\circ J(Z,Z')}\ar[d] & &
C_{d-i}(i-d)(N)\ar[d]\ar[r]^{r(\tau')} & 
C_{e-i}(i-e)(Z')\ar[d]\\
C^i(i)(X)\ar[rr]^{J(X, Y)}\ar@/_2pc/[rrr]_{I(f)} & &
C^i(i)(N_{Y/X})\ar[r]^{r(\tau)}&
C^i(i)(Y),
}}
the vertical maps are all induced by pushforward along the respective closed immersion. 
It follows directly from the definition of the maps $r(\tau)$ and $r(\tau')$ in \cite[(9.1)-(9.4)]{Rost}, that 
the right square of \eqref{lem:pbKCH4} commutes. For the left square note that 
$D(Z,Z')$ is an integral and closed subscheme of $D(X,Y)$, hence it is the closure of
$Z\times (\A^1\setminus \{0\})$ in $D(X,Y)$; furthermore $D(Z,Z')\cap N_{Y/X}= D(Z,Z')\cap N= C_{Z'/Z}$.
Thus by definition
\[J(X,Y)([Z])= [C_{Z'/Z}]\quad \text{in } C^i(i)(N_{Y/X}). \]
This yields the commutativity of the left square in \eqref{lem:pbKCH4}.
\end{proof}

\begin{para}\label{para:ext-prod}
In \cite[14.1]{Rost} a cross product 
\[C^p(i)(X)\times C^q(j)(Y)\to C^{p+q}(i+j)(X\times Y), \quad (a, b)\mapsto a\times b\]
is defined by sending $(a_x,b_y)$, where $a_x\in K^M_{i-p}(x)$, $x\in X^{(p)}$, and $b_y\in K^M_{j-q}(y)$,
$y\in Y^{(q)}$, to 
\eq{para:ext-prod1}{\oplus_{z\in (x\times y)^{(0)}} l_z \, a_{x|z}\cdot b_{y|z},}
where $x\times y$ denotes the fiber product of $k$-schemes, 
$l_z$ denotes the length of the local ring of $x\times y$ at $z$ and 
$a_{x|z}\in K^M_{i-p}(z)$ denotes the pullback of $a_x$ to $z$ and similar with $b_{y|z}$.
Note that $z\in (x\times y)^{(0)}$ implies $z\in (X\times Y)^{(p+q)}$.
By \cite[(14.4)]{Rost} we have 
\eq{para:ext-prod2}{d(a\times b)= (da)\times b + (-1)^{i-p} (a\times db),}
where $d$ denotes the differential of the Gersten complex.
(There seems to be a typo in the formula in {\em loc. cit.}:
the $(-1)^n$ in that formula should be a $(-1)^{n+p}$ as follows from what is said in the proof and 
\cite[R3f, R3d]{Rost}; this formula is for the homological notation, if one translates to 
the cohomological notation via \eqref{para:KIT1.5} one obtains \eqref{para:ext-prod2}.)
We have to modify the cross product to obtain a morphism of complexes (with the usual sign convention for 
a tensor product of complexes). For $a\in C^p(i)(X)$ and $b\in C^q(j)(Y)$ we set
\eq{para:ext-prod3}{a\boxtimes b:= (-1)^{i(q+j)} a\times b.}
Then we obtain
\[d(a\boxtimes b)= (da)\boxtimes b + (-1)^p \, a\boxtimes db.\]
Thus $\boxtimes$ induces a morphism of complexes
\eq{para:ext-prod4}{\boxtimes: {\rm tot}(C^\bullet(i)(X)\otimes_{\Z} C^\bullet(j)(Y))\to C^\bullet(i+j)(X\times Y),}
which via the augmentation from Milnor $K$-theory is compatible with
\eq{para:ext-prod5}{K^M_i(X)\otimes_\Z K^M_j(Y)\to K^M_{i+j}(X\times Y), 
\quad a\otimes b\mapsto \pi_X^*a\cdot \pi_Y^* b,}
with $\pi_X: X\times Y\to X$ the projection.
In degree $i+j$  the map \eqref{para:ext-prod4} is given by
\[(\bigoplus_{x\in X^{(i)}}\Z\cdot x) \otimes_{\Z} (\bigoplus_{y\in Y^{(j)}}\Z\cdot y)\to 
\bigoplus_{z\in X\times Y^{(i+j)}} \Z\cdot z,\]
\[x\otimes y\mapsto \bigoplus_{z\in (x\times y)^{(0)}} (-1)^{i(j+j)}l_z \cdot z=
\bigoplus_{z\in (x\times y)^{(0)}} l_z \cdot z.\]
Hence for families of support $\Phi$ on $X$ and $\Psi$ on $Y$ the following diagram commutes
\eq{para:ext-prod6}{\xymatrix{
\CH^i_{\Phi}(X)\times \CH^j_{\Psi}(Y)\ar[r]^-{\boxtimes}\ar[d]_{\simeq} & 
\CH^{i+j}_{\Phi\times \Psi}(X\times Y)\ar[d]^{\simeq}\\
H^i_{\Phi}(X, K^M_i)\times H^j_{\Psi}(Y, K^M_j)\ar[r]^-{\boxtimes} & 
H^{i+j}_{\Phi\times \Psi}(X\times Y, K^M_{i+j}),
}}
where the upper horizontal map is the exterior product of cycles, see \cite[1.10]{Fu} and 
$\Phi\times \Psi$ denotes the smallest family of supports containing $Z_1\times Z_2$, 
for all $Z_1\in \Phi$ and $Z_2\in\Psi$.
We note that if $\tau: X\times Y\to Y\times X$ is the switching morphism, then
\eq{para:ext-prod7}{\tau_*(a\boxtimes b) = (-1)^{ij+pq} (b\boxtimes a), \quad a\in C^p(i)(X), b\in C^q(j)(Y), }
as follows directly from \eqref{para:ext-prod1} and \eqref{para:ext-prod3}.
The above and Lemma \ref{lem:pbKCH} implies that the intersection product with support
\eq{para:ext-prod8}{\Delta^*\circ \boxtimes :\CH^i_{\Phi}(X)\times \CH^j_{\Psi}(X)\to \CH^{i+j}_{\Phi\cap \Psi}(X),}
from \cite[8]{Fu} corresponds via Bloch's formula to 
\eq{para:ext-prod9}{\Delta^*\circ\boxtimes :H^i_{\Phi}(X, K^M_i)\times H^j_{\Psi}(X, K^M_j)
           \to H^{i+j}_{\Phi\cap \Psi}(X, K^M_{i+j}),}
where $\Phi\cap\Psi=\{Z_1\cap Z_2\mid Z_1\in \Phi, Z_2\in \Psi\}$.
\end{para}

\subsection{Cupping}


\begin{para}\label{para:Tmap}
Let $F,G\in \uMNST$ and let $X$ be a $k$-scheme and $D$ and $E$ effective Cartier divisors on it, such that $(X,D)$, $(X,E)\in \uMCor$.
We recall that there is a natural morphism
of Nisnevich sheaves on $X$
\eq{para:Tmap1}{ F_{(X,D)}\otimes_\Z G_{(X,E)}\to (F\otimes_{\uMNST} G)_{(X,D+E)}}
which is defined as follows:
For $U\to X$ we have a surjection (see the proof of Lemma \ref{lem:surjT})
\mlnl{\pi:\bigoplus_{\sY, \sZ\in \uMCor} F(\sY)\otimes_{\Z} G(\sZ)\otimes_\Z \uMCor((U, (D+E)_U), \sY\otimes \sZ)\\
\surj (F\otimes_{\uMPST} G)(U, (D+E)_U).}
Composition with $\pi$ gives then a morphism 
\begin{equation}\label{eq:Tmap2}F(U,D_U)\otimes G(U,E_U)\to (F\otimes_{\uMPST} G)(U,(D+E)_U),\end{equation}
\[ a\otimes b\mapsto \pi(a\otimes b\otimes \Delta_U),\]
where $\Delta_U\in \uMCor((U,(D+E)_U), (U,D_U)\otimes (U,E_U))$ is the diagonal (note that it is indeed an admissible correspondence). If we now compose \eqref{eq:Tmap2}
with the value on $(U,(D+E)_U)$ of  the natural map (the sheafification)
\[(F\otimes_{\uMPST} G)\to \ul{a}_\Nis(F\otimes_{\uMPST} G)=F\otimes_{\uMNST} G\]
we get the desired map \eqref{para:Tmap1}. 
\end{para}

\begin{lemma}\label{lem:adj-twist}
Let $F\in\CItspNis$ and $X,D,E$ as in \ref{para:Tmap}. Assume $X$ is connected.
Consider the map
\begin{equation}\label{lem:adj-twist1}
    (\gamma^1F(X,D))\otimes_{\Z} \G_m(X\setminus E) \to F(X, D+E)
\end{equation}
defined as composition
\begin{align*}
    (\gamma^1F(X,D))\otimes_{\Z} \G_m(X\setminus E)&\xr{\eqref{para:Tmap1}}
(\gamma^1F \otimes_{\uMNST} \uomega^*\G_m)(X,D+E)\\
&\xr{\rm adj.} F(X, D+E),\end{align*}
where the morphism 'adj.' is induced by the counit of the adjunction 
$(-)\otimes_{\uMPST} \uomega^*\G_m\dashv \ul{\Hom}_{\uMPST}(\uomega^* \G_m, -)$.
Then the precomposition of \eqref{lem:adj-twist1}
with the natural map
\[F(\bcubee\otimes (X,D)) \otimes_{\Z}\Ztr(\bcubee)(X,E)  \to 
\gamma^1F(X,D)\otimes_{\Z} \G_m(X\setminus E),\]
stemming from \eqref{cor:gtwist1.5} and Lemma \ref{lem:Gm}, is given by 
\[F(\bcubee\otimes (X,D)) \otimes_{\Z}\Ztr(\bcubee)(X,E)\to F(X, D+E),\]
\[a\otimes f\mapsto \Delta_X^*((f- \deg(f)\cdot s_1)\otimes \id_{(X,D)})^*a,\]
where $s_1\in\uMCor((X,E), \bcube^{(1)})$ 
is the graph of  $X\to \Spec k=\{1\}\inj \P^1$ and
$\Delta_X\in \uMCor((X, D+E), (X,E)\otimes (X,D))$ is the graph of the diagonal.
\end{lemma}
\begin{proof}
Note that under the identification $\Ztr(\bcube^{(1)})=\bcube^{(1)}_{\rm red}\oplus \Z$,
see \eqref{cor:gtwist1.1}, the projection to the first factor is given by
\[\Ztr( \bcube^{(1)})(X,E)\ni f\mapsto (f- \deg(f)\cdot s_1)\in \bcube^{(1)}_{\red}(X,E).\]
 Since by Lemma \ref{lem:Gm} we have 
$h_{0,\Nis}^{\bcube}(\bcube^{(1)}_{\red})=\uomega^*\G_m$, the statement of the lemma
 is direct from the explicit description of \eqref{para:Tmap1} in \ref{para:Tmap}.
\end{proof}
\begin{lemma}\label{lem:ten-supp}
Let $X$ be a scheme and $Z\subset X$ a closed subset $X$.
Let $A,B\in D(X_{\Nis})$ and assume that the cohomology sheaves $H^i(B)$ have support in $Z$, for all $i\in\Z$.
Then the natural map 
\[R\ul{\Gamma}_Z(A\otimes^L_{\Z}  B)\xr{\simeq} A\otimes^L_{\Z} B\]
is an isomorphism. In particular, for any $C\in D(X_{\Nis})$ we obtain the canonical morphism
\eq{lem:ten-supp1}{
A\otimes^L_{\Z} R\ul{\Gamma}_Z C\cong R\ul{\Gamma}_Z(A\otimes^L_{\Z} R\ul{\Gamma}_Z C)
\to R\ul{\Gamma}_Z(A\otimes^L_{\Z} C).}
\end{lemma}
\begin{proof}
Denote by $j: U=X\setminus Z\inj X$ the open immersion.
By assumption $H^i(B_{|U})= H^i(B)_{|U}=0$, i.e.,  $B_{|U}=0$ in $D(X_{\Nis})$.
Therefore the statement follows from the distinguished triangle (see \cite[\href{https://stacks.math.columbia.edu/tag/09XP}{Tag 09XP}]{stacks-project})
\[R\ul{\Gamma}_Z(A\otimes^L_{\Z}  B)\to A\otimes^L_{\Z} B\to Rj_* (A\otimes^L_{\Z} B)_{|U}\xr{+1}\]
and the isomorphism $(A\otimes^L_\Z B)_{|U}= A_{|U}\otimes^L_\Z B_{|U}$.
\end{proof}

\begin{para}\label{para:cyclecup}
Let $F\in \CItspNis$. Let $\sX=(X, D)\in \uMCor$ with $X\in \Sm$, let $\Phi$ be a family of supports on $X$,
and  $\alpha\in \CH^i_\Phi(X)$, see \ref{para:supports}.
We define the morphism 
\eq{para:cyclecup1}{c_\alpha: (\gamma^iF)_{\sX}[-i]\to R\ul{\Gamma}_{\Phi}F_{\sX} \quad \text{in } D(X_{\Nis})}
as follows:
choose a representative $\tilde{\alpha}\in \CH^i_Z(X)$, $Z\in \Phi$, of $\alpha$;
by the identification
\[\CH^i_Z(X)=H^i_Z(X, K^M_i)=\Ext^i_{X_{\Nis}}(\Z_{X}, R\ul{\Gamma}_Z K^M_{i,X})\]
the cycle
$\tilde{\alpha}$ induces a morphism in $D(X_{\Nis})$ (again denoted by $\tilde{\alpha}$)
\[\tilde{\alpha}: \Z_{X}[-i]\to R\ul{\Gamma}_Z K^M_{i,X}= R\ul{\Gamma}_Z(\uomega^* K^M_i)_{(X,\emptyset)}.\]
We define $c_\alpha$ as the composition in $D(X_\Nis)$
\begin{align*}
(\gamma^iF)_{\sX}[-i]
 &\xr{\id \otimes \tilde{\alpha}} (\gamma^i F)_{\sX}\otimes^L_{\Z} R\ul{\Gamma}_Z K^M_{i,X}\\
 &\xr{\eqref{lem:ten-supp1}} R\ul{\Gamma}_Z((\gamma^i F)_{\sX}\otimes^L_{\Z} K^M_{i,X})\\
 &\xr{\rm els} R\ul{\Gamma}_\Phi((\gamma^i F)_{\sX}\otimes^L_{\Z} K^M_{i,X})\\
&\to R\ul{\Gamma}_\Phi((\gamma^i F)_{\sX}\otimes_{\Z} (\uomega^*K^M_{i})_{(X,\emptyset)})\\
&\xr{\eqref{para:Tmap1}} R\ul{\Gamma}_\Phi(\gamma^i F\otimes_{\uMNST} \ulomega^*K^M_i)_{\sX}\\
&\xr{\rm adj} R\ul{\Gamma}_\Phi F_{\sX},
\end{align*}
where the map {\em els} is the enlarge-support-map,
the fourth map is induced by the quotient map $A\otimes^L B\to H_0(A\otimes^L B)=A\otimes B$,
and  {\em adj} is induced by adjunction via Corollary \ref{cor:gtwist}. 
It is direct to check that the definition of $c_{\alpha}$ does not depend on the choice of $\tilde{\alpha}$.
\end{para}

The morphism $c_\alpha$ satisfies the following functorial properties. 
\begin{lemma}\label{lem:cyclecup}
Let $F\in \CItspNis$. Let $\sX=(X, D)\in \uMCor$ with $X\in \Sm$, let $\Phi$ be a family of supports on $X$.
\begin{enumerate}[label=(\arabic*)]
\item\label{lem:cyclecup1} We have $c_{\alpha+\beta}=c_\alpha+c_\beta$, for $\alpha,\beta\in \CH^i_\Phi(X)$.
\item\label{lem:cyclecup1.5} Let $\Psi$ be another family of supports containing $\Phi$. Denote
by the same letter $\imath$ the natural maps  $\CH_\Phi\to \CH_\Psi$ and $R\ul{\Gamma}_\Phi\to R\ul{\Gamma}_\Psi$.
Then $\imath c_\alpha= c_{\imath\alpha}$, for any $\alpha\in \CH^i_\Phi(X)$.
\item\label{lem:cyclecup2} Let $\sY=(Y, E)\in\uMCor$.  Let $f: Y\to X$ be a morphism in $\Sm$, such that 
$E\ge f^*D$, and let $\alpha\in \CH^i_{\Phi}(X)$. Consider the pullback cycle $f^*\alpha\in \CH^i_{f^{-1}\Phi}(Y)$
(see Lemma \ref{lem:pbKCH}). The following diagram commutes 
\[\xymatrix{
(\gamma^i F)_{\sX}[-i]\ar[r]^-{c_\alpha}\ar[d]_{f^*} & R\ul{\Gamma}_\Phi F_{\sX}\ar[d]_{f^*}\\
Rf_* (\gamma^i F)_{\sY}[-i]\ar[r]^-{c_{f^*\alpha}} & 
                 Rf_* R\ul{\Gamma}_{f^{-1}\Phi} F_{\sY}=R\ul{\Gamma}_\Phi Rf_*  F_{\sY}.
}\]
\item\label{lem:cyclecup3} For $\alpha\in \CH^i_\Phi(X)$, $\beta\in \CH^j_\Psi(X)$
denote by $\alpha\cdot \beta\in \CH^{i+j}_{\Phi\cap \Psi}(X)$ the intersection product of $\alpha$ and $\beta$,
see \eqref{para:ext-prod8}.
The following diagram commutes
\[\xymatrix{
\gamma^i(\gamma^jF)_{\sX}[-i][-j]\ar@{=}[d]^{\text{Cor \ref{cor:gtwist}}}\ar[r]^-{c_\alpha[-j]} 
             &R\ul{\Gamma}_\Phi (\gamma^j F)_{\sX}[-j]\ar[d]^{c_\beta}\\
(\gamma^{i+j}F)_{\sX}[-(i+j)]\ar[r]_-{c_{\alpha\cdot\beta}} & R\ul{\Gamma}_{\Phi\cap \Psi}F_{\sX}.
}\]
\end{enumerate}
\end{lemma}
\begin{proof}
\ref{lem:cyclecup1} and \ref{lem:cyclecup1.5} are immediate to check.
For \ref{lem:cyclecup2} we may assume $\Phi=\Phi_Z$, for some closed subset $Z\in X$.
It suffices to show the commutativity of the adjoint square,
which we can decompose into the following two diagrams (we write $G:=\gamma^i F$)
\[
\xymatrix{
f^{-1}G_{\sX}[-i]\ar[r]^-{\id\otimes \alpha}\ar[d]\ar@{}[dr]|-*+[o][F-]{1} &
f^{-1}G_\sX\otimes^L f^{-1} R\ul{\Gamma}_Z K^M_{i,X}\ar[d]\ar[r]\ar@{}[dr]|-*+[o][F-]{2} &
f^{-1}R\ul{\Gamma}_Z(G_\sX\otimes K^M_{i,X})\ar[d]\\
G_{\sY}[-i]\ar[r]_-{\id\otimes f^*\alpha} &
G_\sY\otimes^L R\ul{\Gamma}_{f^{-1}Z}K^M_{i,Y}\ar[r] &
R\ul{\Gamma}_{f^{-1}Z} (G_\sY\otimes K^M_{i,Y}),
}
\]
and
\[
\xymatrix{
f^{-1}R\ul{\Gamma}_{Z}(G_\sX\otimes K^M_{i,X})\ar[r]^-{\eqref{para:Tmap1}}\ar[d]\ar@{}[dr]|-*+[o][F-]{3}&
f^{-1}R\ul{\Gamma}_{Z}(G\otimes \uomega^*K^M_i)_{\sX}\ar[r]^-{\rm adj.}\ar[d]\ar@{}[dr]|-*+[o][F-]{4}&
f^{-1}R\ul{\Gamma}_{Z}F_{\sX}\ar[d]\\
R\ul{\Gamma}_{f^{-1}Z}(G_\sY\otimes K^M_{i,Y})\ar[r]_-{\eqref{para:Tmap1}}&
R\ul{\Gamma}_{f^{-1}Z}(G\otimes \uomega^*K^M_i)_{\sY}\ar[r]_-{\rm adj.}&
R\ul{\Gamma}_{f^{-1}Z} G_{\sY},
}\]
where the vertical maps are induced by pullback along $f: \sY\to \sX$ and for the first diagram
we use the canonical identification $f^{-1}A\otimes^L_\Z f^{-1}B= f^{-1}(A\otimes^L_\Z B)$.
The identity $Rf_*R\ul{\Gamma}_{f^{-1}Z}= R\ul{\Gamma}_Z Rf_*$ and the natural map 
$\id\to Rf_* f^{-1}$ yield by adjunction a natural transformation 
$f^{-1}R\ul{\Gamma}_Z \to R\ul{\Gamma}_{f^{-1}Z} f^{-1}$; using this the commutativity of the square 2 
is direct to check;  furthermore 
the proof of the commutativity of the squares 3 and 4 reduces to the case without support 
(i.e., $Z=X$),  which is immediate to check.
The commutativity of square 1 follows from Lemma \ref{lem:pbKCH}.
For \ref{lem:cyclecup3} we may assume $\Phi=\Phi_Z$ and $\Psi=\Phi_{Z'}$. Consider the following diagram
\[\resizebox{12.5cm}{!}{
\xymatrix{
(\gamma^i\gamma^j F\otimes^L K_{i,Z})[-j]\ar[r]\ar[d]^{\id\otimes\beta}&
(\gamma^i\gamma^j F\otimes K_{i})[-j]_Z\ar[r]^-{\varphi_i}\ar[d]^{\id\otimes\beta}&
\gamma^j F[-j]_Z\ar[d]^{\id\otimes \beta}\\
\gamma^i\gamma^j F\otimes^L K_{i,Z}\otimes^L K_{j,Z'}\ar[r]\ar[d]&
((\gamma^i\gamma^j F\otimes K_i)\otimes^L K_{j,Z'})_Z\ar[r]^-{\varphi_i}\ar[d]&
(\gamma^j F\otimes^L K_{j,Z'})_Z\ar[d]\\
\gamma^i\gamma^j F\otimes^L (K_i\otimes K_j)_{Z\cap Z'}\ar[d]^{\mu}\ar[r]&
(\gamma^{i}\gamma^{j} F\otimes K_i\otimes K_{j})_{Z\cap Z'}\ar[r]^-{\varphi_i}\ar[d]^-{\mu}&
(\gamma^j F\otimes K_{j})_{Z\cap Z'}\ar[d]^-{\varphi_j}\\
\gamma^{i+j} F\otimes^L K_{i+j, Z\cap Z'}\ar[r] &
(\gamma^{i+j}F\otimes K_{i+j})_{Z\cap Z'}\ar[r]^-{\varphi_{i+j}}&
F_{Z\cap Z'},
}
}\]
in which we skip the indices $\sX$ and $X$, we write $K$ instead of $K^M$,
we write $C_{Z}:=R\ul{\Gamma}_Z C$, for $C\in D(X_\Nis)$, the tensor products are over $\Z$, the map
$\mu$ is induced by multiplication $K^M_i\otimes K^M_j\to K^M_{i+j}$ and Corollary \ref{cor:gtwist}, 
the maps $\varphi_i$ are induced by the composition
\begin{equation}\label{eq;pairing_cup}(\gamma^i G)_{\sX}\otimes_\Z K^M_{i,X}\xr{\eqref{para:Tmap1}} 
(\gamma^i G\otimes_{\uMNST} \uomega^*K^M_i)_{\sX}\xr{\rm adj} G, \end{equation}
for $G\in \CItspNis$, and the unlabeled arrows are induced maps of the form
\[G\otimes^L R\ul{\Gamma}_Z H\xr{\eqref{lem:ten-supp1}} R\ul{\Gamma}_Z(G\otimes^L H)\to 
  R\ul{\Gamma}_Z(G\otimes H),\]
for sheaves $G,H$.
It is direct to check that this diagram commutes.
Thus it remains to show that  the composition in $D(X_\Nis)$
\ml{lem:cyclecup4}{\Z_X[-j-i]\xr{\alpha[-j]}R\ul{\Gamma}_Z K^M_{i,X}[-j]
\xr{\id\otimes \beta} R\ul{\Gamma}_Z K^M_{i,X}\otimes^L R\ul{\Gamma}_{Z'} K^M_{i,X}\\
\to R\ul{\Gamma}_{Z\cap Z'}(K^M_{i,X}\otimes K^M_{i,X})\to R\ul{\Gamma}_{Z\cap Z'} K^M_{i+j,X}}
is equal to the morphism induced by the intersection product $\alpha\cdot\beta\in \CH^{i+j}_{Z\cap Z'}(X)$,
see \eqref{para:ext-prod8}.
Denote by $p_i: X\times X\to X$ the projection to the $i$th factor and by $\Delta: X\inj X\times X$
the diagonal. Since $\Delta^{-1}p_i^{-1}=\id_X$ we obtain that the above composition is adjoint to
\begin{align*}
p_1^{-1}\Z_X[-i]\otimes^L p_2^{-1}\Z_X[-j] & \xr{\alpha\otimes \beta } 
   p_1^{-1}R\ul{\Gamma}_Z K^M_{i,X}\otimes^L p_2^{-1}R\ul{\Gamma}_{Z'}K^M_{j,X}\\
      &    \xr{p_1^*\otimes p_2^*} R\ul{\Gamma}_{Z\times X}K^M_{i,X\times X} 
              \otimes^L R\ul{\Gamma}_{X\times Z'}K^M_{j, X\times X}\\
   &   \xr{\rm mult} R\ul{\Gamma}_{Z\times Z'} K_{i+j, X\times X} \\
  &   \xr{\Delta^*}   \Delta_*R\ul{\Gamma}_{Z\cap Z'} K^M_{i+j,X}.
\end{align*}
By the compatibility of \eqref{para:ext-prod4} and \eqref{para:ext-prod5} this composition maps in $D(X_\Nis)$
to the composition of complexes
\mlnl{p_1^{-1}\Z_X[-i]\otimes p_2^{-1}\Z_X[-j]\xr{\alpha\otimes \beta }
 p_1^{-1}\ul{\Gamma}_Z C^\bullet_X(i)\otimes p_2^{-1}\ul{\Gamma}_{Z'}C^\bullet_Y(j)\\
 \xr{\boxtimes} \ul{\Gamma}_{Z\times Z'}C^\bullet_{X\times X}(i+j) 
     \xr{I(\Delta)} \Delta_*\ul{\Gamma}_{Z\cap Z'} C^\bullet_{X}(i+j),}
 where $I(\Delta)$ is a morphism as in \eqref{para:KIT3.4.5} (and it is compatible with 
 $\Delta^*: K^M_{X\times X, i+j}\to \Delta_* K^M_{X, i+j}$, see after \eqref{para:KIT3.4.5}).
Thus it follows from the compatibility of \eqref{para:ext-prod8} and \eqref{para:ext-prod9}
that \eqref{lem:cyclecup4} is induced by the intersection product $\alpha\cdot \beta$.
\end{proof}

\begin{lemma}\label{lem:divisor-cup}
Let $F\in \CItspNis$ and $\sX=(X, D)\in \uMCor$ with $X\in \Sm$.
Let $E$ be an effective Cartier divisor on $X$ which we view as an element in $\CH^1_{|E|}(X)$.
Denote by $j: U=X\setminus |E|\inj X$ the open immersion and set $\sU=(U, D_{|U})$.
Then the composition 
\[H^1(c_E): (\gamma^1 F)_{\sX}\to R^1\ul{\Gamma}_{|E|} F_{\sX}\cong j_*F_{\sU}/F_{\sX}, \]
factors via the natural map (which is injective by the semipurity of $F$)
\eq{lem:divisor-cup0}{F_{(X, D+E)}/F_{\sX}\inj j_*F_{\sU}/F_{\sX}.}
If  $\nu: V\inj X$ is an open neighborhood of $|E|$, such that $E_{|V}=\Div(e)$ 
is principal with $e\in \Gamma(V, \sO_V)$, then the induced map
\eq{lem:divisor-cup1}{(\gamma^1 F)_{\sX}\to F_{(X, D+E)}/F_{\sX}\cong F_{(V, (D+E)_{|V})}/F_{(V, D_{|V})} }
sends an element $a\in (\gamma^1 F)(\sX)$, represented by
an element in $F(\bcubee\otimes\sX)$ which restricts to $\tilde{a}\in F((\A^1,0)\otimes \sX)$ to the class 
modulo $F(V,D_{|V})$ of 
\eq{lem:divisor-cup2}{\Delta_V^*(\Gamma_e\times \nu)^*\tilde{a} \in F(V, (D+E)_{|V}),}
where $\Gamma_e\in \uMCor((V,E_{|V}), (\A^1, 0))$
is the  graph of the morphism $e\in \A^1(V)$ and $\Delta_V\in \uMCor((V, (D+E)_{|V}), (V, E_{|V})\otimes (V,D_{|V}))$
is induced by the diagonal.
\end{lemma}
\begin{proof}

Set $G^1= \bigoplus_{x\in |E|^{(0)}} i_{x*} \Z$, where $i_x: \ol{x}\inj X$ is the closed immersion.
The complex $G:=[j_*\G_{m,U}\xr{\Div} G^1]$
 is a $\ul{\Gamma}_{|E|}$-acyclic resolution of $\G_{m,X}$ on $X_\Nis$ and hence
\[R\ul{\Gamma}_{|E|} \G_{m,X}= G^1[-1].\]
The map $\Z_X[-1]\to R\ul{\Gamma}_{|E|} \G_{m,X}$ corresponding to $E\in \CH^1_{|E|}(X)$ (see \ref{para:cyclecup})
is thus induced by the map of complexes $\varphi_E: \Z_X[-1]\to G^1[-1]$ which  sends $1\in\Z_X$ to the
Weil divisor $[E]$. Set $\sX':=(X, D+E)$. We obtain a commutative diagram
\[\xymatrix{
(\gamma^1F)_{\sX}\otimes_\Z \G_{m,X}\ar[rr]^-{\eqref{lem:adj-twist1}}\ar[d] &  &
F_{\sX}\ar[d]\\
(\gamma^1F)_{\sX}\otimes_\Z j_*\G_{m, U}\ar[rr]^-{\eqref{lem:adj-twist1}} & &
F_{\sX'}\ar@{^(->}[r]& j_* F_{\sU}.
}\]
where the last horizontal arrow is injective by semipurity. This yields a map on the cokernels (taken vertically)
\[(\gamma^1F)_{\sX}\otimes_\Z G^1\xr{\theta^1} F_{\sX'}/F_{\sX}\inj j_*F_{\sU}/F_{\sX}.\]
where the last morphism is injective again by semipurity. Set $C:=[F_{\sX'}\to F_{\sX'}/F_{\sX}]$, then we obtain
a morphism of complexes $\theta$ which fits in the commutative diagram of complexes
\eq{lem:divisor-cup4}{\xymatrix{
((\gamma^1F)_{\sX}\otimes_\Z \G_{m,X})[0]\ar[rr]^-{\eqref{lem:adj-twist1}}\ar[d] &  &
F_{\sX}[0]\ar[d]^{\rm qis}\\
(\gamma^1F)_{\sX}\otimes_\Z G\ar[rr]^-{\theta}& &
C.
}}
Consider the following diagram in $D(X_\Nis)$
\[
\xymatrix{
(\gamma^1F)_{\sX}[-1]\ar@{=}[dd]\ar[r]\ar@/^2pc/[rr]^{c_E}\ar[dr] &
 (\gamma^1F)_{\sX}\otimes_\Z^L R\ul{\Gamma}_{|E|}\G_{m,X}\ar[d]^{(*)}\ar[r] & 
 R\ul{\Gamma}_{|E|} F_{\sX}\ar[d]^\simeq\\
        &
 R\ul{\Gamma}_{|E|}((\gamma^1F)_{\sX}\otimes_\Z G)\ar[r]^-{\theta} & 
 R\ul{\Gamma}_{|E|} C\\
(\gamma^1F)_{\sX}[-1]\ar[r]^-{\id\otimes \varphi_E} & 
(\gamma^1F)_{\sX}\otimes_\Z G^1[-1]\ar[r]^-{\theta^1}\ar[u] & 
F_{\sX'}/F_{\sX}[-1],\ar[u]
}
\]
where the map $(*)$ is the composition
\mlnl{(\gamma^1F)_{\sX}\otimes_\Z^L R\ul{\Gamma}_{|E|}\G_{m,X}\xr{\eqref{lem:ten-supp1}}
R\ul{\Gamma}_{|E|}((\gamma^1F)_{\sX}\otimes_\Z^L \G_{m,X})\\
\xr{\simeq}  R\ul{\Gamma}_{|E|}((\gamma^1F)_{\sX}\otimes_\Z^L G)\xr{\rm nat.}
R\ul{\Gamma}_{|E|}((\gamma^1F)_{\sX}\otimes_\Z G)}
and, using that $G^1$ and $F_{\sX'}/F_{\sX}$ have support in $|E|$, the lower vertical maps 
are induced by the natural map $\ul{\Gamma}_{|E|}\to R\ul{\Gamma}_{|E|}$.
The upper half of the diagram commutes by the definition of $c_E$ (see \ref{para:cyclecup})
and the commutativity of \eqref{lem:divisor-cup4};  the lower half of the diagram 
commutes by the definition of the involved maps.
Thus $H^1(c_E)$ factors via \eqref{lem:divisor-cup0}. 
If $V$ is as in \eqref{lem:divisor-cup1}, then we can lift $[E]\in G^1(V)$ to $e\in \Gamma(V,j_*\G_m)$.
Therefore, formula \eqref{lem:divisor-cup2} follows directly from the definition of $\theta^1$ and 
Lemma \ref{lem:adj-twist}, where we have to use the fact that the graph of the composition 
$s_1:V\to \Spec k=\{1\}\inj \A^1$ defines an element in $\uMCor((V,\emptyset), (\A^1,0))$ and hence
$\Delta_V^*(\Gamma_{s_1}\times \nu)^*\tilde{a}$ vanishes modulo $F(V, D_{|V})$.
\end{proof}
\begin{remark} Let $C_*(-)\colon {\rm Comp}^+(\NST) \to {\rm Comp}^+(\NST)$ be the classical $\A^1$-fibrant replacement functor given by the Suslin complex \cite[2.14]{MVW}. When applied to $\Z_{\tr}(\G_m^{\wedge i})$, it gives an explicit model for the weight $i$ motivic complex $\Z(i)[i] = C_*\Z_{\tr}(\G_m^{\wedge i})$ \cite[3.1]{MVW}. Let $F\in \uMNST$. By adjunction, we get an evaluation pairing
    \[ \ul{\omega}^*(\Z(i)[i]) \otimes^{\mathbb{L}} \ul{\Hom}_{D(\uMNST)}(\ul{\omega}^*(\Z(i)[i]) , F[0]) \to F[0]\]
    where we note that $\ul{\omega}^*(\Z(i)[i])$ is still a bounded below complex of $\uMNST$, since the functor $\ul{\omega}^*$ is exact. 
    By taking $\mathcal{H}^0$ in the above pairing, we get an induced map
    \[ \ul{\omega}^*\mathcal{H}^0((\Z(i)[i])) \otimes \ul{\Hom}_{\uMNST}(\ul{\omega}^*\mathcal{H}^0(\Z(i)[i]) , F) \to F \]
    noting that $\ul{\omega}^*(\Z(i)[i])$ is concentrated in non negative degrees, and since there is an isomorphism $\ul{\omega}^*\mathcal{H}^0(\Z(i)[i]) \cong  \ul{\omega}^*K^M_i$, this reads, for every $\sX= (X,D) \in \uMCor$, as 
    \[K^M_{i,X} \tensor (\gamma^iF)_{\sX} \to F.\]
    By construction, it agrees with the cup product pairing \eqref{eq;pairing_cup}. We will not use this extended version of the pairing in the rest of the paper. 
\end{remark}

\section{Projective bundle formula}\label{sec:pbf}

\begin{para}\label{para:pbf}
Let $F\in \CItspNis$ and $\sX=(X,D)\in\uMCorls$. Let $V$ be a locally free $\sO_X$-module of rank $n+1$.
Denote by 
\[\pi: P= \P(V)=\Proj(\Sym^\bullet_{\sO_X}V)\to X\]
the projection of the corresponding projective bundle and set $\sP:=(P, \pi^*D)$.
Let $\xi:=c_1(\sO_P(1))\in \CH^1(P)$ be the first Chern class of the hyperplane 
line bundle $\sO_P(1)$  and denote by $\xi^i\in \CH^i(P)$ its $i$-fold self-intersection. 
We denote by $\lambda^i_V$ the composition in $D(X_{\Nis})$
\[\lambda^i_V: (\gamma^i F)_{\sX}[-i]\xr{\pi^*} R\pi_* (\gamma^i F)_{\sP}[-i]
\xrightarrow[\eqref{para:cyclecup1}]{c_{\xi^i}} R\pi_* F_{\sP}, 
\quad 0\le i \le n.\]
We thus get a map
\eq{para:pbf1}{\lambda_V=\sum_{i=0}^n \lambda^i_V: 
\bigoplus_{i=0}^n (\gamma^i F)_{\sX}[-i]\to R\pi_* F_{\sP}.}
\end{para}

\begin{lemma}\label{lem:P1bf}
Let $F$ and $(X,D)$ be as in \ref{para:pbf}. Consider the projection $\pi: \P^1\times X\to X$.
Then 
\[H^1(\lambda_V^1): (\gamma^1F)_{\sX}\xr{\simeq} R^1\pi_* F_{\P^1\otimes \sX}\]
is an isomorphism and $R^i\pi_*F_{\P^1\otimes \sX}=0$, for all $i\ge 2$.
\end{lemma}
\begin{proof}
Set $\sP:=\P^1\otimes \sX$ 
and define $C^1$ by the exact sequence
\eq{lemP1bf}{0\to F_{\sP}\to \underbrace{F_{(\P^1, 0)\otimes \sX}\to C^1}_{=:C}\to 0,}
where the first map is injective by the semipurity of $F$. Since $C^1$ has support in $0\times X$ and 
$R^i\pi_*F_{(\P^1,0)\otimes \sX}=0$, for all $i\ge 1$, by the cube invariance of the cohomology, \cite[Thm 9.3]{S-purity}, it follows that 
$C$ is a $\pi_*$-acylic resolution of $F_{\sP}$ sitting in degree $[0,1]$. This proves the vanishing statement.
Furthermore, we claim that
\[\pi_*C^1= (\gamma^1F)_{\sX}.\]
Indeed,  for $U\to X$ \'etale and $\sU=(U,D_U)$ we have 
\begin{equation}\label{eq:lemP1bf2}(\gamma^1F)_{\sX}(U) = \frac{F((\A^1,0)\otimes \sU)}{F(\A^1\otimes\sU)}       
                                = \pi_*C^1(U),
\end{equation}
where the first equality holds by Remark \ref{rmk:gamma};
the second equality holds by the following two observations:
since $C^1$ is supported on $0\times X$, we have an exact sequence on $\P^1\times X$
\[0\to j_* F_{\A^1\otimes \sX}\to j_*F_{(\A^1, 0)\otimes \sX}\to C^1\to 0,\]
where $j :\A^1\times X\inj \P^1\times X$ is the open immersion, 
and by the Lemmas \ref{lem:comp-SNCD} and \ref{lem:higher-HI} we have 
\[R^1\pi_*( j_* F_{\A^1\otimes \sX})= R^1(\pi j)_* F_{\A^1\otimes \sX}=0.\]
Furthermore, the map $\pi_*F_{\sP}\to \pi_*F_{(\P^1,0)\otimes \sX}$
 is an equality by Corollary \ref{cor:PnInv} and hence 
 {\eqref{lemP1bf}, \eqref{eq:lemP1bf2} and $R^1\pi_*F_{(\P^1,0)\otimes \sX}=0$ imply}
$R^1\pi_* F_\sP=\pi_* C^1=(\gamma^1F)_\sX$. 
 {It remains} to show that  $H^1(\lambda_V^1):(\gamma^1F)_\sX \to R^1\pi_*F_\sP=(\gamma^1F)_\sX$
realises such isomorphism. To this end, note that on $U\to X$ 
we can identify $H^1(\lambda_V)$ by Lemma \ref{lem:divisor-cup}
 with the composition
\mlnl{\gamma^1 F(\sU)\xr{\pi^*} \gamma^1 F(\P^1\otimes \sU)
=\frac{F((\A^1,0)\otimes \P^1\otimes \sU)}{F(\A^1\otimes \P^1\otimes \sU)} \\
\xr{\Delta_{\A^1_U}^* (\Gamma\times \nu)^*} 
\frac{F((\A^1,0)\otimes \sU)}{F(\A^1\otimes\sU)}=\gamma^1 F(\sU),}
where $\nu: \A^1_U\inj \P^1_U$ is the open immersion and
we identify $\xi$ on $\P^1_U$ with the class of the divisor $0\times U$
and  where $\Gamma$ is the graph of the projection $\A^1\times U\to \A^1$.
Thus the equality
\[(\id_{\A^1}\times\pi)\circ(\Gamma\times \nu)\circ\Delta_{\A^1_U}=\id_{\A^1_U}\]
implies the statement.
\end{proof}
We are now ready to prove the projective bundle theorem in our setting. 

 {
\begin{thm}\label{thm:pbf}
The map \eqref{para:pbf1} is an isomorphism in $D(X_{\Nis})$.
\end{thm}
\begin{proof}
The question is local on $X$. Hence the statement follows for $n=1$ from Lemma \ref{lem:P1bf}.
We now assume $n\ge 2$ and $P=\P^n_X$. 
Consider the diagram
\[\xymatrix{
E\ar[r]^{i_E}\ar[d]_{\rho_E} &Y\ar[r]^q\ar[d]^\rho   & E\ar[d]^{\rho_E}\\
X\ar[r]^i &P\ar[r]^{\pi} & X,
}\]
where $\pi$ is the projection, $i$ a section of $\pi$, $\rho$ the blow-up of $P$ in $i(X)$, $i_E$ is the closed immersion 
of the exceptional divisor, and  $q$ is the standard map, which identifies 
$Y$ with $\P(W)$, where $W:=\sO_E\oplus\sO_E(1)$.
Denote by $V$ (resp. $V_E$) the trivial $\sO_X$-module of rank $n+1$ (resp. $n$) defining
the projective bundle $P$ (resp. $E$) over $X$ (recall we work locally on $X$). We set $\sE:=(E, D_{|E})$ and $\sY:=(Y,D_{|Y})$.
By Corollary \ref{cor:PnInv} we have 
\eq{thm:pbf1}{R\pi_*F_{\sP}\cong F_{\sX}\oplus \tau^{\ge 1}R\pi_*F_{\sP},}
where the map $F_{\sX} =\pi_* F_{\sP} \to R\pi_*F_{\sP}$ is split by the section $i$.
Thus applying $R\pi_*$ to the exact triangle from Theorem \ref{thm:bus} induced by the blow-up $Y\to P$,
we can split off $F_\sX$ to obtain the triangle on the bottom of the following diagram:
\[\xymatrix{
\oplus_{i=1}^n (\gamma^i F)_{\sX}[-i]\ar[d]^{\oplus_i \lambda^i_V}\ar[rr]^-{(0, \lambda_{V_E}[-1])}& 
&
R\rho_{E*}F_{\sE}\oplus R\rho_{E*}\gamma^1F_{\sE}[-1] \ar[r]^-{\id+0}\ar[d]^{\lambda_W}&
R\rho_{E*}F_{\sE}\ar@{=}[d]\\
\tau^{\ge 1}R\pi_* F_{\sP}\ar[rr]^{\rho^*_1}&
& 
R\pi_*R\rho_*F_{\sY}\ar[r]^-{i_E^*}&
R\rho_{E*}F_{\sE},
}\]
where the map labeled by $\rho^*_1$ is induced by \eqref{thm:pbf1} and $\rho^*$.
Note that the bottom triangle is actually split since $i_E^*q^*=\id$. 
In the top left $\lambda_{V_E}$ is applied to $\sE/\sX$ and $(\gamma^1F)_{\sX}$ and it is an isomorphism by induction. 
Hence the top sequence is a split triangle as well. Since $\lambda_W$ is also an isomorphism by the $n=1$ case,
it remains to show that the diagram is commutative. 
Let 
$\xi=c_1(\sO_{P}(1))$, $\xi_E= c_1(\sO_{E}(1))$ and $\eta=c_1(\sO_{\P(W)}(1))$
be the first Chern classes of the corresponding fundamental line bundles.
The commutativity of the right square follows from 
$i_E^*\circ c_\eta=c_{i_E^*\eta}\circ i_E^*=0$,
where we use Lemma \ref{lem:cyclecup}\ref{lem:cyclecup2} for the first equality and $i_E^*\sO_{\P(W)}(1)=\sO_E$ for the second.
The commutativity of the left square reduces by Lemma \ref{lem:cyclecup}\ref{lem:cyclecup2} to the equality
\[\rho^*(\xi^i)=  q^*(\xi_E^{i-1})\cdot \eta \quad \in \CH^i(Y),\]
which is well-knwon and straightforward to check.
\end{proof}
}

\section{The Gysin triangle}\label{sec:gysin}
 {We begin with an elementary lemma on split exact triangles in an arbitrary triangulated category. }

\begin{lemma}\label{lem:can-sec}
Let $A\xr{a}B \xr{b} C\xr{\partial}A[1]$ be an exact triangle in a triangulated category $T$.
\begin{enumerate}[label=(\arabic*)]
\item\label{lem:can-sec1} If $\tau\colon C\to B$ is a section of $b$, then there is a unique  map $\sigma\colon B\to A$,
such that $a\circ \sigma = \id_B- e$, where $e:=\tau \circ b\in \Hom_T(B,B)$. Moreover $\sigma$ is a  {retraction} of $a$.
\item\label{lem:can-sec2} If $\sigma:B\to A$ is a  {retraction} of $a$, then there is a unique map $\tau: C\to B$,
such that $\tau\circ b= \id_B- \epsilon$, where $\epsilon:= a\circ \sigma\in \Hom_T(B,B)$. 
Moreover $\tau$ is a section of $b$
\end{enumerate}
In \ref{lem:can-sec1} (resp. \ref{lem:can-sec2}) we call $\sigma$ (resp. $\tau$) 
{\em the canonical  {retraction} defined by $\tau$ (resp.   {the canonical section defined by} $\sigma$)}.
 {Moreover, the canonical section defined by the canonical retraction of $\tau$ is equal to $\tau$, 
and similarly with $\sigma$.}
\end{lemma}
\begin{proof}
\ref{lem:can-sec1}. By the existence of the section $\tau$ the long exact sequence
stemming from applying $\Hom_T(B,-)$ to the exact triangle $(a,b,\partial)$
breaks up into short exact sequences; in particular
we obtain the short exact sequence
\[0\to \Hom_T(B, A)\xr{a\circ} \Hom_T(B,B)\xr{b\circ } \Hom_T(B,C)\to 0.\]
This gives a unique $\sigma$ with $a\circ \sigma= \id_B-e$. 
It follows that $a\circ (\sigma\circ a)= a$. Since
also $a\circ :\Hom_T(A,A)\to \Hom_T(A,B)$ is injective (by the existence of $\tau$), we see that
$\sigma$ is a  {retraction}. 
Similar for \ref{lem:can-sec2}; the other statements are clear.
\end{proof}

\begin{para}\label{para:splitting} {Having the projective bundle formula and the blow-up formula at disposal, we can construct the Gysin triangle by formally following the procedure indicated by Voevodsky in \cite[3.5]{V-TCM}.  Note that our statement is sheaf-theoretic, and therefore the arrows are reversed compared to \emph{loc.cit.} We begin by setting the notation}.
 
Let $F\in \CItspNis$ and $\sX=(X,D)\in \uMCorls$ (see Notation \ref{nota:MCorls}).
Let $i: Z\inj X$ be a smooth closed subscheme intersecting $D$ transversally (see Definition \ref{defn:ti}).
Denote by $\rho\colon \tilde{X}\to X$ the blow-up of $X$ along $Z$,   {and let $\rho_E\colon E = \rho^{-1}(Z) \to Z$ be the exceptional divisor. We define the modulus pairs 
$\tilde{\sX}= (\tilde{X}, D_{|\tilde{X}})$, $\sZ=(Z, D_{|Z})$ and $\sE=(E, D_{|E})$ with the obvious convention on the divisor.

Let $\rho_1\colon Y\to X\times \P^1$ be the blow-up of $X\times \P^1$ along $Z\times 0$ and  let $E_1=\rho_1^{-1}(Z\times 0)$ be the exceptional divisor.
We obtain the following modulus pairs in $\uMCorls$
\[\sY=(Y, (D\times \P^1+ X\times\infty)_{|Y}),  \quad \sE_1= (E_1, (D\times 0)_{|E_1})\]
and the following obvious morphisms in $\uMCor$
\[i: \sZ\to \sX, \quad i_{Z0}: \sZ=\sZ\otimes 0\to \sX\otimes\bcube,\] 
\[\rho: \tilde{\sX}\to \sX, \quad \rho_1: \sY\to \sX\otimes \bcube, \quad i_{\tilde{X}}: \tilde{\sX}\to \sY,\]
\[i_{E}: \sE\to \tilde{\sX}, \quad i_{E_1}: \sE_1\to \sY,\quad i_{EE_1}: \sE\to \sE_1,\]
\[\rho_{E}: \sE\to \sZ, \quad \rho_{E_1}: \sE_1\to \sZ\otimes 0=\sZ,\]
\[i_\epsilon: \sX= \sX\otimes\epsilon \to \sX\otimes\bcube, \quad \epsilon\in\{0,1\},\quad \pi: \sX\otimes\bcube\to \sX,\]
where all the $i$'s are induced by closed immersions of the underlying schemes, in particular $i_{\tilde{X}}$ is induced by
identifying $\tilde{X}$ with the strict transform of $X\times 0$ in $Y$ and $i_{EE_1}$ is the pullback 
of $i_{\tilde{X}}$ along $i_{E_1}$. This gives the following commutative diagram in $\uMCor$: }
\[\begin{tikzcd}[row sep={40,between origins}, column sep={40,between origins}]
      & \sE \arrow[rr, "i_E"] \arrow[dd, "\rho_E", near start]\arrow[dl, "i_{EE_1}"', near start]& & \tilde{\sX} \arrow[dd, "\rho"] \arrow[dl, "i_{\tilde{X}}"] \\
    \sE_1 \arrow[dd, "\rho_{E_1}" ] \arrow[rr, crossing over, "i_{E_1}"', near start]  & & \sY   \\
      & \sZ   \arrow[rr, "i"', near start] \arrow[dl, equal]& & \sX \arrow[dl, "i_0"] \\
   \sZ\otimes 0 \arrow[rr, "i_{Z0}"] && \sX\otimes \bcube \arrow[ru, bend right = 50,  "\pi"']\arrow[uu, crossing over, "\rho_1"', near end, leftarrow] 
\end{tikzcd}\]
 {We will denote the underlying morphisms of schemes by the same letter.
Note that $i_1:\sX=\sX\otimes 1\to \sX\otimes \bcube$ extends canonically to a morphism
\[i_{1,Y}: \sX\to \sY.\]
Finally the morphism underlying $i_{EE_1}$ is equal to the natural inclusion 
\[i_{EE_1}: \P(\sN_{Z/X}^\vee)\inj \P(\sO_{Z}\oplus \sN_{Z/X}^\vee),\]
where $\sN_{Z/X}^\vee=\sI/\sI^2$ is the conormal sheaf, $\sI$ being the ideal sheaf of $Z\inj X$.  We obtain the following diagram, 
\eq{para:splitting2}{\xymatrix{
\bigoplus_{i=0}^{j} (\gamma^i F)_{\sZ}[-i]\ar[d]\ar[rr]_-\simeq^-{\lambda_{\sO_{Z}\oplus \sN_{Z/X}^\vee}} & &
R\rho_{E_1*} F_{\sE_1}\ar[d]^{i_{EE_1}^*}\\
\bigoplus_{i=0}^{j-1} (\gamma^i F)_{\sZ}[-i]\ar[rr]^-{\lambda_{\sN_{Z/X}^\vee}}_-\simeq& &
R\rho_{E*}F_{\sE},
}}
where $j=\codim(Z,X)$, the horizontal maps are the isomorphisms from Theorem \ref{thm:pbf}, and the vertical
map on the left is the projection. Using Lemma \ref{lem:cyclecup} it is direct to check that
\eqref{para:splitting2} commutes. Thus $i_{EE_1}^*$ in \eqref{para:splitting2} has a canonical section
\[s: R\rho_{E*} F_{\sE}\to R\rho_{E_1*}F_{\sE_1},\]
splitting off the summand $(\gamma^j F)_{\sZ}[-j]$.
Let $b_1$ be the morphism in $D(X_{\Nis})$ defined as the composition}
\[\begin{tikzcd}
 R\pi_*R\rho_{1*} F_{\sY} \arrow[r, "{i_{E_1}^*}"] &  i_* R\rho_{E_1*} F_{\sE_1}  \arrow[r,   "\simeq"] & \bigoplus_{i=0}^{j}i_* (\gamma^i F)_{\sZ}[-i] \arrow[r] & \bigoplus_{i=1}^{j}i_* (\gamma^i F)_{\sZ}[-i]
 \end{tikzcd}
\]
 {where the second isomorphism is the inverse of $\lambda_{\sO_{Z}\oplus \sN_{Z/X}^\vee}$, and the rightmost arrow is the canonical projection. Similarly, we define $b$ as the composition}
\[\begin{tikzcd}
 R\rho_{*} F_{\tilde{ \sX}} \arrow[r, "{i_{E}^*}"] &  i_* R\rho_{E*} F_{\sE}  \arrow[r,   "\simeq"] & \bigoplus_{i=0}^{j-1} i_*(\gamma^i F)_{\sZ}[-i] \arrow[r] & \bigoplus_{i=1}^{j-1}i_* (\gamma^i F)_{\sZ}[-i]
 \end{tikzcd}
\]
 {where the second isomorphism is the inverse of $\lambda_{\sN_{Z/X}^\vee}$ and the last map is the canonical projection.  This gives the following commutative diagram of solid arrows in $D(X_{\Nis})$:}
\begin{equation}\label{para:splitting3} 
\begin{tikzcd}[column sep={50}]
R\pi_* F_{\sX\otimes \bcube} \arrow[r, "{\rho_1^*}"] \arrow[d, "i_0^*=i_1^*"', "\simeq"]& R\pi_{*} R\rho_{1 *} F_{\sY} \arrow[l, dashrightarrow, bend right = 38, "\sigma_1"']  \arrow[dl, "i^*_{1,Y}"', description] \arrow[d, "i^*_{\tilde{X}}"]\arrow[r, "b_1"] &  \bigoplus_{i=1}^{j}i_* (\gamma^i F)_{\sZ}[-i] \arrow[d, "h"]  \arrow[l, dashrightarrow, bend right = 30, "\tau_1"'] \\
F_{\sX} \arrow[u, bend left =70, "\pi^*"] \arrow[r, "\rho^*"] & R\rho_*F_{\tilde{\sX}}\arrow[r, "b"] \arrow[lu, phantom, "\not\circlearrowright", near start ] \arrow[l, dashrightarrow, bend left = 38, "\sigma"'] & \bigoplus_{i=1}^{j-1} i_*(\gamma^i F)_{\sZ}[-i], \arrow[u, bend right = 30, "s"'] \arrow[l, dashrightarrow, bend left = 30, "\tau"']
 \end{tikzcd}
\end{equation}
 {where the dashed arrows are defined as follows. First, note that  the bottom horizontal sequence is a distinguished triangle obtained from the distinguished blow-up triangle from Theorem \ref{thm:bus} for $(\sX,\sZ)$. Indeed,  we have the following diagram in $D(X_{\Nis})$:}
\[
\begin{tikzcd}[column sep={50}]
 F_{\sX} \arrow[r, "\rho^* \oplus (-i^*)" ] \arrow[d, equal]& R\rho_* F_{\tilde{\sX}} \oplus i_*F_{\sZ} \arrow[r, "i^*_E+\rho_E^*"] \arrow[d]& i_*R\rho_{E*} F_{\sE} \arrow[d, "\mathrm{can.}"] & \arrow[l, "\lambda_{\sN_{Z/X}^\vee}"', "\simeq"] \bigoplus_{i=0}^{j-1}(\gamma^iF)_{\sZ}[-i] \arrow[d]\\
  F_{\sX} \arrow[r, "\rho^*"] & R\rho_*F_{\tilde{\sX}} \arrow[r, "\mathrm{can.}\circ (i_E^*)"] & i_*\tau_{\geq 1} R\rho_{E*}F_{\sE}  & \arrow[l, "\sum_{i=1}^{j-1}\lambda_{V}^i"', "\simeq"] \bigoplus_{i=1}^{j-1}(\gamma^iF)_{\sZ}[-i]
\end{tikzcd}
\]
 {where $V=\sN_{Z/X}^\vee$, the unlabeled maps are the natural projections, $\mathrm{can.}$ is the canonical map to the truncation $\tau_{\geq 1}$,  and the composition of the second and the (inverse of the) third arrow in the bottom line is $b$. Every square is  commutative, and the middle one is clearly homotopy cartesian, so that the bottom line is part of a distinguished triangle as required. 
In a similar fashion, the top line of \eqref{para:splitting3} is  part of a distinguished triangle obtained from the blow-up triangle for $(\sX\otimes \bcube, Z\times 0)$ after applying $R\pi_*$. 

Going back to \eqref{para:splitting3}, note that the right square is commutative thanks to the definitions of $b$ and $b_1$, and the commutativity of \eqref{para:splitting2}. The square on the left commutes   if the left vertical arrow is  $i_0^*$. By cube-invariance it is an isomorphism, with inverse $\pi^*$ and thus it is equal to $i_1^*$. Replacing $i_0^*$ with $i_1^*$ we see then that the upper triangle in the left square is commutative. The lower triangle on the other hand does not commute. Set
\[\sigma_1:= \pi^*\circ i_{1,Y}^*;\]
it is a retraction of $\rho_1^*$. We define the (dotted) map $\tau_1$ as  the canonical section
defined by $\sigma_1$, see Lemma \ref{lem:can-sec}.
Set
\begin{equation}\label{cor:bud0}\tau:= i_{\tilde{X}}^*\circ\tau_1\circ s;
\end{equation}
it is a section of $b$ (that we can identify with $(\mathrm{can.}\circ i_{E}^*)$ up to the  isomorphism $\tau_{\geq1} R\rho_{E*} F_{\sE} \simeq \bigoplus_{i=1}^{j-1}(\gamma^iF)_{\sZ}[-i]$).
We define $\sigma$ as the canonical section defined by $\tau$ (again see Lemma \ref{lem:can-sec}). The proof of the following corollary is immediate from the previous constructions.
\begin{cor}\label{cor:bud}Let $F$ and $\rho: \tilde{\sX}\to \sX$ and $\sZ$ be as in \ref{para:splitting} and $j=\codim(Z,X)$.
We have isomorphisms in $D(X_\Nis)$ 
\eq{cor:bud1}{ F_{\sX}\oplus \bigoplus_{r=1}^{j-1} i_{*}\gamma^r F_{\sZ}[-r]\xrightarrow[\simeq]{\rho^*+ \tau } R\rho_* F_{\tilde{\sX}}.}
\end{cor}}
\end{para}
 {
\begin{para}\label{para:gysin}
Let $F\in \CItspNis$, $\sX=(X,D)\in \uMCorls$ and let $i:Z\inj X$ be a smooth closed subscheme of
codimension $j\ge 1$ intersecting $D$ transversally.
Set $\sZ=(Z,D_{|Z})$. 
Following \cite[p. 220/21]{V-TCM} we define the Gysin map
\eq{para:gysin1}{g_{\sZ/\sX}: i_{*}(\gamma^j F_{\sZ})[-j]\to F_{\sX}}
as follows: let the notation be as in \eqref{para:splitting3}. We set
\[ v\colon i_*(\gamma^jF)_{\sZ}[-j] \longrightarrow R\pi_*R_{\rho_{1*}} F_{\sY} \]
as the composition $\tau_1 \circ \nu^j$, where $\nu^j$ is the canonical inclusion $i_*(\gamma^jF)_{\sZ}[-j] \to \bigoplus_{i=1}^j i_*(\gamma^iF)_{\sZ}[-i]$. Then the Gysin map is 
\[g_{\sZ/\sX} := - \sigma \circ i_{\tilde{X}}^* \circ v. \]
\end{para}
}
\begin{rmk}\label{rmk:gysin_lambda0} { We can replace the top row in \eqref{para:splitting3} with the equivalent one}
\begin{equation}\label{eq:gysin_lambda1}
    \begin{tikzcd}
    R\pi_* F_{\sX\otimes \bcube} \arrow[r, "{\rho_1^*}"] & R\pi_{*} R{\rho_{1*}} F_{\sY} \arrow[l, dashrightarrow, bend right = 38, "\sigma_1"']  \arrow[r, "\mathrm{can.} \circ i_{E_1}^*"] &  i_*\tau_{\geq1} R\rho_{E_{1*}}F_{\sE_1}. \arrow[l, dashrightarrow, bend right = 30, "\tau_1'"'] 
    \end{tikzcd}
\end{equation} 
 {where $\tau_1'$ is the canonical section induced by $\sigma_1$. Note that $\tau_1 = \tau_1'\circ \sum_{i=1}^{j} \lambda^i_{V_1}$ where $V_1=  \sO_Z\oplus \sN^\vee_{Z/X}$, 
and  $\mathrm{can.}$ is the canonical map $i_* R\rho_{E_{1*}}F_{\sE_1}\to i_*\tau_{\geq1} R\rho_{E_{1*}}F_{\sE_1}$. In particular the Gysin map satisfies 
\begin{equation}\label{eq:gysin_lambda2}
    g_{\sZ/\sX} = -\sigma\circ i_{\tilde{X}}^* \circ \tau_1'\circ \lambda_{V_1}^j.
\end{equation}}
\end{rmk}

 {
\begin{rmk}
The Gysin map can be described alternatively as follows.  Set 
\[\beta:=  i_{1, Y}^* - \sigma\circ i_{\tilde{X}}^*: R\pi_* R\rho_{1*}F_{\sY} \to F_{\sX}.\]
We have $\beta\circ (\rho_1^*)=0: R\pi_* F_{\sX\otimes\bcube}\to F_{\sX}$, and since the top row in \eqref{para:splitting3} is split exact, there exists a unique map \[\beta_1 \colon  \bigoplus_{i=1}^j i_*(\gamma^iF)_{\sZ}[-i] \to  F_{\sX}\]
such that $\beta = \beta_1 \circ b_1$. Then a diagram chase shows that
\begin{equation}\label{lem:gysin} g_{\sZ/\sX}= \beta_1 \circ \nu^j.\end{equation}
\end{rmk}

\begin{rmk}
Let $i:Z\inj X$ be as above. One can consider the class $[Z]$ of $Z$ in the Chow group with support $\CH^j_Z(X)$, and the cup product construction \eqref{para:cyclecup1} gives a morphism 
\begin{equation}\label{eq:cyclecup-vs-gysin}c_{[Z]}\colon (\gamma^j F)_{\sX}[-j]\to  R\ul{\Gamma}_{Z}F_{\sX} \to F_{\sX}
\end{equation}
 in  $D(X_{\Nis})$, where the last morphism is the forget support map. It is a natural question to compare \eqref{eq:cyclecup-vs-gysin} and \eqref{para:gysin1}: this is done in Theorem \ref{prop:pf-gysin} below. 
\end{rmk}
}

 {
\begin{prop}\label{prop:gysin-cup}
Let $F\in \CItspNis$, $\sX=(X,D)$, $\sZ=(Z, D_{|Z})$ be as in \ref{para:gysin} above. 
Assume $X$ is equidimensional and let $\Phi$ be a family of supports on $X$ and  $\alpha\in \CH_\Phi^r(X)$. Then the following diagram
\[\xymatrix{
i_*(\gamma^{j+r}F)_{\sZ}[-j-r] \ar[r]^-{g_{\sZ/\sX}} \ar[d]^{c_{i^* \alpha}} & \gamma^r F_{\sX}[-r] \ar[d]^{c_{\alpha}}\\
i_* R\ul{\Gamma}_{Z\cap \Phi}  \gamma^j F_{\sZ}[-j]= R\ul{\Gamma}_{\Phi}i_* \gamma^j F_{\sZ}[-j] \ar[r]^-{g_{\sZ/\sX}} & R\ul{\Gamma}_{\Phi}F_{\sX}
}
\]
commutes, where $i^*\alpha\in \CH^r_{\Phi\cap Z}(Z)$ is the refined pullback of $\alpha$ (see Lemma \ref{lem:pbKCH}).
\end{prop}
\begin{proof}
We use the notation from \ref{para:splitting} and in particular \eqref{para:splitting3} above.
We compute:
\begin{align*}
g_{\sZ/\sX}\circ c_{i^*\alpha} & 
\stackrel{(*)}{=} -\sigma\circ i_{\tilde{X}}^*\circ \tau_1'\circ \lambda_{V_1}^j\circ c_{i^*\alpha}\\
& \stackrel{(2*)}{=} -\sigma\circ i_{\tilde{X}}^*\circ \tau_1'\circ c_{\xi^j}\circ \rho_{E_1}^*\circ c_{i^*\alpha}\\
&\stackrel{(3*)}{=} -\sigma\circ i_{\tilde{X}}^*\circ \tau_1'\circ c_{(i\rho_{E_1})^*\alpha} \circ \lambda_{V_1}^j \\
&\stackrel{(4*)}{=} -\sigma\circ i_{\tilde{X}}^* \circ \tau_1'\circ c_{(i\rho_{E_1})^*\alpha}\circ ((\mathrm{can.}\circ i_{E_1}^*) \circ \tau_1')\circ \lambda_{V_1}^j\\
&\stackrel{(5*)}{=} - \sigma\circ i_{\tilde{X}}^*\circ \tau_1'\circ  (\mathrm{can.}\circ i_{E_1}^*)\circ
                  c_{\rho_1^*\pi^*\alpha}  \circ \tau_1'\circ \lambda_{V_1}^j \\
&\stackrel{(6*)}{=} - \sigma\circ i_{\tilde{X}}^* \circ (\id- \rho_1^*\circ \sigma_1)\circ
                  (c_{\rho_1^*\pi^*\alpha}) \circ \tau_1'\circ \lambda_{V_1}^j \\
&\stackrel{(7*)}{=} - \sigma\circ i_{\tilde{X}}^*\circ  c_{\rho_1^*\pi^*\alpha}   \circ 
                             (\id - \rho_1^*\circ \sigma_1)\circ \tau_1'\circ \lambda_{V_1}^j \\
&\stackrel{(8*)}{=} -\sigma\circ i_{\tilde{X}}^*\circ  c_{\rho_1^*\pi^*\alpha} \circ 
                             \tau_1'\circ \lambda_{V_1}^j \\
&\stackrel{(9*)}{=} - \sigma\circ  c_{\rho^*\alpha} \circ i_{\tilde{X}}^* \circ 
                             \tau_1'\circ \lambda_{V_1}^j \\
&\stackrel{(10*)}{=} - \sigma\circ  c_{\rho^*\alpha} \circ (\rho^*\circ \sigma +\tau\circ b)
                              \circ i_{\tilde{X}}^* \circ \tau_1'\circ \lambda_{V_1}^j \\
&\stackrel{(11*)}{=} - \sigma\circ \rho^*  \circ c_{\alpha}  \circ \sigma \circ i_{\tilde{X}}^*\circ \tau_1\circ \nu^j \\
&\stackrel{(12*)}{=}- c_{\alpha}  \circ \sigma \circ i_{\tilde{X}}^* \circ \tau_1\circ \nu^j \\
&\stackrel{(13*)}{=}c_\alpha\circ g_{\sZ/\sX},
\end{align*}
where 
\begin{enumerate}
\item[$(*)$] holds by \eqref{eq:gysin_lambda2},
\item[$(2*)$] holds by the definition of $\lambda_{V_1}^j$ in \ref{para:pbf},
\item[$(3*)$] holds by Lemma \ref{lem:cyclecup}\ref{lem:cyclecup2}, \ref{lem:cyclecup3} and 
                 the definition of $\lambda_V^j$,
\item[$(4*)$] holds by $ ((\mathrm{can.}\circ i_{E_1}^*) \circ \tau_1') = \id$ on $i_*\tau_{\geq 1 }R \rho_{E_{1*}}F_{\sE_1}$,
\item[$(5*)$] holds  by  $i\rho_{E_1}=\pi\rho_1i_{E_1}$,
                      and Lemma \ref{lem:cyclecup}\ref{lem:cyclecup2},
\item[$(6*)$] holds since  $\tau_1'$ is the section defined by the section $\sigma_1$ (see \ref{lem:can-sec}),
\item[$(7*)$] follows again from Lemma \ref{lem:cyclecup}\ref{lem:cyclecup2} and the fact that $\sigma_1 = \pi^*\circ i_{1, Y}^*$ and $\pi\rho_1 i_{1,Y}=\id$. 
\item[$(8*)$] holds by $(\id - \rho_1^*\circ \sigma_1)\circ \tau_1'=\tau_1'$,
\item[$(9*)$] holds by  $\pi\rho_1 i_{\tilde{X}}= \rho$,
and Lemma \ref{lem:cyclecup}\ref{lem:cyclecup2},
\item[$(10*)$] holds by $\id= \rho^*\circ \sigma + \tau\circ b$,
\item[$(11*)$] holds by  Lemma \ref{lem:cyclecup}\ref{lem:cyclecup2}, $\tau_1'\circ \lambda^j_V=\tau_1\circ\nu^j$, see Remark \ref{rmk:gysin_lambda0}, and 
\[b\circ i_{\tilde{X}}^*\circ \tau_1'\circ \lambda_V^j= h\circ b_1\circ \tau_1 \circ \nu^j= h\circ \nu^j=0,\]
\item[$(12*)$] holds by $\sigma\circ \rho^*=\id$,
\item[$(13*)$] holds by definition of the Gysin map in \ref{para:gysin};
\end{enumerate}
whence the statement.
\end{proof}

}

\begin{prop}\label{prop:gysin-bc}
Let $\sX=(X,D)\in\uMCorls$ and  $F\in\CItspNis$.
Let $i:Z\inj X$ be a smooth closed subscheme of codimension $j$, 
intersecting $D$ transversally (see Definition \ref{defn:ti}) and set $\sZ=(Z, D_{|Z})$. 
Let $X'\in \Sm$ and let 
 {
$f: X'\to X$ be a morphism, such that 
$i': Z':=X'\times_X Z \inj X'$ is a smooth closed subscheme of pure codimension $m\le j$
and $|f^*D|$ is a simple normal crossing divisor intersecting $Z'$ transversally.
Set $\sX':= (X', f^*D)$ and $\sZ':=(Z', (f^*D)_{|Z'})$. Let $f_Z:Z'\to Z$ be the base change of $f$ and 
consider the excess normal bundle
\[\sE xc= f_{Z}^*(\sN_{Z/X})/\sN_{Z'/X'}.\]
Then the following diagram commutes
\begin{equation}\label{eq:exc-inters-formulaGysin}\xymatrix{
i_*\gamma^j F_{\sZ}[-j]\ar[d]_{f^{*}_Z}\ar[rr]^-{g_{\sZ/\sX}} & &
F_{\sX}\ar[d]^{f^*} \\
Rf_* (i'_*\gamma^j F_{\sZ'}[-j])\ar[r]^-{c_\beta} & 
Rf_* (i'_*\gamma^{m} F_{\sZ'}[-m])\ar[r]^-{g_{\sZ'/\sX'}} &
Rf_* F_{\sX'},
}\end{equation}
where $\beta:=c_{j-m}(\sE xc)\cap Z'\in \CH^{j-m}(Z')$ and $c_{j-m}(\sE xc)$
is the $(j-m)$th Chern class of $\sE xc$.
In particular, the Gysin map is compatible with smooth base change.
}
\end{prop}
\begin{proof}
First observe that given two distinguished triangles $\Delta$ and $\Delta'$ with sections
$\sigma$ and $\sigma'$ (resp. $\tau$ and $\tau'$) as in Lemma \ref{lem:can-sec},
and $\varphi:\Delta\to \Delta'$ is a morphism of triangles which commutes with
$\sigma$ and $\sigma'$ (resp. $\tau$ and $\tau'$), then $\varphi$  also commutes with the
canonical sections defined by these sections.
From this and the commutative diagram  
\eq{prop:gysin-bc1}{\xymatrix{
{\rm Bl}_{Z'}(X')\ar@{^(->}[r]\ar@{^(->}[d] &  
{\rm Bl}_{Z}(X)\times_X X'\ar[r]\ar@{^(->}[d] &
{\rm Bl}_{Z}(X)\ar@{^(->}[d]\\ 
{\rm Bl}_{Z'\times 0}(X'\times \P^1)\ar@{^(->}[r]\ar[dr] &  
{\rm Bl}_{Z\times 0}(X\times \P^1)\times_X X'\ar[r]\ar[d] &
{\rm Bl}_{Z\times 0}(X\times \P^1)\ar[d]\\
& X'\times \P^1\ar[r] & X\times \P^1
}}
(and all the diagrams which this diagram induces)
we see that the various pullbacks induced by $f: X'\to X$ commute with the section
$\sigma_1$, thus with  {$\tau_1'$ from \eqref{eq:gysin_lambda1}. 
It is direct to check that by \eqref{eq:gysin_lambda2}}
we are reduced to show the commutativity of the following diagram
\eq{prop:gysin-bc1.1}{\xymatrix{
i_*\gamma^j F_{\sZ}[-j]\ar[r]^-{\lambda^j_V}\ar[d]_-{f_Z^*} &
R\rho_{E_1*}F_{\sE_1}\ar[r]^-{\tilde{f}_{E_1}^*} &
Rf_{Z*}R\rho_{E_1'*} F_{\sE'_1}\\
Rf_*(i'_*\gamma^j F_{\sZ'}[-j])\ar[r]^-{c_\beta}&
Rf_*(i'_*\gamma^{m} F_{\sZ'}[-m])\ar[ur]_-{\lambda^m_{V'}}&
}}
where $V=\sN_{Z/X}^\vee\oplus \sO_Z$, $V'= \sN_{Z'/X'}^\vee\oplus \sO_{Z'}$,
$E_1$ (resp. $E_1'$) is the exceptional divisor in ${\rm Bl}_{Z\times 0}(X\times \P^1)$ (resp. in ${\rm Bl}_{Z'\times 0}(X'\times \P^1)$),
the maps are as in the following commutative diagram 
(which is induced by \eqref{prop:gysin-bc1} and is in general not cartesian)
\[\xymatrix{
E'_1\ar[r]^{\tilde{f}_{E_1}}\ar[d]_{\rho_{E_1'}} &
E_1\ar[d]^{\rho_{E_1}}\\
Z'\ar[r]^{f_Z} &
Z,
}\]
and $\sE_1=(E_1, D_{|E_1})$ (resp. $\sE'_1=(E'_1, D_{|E'_1})$). In view of the definition of $\lambda_V$ (see \ref{para:pbf})
and Lemma \ref{lem:cyclecup}\ref{lem:cyclecup2},\ref{lem:cyclecup3} we are reduced to show
\eq{prop:gysin-bc2}{\tilde{f}_{E_1}^*(\xi^j)= \zeta^m\cdot \rho_{E'_1}^*(\beta)\quad \text{in } \CH^j(E_1'),}
where $\xi=c_1(\sO_{E_1}(1))\in \CH^1(E_1)$ and $\zeta=c_1(\sO_{E'_1}(1))\in \CH^1(E'_1)$.
 {
This follows from the excess intersection formula:
Indeed, consider the cartesian  diagram
\[\xymatrix{
Z'\ar@{^(->}[d]_{s'}\ar@{=}[r] &
Z\times_X X'\ar[r]^-{f_Z}\ar@{^(->}[d] &
Z\ar@{^(->}[d]^s\\
E_1'\ar@{^(->}[r] &
E_1\times_X X'\ar[r]&
E_1,
}\]
where  the lower horizontal map is a factorization of $\tilde{f}_{E_1}$ and $s$ is the closed immersion
induced by the zero section $Z\inj N_{Z/X}$ followed by the open immersion $N_{Z/X}\inj E_1$
and similar with $s'$.  We obeserve
\[{s'}^*\sN_{E_1'/E_1\times_X X'}= \sN_{Z'/E_1\times_X X'}/ \sN_{Z'/E'_1}=f_Z^*\sN_{Z/N_{Z/X}}/\sN_{Z'/N_{Z'/X'}}= \sE xc, \]
see, e.g., \cite[Exa 6.3.2]{Fu} for the first equality.
Furthermore, we have  $\xi^j= s_* Z$ and $\zeta^m= s'_* (Z')$.
Thus the projection formula \cite[Exa 8.1.7]{Fu} and the above yield
\[\zeta^m\cdot \rho_{E'_1}^*(\beta)=s'_*(c_{j-m}(\sE xc)\cap Z')=s'_*(c_{j-m}({s'}^*\sN_{E_1'/E_1\times_X X'})\cap Z'), \]
which is equal to $\tilde{f}_{E_1}^*(s_* Z)$ by excess intersection \cite[Prop 6.6(c)]{Fu}.
}
\end{proof}

\begin{lemma}\label{lem:local-gysin}
Let the notations be as in \ref{para:splitting} and \ref{para:gysin} and assume $j=\codim_X(Z)\ge 1$.
Then the pullback maps $i_{E_1}^*$ and $\rho^*$ induce the following isomorphism on cohomology with supports
\eq{lem:loc-gysin1}{
i_{E_1}^*: R^j\pi_*R\rho_{1*}R\ul{\Gamma}_{E_1}F_{\sY}\xr{\simeq} i_{*}R^j\rho_{E_1*} F_{\sE_1}
}
and 
\eq{lem:loc-gysin2}{
\rho^*: R^j\ul{\Gamma}_Z F_{\sX}\xr{\simeq} R^j\rho_*R\ul{\Gamma}_{E}F_{\tilde{\sX}}.
}
Furthermore, if we define the {\em local Gysin map} $g_{\sZ/\sX, Z}^j$ by the composition
\mlnl{g_{\sZ/\sX, Z}^j: i_{*}\gamma^j F_{\sZ}\xr{\lambda^j_V} i_{*}R^j\rho_{E_1*} F_{\sE_1}
\xr{-(i_{E_1}^*)^{-1}} R^j\pi_*R\rho_{1*}R\ul{\Gamma}_{E_1}F_{\sY} \\
\xr{i_{\tilde{X}}^*}  R^j\rho_* R\ul{\Gamma}_E F_{\tilde{\sX}}\xr{(\rho^*)^{-1}} R^j\ul{\Gamma}_Z F_{\sX},}
where $V=\sN_{Z/X}^\vee\oplus \sO_Z$, 
then the Gysin map $g_{\sZ/\sX}$ is equal to the composition
\[i_{*}\gamma^j F_{\sZ}[-j]\xr{g_{\sZ/\sX, Z}^j[-j]} R^j\ul{\Gamma}_Z F_{\sX}[-j]\cong R\ul{\Gamma}_Z F_{\sX}
\to F_{\sX},\]
where the isomorphism $R^j\ul{\Gamma}_Z F_{\sX}[-j]\cong R\ul{\Gamma}_Z F_{\sX}$
is \cite[Cor 8.6(3)]{S-purity} and the last arrow is the forget-supports-map.
Moreover, the Gysin map determines the local Gysin map via
\[g_{\sZ/\sX,Z}^j= R^j\ul{\Gamma}_Z(g_{\sZ/\sX})[-j].\]
\end{lemma}
\begin{proof}
We will use without further notice the isomorphism 
$R\ul{\Gamma}_A Rf_*= Rf_*R\ul{\Gamma}_{f^{-1}(A)}$, for a morphism $f:V\to W$ and  a closed subset $A\subset W$.
For \eqref{lem:loc-gysin1} apply $R\pi_*R\ul{\Gamma}_{Z\times 0}$ to the blow-up sequence
of $(X\otimes \bcube, Z\times 0)$ (see Theorem \ref{thm:bus}) to obtain the following long exact sequence
\[\ldots\to R^j\pi_*R\ul{\Gamma}_{Z\times 0} F_{\sX\otimes \bcube} 
\to R^j\pi_*R\rho_{1*}R\ul{\Gamma}_{E_1}F_{\sY}
\xr{i_{E_1}^*} i_{*}R^j \rho_{E_1*}F_{\sE_1}\to \ldots\]
Since $Z\times 0$ has codimension $j+1$ in $X\times \P^1$ the term on the left vanishes by
\cite[Cor 8.6(3)]{S-purity}; by existence of the canonical section as in \eqref{para:splitting}
(for $(\sX\times\bcube, Z\times 0)$ instead of $(\sX, Z)$) we see that $i_{E_1}^*$ is split surjective.
This yields the isomorphism \eqref{lem:loc-gysin1}. For \eqref{lem:loc-gysin2} apply $R\ul{\Gamma}_Z$ to the 
blow-up sequence of $(\sX,Z)$ to obtain the long exact sequence
\[\ldots\to R^j\ul{\Gamma}_Z F_{\sX}\xr{\rho^*} R^j\rho_*R\ul{\Gamma}_E F_{\tilde{\sX}}\to 
i_{*}R^j \rho_{E*} F_{\sE}\to \ldots \]
By the existence of the canonical section $\sigma$ as in \eqref{para:splitting} the map $\rho^*$ is injective and 
by the  projective bundle formula, see Theorem \ref{thm:pbf}, the right hand side vanishes.
This yields the isomorphism \eqref{lem:loc-gysin2}.

Since $R\ul{\Gamma}_Z=i_*Ri^!$ and $Ri^!$ is right adjoint to $i_*$, 
the Gysin morphism factors via the counit $R\ul{\Gamma}_Z F_\sX\to F_\sX$ 
and by the isomorphism $R^j\ul{\Gamma}_Z F_{\sX}[-j]\cong R\ul{\Gamma}_Z F_{\sX}$ also via
\[g_{\sZ/\sX}: i_{*}\gamma^j F_{\sZ}[-j]\xr{R^j\ul{\Gamma}_Z(g_{\sZ/\sX})[-j]} 
 R^j\ul{\Gamma}_Z F_\sX[-j]\cong R\ul{\Gamma}_Z F_{\sX}\to F_{\sX}.\]
Thus it remains to show 
\[g_{\sZ/\sX, Z}^j= R^j\ul{\Gamma}_Z(g_{\sZ/\sX}).\]
 {
By \eqref{eq:gysin_lambda2} it remains to show that 
\eq{lem:loc-gysin3}{R^j\ul{\Gamma}_Z(\sigma)= \eqref{lem:loc-gysin2}^{-1},}
and  that $\tau_1^j:=R^j\ul{\Gamma}_Z(\tau_1')$ is equal to the composition
\eq{lem:loc-gysin4}{i_{*}R^j\rho_{E_1*}F_{\sE_1}
\xr{\eqref{lem:loc-gysin1}^{-1}} R^j\pi_* R\rho_{1*} R\ul{\Gamma}_{\rho_1^{-1}(Z\times 0)} F_{\sY}
\xr{e} R^j\pi_* R\rho_{1*}R\ul{\Gamma}_{\rho_1^{-1}(Z\times\P^1)} F_{\sY},}
}
where $e$ is the enlarge-support map. 
(Note that $i_{\tilde{X}}^*$ from the definition of the local Gysin map factors via $e$.)
For \eqref{lem:loc-gysin3} we observe that $R^j\ul{\Gamma}_Z(\sigma)$ is by definition 
a section of the isomorphism \eqref{lem:loc-gysin2}, hence has to be the inverse of that isomorphism.
To compute $\tau_1^j$ we consider the following diagram:
\eq{lem:loc-gysin5}{\xymatrix{
 & R^j\pi_* R\rho_{1*} R\ul{\Gamma}_{\rho_1^{-1}(Z\times 0)} F_{\sY}\ar[d]^{e}\ar[dr]^{b''} &\\
R^j\ul{\Gamma}_Z R\pi_* F_{\sX\otimes\bcube}\ar[d]_{i_0^*=i_1^*}^{\simeq}\ar[r]_-{\rho_1^*} & 
R^j\ul{\Gamma}_Z R\pi_*R\rho_{1*}F_{\sY}\ar[dl]^-{i_{1,Y}^*}\ar[r]^-{b_1} \ar@/_1.5pc/[l]_{\sigma_1^j}& 
i_{*}R^j\rho_{E_1*} F_{\sE_1}\ar@/^1.5pc/[l]^{\tau_1^j}\\
R^j\ul{\Gamma}_ZF_{\sX}. &    & \\
}}
Here the middle line and the triangle on the left lower side are induced by applying
$R^j\ul{\Gamma}_Z$ to the diagram \eqref{para:splitting3} 
up to the isomorphism $\bigoplus_{i=1}^{j} (\gamma^i F)_{\sZ}[-i]\simeq \tau_{\geq 1} R\rho_{E_1*} F_{\sE_1}$ from \eqref{para:splitting2},
 {with the obvious notation for $\sigma_1^j$ and $\tau_1^j$,} and $b''$ is the isomorphism 
\eqref{lem:loc-gysin1}.  {By definition of $\tau_1^j$ (see \ref{para:splitting} and Lemma \ref{lem:can-sec})
we have 
\[\tau^j_1 \circ b_1=\id -\rho_1^*\sigma_1^j.\]
Thus 
\eq{lem:loc-gysin6}{\tau_1^j=\tau_1^j b_1 e(b'')^{-1}= e (b'')^{-1}- \rho_1^*\sigma_1^j e (b'')^{-1}.}
By definition $\sigma_1^j=(i_0^*)^{-1}i_{1, Y}^* = \pi^* \circ i_{1, Y}^*$. We claim that $ i_{1, Y}^*\circ e=0$.
Indeed $i_{1, Y}^*$ is induced by $F_\sY\to i_{1, Y*}F_{\sX\otimes 1}$.
Consider the natural  commutative diagram
\[\xymatrix{
R\ul{\Gamma}_{\rho_1^{-1}(Z\times 0)} F_\sY\ar[r]\ar[d] & 
R\ul{\Gamma}_{\rho_1^{-1}(Z\times 0)} (i_{1, Y*}F_{\sX\otimes 1})=0\ar[d]\\
F_{\sY}\ar[r] & i_{1, Y*}F_{\sX\otimes 1},
}\]
where the top right corner vanishes since $\rho_1^{-1}(Z\times 0)\cap i_{1, Y}(X\times\{1\})=\emptyset$.
The map $ i_{1, Y}^*\circ e$ is induced by applying $R^j\ul{\Gamma}_Z R\pi_*R\rho_{1*}$ to this diagram 
and going counter-clockwise starting at the top left corner; hence the vanishing $ i_{1, Y}^*\circ e=0$.
Thus \eqref{lem:loc-gysin6} yields $\tau_1^j= e (b'')^{-1}$, which proves that $\tau^j_1$ is equal
to the composition \eqref{lem:loc-gysin4}.}
\end{proof}

\begin{para}\label{para:buf}
We recall a general formula for the refined Gysin morphism of a blow-up.
Let $a:V\inj  W$ be a regular closed immersion of quasi-projective $k$-schemes 
with normal sheaf $\sN_{V/W}$ and denote by $f: \tilde{W}\to W$ the blow-up of $W$ along $V$. We have then the cartesian square
\[ \xymatrix{ \tilde{V}:=f^{-1}(V) \ar[r]^-b\ar[d]^g & \tilde{W} \ar[d]^f\\
V\ar[r]^a & W
}
\]
Since $\tilde{V}=\P(\sN_{V/W}^\vee)$ the excess normal bundle  satisfies
\[\sE xc = g^*\sN_{V/W}/ \sN_{\tilde{V}/\tilde{W}} = g^*\sN_{V/W}/ \sO_{\tilde{V}}(-1),\]
i.e. it is the universal quotient bundle on $\P(\sN_{V/W}^\vee)$ (see e.g. \cite[6.7]{Fu}).

Let $W'\to W$ be a finite-type morphism and denote by $V', \tilde{W}', \tilde{V}'$ and  
by $f', g',a', b'$ the base-changes along $W'\to W$. 
The refined Gysin morphism of $f$ is a map $f^!: \CH_m(W')\to \CH_m(\tilde{W}')$.
Let $T\subset W'$ be an $m$-dimensional integral closed subscheme,
and denote by $[T]\in \CH_m(W')$  its cycle class.
Then by \cite[Exa 6.7.1, Prop 17.5]{Fu} 
\eq{para:buf1}{ f^![T]= [\tilde{T}]+ b'_*\left\{c(\sE xc_{|\tilde{V}'})\cap {g'}^*s(T\cap V', T)\right\}_m\quad 
                                                 \text{in } \CH_m(\tilde{W}'),}
where $\tilde{T}={\rm Bl}_{T\cap V'}(T)\subset {\rm Bl}_{V'}(W')\subset \tilde{W}'$ is the blow-up of $T$ in $T\cap V'$,
$c(\sE xc_{|\tilde{V}'})$ is the total Chern class of the pullback of $\sE xc$ to $\tilde{V}'$,
and $s(T\cap V', T)$ is the Segre class of $T\cap V'$ in $T$ defined in \cite[4.2]{Fu}.
\end{para}

\begin{thm}\label{prop:pf-gysin}
Let $F\in \CItspNis$, $\sX=(X,D)\in \uMCorls$ and $i: Z\inj X$ a smooth closed subscheme of codimension $j$
intersecting $D$ transversally. Set $\sZ=(Z, D_{|Z})$. 
Then  we have the following equality of maps of sheaves on $X_\Nis$
\eq{prop:pf-gysin1}{H^j(c_Z)= g_{\sZ/\sX,Z}^j\circ i^*: 
\gamma^j F_{\sX}  \to i_*\gamma^j F_{\sZ}\to R^j\ul{\Gamma}_Z F_{\sX},}
where $c_Z$ is the morphism \eqref{para:cyclecup1} for $Z$ viewed as cycle in $\CH^j_Z(X)$.
In particular, the following diagram commutes in $D(X_\Nis)$
\[\xymatrix{
\gamma^j F_{\sX}[-j]\ar[d]^{i^*}\ar[r]^-{c_Z} & R\ul{\Gamma}_{Z} F_{\sX}\ar[d]\\
i_*\gamma^j F_{\sZ}[-j]\ar[r]^-{g_{\sZ/\sX}} & F_{\sX}.
}\]
In particular, if $X$ admits a $k$-morphism $q: X\to Z$, such that $q\circ i=\id_Z$
(locally in the Nisnevich topology this is always possible, see Lemma \ref{lem:loc-cl-sec} below),
then
\eq{prop:pf-gysin1.1}{g_{\sZ/\sX}= c_Z\circ q^* : i_*\gamma^j F_{\sZ}[-j]\to R\ul{\Gamma}_Z F_{\sX}\to F_{\sX}.}
\end{thm}
\begin{proof}
The equivalence of the two statements follows from 
the isomorphism $R\ul{\Gamma}_Z F_{\sX}\cong R^j\ul{\Gamma}_Z F_{\sX}[-j]$ (see \cite[Cor 8.6(3)]{S-purity})
and the definition of the local Gysin map, see Lemma \ref{lem:local-gysin}.
We show the equality \eqref{prop:pf-gysin1}. This is a local question and we can therefore assume 
that the normal sheaf of $Z$ in $X$ is trivial, i.e.,
\eq{prop:pf-gysin1.5}{\sN_{Z/X}\cong \sO_Z^{\oplus j}.}
Let the notation be as in \ref{para:splitting}. Set $\xi=c_1(\sO_{E_1}(1))\in \CH^1(E_1)$.
Note that the pullback $i_{E_1}^*: \CH^{j}_{E_1}(Y)\to \CH^{j}_{E_1}(E_1)$
is by definition (see Lemma \ref{lem:pbKCH}) equal to the refined Gysin map (see \cite[6.2]{Fu})
\[i_{E_1}^!: \CH^{j-1}(E_1)\to \CH^j(E_1)\]
corresponding to the diagram
\[\xymatrix{
E_1\ar@{=}[r]\ar@{=}[d] & E_1\ar[d]^{i_{E_1}}\\
E_1\ar[r]^{i_{E_1}} & Y.
}\]
The normal sheaf  of the immersion $i_{E_1}$ is 
\[\sN_{E_1/Y}=\sO_Y(E_1)_{|E_1}= \sO_{E_1}(-1).\]
Thus by the excess intersection formula (see \cite[Cor 6.3]{Fu}) we find 
\eq{prop:pf-gysin1.6}{-i_{E_1}^!(\xi^{j-1})= \xi^j.}
Set $\eta:=\xi^{j-1}$ viewed as an element in $\CH^j_{E_1}(Y)$.
Consider the following diagram
\[\xymatrix{
\gamma^j F_{\sX}[-j]\ar[d]^{i^*}\ar[r]^-{(\pi\rho_1)^*} & 
  R(\pi\rho_{1})_*\gamma^j F_{\sY}[-j]\ar[r]^-{c_{\eta}}\ar[d]^{i_{E_1}^*}&
R(\pi\rho_{1})_*R\ul{\Gamma}_{E_1}F_{\sY}\ar[d]_{-i_{E_1}^*}\\
i_*\gamma^j F_{Z}[-j]\ar@/_1.5pc/[rr]_-{\lambda^j_V}\ar[r]^-{\rho_{E_1}^*}  & 
i_{*}R\rho_{E_1*} \gamma^j F_{\sE_1}[-j]\ar[r]^-{c_{\xi^j}}&
i_*R\rho_{E_1*}F_{\sE_1}.
}\]
The right square commutes by \eqref{prop:pf-gysin1.6} and Lemma \ref{lem:cyclecup}\ref{lem:cyclecup2},
the left square clearly commutes, hence so does the big outer square.
Thus by Lemma \ref{lem:local-gysin}
\eq{prop:pf-gysin2}{\rho^* \circ g^j_{\sZ/\sX,Z}\circ i^*= i_{{\tilde{X}}, E_1}^*\circ H^j(c_{\eta})\circ (\pi\rho_1)^*: 
i_*\gamma^j F_{\sX}\to R^j\rho_*R\ul{\Gamma}_E F_{\tilde{\sX}},}
where $i_{{\tilde{X}},E_1}^*:= R^j(\pi\rho_{1})_{*} R\ul{\Gamma}_{E_1} (i_{\tilde{X}}^*)$.

Set $Z_1:=\rho_1^{-1}(Z\times\P^1)$.
We have $E_1\subset Z_1$ and $i_{{\tilde{X}}}^{-1}(E_1)= i_{\tilde{X}}^{-1}(Z_1)=E$.
Denote by the same letter $\imath$ the two enlarge-support maps 
$\imath: R\ul{\Gamma}_{E_1}\to R\ul{\Gamma}_{Z_1}$ and $\imath:\CH^j_{E_1}(Y)\to \CH^j_{Z_1}(Y)$.
We also denote by $i_{{\tilde{X}},Z_1}^*$ the two maps induced by $i_{\tilde{X}}^*$
\[i_{{\tilde{X}},Z_1}^*\,\,: \,\,
R^j(\pi\rho_1)_* R\ul{\Gamma}_{Z_1}F_{\sY}\to R^j\rho_* R\ul{\Gamma}_{E} F_{\tilde{\sX}},
\qquad \CH^j_{Z_1}(Y)\to \CH^j_E(\tilde{X}).\]
Clearly we have in both cases
\[i_{{\tilde{X}},E_1}^*= i_{{\tilde{X}},Z_1}^*\circ\imath.\] 
Thus \eqref{prop:pf-gysin2} yields
\eq{prop:pf-gysin3}{\rho^* \circ g^j_{\sZ/\sX,Z}\circ i^*= 
i_{{\tilde{X}}, Z_1}^*\circ H^j(c_{\imath(\eta)})\circ (\pi\rho_1)^*: 
i_*\gamma^j F_{\sX}\to R^j\rho_*R\ul{\Gamma}_E F_{\tilde{\sX}}.}
The strict transform of $Z\times \P^1\subset X\times \P^1$ in $Y$ is the blow-up of
$Z\times \P^1$ in the Cartier divisor $Z\times 0$ and is therefore isomorphic to $Z\times \P^1$.
We obtain 
\[Z_1= (Z\times \P^1)\cup E_1\]
and $E_1\cap (Z\times \P^1)= Z\times 0$ is embedded as the zero section in the normal bundle $N_{E_1}Y$,
which is equal to $E_1\setminus E$. On the other hand, $\tilde{X}\subset Y$ is the strict transform of $X\times 0$ and
intersects $E_1$ in $E$,  see, e.g.,  \cite[5.1]{Fu}. Thus in $Y$ we have $\tilde{X}\cap (Z\times \P^1)=\emptyset$.

\begin{claim}\label{prop:pf-gysin4}
We claim that the following equality holds in $\CH^j_{Z_1}(Y)$ 
\[(\pi\rho_1)^! [Z] = \imath(\eta) +[Z\times \P^1], \]
where 
$(\pi\rho_1)^!: \CH_{d-j}(Z)\to \CH_{d+1-j}(Z_1)$
is the refined Gysin map corresponding to the cartesian diagram
\[\xymatrix{
Z_1\ar[r]\ar@{^(->}[d] & Z\ar@{^(->}[d]^i\\
Y\ar[r]^{\pi\rho_1} & X.
}\]
\end{claim}
Assuming Claim \ref{prop:pf-gysin4}, we can conclude as follows.
The composition
\[ \gamma^j F_{\sY}[-j]\xr{c_{[Z\times \P^1]}} R\ul{\Gamma}_{Z_1} F_{\sY}\xr{i_{\tilde{X}}^*}
R\ul{\Gamma}_E F_{\tilde{\sX}}\]
factors via 
\[R\ul{\Gamma}_{Z\times \P^1} F_{\sY}\xr{i_{\tilde{X}}^*} R\ul{\Gamma}_{(Z\times \P^1)\cap E} F_{\tilde{\sX}}.\]
Since $\tilde{X}\cap (Z\times \P^1)=\emptyset$ by what was said after \eqref{prop:pf-gysin3}, we have
\eq{prop:pf-gysin5}{i_{{\tilde{X}},Z_1}^*\circ c_{[Z\times \P^1]}=0.}
Thus we obtain the following equality of maps $i_*\gamma^j F_{\sX}\to R^j\rho_*R\ul{\Gamma}_E F_{\tilde{\sX}}$
\begin{align*}
\rho^* \circ g^j_{\sZ/\sX,Z}\circ i^* & = i_{{\tilde{X}}, Z_1}^*\circ H^j(c_{\imath(\eta)})\circ (\pi\rho_1)^*, & & 
      \text{by \eqref{prop:pf-gysin3}},\\
                          &= i_{{\tilde{X}}, Z_1}^*\circ H^j(c_{(\pi\rho_1)^*Z -[Z\times \P^1]})\circ (\pi\rho_1)^*, & &
                                                     \text{by \ref{prop:pf-gysin4}},\\
                                                     & =i_{{\tilde{X}}, Z_1}^*(\pi\rho_1)^*\circ H^j(c_Z), & & 
                                                     \text{by \eqref{prop:pf-gysin5}, \ref{lem:cyclecup}},\\
                                                     &= \rho^* \circ H^j(c_Z).
\end{align*}
The statement follows since $\rho^*$ here is an isomorphism, see \eqref{lem:loc-gysin2}.

{\em Proof of  Claim \ref{prop:pf-gysin4}.}
First note 
\[(\pi\rho_1)^! [Z]= \rho_1^! [Z\times \P^1].\]
To compute this expression we apply the formula \eqref{para:buf1}
in the case where $a=i_{Z0}: Z\times 0\inj X\times \P^1$, $f=\rho_1: Y\to X\times \P^1$,
$(W'\to W)= (Z\times \P^1\inj X\times \P^1)$ and $T=Z\times \P^1$ and $m= d+1-j$, where $d=\dim X$.
In particular, we have 
\[\tilde{W}'=Z_1,\quad V'=V=Z\times 0,\quad  \tilde{V'}=\tilde{V}=E_1,\]
and $\tilde{T}=Z\times \P^1$.
Since the conormal bundle of $Z\times 0\inj Z\times \P^1$ is trivial, \cite[Prop 4.1(a)]{Fu} yields
\[\rho_{E_1}^*s(Z\times 0, Z\times \P^1)= [E_1].\]
By  \eqref{prop:pf-gysin1.5} we have $\sN_{Z\times 0/X\times \P^1}= \sO_{Z}^{\oplus j+1}$.
The Whitney formula yields
\[c(\sE xc_{|E_1})= (1-\xi)^{-1}= 1+\xi+\ldots+ \xi^{j},\]
where $\xi=c_1(\sO_{E_1}(1))$.
Thus Claim \ref{prop:pf-gysin4} follows from \eqref{para:buf1}.
\end{proof}

\begin{cor}\label{cor:pf-gysin}
Let $F$, $\sX=(X,D)$, $i: Z\inj X$,  and  $\sZ$ as 
in Theorem \ref{prop:pf-gysin} above. Assume $D=\emptyset$ (thus $\sX=X$ and $\sZ=Z$).
Let $\Phi$ be a family of supports on $Z$ and 
$\alpha\in \CH^r_{\Phi}(Z)$. Then
\mlnl{c_{i_*\alpha}= g_{Z/X}\circ c_\alpha\circ i^*: \\
\gamma^{j+r} F_{X}[-j-r]\to \gamma^{j+r} F_{Z}[-j-r]
\to R\ul{\Gamma}_{\Phi}\gamma^{j} F_{Z}[-j]\to R\ul{\Gamma}_\Phi F_{X},}
where we view $i_*\alpha\in \CH^{j+r}_{\Phi}(X)$.
\end{cor}
\begin{proof}
It suffices to consider $\alpha=[V]$, with $V\subset Z$ irreducible and of codimension $r$, and $\Phi=\Phi_V$.
By \cite[Cor 8.6(1)]{S-purity} we have $R\ul{\Gamma}_V F_{X}\cong \tau_{\ge j+r}R\ul{\Gamma}_V F_{X}$
(here we need $D=\emptyset$). Hence we have a natural map 
$R^{j+r}\ul{\Gamma}_V F_X[-j-r]\to R\ul{\Gamma}_V F_{X}$ in the derived category 
and the two maps in the statement are induced by composing this map
with the two morphisms of sheaves
\eq{cor:pf-gysin1}{H^{j+r}(c_{i_*V}),\,\, H^{j+r}(g_{Z/X}\circ c_V\circ i^*): 
\gamma^{j+r}F_X\to R^{j+r}\ul{\Gamma}_V F_X.}
Thus it suffices to show that the two maps in \eqref{cor:pf-gysin1} are equal.
By \cite[Cor 8.6(1)]{S-purity} the restriction 
$R^{r+j}\ul{\Gamma}_V F_{X}\to 
\nu_* R^{r+j}\ul{\Gamma}_{V\setminus V_{\rm sing}}  
F_{X\setminus V_{\rm sing}}$
is injective,   where $\nu : X\setminus V_{\rm sing}\inj X$ is the open immersion.
Thus we may furthermore assume that $V$ is smooth.
The question is local on $X$ and we can therefore assume that there exists a closed subset
$W\subset X$ of pure codimension $r$ such that $V=i^* W$ in $\CH^r_V(Z)=\CH^0(V)$. 
In this situation we have the following equality of maps 
$\gamma^{j+r}F_X[-j-r]\to R\ul{\Gamma}_V F_X$
\[g_{Z/X}\circ c_V \circ i^*  =  g_{Z/X}\circ c_{i^*W} \circ i^*
                                     = g_{Z/X}\circ i^* \circ c_W
                                     = c_Z\circ c_W
                                     = c_{Z\cdot W}
                                     = c_{i_*V},
\]
where the second and forth equality hold by Lemma \ref{lem:cyclecup} and  the third equality 
by Theorem \ref{prop:pf-gysin}. This implies the statement.
\end{proof}

\begin{lemma}\label{lem:loc-cl-sec}
Let $S$ be an affine scheme and $Z\inj X$ a closed immersion of affine $S$-schemes.
Assume that $X$ is noetherian, integral and normal and  $Z$ is irreducible and formally smooth and of finite type over $S$.
Then there exists a Nisnevich neighborhood $u: X'\to X$ of $Z$
(i.e., $u$ is \'etale and induces an isomorphism $u^{-1}(Z)\xr{\simeq} Z$)
which admits an $S$-morphism $X'\to Z$, such that the composition
$Z\cong u^{-1}(Z)\inj X'\to Z$ is the identity.
\end{lemma}
\begin{proof}
We follow an argument in the proof of \cite[Lem 8.5]{S-purity}.
Write $S=\Spec R$, $X=\Spec A$ and $Z=\Spec A/I$.
Set $Z_n:=\Spec A/I^n$ and $\hat{X}_Z=\Spec \hat{A}_I$, where $\hat{A}_I=\varprojlim_n (A/I^n)$.
Since $Z$ is formally smooth over  $S$,
we find a compatible system of $S$-morphisms $\{Z_n\to Z\}$ which
reduce to the identity on $Z$; it induces a morphism of $S$-schemes
$\hat{\pi}: \hat{X}_Z\to Z$ of which the natural closed immersion $Z\inj \hat{X}_Z$ is a section.
We can form the closed immersion 
$\hat{\epsilon}:=\id_{\hat{X}_Z}\times \hat{\pi}: \hat{X}_Z\inj \hat{X}_Z\times_S Z$ which restricts to the diagonal on
$Z\times_S Z$. By \cite[Thm 2bis]{Elkik} we find therefore an $X^h_Z$-morphism
$\epsilon^h: X_Z^h\inj  X_Z^h\times_S Z$ which restricts to the diagonal  on $Z\times_S Z$,
where $X^h_Z=\Spec A^h_I$ is the henselization of the pair $(X, Z)$.
Composing   $\epsilon^h$ with the projection to $Z$ yields an $S$-morphism
$u^h: X_Z^h\to Z$. Since  $X$ is normal and noetherian, so is any affine \'etale scheme $Y$  over $X$;
in particular such a $Y$ is a disjoint union of integral normal $X$-schemes.
Since $Z$ is irreducible, any Nisnevich neighborhood $Y\to X$ of $Z$ can be refined to a Nisnevich neighborhood 
$Y'\to X$ of $Z$ with $Y'$ integral. It follows that we can write  
$A^h_I=\varinjlim B$, where the limit is over all \'etale maps $A\to B$ inducing an isomorphism $A/I\to B/IB$
with $B$ integral; the transition maps $B\to B'$ in this system are automatically \'etale and hence 
(since $B$ and $B'$ are integral) also injective; thus also $B\to A^h_I$ is injective.
Since $A/I$ is of finite type over $R$ it follows that the $R$-algebra map ${u^h}^*:A/I\to A^h_I$
factors via an $R$-algebra map $A/I\to B$ for some $B$ as above. This yields the statement.
\end{proof}

\begin{cor}\label{cor:gysin-fun}
Let $F\in \CItspNis$ and $\sX=(X,D)\in \uMCorls$. 
Let $i:Z\inj Z'\xhookrightarrow{i'} X$ be closed immersion of smooth schemes
of codimension $a=\codim(Z, Z')$ and $b=\codim(Z', X)$,
such that $D$ intersects $Z$ and $Z'$ transversally. Set $Z=(Z, D_{|Z})$ and $Z'=(Z', D_{|Z'})$.
We have the following equality
\[R^a\ul{\Gamma}_{Z}(g^b_{\sZ'/\sX, Z'})\circ g^a_{\sZ/\sZ', Z}= g^{a+b}_{\sZ/\sX, Z}: 
i_*\gamma^{a+b}F_\sZ\to R^{a+b}\ul{\Gamma}_{Z} \sF_{\sX}.\]
In particular, the following diagram commutes
\[\xymatrix{
i_*\gamma^{a+b}F_{\sZ}[-a-b]\ar[rr]^{g_{\sZ/\sX}}\ar[dr]_{g_{\sZ/\sZ'}[-b]} &  &
F_\sX\\
&    i'_* \gamma^b F_{\sZ'}[-b].\ar[ur]_{g_{\Z'/\sX}} &
}\]
\end{cor}
\begin{proof}
The second statement follows from the first and Lemma \ref{lem:local-gysin}.
The first statement is local in $X$ and we may therefore assume that we find a smooth
closed subscheme $Z''\subset X$ of codimension $a$ such that $Z=Z'\times_X Z''$.
Since $i^*: \gamma^{a+b}F_{\sX}\to i_*\gamma^{a+b}F_{\sZ}$
is surjective by Lemma \ref{lem:loc-cl-sec} it suffices to show the equality after precomposition
with $i^*$. Consider the following diagram
\[\xymatrix{
\gamma^{a+b} F_{\sX}\ar@/_3pc/[dd]_{i^*}\ar[d]^{{i'}^*}\ar[rr]^-{c_{Z''}} & &
R^a\ul{\Gamma}_{Z''}\gamma^b F_{\sX}\ar[d]^{{i'}^*}\ar[drr]^-{c_{Z'}} 
& &\\
i'_*\gamma^{a+b} F_{\sZ'}\ar[rr]_-{c_Z}\ar[d] & &
i'_*R^a\ul{\Gamma}_{Z} \gamma ^b F_{\sZ'}\ar[rr]_-{g^b_{\sZ'/\sX, Z'}} & &
R^{a+b}\ul{\Gamma}_Z F_{\sX}\\
i_*\gamma^{a+b} F_{\sZ},\ar[rru]_{g^a_{\sZ/\sZ', Z}} & &   & &
}\]
where the maps $c_Z$, $c_{Z'}$, and $c_{Z''}$ are defined viewing $Z$, $Z'$,  and  $Z''$ 
as cycles in $\CH^a_Z(Z')$, $\CH^b_{Z'}(X)$, and $\CH^a_{Z''}(X)$, respectively.
The square commutes by Lemma \ref{lem:cyclecup}\ref{lem:cyclecup2} and the triangles commute
 by Theorem \ref{prop:pf-gysin}. 
By definition of the refined intersection product we have $Z'\cdot Z''= Z$ in $\CH^{a+b}_Z(X)$. 
Thus the statement follows from Lemma \ref{lem:cyclecup}\ref{lem:cyclecup3} and Theorem \ref{prop:pf-gysin}.
\end{proof}

\begin{thm}\label{thm:gysin-tri}
Let $\sX=(X,D)\in\uMCorls$ and  $F\in\CItspNis$.
Let $i:Z\inj X$ be a smooth closed subscheme of codimension $j$ intersecting $D$ transversally
and set $\sZ=(Z, D_{|Z})$. 
Then there is a canonical distinguished triangle in $D(X_{\Nis})$
\eq{thm:gysin-tri0}{i_*\gamma^j F_{\sZ}[-j]\xr{g_{\sZ/\sX}} F_{\sX}\xr{\rho^*} 
                             R\rho_* F_{(\tilde{X}, D_{|\tilde{X}}+ E)}\xr{\partial} i_*\gamma^j F_{\sZ}[-j+1],}
where $\rho:\tilde{X}\to X$ is the blow-up of $X$ along $Z$ and $E=\rho^{-1}(Z)$.
\end{thm}
\begin{proof}
We first consider the case $j=1$.
In this case $\tilde{X}=X$ and $Z=E$.
Denote by $j: U=X\setminus Z\inj X$ the open immersion and set $\sU=(U, D_{|U})$ and $\sX':=(X, D+Z)$.
Consider the following diagram of solid arrows of sheaves on $X_{\Nis}$
\eq{thm:gysin-tri1}{\xymatrix{
\gamma^1 F_{\sX}\ar[r]_-{\eqref{lem:divisor-cup1}}\ar@{->>}[d]_{i^*}\ar@/^2pc/[rrr]^{c_Z} & 
F_{\sX'}/F_{\sX}\ar@{^(->}[r] & 
j_*F_{\sU}/F_{\sX}\ar[r]^{\simeq} &
R^1\ul{\Gamma}_Z F_{\sX}.\\
i_*\gamma^1 F_{\sZ}\ar@{.>}[ur]_{(*)}\ar@/_1pc/[urrr]_{g_{\sZ/\sX,Z}^1}
}}
The diagram commutes by Lemma \ref{lem:divisor-cup} and Theorem \ref{prop:pf-gysin}, the vertical
arrow  is surjective by Lemma \ref{lem:loc-cl-sec}. 
\begin{claim}\label{thm:gysin-tri1.5}
The dotted arrow $(*)$ exists,  makes the diagram commute, and is an isomorphism.
(It is automatically uniquely determined).
\end{claim}
Indeed, the question is local around the points of $Z$. We may therefore assume 
that we have an \'etale morphism $u: X\to S[t]$, such that $S=\Spec K\{x_1,\ldots, x_n\}$, 
with a function field $K$,  $D=u^*\Div(x_1^{r_1}\cdots x_s^{r_s})$, and $u$ induces an isomorphism
$Z\cong u^{-1}(t=0)\xr{\simeq} S$. 
In particular, we have a morphism $q: X\to S$ such that the composition 
\eq{thm:gysin-tri2}{q\circ i : Z\xr{\simeq} S }
is an isomorphism.
Thus the arrow $(*)$ exists by \eqref{prop:pf-gysin1.1} as the composition
\[i_*\gamma^1 F_\sZ\xr{q^*} \gamma^1 F_{\sX}\xr{\eqref{lem:divisor-cup1}} F_{\sX'}/F_{\sX}.\]
 By Lemma \ref{lem:divisor-cup} the map $(*)$
is induced by pullback along the composition
\[X\xr{\Delta_X} X\times X\xr{u^*(t)\times \id_X} \A^1\times X \xr{\id_{\A^1}\times q} \A^1\times S;\]
this composition is equal to $u$. Set $\sS:= (S, \Div(x_1^{r_1}\cdots x_s^{r_s}))$. 
Hence the map $(*)$ is on $Z$ equal to (cf. Remark \ref{rmk:gamma}) 
\[\frac{F((\A^1, 0)\otimes \sZ)}{F(\A^1\otimes \sZ)}\cong \frac{F((\A^1, 0)\otimes \sS)}{F(\A^1\otimes \sS)}
\xr{u^*} \frac{F(X, D+Z)}{F(X,D)}.\]
It remains to show that $u^*$ becomes an isomorphism if we replace $X$ by a Nisnevich neighborhood around 
the point $(x_1,\ldots, x_n, t)$. By the usual trace argument we may assume that the field $K$ is infinite.
By \cite[Lem 6.7]{S-purity} we may therefore assume that $(X, Z)$ is a $V$-pair over $S$
(in the sense of \cite[Def 2.1]{S-purity}). Clearly $(\A^1_S, 0_S)$ is also a $V$-pair over $S$
and \eqref{thm:gysin-tri2} gives an identification $Z\cong 0_S$. Thus $u^*$ is an isomorphism
by \cite[Cor 2.21]{S-purity}. This proves Claim \ref{thm:gysin-tri1.5}.

We construct the triangle from the statement in the case $j=1$.
Set 
\[\alpha:=(*)^{-1}: F_{\sX'}/F_{\sX}\xr{\simeq} i_*\gamma^1 F_{\sZ}.\]
Denote by $r: F_{\sX}\inj F_{\sX'}$ the inclusion. For exact triangles we adopt the sign
conventions from \cite[1.3]{Co}. Thus the boundary map $\cone(r)\to F_{\sX}[1]$ 
of the exact triangle determined by $r$,  is given by $-\id_{F_{\sX}}$ in degree $-1$.
We define the boundary map $\partial$ as the composition
\[\partial: F_{\sX'}\to F_{\sX'}/F_{\sX}\xr{\alpha} i_*\gamma^1 F_{\sZ}\]
and we define a quasi-isomorphism $\varphi$ as the composition
\[\varphi: \cone(r)\xr{{\rm qis}} F_{\sX'}/F_{\sX} \xr{\alpha } i_*\gamma^1 F_{\sZ},\]
where the first map is induced by the quotient map in degree 0.
It remains to show that the following diagram is commutative in $D(X_{\Nis})$
\eq{thm:gysin-tri3}{
\xymatrix{
\cone(r)[-1]\ar[r]\ar[d]_{\varphi[-1]} & 
F_{\sX}\ar[r]^{r}\ar@{=}[d]& 
F_{\sX'}\ar@{=}[d]\ar[r] & 
\cone(r)\ar[d]^{\varphi}\\
i_*\gamma^1 F_{\sZ}[-1]\ar[r]^-{g_{\sZ/\sX}} &
F_{\sX}\ar[r]^r&
\sF_{\sX'}\ar[r]^-{\partial}&
i_*\gamma^1 F_{\sZ}.
}}
By definition the square on the right commutes; 
by  Lemma \ref{lem:local-gysin} the square on the left is the big outer square of the following diagram
\eq{thm:gysin-tri4}{\xymatrix{
\cone(r)[-1]\ar[r]\ar[d]_{\rm qis} & 
 \cone(F_{\sX}\to j_* F_{\sU})[-1]\ar[r]\ar[d]^{\rm qis}&
F_{\sX}\\
F_{\sX'}/F_{\sX}[-1]\ar[r]\ar[d]_{\alpha} & 
j_*F_{\sU}/F_{\sX}[-1]\ar[r]^{\simeq}&
R^1\ul{\Gamma}_Z F_{\sX}[-1]\ar[u] 
\\
i_{*}\gamma^1 F_{\sZ}[-1],\ar[urr]_{g^1_{\sZ/\sX,Z}} & &
 }}
where the vertical arrow on the top right is the composition
\[R^1\ul{\Gamma}_Z F_{\sX}[-1]\cong R\ul{\Gamma}_Z F_{\sX}\to F_{\sX},\]
where the isomorphism comes from \cite[Cor 8.6(3)]{S-purity}.
The lower triangle in \eqref{thm:gysin-tri4} commutes by the definition of $\alpha$,
the left top square commutes by functoriality, the right top square commutes by the definitions of the involved maps.
Thus the square on the left in \eqref{thm:gysin-tri3} commutes.
We have constructed the canonical distinguished Gysin triangle in codimension 1.

We consider the general case $j\ge 1$.
Let $\rho: \tilde{X}\to X$ be the blow-up along $Z$ and $E$ the exceptional divisor.
Set $\tilde{\sX}=(\tilde{X}, D_{|\tilde{X}})$ and $\tilde{\sX}'=(\tilde{X}, D_{|\tilde{X}}+E)$
(note that $|E+D_{|\tilde{X}}|$ is a SNCD),
moreover we use the notation from \ref{para:splitting}.
Set 
\[C:= \bigoplus_{r=1}^{j-1} i_*\gamma^r F_{\sZ}[-r]\]
and consider the following diagram in $D(X_\Nis)$
\eq{thm:gysin-tri5}{\xymatrix{
i_*\gamma^jF_{\sZ}[-j]\oplus C\ar[d]_{g_{\sZ/\sX}\oplus(-\id)}\ar[r]^-{\eqref{para:pbf1}[-1]}_-{\simeq} &
R\rho_{E*} \gamma^1 F_{\sE}[-1]\ar[d]^{g_{\sE/\tilde{\sX}}}\\
F_{\sX} \oplus C \ar[d]_{\rho^*+0}\ar[r]^-{\eqref{cor:bud1}}_-{\simeq}&
R\rho_* F_{\tilde{\sX}}\ar[d]\\
R\rho_* F_{\tilde{\sX}'}\ar[d]_{\partial_Z \oplus 0}\ar@{=}[r] &
R\rho_* F_{\tilde{\sX}'}\ar[d]^{\partial_E}\\
i_*\gamma^{j}F_{\sZ}[-j+1]\oplus C[1]\ar[r]^-{\eqref{para:pbf1}}_-{\simeq}&
R\rho_{E*}\gamma^1 F_{\sE},
}}
where the right column is $R\rho_*$ applied to the Gysin triangle for $E\inj \tilde{X}$ stemming from the
codimension 1 case above 
and the map $\partial_Z$ is defined so that the lower square commutes.
This defines the triangle \eqref{thm:gysin-tri0}. Note that the right column is a distinguished triangle and
the left column is the direct sum of \eqref{thm:gysin-tri0} and  $C\xr{-\id} C\to 0\to C[1]$.
If the top square commutes, then $\eqref{thm:gysin-tri0}$ is therefore a distinguished triangle,
by \cite[Prop 1.2.3]{Neeman}.
Thus it remains to show:
\begin{claim}\label{thm:gysin-tri-claim2}
The top square in diagram \eqref{thm:gysin-tri5} commutes.
\end{claim}
This is equivalent to the commutativity of the squares resulting from precomposition
with the canonical maps $i_*\gamma^r F_{\sZ}[-r]\to i_*\gamma^j F_{\sZ}[-j]\oplus C$, for $r=1,\ldots, j$.
We consider two cases.

{\em 1st case: $r=1,\ldots, j-1$.} In  this case we have  to show the commutativity of the following diagram
\eq{thm:gysin-tri5.1}{\xymatrix{
i_*\gamma^r F_{\sZ} [-r]\ar[d]_{\lambda^{r-1}_V}\ar[r]^-{\lambda^r_V} &
i_*R\rho_* F_{\sE}\ar[d]^{-\tau}\\
i_*R\rho_*\gamma^1 F_{\sE}[-1]\ar[r]^-{g_{\sE/\tilde{\sX}}} &
R\rho_* F_{\tilde{\sX}},
}}
with $\tau$ as in \eqref{cor:bud0} and $V=\sN_{Z/X}^\vee$. Note that the two compositions
$-\tau\circ \lambda_V^r$ and $g_{\sE/\tilde{\sX}}\circ \lambda_V^{r-1}$ factor automatically
via the forget-support map $R\ul{\Gamma}_Z R\rho_* F_{\tilde{\sX}}\to R\rho_* F_{\tilde{\sX}}$.
By applying $R\ul{\Gamma}_Z$ to the second isomorphism in \eqref{cor:bud1} and using
the isomorphism $R^j\ul{\Gamma}_Z F_{\sX}[-j] \cong R\ul{\Gamma}_Z F_{\sX}$ from \cite[Cor 8.6(3)]{S-purity}
we obtain
\[R\ul{\Gamma}_Z R\rho_* F_{\tilde{\sX}}\cong
R^j\ul{\Gamma}_Z F_{\sX}[-j]\oplus i_*\tau_{\ge 1} R\rho_{E*} F_{\sE}.\]
Since $\Hom_{D(X_\Nis)}(i_*\gamma^r F_{\sZ}[-r], R^j\ul{\Gamma}_Z F_{\sX}[-j])= 0$ for $r<j$,
we see that it suffices to show the equality \eqref{thm:gysin-tri5.1} after composing with 
\eq{thm:gysin-tri5.2}{\imath_E: 
R\rho_*F_{\tilde{\sX}}\xr{i_E^*} i_*R\rho_{E*}F_{\sE}\xr{\rm can.} i_*\tau_{\ge 1}R\rho_{E*}F_{\sE}.}
Since $\imath_E$ is a section of $\tau$  we are reduced to show
\[R\rho_*(i_E^*\circ g_{\sE/\tilde{\sX}})\circ \lambda^{r-1}_V=
   -\lambda^r_V: i_*\gamma^{r}F_\sZ[-r]\to i_*R\rho_{E*}F_{\sE}.\]
By the definition of $\lambda_V$ (see \eqref{para:pbf1})  and Lemma \ref{lem:cyclecup}\ref{lem:cyclecup3}
it remains to check
\eq{thm:gysin-tri6}{i_E^* \circ g_{\sE/\tilde{\sX}, E}= -c_\xi: \gamma^1 F_{\sE}[-1] \to F_\sE.}
which follows from Proposition \ref{prop:gysin-bc} applied to the cartesian diagram
\[\xymatrix{
E\ar[r]^{=}\ar[d]_{=} & E\ar[d]^{i_E}\\
E\ar[r]^{i_E} & \tilde{X}.
}\]

{\em 2nd case: $r=j$.}
 In this case we have  to check the commutativity of the square 
 \eq{thm:gysin-tri6.6}{\xymatrix{
i_*\gamma^j F_{\sZ}[-j]\ar[d]^{g_{\sZ/\sX}}\ar[rr]^{\lambda_V^{j-1}} & &
R\rho_{E*}\gamma^1 F_\sE[-1]\ar[d]^{R\rho_*(g_{\sE/\tilde{X}})}\\
F_{\sX}\ar[rr]^{\rho^*} & & R\rho_* F_{\tilde{\sX}}.
}}
By the the second isomorphism in \eqref{cor:bud1} we have the vanishing $\imath_E\circ \rho^*=0$,
with $\imath_E$  the map from \eqref{thm:gysin-tri5.2}; hence also $\imath_E\circ \rho^*\circ g_{\sZ/\sX}=0$.
On the other hand, the vanishing 
$\imath_E\circ R\rho_*(g_{\sE/\tilde{\sX}})\circ \lambda^{j-1}_V=0$ 
follows from \eqref{thm:gysin-tri6},
Lemma \ref{lem:cyclecup}\ref{lem:cyclecup3}, and the vanishing $\xi^j=0$ in $\CH^j(E)$, 
which holds since $E$ has relative dimension $j-1$ over $Z$. 
By Corollary \ref{cor:bud} and \cite[Cor 8.6(3)]{S-purity} we have
\begin{align*}
R\ul{\Gamma}_Z R\rho_* F_{\tilde{\sX}} &\cong 
R^j\ul{\Gamma}_Z F_{\sX}[-j]\oplus i_*\tau_{\ge 1}R\rho_{E*} F_{\sE}\\
& \cong R^{j-1}\rho_* R^1\ul{\Gamma}_{E} F_{\tilde{\sX}}[-j]\oplus i_*\tau_{\ge 1}R\rho_{E*} F_{\sE}. 
\end{align*}
Hence it suffices to show  the commutativity of the diagram of sheaves
\eq{thm:gysin-tri6.7}{\xymatrix{
i_*\gamma^j F_{\sZ}\ar[d]^{g_{\sZ/\sX,Z}^j}\ar[rr]^-{H^{j-1}(\lambda_{V})} & &
R^{j-1}\rho_{E*}(\gamma^1 F_{\sE})\ar[d]^{g_{\sE/\tilde{\sX},E}^1}  \\
R^j\ul{\Gamma}_Z F_\sX\ar[rr]^{\rho^*} & &
R^{j-1}\rho_* R^1\ul{\Gamma}_E F_{\tilde{\sX}}.
}}
This is a local question and we may therefore assume that $\sN_{Z/X}=\sO_Z^{\oplus j}$.
Thus we are back at showing the commutativity of \eqref{thm:gysin-tri6.6},
under the additional assumption $\sN_{Z/X}=\sO_Z^{\oplus j}$.
Hence the statement follows from Proposition \ref{prop:gysin-bc}
by observing that the excess normal sheaf in question is in this case is equal to
(see \cite[6.7]{Fu})
\[\sE xc= \rho_E^*\sN_{Z/X}/\sO_E(-1)= \sO_E^{\oplus j}/\sO_E(-1)\]
and that the Whitney sum formula  in this case yields
\[c_{j-1}(\sE xc)\cap E= \xi^{j-1}.\]
This shows the commutativity in the second case $r=j$ and hence completes the proof 
of Claim \ref{thm:gysin-tri-claim2} and the theorem.
\end{proof}
\begin{rmk}The reader should compare Theorem \ref{thm:gysin-tri} with the classical Gysin triangle in the $\A^1$-motivic setting. Recall that for $X\in \Sm$ and $i\colon Z\hookrightarrow X$ a smooth closed subscheme of codimension $i$, there is a distinguished triangle in $\mathbf{DM}_{gm}(k)$, called the Gysin triangle (see e.g. \cite[2.20]{Deg12}),
    \begin{equation}\label{eq:gysinVoev1} M(X-Z) \xrightarrow{j_*} M(X)\xrightarrow{i^*} M(Z)(i)[2i] \xrightarrow{\partial_{X,Z}}  M(X-Z)[1]
    \end{equation}
which gives, after applying any realisation functor $H^{*,*}(-)$ the localisation long exact sequence
\begin{equation}\label{eq:gysinVoev2}
\cdots\to H^{n-2i, j-i}(Z) \xrightarrow{i_*} H^{n,j}(X) \xrightarrow{j^*}H^{n,j}(X-Z) \xrightarrow{\partial_{X,Z}} H^{n+1-2i, j-i}(Z)\to \cdots.
\end{equation}
The most significant difference between our formulation, even when $D=\emptyset$, and the formulation in the $\A^1$-setting is that the cohomology of the open complement $U = X-Z$ of $Z$ in $X$, which appears in \eqref{eq:gysinVoev2} and \eqref{eq:gysinVoev1}, is replaced by the cohomology of the pair $(\tilde{X}, E)$, where $\tilde{X}$ is the blow-up of $Z$ in $X$, and $E$ is the exceptional divisor. In the modulus setting, where smooth schemes get replaced by ``compactifications'' $\sX= (X,D)$, we need then to ``compactify'' $(X-Z)$ without changing its ``homotopy type'', and the pair $(\tilde{X},E)$ does the job. For reduced modulus the formula in Theorem \ref{thm:gysin-tri} is also witnessed in the logarithmic setting, see \cite[7.5]{logmot}.
\end{rmk}
\section{Pushforward}\label{sec:pf}
In this section we  construct a pushforward for $F\in \CItspNis$ along projective morphisms by
using the projective bundle formula from section \ref{sec:pbf} and the Gysin map from section \ref{sec:gysin}.
This is a classical approach which can be found,
e.g., in \cite[III]{Ha66} (for coherent sheaves) and the dual version in \cite[6.]{Fu} (Chow groups) and 
\cite[5.]{Deg08} (motives).
In fact we construct, the pushforward with proper support along quasi-projective morphisms,
which for the K{\"a}hler (resp. the de Rham-Witt) differentials was done in \cite{CR11} (resp. \cite{CR12}).

\begin{defn}\label{para:prop-supp}
We say a family of supports $\Phi$ on an $S$-scheme $X$ is a {\em family of proper supports for  $X/S$},
if $\Phi$ consists of closed subsets in $X$, which are proper over $S$.
\end{defn}

\begin{defn}\label{defn:traceP}
Let $F\in\CItspNis$ and $\sX=(X,D)\in\uMCorls$. Let $V$ be a locally free $\sO_X$-module of rank $n+1$
and denote by $\pi: P=\P(V)\to X$ the projection and set $\sP=(P, \pi^*D)$. 
Let $j:U\inj P$ be an open immersion and  denote by $\pi_U: U\to X$ the restriction of $\pi$
and set $\sU=(U,\pi_U^*D)$.
Let $\Phi$ be a family of proper supports for $U/X$ and let $\Psi$ be a family of supports on $X$
satisfying $\Phi\subset \pi^{-1}_U \Psi$.
We define the morphism in $D(X_\Nis)$
\eq{defn:traceP1}{\tr_{(\sU,\Phi)/(\sX,\Psi)}: R\pi_{U*}R\ul{\Gamma}_\Phi F_{\sU}\to 
R\ul{\Gamma}_\Psi(\gamma^n F)_{\sX}[-n]  }
as the composition 
\mlnl{
R\pi_{U*}R\ul{\Gamma}_\Phi F_{\sU}   \cong R\pi_*R\ul{\Gamma}_\Phi F_{\sP} 
                                                       \xr{\text{enlarge supp}} R\pi_*R\ul{\Gamma}_{\pi^{-1}\Psi}F_{\sP}\\
                 \xrightarrow[\simeq]{\eqref{para:pbf1}^{-1}} \bigoplus_{i=0}^n R\ul{\Gamma}_\Psi(\gamma^i F)_{\sX}[-i]
                                                       \xr{\text{proj.}} R\ul{\Gamma}_\Psi(\gamma^n F)_{\sX}[-n],}
where the first isomorphism is induced from the excision isomorphism
\[Rj_*R\ul{\Gamma}_\Phi F_{\sU}= R\ul{\Gamma}_\Phi Rj_*j^{-1} F_{\sP} = R\ul{\Gamma}_\Phi F_{\sP}\]
stemming from the fact that $\Phi$ is by assumption also a family of supports on $P$.
If  it is clear from the context which families of support we take, we also write $\tr_{\sU/\sX}$ instead of 
$\tr_{(\sU,\Phi)/(\sX,\Psi)}$.
In particular, we write  (see \ref{para:supports} for notation)
\[\tr_{\sP/\sX}:= \tr_{(\sP,\Phi_P)/(\sX,\Phi_X)}: R\pi_{*}F_{\sP}\to (\gamma^n F)_{\sX}[-n]\]
which is simply the projection to the $n$-th component of the inverse of the projective bundle formula \eqref{para:pbf1}.
\end{defn}

\begin{lemma}\label{lem:traceP}
Assumptions and notations as in Definition \ref{defn:traceP}.
\begin{enumerate}[label=(\arabic*)]
\item\label{lem:traceP1} Let $\imath: R\pi_{U*} R\ul{\Gamma}_\Phi F_\sU\to 
R\pi_* R\ul{\Gamma}_{\pi^{-1}\Psi}F_\sP = R\ul{\Gamma}_{\Psi}R\pi_* F_\sP$
 be the natural map ({\em excision} composed with {\em enlarge-supports}).
 Then
\[\tr_{(\sU,\Phi)/(\sX,\Psi)}=R\ul{\Gamma}_{\Psi}(\tr_{\sP/\sX})\circ \imath.\]
\item\label{lem:traceP2} Let $f:Y\to X$ be a morphism in $\Sm$ such that $\sY:=(Y, f^*D)\in\uMCorls$.
We obtain a diagram
\[\xymatrix{
 U_Y\ar@{^(->}[r]\ar[d]_{f_U}\ar@/^1.5pc/[rr]|{\pi_{U_Y}} &    P_Y\ar[d]_{f_P}\ar[r]_{\pi_Y} &    Y\ar[d]^{f}\\
U\ar@{^(->}[r] &  P\ar[r]^{\pi} &      X,\\
}\]
in which the two squares are cartesian.
Set  $\sU_Y=(U_Y, \pi_Y^*f^*D)$. The following diagram commutes
\[\xymatrix{
Rf_* R\pi_{U_Y *} R\ul{\Gamma}_{f^{-1}_U\Phi} F_{\sU_Y}\ar[rr]^{\tr_{\sU_Y/\sY}} & &
Rf_* R\ul{\Gamma}_{f^{-1}\Psi}(\gamma^n F)_{\sY} [-n]\\
R\pi_{U *} R\ul{\Gamma}_{\Phi} F_{\sU}\ar[rr]^{\tr_{\sU/\sX}}\ar[u]^{f_U^*} & &
R\ul{\Gamma}_\Psi(\gamma^n F)_{\sX} [-n].\ar[u]^{f^*}
}\]
\item\label{lem:traceP3} Let $s: X \inj P$ be a section of $\pi$ which is contained in $U$, i.e., $s(X)\subset U$.
Then $s(X)$ defines a proper family of supports for $U/X$ and the following diagram commutes
\[\xymatrix{ 
R\pi_{U*} s_* \gamma^n F_{\sX}[-n]\ar[rr]^{g_{\sX/ \sU, s(X)}}\ar@{=}[drr] & &
R\pi_{U*} R\ul{\Gamma}_{s(X)} F_{\sU}\ar[d]^{\tr_{(\sU, s(X))/ \sX}}\\
& & \gamma^n F_{\sX}[-n],
}\] 
where $g_{\sX/\sU, s(X)}=R\ul{\Gamma}_{s(X)}(g_{\sX/\sU}): 
s_* \gamma^n F_{\sX}[-n] \to R\ul{\Gamma}_{s(X)} F_\sU$
is induced by the Gysin map \eqref{para:gysin1}.
\item\label{lem:traceP4}
Let $V'$ be another locally free $\sO_X$-module of rank $n'+1$ and let $\pi':P':=\P(V')\to X$ be the projection.
Let $U'\subset P'$ be open and $\Phi'$ be a family of proper supports for $U'/X$.
Denote by $\pi'_{U'}$ the restriction of $\pi'$ to $U'$ and set $\sU':=(U', {\pi'_{U'}}^*D)$.
Then $\Xi:= \Phi\times_X \Phi'$ is a proper family of supports for 
$U\times_X U'/ U'$ and for $U\times_X U'/U$ and the following 
diagram commutes
\[\xymatrix{
R(\pi_U\times_X \pi'_{U'})_* R\ul{\Gamma}_{\Xi} (F_{\sU\otimes_{\sX} \sU'})
\ar[rr]^{\tr_{(\sU\otimes_{\sX} \sU',\Xi)/(\sU,\Phi)}} \ar[d]_-{\tr_{(\sU\otimes_{\sX} {\sU'},\Xi)/(\sU',\Phi')}} & &
R\pi_{U*} R\ul{\Gamma}_{\Phi} (\gamma^{n'}F)_{\sU}[-n']
\ar[d]^{\tr_{(\sU, \Phi)/(\sX,\Psi)}}\\
R\pi'_{U'*} R\ul{\Gamma}_{\Phi'} (\gamma^{n}F)_{\sU'}[-n]
\ar[rr]^-{\tr_{(\sU',\Phi')/(\sX,\Psi)}} & &
R\ul{\Gamma}_{\Psi}(\gamma^{n+n'}F)_{\sX}[-(n+n')],
}\]
where $\sU\otimes_{\sX} \sU'=(U\times_X U', (\pi_U\times_X \pi'_{U'})^*D)$.
\item\label{lem:traceP5} Let $i: Z\inj X$ be a smooth closed subscheme of codimension $c$ 
intersecting $D$ transversally and set $\sZ:=(Z, i^*D)\in \uMCorls$.
We obtain the diagram 
\[\xymatrix{
 U_Z\ar@{^(->}[r]\ar@{^(->}[d]_{i_U}\ar@/^1.5pc/[rr]|{\pi_{U_Z}} & 
 P_Z\ar@{^(->}[d]_{i_P}\ar[r]_{\pi_Z} &    
 Z\ar@{^(->}[d]^{i}\\
U\ar@{^(->}[r] &  P\ar[r]^{\pi} &      X,\\
}\]
in which the two squares are cartesian. Set $\sU_Z:=(U_Z, \pi_{U_Z}^*i^*D)$.
Then the following square commutes
\[\xymatrix{
R\pi_{U*} R\ul{\Gamma}_{\Phi}i_{U*} (\gamma^c F)_{\sU_Z}[-c]\ar[r]^-{g_{\sU_Z/\sU}}\ar[d]^{\tr_{\sU_Z/\sZ}}&
R\pi_{U*}R\ul{\Gamma}_{\Phi} F_{\sU}\ar[d]^{\tr_{\sU/\sX}}\\
R\ul{\Gamma}_{\Psi} i_* (\gamma^{c+n}F)_{\sZ}[-c-n]\ar[r]^-{g_{\sZ/\sX}}&
R\ul{\Gamma}_\Psi(\gamma^n F)_{\sX}[-n],
}\]
where $g_{\sU_Z/\sU}$ and $g_{\sZ/\sX}$ are the Gysin maps.
\item\label{lem:traceP5.5} Let $i:Z\inj X$ and $\sZ$ be as in \ref{lem:traceP5} above
and assume that $i$ factors as $Z\inj U\xr{\pi_U} X$, such that $\Psi\cap Z= \Phi\cap Z$.
Then $Z\inj U$ is a closed immersion of codimension $c+n$ and then following square commutes
\[\xymatrix{
     & R\pi_{U*}R\ul{\Gamma}_\Phi F_{\sU}\ar[d]^{\tr_{\sU/\sX}}\\
R\ul{\Gamma}_\Psi i_* \gamma^{n+c}F_{\sZ}\ar[ru]^{g_{\sZ/\sU}}[-c-n]\ar[r]^-{g_{\sZ/\sX}}&
R\ul{\Gamma}_\Psi \gamma^n F_{\sX}[-n].
}\]                           
\end{enumerate}
\end{lemma}
\begin{proof}
\ref{lem:traceP1} holds by definition. For \ref{lem:traceP2} first observe that 
$f_U^{-1}\Phi$ is a family of proper supports
for $U_Y/Y$; by \ref{lem:traceP1} we are reduced to show
\[f^*\circ \tr_{\sP/\sX} = \tr_{\sP_Y/\sY}\circ f_P^*.\]
By definition and with the notation from \ref{para:pbf} this follows from the equality
\eq{lem:traceP6}{c_{\xi_Y^i}\circ \pi_Y^*\circ  f^*= f_P^* \circ c_{\xi^i}\circ \pi^*,\quad i=0,\ldots, n,}
where $\xi_Y=c_1(\sO_{P_Y}(1))\in \CH^1(P_Y)$ and $\xi=c_1(\sO_P(1))\in \CH^1(P)$. Since 
$f_P^*\xi= \xi_Y$ the equality \eqref{lem:traceP6} follows from Lemma \ref{lem:cyclecup}\ref{lem:cyclecup2}.
For \ref{lem:traceP3} it suffices to show
\eq{lem:traceP7}{\tr_{\sP/\sX}\circ g_{\sX/\sP}=\id_{(\gamma^n F)_{\sX}[-n]}.}
Indeed this follows from the equality $R\ul{\Gamma}_{s(X)} F_{\sU}= R\ul{\Gamma}_{s(X)} F_{\sP}$,
the compatibility of the Gysin with restriction along open immersions (see Proposition \ref{prop:gysin-bc}),
and from \ref{lem:traceP1}. By \eqref{prop:pf-gysin1.1} we have 
$g_{\sX/\sP}= c_{s(X)}\circ \pi^*$, where we view $s(X)\in \CH^n(P)$.
The projective bundle formula yields
\[s(X)=\sum_{i=0}^n \pi^*(\alpha_i)\cdot \xi^i, \quad \text{for certain }\alpha_i\in \CH^{n-i}(X).\]
Applying $\pi_*$ we obtain $\alpha_n=X$ from \cite[Exa 3.3.3]{Fu} and the fact that $s$ is a section of $\pi$.
By Lemma \ref{lem:cyclecup}\ref{lem:cyclecup2}, \ref{lem:cyclecup3} we obtain (with the notation from \ref{para:pbf})
\[g_{\sX/\sP}= \sum_{i=0}^{n-1} c_{\xi^i}\circ \pi^*\circ c_{\alpha_i} + c_{\xi^n}\circ \pi^*
=\sum_{i=0}^{n-1}\lambda_V^i\circ c_{\alpha_i} + \lambda^n_V.\]
Thus equality \eqref{lem:traceP7} follows directly from the definition of $\tr_{\sP/\sX}$.
Next \ref{lem:traceP4}.  Note that $\Xi$ is  by definition the smallest family of supports
on $U\times_X U'$ containing all closed subsets of the form $Z\times_X Z'$ with $Z\in\Phi$ and $Z'\in \Phi'$.
Thus $\Xi$ is clearly a family of proper supports over $U$ and $U'$, respectively,
and we have $\Xi\subset (\pi_U\times_X \pi'_{U'})^{-1}(\Psi)$. Using \ref{lem:traceP1} it is easy to see
that the commutativity of the square in \ref{lem:traceP4} is implied by the commutativity of the following diagram
\eq{lem:traceP8}{\xymatrix{
R(\pi\times_X \pi')_* F_{\sP\otimes_\sX \sP'}
\ar[r]^{\tr_{\sP\otimes_{\sX}\sP'/\sP}}\ar[d]_{\tr_{\sP\otimes_{\sX}\sP'/\sP'}}& 
R\pi_* (\gamma^{n'}F)_{\sP}[-n']\ar[d]^{\tr_{\sP/\sX}}\\
R\pi'_*(\gamma^n F)_{\sP'}[-n]\ar[r]^{\tr_{\sP'/\sX}}  & (\gamma^{n+n'}F)_{\sX} [-(n+n')].
}}
Let $\xi=c_1(\sO_P(1))\in \CH^1(P)$ and $\eta=c_1(\sO_{P'}(1))\in \CH^1(P')$.
Denote by $p: P\times_X P'\to P$ and $q: P\times_X P'\to P'$  the projections.
With the notation from \ref{para:pbf} we have for $i, j=0,\ldots, n$, 
\begin{align*}
\lambda_{\pi^*V'}^j\circ \lambda_V^i
&=c_{q^*\eta^j}\circ p^*\circ c_{\xi^i}\circ \pi^* , & & \text{by defn},\\
& = c_{q^*\eta^j}\circ c_{p^* \xi^i}\circ p^*\pi^*, & &  \text{by \ref{lem:cyclecup}\ref{lem:cyclecup2}},\\
&= c_{(q^*\eta^j)\cdot (p^* \xi^i)} \circ q^*{\pi'}^*, & &\text{by \ref{lem:cyclecup}\ref{lem:cyclecup3}},\\
& = c_{p^*\xi^i}\circ c_{q^* \eta^j}\circ q^*{\pi'}^*, & &  \text{by \ref{lem:cyclecup}\ref{lem:cyclecup3}},\\
& = c_{p^*\xi^i}\circ q^*\circ c_{\eta^j} \circ {\pi'}^*,  & &\text{by \ref{lem:cyclecup}\ref{lem:cyclecup2}},\\
& = \lambda_{{\pi'}^*V}^i \circ \lambda_{V'}^j, & &\text{by defn.}
\end{align*}
Now the commutativity of the diagram \eqref{lem:traceP8} follows from this and the definition of $\tr$.
For \ref{lem:traceP5} it suffices as above to show that the following diagram commutes
\[\xymatrix{
R\pi_* i_{P*}(\gamma^c F)_{\sP_Z}[-c]\ar[r]^-{g_{\sP_Z/\sP}}\ar[d]^{\tr_{\sP_Z/\sZ}} &
R\pi_* F_{\sP}\ar[d]^{\tr_{\sP/\sX}}\\
(\gamma^{c+n}F)_{\sZ}[-c-n]\ar[r]^-{g_{\sZ/\sX}} &
(\gamma^n F)_{\sX}[-n].
}\]
By definition of $\tr$ it suffices to show for all $j=0,\ldots, n$
\[g_{\sP_Z/\sP}\circ \lambda^j_{i^*V}=\lambda^j_V\circ g_{\sZ/\sX}: (\gamma^{c+j}F)_{\sZ}[-c-j]\to R\pi_* F_\sP.\]
Since $\lambda_V^j=c_{\xi^j}\circ \pi^*$ and $\lambda_{i^*V}^j= c_{i_P^*\xi^j}\circ \pi_Z^*$
the above equality follows from the Propositions \ref{prop:gysin-cup} and \ref{prop:gysin-bc}.
Finally \ref{lem:traceP5.5}. 
By considering the diagram 
\[\xymatrix{
Z\ar@{^(->}[r]\ar@{=}[dr]  & U\times_X Z\ar[d]\ar@{^(->}[r] & U\ar[d]^{\pi_U}\\
                                      &    Z\ar@{^(->}[r]^i & X
}\]
with cartesian square,  we see that the statement  follows from \ref{lem:traceP5}, \ref{lem:traceP3}, and 
the functoriality of the Gysin map, see Corollary \ref{cor:gysin-fun}.
\end{proof} 

\begin{para}\label{para-qp}
Recall from \cite[Exa 2.1.2(d) and Lem 2.1.3]{TT} that a morphism $f:Y\to X$ in $\Sm$ is quasi-projective in
the sense of \cite[Def (5.3.1)]{EGAII}  if and only if there is a locally free $\sO_X$-module of finite rank $V$ such that 
$f$  factors as an immersion $Y\inj  \P(V)$ followed by the projection $\P(V)\to X$.

We say such a morphism $f$ has {\em relative dimension $r$}, if $r=\dim Y_i-\dim X_j$ is  constant, 
for $Y_i$ ranging through the connected components of $Y$ mapping to the connected component $X_j$ of $X$.
\end{para}


\begin{defn}\label{defn:pfs}
Let $F\in \CItspNis$.
Let $\sX=(X,D)\in\uMCorls$ and let $f: Y\to X$ be a quasi-projective morphism in $\Sm$ 
of relative dimension $r\in \Z$, which is transversal to $D$ (see Definition \ref{defn:ti}). 
Let $\Phi$ be a family of proper supports for $Y/X$ and let $\Psi$ be a family of supports on
$X$ such that $\Phi\subset f^{-1}\Psi$. 
Choose a factorization
\eq{defn:pfs2}{f:Y\xr{i} U\xr{\pi} X,}
where $i$ is a closed immersion of codimension $c$ and $\pi$ is the composition
of an open immersion into a projective bundle over $X$, $U\inj P$, followed by the projection $P\to X$.
Let $n$ be the relative dimension of $\pi$,  so that $r=n-c$.
Set $\sY:=(Y, f^*D)$ and $\sU=(U, \pi_U^*D)$.

For  $e\ge c=\codim(Y,U)$  we define the map 
\eq{defn:pfs1}{(i,\pi)^{e}_*: 
Rf_* R\ul{\Gamma}_\Phi \gamma^e F_{\sY}[-e]\to R\ul{\Gamma}_\Psi \gamma^{e+r}F_\sX[-e-r]}
as the following  composition:
\[
Rf_* R\ul{\Gamma}_\Phi \gamma^e F_{\sY}[-e]  \xr{g_{\sY/\sU}} R\pi_{*} R\ul{\Gamma}_\Phi \gamma^{e-c}F_{\sU}[-e+c]
\xr{\tr_{\sU/\sX} } R\ul{\Gamma}_\Psi\gamma^{e+r}F_\sX[-e-r].
\]

\end{defn}

{
\begin{prop}\label{prop:pfs}
Assumptions as in Definition \ref{defn:pfs}.
\begin{enumerate}[label=(\arabic*)]
\item\label{prop:pfs2} Let $Y\xr{i'} U'\xr{\pi'} X$ be another factorization as in \eqref{defn:pfs2}.
Then  
\[(i,\pi)^{e}_*=(i',\pi')^{e}_*, \quad \text{for all } e\ge \codim(Y, U\times_X U').\]
\item\label{prop:pfs3} Let $g: Z\to Y$ be  a  quasi-projective morphism in $\Sm$ of relative dimension $s$, which is  transversal to $f^*D$.
Let $\Xi$ be a family of proper supports for $Z/Y$ such that $\Xi\subset g^{-1}\Phi$.
Let $Z\xr{i'} U'\xr{\pi'} X$ be a factorization of $fg$ as in \eqref{defn:pfs2} with $\codim(Z, U')=c'$. 
Set $U'_Y=U'\times_X Y$ and $i'_Y:= i'\times g: Z\inj U'_Y$ and 
$\pi'_Y:= p_Y: U'_Y\to Y$.  Set $\sZ:=(Z, g^*f^*D)$ etc.

Then  we have $\codim(Z, U'_Y)=c'+r$, 
$\Xi$ is also a proper family  with supports for $Z/X$
and for  $e\ge  c'+ c +r$ we have a commutative diagram
\[\xymatrix{
R (fg)_* R\ul{\Gamma}_\Xi \gamma^eF_{\sZ}[-e]\ar[rr]^{(i', \pi')^{e}_*}\ar[dr]_-{(i'_Y, \pi'_Y)^{e}_*} &   &
 R\ul{\Gamma}_\Psi \gamma^{e+r+s} F_\sX[-e-r-s]\\
&  Rf_* R\ul{\Gamma}_\Phi \gamma^{e+s}F_\sY[-e-s].\ar[ur]_-{(i,\pi)^{e+s}_*} & 
 }\]
\end{enumerate}
\end{prop}
\begin{proof}
\ref{prop:pfs2}.
We obtain the following diagram in which $ST= S\times_X T$ and all maps are the obvious ones 
\[\xymatrix{
Y\ar[d]\ar[r]   & YU'\ar[r]\ar[d]  &  Y\ar[d]^{i'}  \\
  UY\ar[r]\ar[d]  & UU'\ar[r]\ar[d] & U'\ar[d]^{\pi'} &\\
 Y\ar[r]^{i}& U\ar[r]^\pi  &  X.
}\]
We form the modulus pairs $\sU'$, $\sU\sU'$, etc. in the obvious way by pulling back the divisor from $X$;
all these pairs are in $\uMCorls$ and all morphisms are transversal to the corresponding pullback of $D$.
We can view $\Phi$ as a family of supports on $Y,U,U',P,P'$, and
$\Xi:=\Phi\times_X \Phi$ as a family of supports on $UY, YU', UU', PP',$ etc.
Let $c'=\codim(Y,U')$, $n'=\dim(U'/X)$. We have 
\[\codim(Y, U\times_X U')=n+c'=c+n'=:m.\]
We obtain the following diagram in which the grayish entries keep track of the $\gamma$-twist, the
modulus pair, and the support,  the rest is omitted for readability:
\[\xymatrix@+1pc{
 \textcolor{gray}{\gamma^e_{\sY,\Phi}}
\ar[d]_{g_{\sY/\sU\sY}}\ar[r]^{g_{\sY/\sY\sU'}}\ar@{}[dr]|*+[o][F-]{1}  & 
\textcolor{gray}{\gamma^{e-n'}_{\sY\sU', \Xi}}
\ar[r]^{\tr_{\sY\sU'/\sY}}\ar[d]|{g_{\sY\sU'/\sU\sU'}}\ar@{}[dr]|*+[o][F-]{2}  &  
\textcolor{gray}{\gamma^e_{\sY,\Phi}}\ar[d]^{g_{\sY/\sU'}}  \\
 \textcolor{gray}{\gamma^{e-n}_{\sU{\sY}, \Xi}}
\ar[r]^{g_{\sU\sY/\sU\sU'}}\ar[d]_{\tr_{\sU\sY/\sY}}\ar@{}[dr]|*+[o][F-]{3} & 
\textcolor{gray}{\gamma^{e-m}_{\sU\sU',\Xi}}
\ar[r]^{\tr_{\sU\sU'/\sU'}}\ar[d]|{\tr_{\sU\sU'/\sU}}\ar@{}[dr]|*+[o][F-]{4} & 
\textcolor{gray}{\gamma^{e-c'}_{\sU', \Phi}}\ar[d]^{\tr_{\sU'/\sX}} \\
 \textcolor{gray}{\gamma^e_{\sY, \Phi}}\ar[r]^{g_{\sY/\sU}}& 
\textcolor{gray}{\gamma^{e-c}_{\sU,\Phi}}\ar[r]^{\tr_{\sU/\sX}}  & 
\textcolor{gray}{\gamma^{e+r}_{\sX, \Psi}}\\
}\]
The square \textcircled{1} is commutative by Corollary \ref{cor:gysin-fun},
the squares \textcircled{2} and \textcircled{3} commute by Lemma \ref{lem:traceP}\ref{lem:traceP5},
the square \textcircled{4} commutes by Lemma \ref{lem:traceP}\ref{lem:traceP4}, and finally
we have $\tr_{\sY\sU'/\sY}\circ g_{\sY/\sY\sU'}=\id$ and $\tr_{\sU\sY/\sY}\circ g_{\sY/\sU\sY}=\id$
by Lemma \ref{lem:traceP}\ref{lem:traceP3}.
Thus the whole diagram commutes. It follows that going counterclockwise  from the top left to the bottom right corner
gives the pushforward using the factorization \eqref{defn:pfs2}, 
whereas going clockwise yields the pushforward using the primed-version
of this factorization and therefore these two pushforwards agree.

\ref{prop:pfs3}. We have the commutative diagram
\[\xymatrix@!@C-0.5pc@R-2.5pc{
Z\ar[r]|-{i'_Y}\ar[dr]_g\ar@/^1.5pc/[rrr]|{i'} &  
U'_Y\ar[r]\ar[d]^(0.4){\pi'_Y} & U'\times_X U\ar[d]\ar[r]&
U',\ar[d]^{\pi'}\\
& Y\ar@/_1pc/[rr]|f \ar[r]^-i & U\ar[r]^\pi & X    
}\]
in which the squares are cartesian.
Then \ref{prop:pfs3} follows directly from
Lemma \ref{lem:traceP}\ref{lem:traceP4}, \ref{lem:traceP5}, \ref{lem:traceP5.5} and Corollary \ref{cor:gysin-fun}.
\end{proof}
}


\begin{para}\label{para:pfs-rel0}
Recall from  \ref{para:RSC} that the functor  $ \uomega^{\CI}: \RSC_{\Nis}\to \CItspNis$ is right adjoint to $\uomega_!$.
Let  $F\in\RSC_\Nis$ and set $\tF:=\uomega^{\CI}F$. 
By the weak cancellation theorem  \cite[Cor 3.6]{MS} the natural map from \ref{para:cancel}
\eq{para:tutRSC3}{\kappa_e: \tF\xr{\simeq} \gamma^e(\tF(e)), \quad e\ge 0.}
is an isomorphism. 

Let $\sX$, $f:Y\to X$ and $\Phi, \Psi$ be as in Definition \ref{defn:pfs} and assume the relative dimension of $f$ 
is $r=0$.
We define
\[f_*:= 
\kappa_{e}^{-1}\circ (i,\pi)^{e}_*\circ\kappa_e : Rf_*R\ul{\Gamma}_\Phi \tF_{\sY}\to R\ul{\Gamma}_\Psi \tF_{\sX},\]
where $e\gg 0$. It follows from Proposition \ref{prop:pfs}, that $f_*$ is independent of the choice of a factorization
\eqref{defn:pfs1} and it follows from the commutativity of \eqref{para:cancel2.5} that  it is independent of the choice of $e$.
\end{para}

\begin{proposition}\label{prop:pfs0}
Let $F\in\RSC_\Nis$ and set $\tF:=\uomega^{\CI}F\in \CItspNis$.
Let $\sX=(X,D)$, $f:Y\to X$, $\sY$, $\Psi$ and $\Phi$ be as in \ref{defn:pfs} above and 
assume that $f$ is of relative dimension $r=0$.
\begin{enumerate}[label=(\arabic*)]
\item\label{prop:pfs01}  Let $g: Z\to Y$ be a quasi-projective morphism of relative dimension $0$ in $\Sm$ and 
assume that $g$ is transversal  $f^*D$. Set $\sZ:=(Z, g^*f^*D)$.
Let $\Xi$ be a family of proper supports
for $Z/Y$ such that $\Xi\subset g^{-1}\Phi$. Then $\Xi$ is also a family of proper supports for
$Z/X$ and we have 
\[(f\circ g)_*=f_*g_*: 
Rg_*Rf_* R\ul{\Gamma}_{\Xi} \tF_{\sZ}\to Rf_*R\ul{\Gamma}_{\Phi}\tF_{\sY}\to R\ul{\Gamma}_{\Psi}\tF_\sX.\] 
\item\label{prop:pfs02} Assume $X$ and $Y$ are connected and $\Phi=f^{-1}\Psi$. Then 
\[\deg(Y/X) \cdot = f_* \circ f^* :  
R\ul{\Gamma}_{\Psi} \tF_{\sX}\to Rf_* R\ul{\Gamma}_{\Phi}\tF_{\sY}\to R\ul{\Gamma}_{\Psi}\tF_{\sX},\]
where
\[\deg(Y/X):=\begin{cases} [k(Y):k(X)] & \text{if $f$ is dominant}\\ 0 & \text{else.}\end{cases}\]
\item\label{prop:pfs03}
 Assume $X$ and $Y$ are connected and $f$ is proper and its restriction 
 $f_{|Y\setminus |f^*D|}: Y\setminus |f^*D|\to X\setminus |D|$  is finite and surjective.   Then 
       \[H^0(f_*)=(\Gamma^t_f)^*: f_*\tF_{\sY}\to \tF_{\sX},\]
       where $\Gamma^t_f\in \uMCor(\sX, \sY)$ is the transpose of the graph of  $f$.
\end{enumerate}
\end{proposition}
\begin{proof}
\ref{prop:pfs01}. This follows from Proposition \ref{prop:pfs}\ref{prop:pfs3}.
\ref{prop:pfs02}. Choose a factorization \eqref{defn:pfs2} and $e\ge c=n$, then $f_* \circ f^*$ 
is by Theorem \ref{prop:pf-gysin} equal to the composition
\begin{align*}
R\ul{\Gamma}_{\Psi}\tF_{\sX} & \xr{\kappa_e} R\ul{\Gamma}_{\Psi}\gamma^e(\tF(e))_{\sX} \\
                                             &     \xr{\pi_U^*} R\pi_{U*}R\ul{\Gamma}_{\Phi} \gamma^e(\tF(e))_{\sU}\\
                                             &   \xr{c_Y} R\pi_{U*}R\ul{\Gamma}_{\Phi} \gamma^{e-n}(\tF(e))_{\sU}[n]\\
                                            &     \xr{\tr_{\sU/\sX}} R\ul{\Gamma}_{\Psi} \gamma^{e}(\tF(e))_{\sU}\\
                                         &     \xrightarrow[\simeq]{\kappa_e^{-1}} R\ul{\Gamma}_{\Psi}\tF_{\sX}.
\end{align*}
Let  $\ol{Y}\subset P$ be the closure of $Y$; it induces a cycle in $\CH^n(P)$.
It remains to show:
\begin{claim}\label{prop:pfs0-claim} For $G\in \CItspNis$ the composition
\[\gamma^n G_{\sX}\xr{\pi^*} R\pi_* \gamma^n G_{\sP} \xr{c_{\ol{Y}}} R\pi_* G_{\sP}[n]\xr{\tr_{\sP/\sX}}
\gamma^n G_{\sX},\]
is equal to the multiplication with $\deg(Y/X)$. 
\end{claim}
To this end, let $\xi=c_1(\sO_P(1))\in \CH^1(P)$.
By the projective bundle formula there exist cycles $\alpha_i\in \CH^{n-i}(X)$, such that 
\[\ol{Y}= \sum_{i=0}^n \pi^*\alpha_i\cdot \xi^i,\quad \text{in }\CH^n(P).\]
Applying $\pi_*: \CH^n(P)\to \CH^0(X)=\Z$ we find 
$\deg(Y/X)=\alpha_n$ and hence Claim \ref{prop:pfs0-claim}  follows from 
Lemma \ref{lem:cyclecup}\ref{lem:cyclecup1},\ref{lem:cyclecup2},\ref{lem:cyclecup3}
and the definition of $\tr_{\sP/\sX}$.

\ref{prop:pfs03}. By semipurity we can assume $D=\emptyset$. Since restriction to a dense open subset
is injective for $F\in \RSC_\Nis$ (e.g., \cite[Thm 3.1]{S-purity}) we can reduce to the case where
$X$ and $Y$ are points and $f$ is induced by a finite field extension; since both sides of the equality in
\ref{prop:pfs03} are transitive we can assume that this field extension is simple, so that $f$ factors as
a closed immersion $Y\inj\P^1_X$ followed by the projection $\P^1_X\to X$.
In this situation $H^0(f_*)$ is equal to the composition 
\[F(Y)\xr{\kappa_1} (\gamma^1\tF(1))(Y) \xr{g_{Y/\P^1_X}} H^1(\P^1_Y, \tF(1)_{(\P^1_Y,\emptyset)})
\xr{\tr_{\P^1_X/X}}  (\gamma^1\tF(1))(X)\xr{\kappa_1^{-1}} F(X).\]
In the following we set
\[\Tr_h:=(\Gamma_h^t)^*: h_* F_{U'}\to F_U,\]
for a finite surjective morphism $h: U'\to U$ in $\Sm$. 
\begin{claim}\label{prop:pfs04}
Let $V$ be a locally free $\sO_U$-module of rank $n+1$. Set $P=\P(V)$ and $P'=\P(h^*V)$
and denote by $h':P'\to P$ the base change of $h: U'\to U$, so that we have a commutative diagram
\[\xymatrix{
P'\ar[r]^{h'} \ar[d]^{\pi'} & P \ar[d]^{\pi} \\
U' \ar[r]^{h} & U.
}
\]
Let $G\in\CItspNis$.
Then 
\[\Tr_h\circ \tr_{P'/U'}= \tr_{P/U}\circ \Tr_{h'}: 
H^i(P', G_{P'})\to H^{i-n}(U, (\gamma^n G))_{U}).\]
for all $i$. 
\end{claim}
We prove the claim. Let $\lambda_V^i= c_{\xi^i}\circ \pi^*$ be as in \eqref{para:pbf}, where 
$\xi=c_1(\sO_P(1))\in \CH^1(P)$ and $\pi: P\to U$ is the projection. 
Then by the definition of $\tr_{P/U}$ it suffices to show
\[\Tr_{h'}\circ \lambda^i_{h^*V} =\lambda^i_{V}\circ \Tr_h: 
h_*(\gamma^iG)_{U'}[-i]\to R\pi_* G_{P},\quad \text{for all }i. \]
We know $\pi^*\circ \Tr_h= \Tr_{h'}\circ {\pi'}^*$, since the pullback is compatible with the composition of finite correspondences. 
Thus we are left to show the commutativity of 
\begin{equation} \label{prop:pfs05} \xymatrix{ 
(\gamma^i F)_{P}[-i] \ar[rr]^{c_{\xi^i}} && F_P \\
h'_*(\gamma^i F)_{P'}[-i] \ar[u]_{\Tr_{h'}} \ar[rr]_{c_{{h'}^*\xi^i }} &&  h'_* F_{P'} \ar[u]^{\Tr_{h'}}
}
\end{equation}
Using the definition of $c_{\xi^i}$ in \ref{para:cyclecup} 
and the explicit description \eqref{eq:Tmap2} of the map \eqref{para:Tmap1}
we see that \eqref{prop:pfs05} follows from the projection formula
\[(\Tr_{h'}(a)\otimes \beta\otimes \Delta_P)= \Tr_{h'}( a\otimes {h'}^*\beta\otimes\Delta_{P'})\quad \text{in }
(G\otimes_{\uMPST} \uomega^{*}K^M_i)(P),\]
where $a\in G(P')$, $\beta\in K^M_i(P)$, $\Delta_P$ and $\Delta_{P'}$ are the respective diagonals.
The projection formula follows from the description of $\otimes_{\uMPST}$,  e.g., \cite[Lem 4.3]{RSY},
and the equality of finite correspondences
\eq{prop:pfs06}{(\id_{P'}\times \Gamma_{h'})\circ \Gamma_{\Delta_{P'}}\circ \Gamma^t_{h'}= 
(\Gamma_{h'}^t\times\id_P)\circ \Gamma_{\Delta_P} \in \mathbf{Cor}(P, P'\times P), }
which can be deduced from the cartesian diagram
\[\xymatrix{
P'\ar[rr]^-{(\id_{P'}\times h')\circ \Delta_{P'}}\ar[d]^{h'} & & P'\times P\ar[d]^{h'\times \id_P}\\
P\ar[rr]^{\Delta_P}     &   & P\times P.
}\]
This completes the proof of the claim.

\medskip
We come back to the proof  \ref{prop:pfs03}. Consider the following commutative diagram
\eq{prop:pfs03.5}{
\xymatrix{
Y\ar@{^(->}[r]\ar@{=}[dr] & \P^1_Y\ar[d]\ar[r]^{f_1} & \P^1_X\ar[d]\\
    & Y\ar[r]^{f}    & X,
}
}
in which the vertical maps are the projections and the square is cartesian. 
Clearly it suffices to show for $G=\tF(1)$ and with the notation from above
\eq{prop:pfs07}{\Tr_f=\tr_{\P^1_X/X}\circ g_{Y/\P^1_X}: (\gamma^1G)(Y)\to (\gamma^1G)(X).}
We compute using \eqref{prop:pfs03.5}
\begin{align*}
\Tr_f& = \Tr_f\circ (\tr_{\P^1_Y/Y}\circ g_{Y/\P^1_Y}) & & \text{by  Lem \ref{lem:traceP}\ref{lem:traceP3}}\\
          &= \tr_{\P^1_X/X}\circ Tr_{f_1}\circ g_{Y/\P^1_Y} & & \text{by  Claim \ref{prop:pfs04}}\\
          &= \tr_{\P^1_{X}/X}\circ g_{Y/\P^1_X} & & \text{by  Lem \ref{lem:gysinTRco1} below.}
\end{align*}
This proves \eqref{prop:pfs07} and finishes the proof of the proposition.
\end{proof}

\begin{lemma}\label{lem:gysinTRco1}
Let $F\in \RSC_{\Nis}$ and $\tF=\uomega^{\CI}F\in \CItspNis$.
Let $f: X_1\to X$ be a finite and surjective morphism in $\Sm$ and let $Z$ be a smooth
$k$-scheme which comes with two closed immersions  $i:Z\inj X$ and $i_1: Z\inj X_1$ both of codimension 1 
such that $i=f\circ i_1$. Then the following diagram commutes in $D(X_{\Nis})$
\[\xymatrix{
f_*i_{1*}(\gamma^1 \tF)_Z [-1]\ar[r]^-{g_{Z/X_1}}\ar@{=}[d] &  f_*\tF_{X_1}\ar[d]^{(\Gamma_f^t)^*}\\
i_*(\gamma^1 \tF)_Z[-1]\ar[r]^-{g_{Z/X}}                                & \tF_X,
}\]
where $\Gamma_f^t$ is the transpose of the graph of $f$ which we view in 
$\uMCor((X,\emptyset), (X_1,\emptyset))=\Cor(X_1,X)$.
\end{lemma}
\begin{proof}
Note that $(\Gamma_f^t)^*$ also induces a morphism on the cohomology with supports
\eq{lem:gysinTRco11}{
f_*R^1\ul{\Gamma}_{Z_1} F_{X_1}\to f_* R^1\ul{\Gamma}_{f^{-1}(Z)}F_{X_1}= R^1\ul{\Gamma}_Z(f_* F_{X_1})
\xr{(\Gamma_f^t)^*} R^1\ul{\Gamma}_Z F_X,
}
where we use  $R^if_*=0$ for all $i>0$, which holds by the finiteness of $f$. Here, $Z_1 = i_1(Z)$, and the first arrow is the enlarge support map. 
Using this map and the local Gysin map from Lemma \ref{lem:local-gysin} we see that the statement is local around
$Z$. By Lemma \ref{lem:loc-cl-sec} we can replace $X$ by a Nisnevich neighborhood of $Z$ 
to find a morphism $q: X\to Z$ such that $q\circ i=\id_Z$.
Set $q_1:= q\circ f: X_1\to Z$; it satisfies $q_1\circ i_1=\id_Z$. 
Whence Theorem \ref{prop:pf-gysin} yields 
\[g_{Z/X}= \imath\circ c_Z\circ q^*: i_*(\gamma^1F)_Z[-1]\to R\ul{\Gamma}_ZF_X\to F_X,\]
\[(\Gamma_f^t)^*\circ g_{Z/X_1}=(\Gamma_f^t)^*\circ  \imath\circ c_{Z_1}\circ q_1^*: 
f_* i_{1*} (\gamma^1 F)_{Z_1}[-1]\to f_*R\ul{\Gamma}_{Z_1}F_{X_1}
\to f_*F_{X_1},\]
where $c_Z$ (resp. $c_{Z_1}$) is defined viewing $Z\in \CH^1_Z(X)$ (resp. $Z_1\in \CH^1_{Z_1}(X_1)$)
and $\imath$ is the forget support map in both cases.
Thus it suffices to show the equality
\eq{lem:gysinTRco12}{
c_Z=\eqref{lem:gysinTRco11}\circ c_{Z_1}\circ f^*: (\gamma^1 F)_X[-1]\to R\ul{\Gamma}_{Z}F_X, }
since clearly we also have the commutativity
\[ (\Gamma^t_f)^*\circ \iota =\iota \circ \eqref{lem:gysinTRco11}  \]
where $\iota$ is again the forget support map. 

Again this statement is local in $Z$ and we can therefore assume that there are global functions
$d\in H^0(X, \sO_X)$ and $d_1\in H^0(X_1, \sO_{X_1})$ with
\[Z= \Div_X(d) \quad \text{and}\quad Z_1=\Div_{X_1}(d_1).\]
Since as cycles we have $f_*Z_1=Z$ we can (by \cite[Prop 1.4]{Fu}) additionally choose $d$ and $d_1$ such that 
\eq{lem:gysinTRco12.5}{\Nm_{X_1/X}(d_1)=d.}
Note that $\Gamma_f^t$ also defines an element in $\uMCor((X,Z), (X_1, Z_1))$ and thus Lemma \ref{lem:divisor-cup}
together with Remark \ref{rmk:gamma} show that   \eqref{lem:gysinTRco12} is implied by
the commutativity of the following diagram
\[\xymatrix{
\tF((\A^1,0)\otimes X_1)\ar[r]^{\eqref{lem:divisor-cup1}} &  \tF(X_1,Z_1)/F(X_1)\ar[d]^{(\Gamma_f^t)^*}\\
\tF((\A^1,0)\otimes X)\ar[u]_{(\id\times f)^*}\ar[r]^{\eqref{lem:divisor-cup1}} & \tF(X,Z)/F(X).
}\]
We prove the commutativity of the above diagram.
By \eqref{lem:divisor-cup2} it suffices to show
\eq{lem:gysinTRco13}{(\Gamma_f^t)^*\Delta_{X_1}^*(d_1\times \id_{X_1})^*(\id_{\A^1}\times f)^*= 
\Delta_X^*(d\times \id_X)^*: \tF((\A^1,0)\times X)\to \tF(X,Z)/F(X),}
where by abuse of notation we denote by $d: (X, Z)\to (\A^1,0)$ 
(resp. $d_1: (X_1, Z)\to (\A^1,0)$) the morphisms of modulus pairs induced by $d$ (resp. $d_1$)
and by $\Delta_X:(X, Z)\to (X,Z)\otimes X$  (resp $\Delta_{X_1}: (X_1,Z)\to (X_1,Z)\otimes X_1$)  the diagonal map.
We have 
\begin{align}
(\Gamma_f^t)^*\Delta_{X_1}^*(d_1\times \id_{X_1})^*(\id_{\A^1}\times f)^* &=
(\Gamma_f^t)^*\Delta_{X_1}^*(\id_{X_1}\times f)^*(d_1\times \id_{X})^*\label{lem:gysinTRco14}\\
     & = \Delta_X^* (\Gamma_f^t\times \id_X)^*(d_1\times \id_{X})^*,\notag
\end{align}
where the second equality is induced by \eqref{prop:pfs06}.
Note that the graph of $d_1$ in $X_1\times\A^1$ is given by $V(t-d_1)$, where $t$ is the coordinate of $\A^1$.
As in \cite[Prop 16.1.1, Prop 1.4]{Fu} we find
\eq{lem:gysinTRco15}{
(\Gamma_f^t\times \id_X)^*(d_1\times \id_{X})^*= (( f\times \id_{\A^1})_*(\Gamma_{d_1})\times \id_X)^*
= (\Div_{X\times\A^1}(P(t))\times \id_X)^*,}
where $P(t)=\Nm_{X_1[t]/X[t]}(t-d_1)\in \sO(X)[t]$ is the minimal polynomial of $d_1$ over $k(X)$.
Note that by \eqref{lem:gysinTRco12.5} we have
\[P(t)= t^n - \Tr_{X_1/X}(d_1) t^{n-1}+\ldots + (-1)^n d, \]
where $n=\deg(X_1/X)$.
Putting \eqref{lem:gysinTRco14} and \eqref{lem:gysinTRco15} together and setting
 $G=\uHom_{\uMPST}(\Ztr(X), \tF)\in\CItspNis$  we find that  \eqref{lem:gysinTRco13} is implied by
\eq{lem:gysinTRco16}{\Div_{X\times \A^1}(t-d)^*=\Div_{X\times\A^1}(P(t))^*: G(\A^1,0)\to G(X,Z)/G(X),}
where we view $\Div_{X\times\A^1}(t-d), \Div_{X\times\A^1}(P(t))\in \uMCor((X,Z),(\A^1,0))$.
To show this we can shrink $X$ around the generic point of $Z$ (by purity, see \cite[Cor 8.6(1)]{S-purity}).
Thus in the following we assume $Z\in X$ is a point with residue  $K=k(Z)$. 
The map $q:X\to Z$ from the beginning of the proof induces a morphism
$X\to \Spec K[d]$ which is {\'e}tale.
 By \cite[Rmk 2.2(1), Lem 4.2, Lem 4.3]{S-purity}
we see that $(\A^1_K=\Spec K[t],0_K)$ and $(X, Z)$ are $V$-pairs over $K$ in the sense of 
\cite[Def 2.1]{S-purity}. There is a canonical identification $0_K\cong Z$.
We claim that $g\in \{\frac{t-d}{t-1}, \frac{P(t)}{(t-1)^n}\}$ is admissible for the pair 
$((\A^1_K,0_K), (X,Z))$ in the sense of \cite[Def 2.3]{S-purity}, i.e., we have to show 
\begin{enumerate}[label=(\arabic*)]
\item $g$ is regular in a neighborhood of $X\times_K 0_K$;
\item\label{lem:gysinTRco17} $\Div_{X\times \A^1}(g)\times_{\A^1} 0_K=\Delta_{0_K}$, where 
$\Delta_{0_K}: 0_K\inj X\times_K 0_K$ is the diagonal (via the identification $0_K=Z$ fixed above);
\item\label{lem:gysinTRco18} $g$ extends to an invertible  function in a neighborhood of $X\times_K \infty_K$ 
in $X\times_K \P^1_K$.
\end{enumerate}
All points are immediate to check.
Therefore \cite[Thm 2.10(2)]{S-purity} yields
\[\Div_{X\times_K \A^1_K}(\tfrac{t-d}{t-1})^*= \Div_{X\times_K \A^1_K}(\tfrac{P(t)}{(t-1)^n})^*: 
G(\A^1_K,0_K)/G(\A^1_K)\to G(X,Z)/G(X).\]
Since $\Div_{X\times \A^1}(t-1)^*(G(\A^1,0))\subset G(X)$, 
this implies \eqref{lem:gysinTRco16} and completes the proof of the lemma.
\end{proof}

\section{Proper correspondence action on reciprocity sheaves}\label{sec:prop-cor-action}
In this  section we fix a reciprocity sheaf $F\in\RSC_\Nis$ and
set $\tF:=\uomega^{\CI}F\in\CItspNis$, see \eqref{para:tut2}.

\subsection{Pushforward and cycle cupping for reciprocity sheaves}\label{subsec:pf-RSC}
\begin{para}\label{para:RSCtwist}
Recall the twists of reciprocity sheaves from \cite[5e]{RSY}:
For $n\ge 1$ we define recursively
\eq{para:RSCtwist1}{F\la 0\ra:= F, \quad F\la n\ra := \uomega_!(\uomega^{\CI}(F\la n-1\ra )(1)),}
where $(-)(1)$ denotes the twist from Definition \ref{def:gtwist}.
Thus $F\la n\ra\in\RSC_\Nis$ and $F\la m+n\ra=F\la m\ra\la n\ra$, for all $m,n\ge 0$.
There exists a natural surjective map
\eq{para:RSC:twist2}{\uomega_!(\tF(n))\surj F\la n\ra,}
which is defined as follows:
Let $G\in\CItspNis$; twisting the adjunction map $G\to \uomega^{\CI}\uomega_! G$ by $(1)$ and applying
$\uomega_!$, yields a map $\uomega_! (G(1))\to (\uomega_! G)\la 1\ra $.
For $G=\tilde{F}$ this yields \eqref{para:RSC:twist2} for $n=1$, and in general we define it recursively by
\[\uomega_!(\tF(n))=\uomega_!(\tF(n-1)(1))\to \uomega_!(\tF(n-1))\la 1\ra\to F\la n-1\ra \la1\ra= F\la n\ra. \]
The surjectivity holds by \cite[5e(4)]{RSY}.
It is not known in general whether this is an isomorphism.

We also define recursively for $n\ge 0$
\eq{para:RSCtwist3}{\gamma^0 F:= F, \quad \gamma^n F:=\uHom_{\PST}(\G_m, \gamma^{n-1}F).}
It follows from  \cite[Prop 2.10]{MS} that for $G\in \CItspNis$ we have 
\eq{para:RSCtwist4}{\gamma^n \uomega_! G = \uomega_! \gamma^n G \in \RSC_\Nis.}
We obtain an isomorphism
\ml{para:RSCtwist6}{\gamma^n F\la n\ra\cong \gamma^{n-1} \uomega_!\gamma^1(\widetilde{(F\la n-1\ra )}(1)) 
\cong \gamma^{n-1} \uomega_!\widetilde{(F\la n-1\ra )}\\
\cong \gamma^{n-1} F\la n-1\ra \cong F, }
where the first isomorphism holds by definition and \eqref{para:RSCtwist4},
the second by the weak cancellation theorem \cite[Cor 3.6]{MS},
the third by $\uomega_!\tau_!\omega^{\CI}=\id$, and the forth by induction.
It is direct to check that the composition 
\eq{para:RSCtwist5}{F=\uomega_!\tF\xrightarrow[\simeq]{\kappa_n}
\uomega_!(\gamma^n\tF(n))\xrightarrow[\simeq]{\eqref{para:RSCtwist4}} 
\gamma^n\uomega_!\tF(n)\xr{\eqref{para:RSC:twist2}}\gamma^n F\la n\ra 
\xrightarrow[\simeq]{\eqref{para:RSCtwist6}} F}
is the identity.
\end{para}

\begin{lemma}\label{lem:gamma-exact}
The functor $\gamma^n: \RSC_\Nis\to \RSC_\Nis$ is exact, for all $n\ge 0$.
Furthermore, if char$(k)=0$,  and 
\[0\to G_1\to G_2\to G_3\to 0\]
is an exact sequence in $\uMNST$ with $G_i\in \CItspNis$, then so is
\[0\to \gamma^n(G_1)\to \gamma^n(G_2)\to \gamma^n(G_3)\to 0.\]
\end{lemma}
\begin{proof}
First recall that $\RSC_\Nis$ is an abelian category, by \cite[Thm 0.1]{S-purity}, and 
that a sequence $0\to F_1\to F_2\to F_3\to 0$ in $\RSC_\Nis$ is exact if and only if
the sequence $0\to (F_1)_X\to (F_2)_X\to (F_3)_X\to 0$ of sheaves on $X_\Nis$ is exact, for any $X\in\Sm$.
It suffices to consider the case $n=1$.
Given a short exact sequence in $\RSC_\Nis$ as above we obtain for $X\in \Sm$ 
a short exact sequence on $\P^1_{X,\Nis}$
\[0\to (\tF_1)_{(\P^1_X,\emptyset)}\to (\tF_2)_{(\P^1_X,\emptyset)}\to(\tF_3)_{ (\P^1_X,\emptyset)}\to 0.\]
Applying $R\pi_*$, with $\pi: \P^1_X\to X$ the structure map we get a short exact sequence
\[0\to R^1 \pi_*(\tF_1)_{(\P^1_X,\emptyset)}\to R^1\pi_*(\tF_2)_{(\P^1_X,\emptyset)}\to
R^1\pi_*(\tF_3)_{ (\P^1_X,\emptyset)}\to 0\]
using the fact that $\pi_*(\tF_3)_{(\P^1_X,\emptyset)}= (\tF_3)_{(X,\emptyset)}$.
Applying the projective bundle formula (see Theorem \ref{thm:pbf}) yields an exact sequence
\[0\to (\uomega_!\gamma^1\tF_1)_X\to (\uomega_!\gamma^1\tF_2)_X\to (\uomega_!\gamma^1\tF_3)_X\to 0.\]
The first statement follows from \eqref{para:RSCtwist4}.

Now assume char$(k)=0$. We show the second statement. Since $\gamma$ is left exact on $\uMPST$
it suffices to show the surjectivity. By  Lemma \ref{lem:crit-surj}, Corollary \ref{cor:gtwist},
and resolution of singularities it suffices to show $\gamma (G_2)_{\sX}\to \gamma(G_3)_{\sX}$
is surjective for all $\sX\in \uMCorls$. This follows from the projective bundle formula, Theorem \ref{thm:pbf}, as above.
\end{proof}

\begin{prop}\label{prop:twistMilnor}For $F\in \CI^{\tau, sp}_{\Nis}$, we have $\ul{\omega}_!\gamma^n  F = \uHom_{\PST}(\sK^M_n, \ul{\omega}_!F)$. In particular, for $F\in \RSC_{\Nis}$, we have $\gamma^n F = \uHom_{\PST}(\sK^M_n, F)$. (This is \cite[Prop 2.10]{MS} for $n=1$.)
\end{prop}
\begin{proof}Thanks to \eqref{para:RSCtwist4} and \eqref{cor:gtwist1}, we have to show that the natural morphism
    \begin{equation}\label{eq.prop:twistMilnor} {\uomega}_! \uHom_{\uMPST}(\uomega^* \sK^M_n, F) \to \uHom_{\PST} (\uomega_! \uomega^*\sK^M_n, \uomega_!F) = \uHom_{\PST} (\sK^M_n, \uomega_!F)\end{equation}
    is an isomorphism for every $F\in \CI^{\tau, sp}_{\Nis}$.  Evaluating both sides of \eqref{eq.prop:twistMilnor} on $X\in \Sm$, we see that we can replace $F$ by $F^X = \uHom_{\uMPST}(\Z_{tr}(X), F)$ and are left to show that
    \[ \Hom_{\PST}(\sK^M_n, \uomega_!F) = (\uomega_! \gamma^n F) (k)=(\gamma^n F)(k).\]
Indeed, we have 
\begin{align*}
    \Hom_{\PST}(\sK^M_n, \uomega_! F) &\cong^{(1)} \Hom_{\uMPST}(\uomega^* \sK^M_n, \uomega^* \uomega_! F)\\
    &\cong^{(2)} \Hom_{\uMPST}(\uomega^* \sK^M_n, \uomega^{\CI} \uomega_! F)\\
    & \cong^{(3)}  (\gamma^n \uomega^{\CI} \uomega_! F)(k) \\
    &\cong^{(4)} H^n(\P^n_k, (\uomega^{\CI} \uomega_! F)_{(\P^n_k,\emptyset)})\\
    &\cong^{(5)} H^n(\P^n_k, F_{(\P^n_k,\emptyset)}) \cong^{(6)} (\gamma^n F)(k)
\end{align*}
where the isomorphism $(1)$ follows from the fact that $\uomega^*$ is fully faithful, $(2)$ follows from the fact that 
$\uomega^* \sK^M_n\in \CI^{\tau}$, by adjunction and definition of $\uomega^{\CI}$, $(3)$ is \eqref{cor:gtwist1},
the isomorphisms $(4)$ and $(6)$ follow from Theorem \ref{thm:pbf} and isomorphism $(5)$ by
$\uomega_!\uomega^{\CI}=\id$.
\end{proof}
\begin{para}\label{para:RSCcyclecup}
Let $X\in \Sm$, let $\Phi$ be a family of supports on $X$, and let $\alpha\in \CH^r_\Phi(X)$.
Then we define
\[C_\alpha: F_X\to R\ul{\Gamma}_{\Phi} F\la r\ra_X [r]\]
as the composition
\[F_X=\tF_{(X,\emptyset)}\xr{\kappa_r} \gamma^r \tF(r)_{(X,\emptyset)}
\xr{c_\alpha} R\ul{\Gamma}_{\Phi}\tF(r)_{(X,\emptyset)}[r]
\xr{\eqref{para:RSC:twist2}} R\ul{\Gamma}_{\Phi}F\la r\ra_X [r].\]
We note that $C_\alpha$ satisfies the analogous properties of $c_\alpha$ listed in Lemma \ref{lem:cyclecup}.
This is immediate for \ref{lem:cyclecup}\ref{lem:cyclecup1} - \ref{lem:cyclecup2}.
We give the argument for the analogous property of Lemma \ref{lem:cyclecup}\ref{lem:cyclecup3}:
Let  $\alpha\in \CH^r_\Phi(X)$ and $\beta\in \CH^s_\Psi(X)$. 
Set $m=i+j$ and consider the following diagram in which we omit all the supports for readability
\[
\xymatrix{
F_X\ar[r]^-{\kappa_i}\ar[dr]_{\kappa_{m}} & 
\gamma^i\tF(i)_{(X,\emptyset)} \ar[d]^{\gamma^i(\kappa_j)}\ar[r]^{c_\alpha}&
\tF(i)_{(X,\emptyset)}[i]\ar[r]^{\eqref{para:RSC:twist2}}\ar[d]^{\kappa_j}&
F\la i\ra_X[i]\ar[d]^{\kappa_j}\\
 &
\gamma^{m}\tF(m)_{(X,\emptyset)}\ar[r]^{c_\alpha}\ar[dr]_{c_{\alpha\cdot\beta}}&
\gamma^j \tF(m)_{(X,\emptyset)}[i]\ar[d]^{c_\beta}\ar[r]^{\eqref{para:RSC:twist2}^a}&
\gamma^j\widetilde{F\la i\ra}(j)_{(X,\emptyset)}[i]\ar[d]^{c_\beta}\\
& &
\tF(m)_{(X,\emptyset)} [m]\ar[r]^-{\eqref{para:RSC:twist2}^a}\ar[dr]_{\eqref{para:RSC:twist2}}&
\widetilde{F\la i\ra}(j)_{(X,\emptyset)}[m]\ar[d]^{\eqref{para:RSC:twist2}}\\
& & &
F\la m\ra_X [m],
}
\]
where  $\eqref{para:RSC:twist2}^a: \tG(n)\to \widetilde{G\la n\ra}$, $G\in\RSC_\Nis$, is induced 
by adjunction from \eqref{para:RSC:twist2}. All the squares in the diagram commute by functoriality,
the top triangle commutes  by \eqref{para:cancel2.5}, the triangle in the middle commutes by
Lemma  \ref{lem:cyclecup}\ref{lem:cyclecup3}, and the commutativity of the bottom triangle follows from 
the definition of the map \eqref{para:RSC:twist2}. Hence the whole diagram commutes and we obtain
$C_{\alpha\cdot \beta}=C_{\beta}\circ C_{\alpha}$.
\end{para}

\begin{para}\label{para:RSC-pf}
Let $f: Y\to X$ be quasi-projective in $\Sm$, let $\Phi$ be a family of proper supports for $Y/X$
and $\Psi$ a family of supports on $X$ such that $\Phi\subset f^{-1}\Psi$.
Set $r:=\dim Y-\dim X\in \Z$. For $a\ge 0$ with $a+r\ge 0$ and $b \ge 0$ we define
\[f_* : Rf_*R\ul{\Gamma}_{\Phi} \gamma^b (F\la a+ r\ra)_Y[r] \to R\ul{\Gamma}_{\Psi} \gamma^b(F\la a\ra)_X\]
as follows: let $Y\xr{i} U\xr{\pi} X$ be a factorization as in \eqref{defn:pfs2} with $c=\codim(Y,U)$ and $n=\dim(U/X)$, 
so that $r=n-c$; let $e\ge c$, then we define $f_*$ as the composition
\begin{align*}
Rf_* R\ul{\Gamma}_\Phi \gamma^b(F\la a+r\ra)_Y[r]&
\xr{\eqref{para:RSCtwist4}} Rf_* R\ul{\Gamma}_\Phi \gamma^b(\widetilde{F\la a+r\ra})_{(Y,\emptyset)}[r]\\
&\xr{\kappa_e} Rf_* R\ul{\Gamma}_\Phi \gamma^{b+e}(\widetilde{F\la a+r\ra}(e))_{(Y,\emptyset)}[r]\\
&\xr{(i,\pi)_*^e} R\ul{\Gamma}_\Psi \gamma^{b+e+r}(\widetilde{F\la a+r\ra}(e))_{(X,\emptyset)}\\
&\xr{\eqref{para:RSCtwist4}} R\ul{\Gamma}_\Psi \gamma^{b+e+r}\uomega_!(\widetilde{F\la a+r\ra}(e))_X\\
&\xr{\eqref{para:RSC:twist2}}  R\ul{\Gamma}_\Psi \gamma^{b+e+r}(F\la a+e+r\ra)_X\\
&\xr{\eqref{para:RSCtwist6}} R\ul{\Gamma}_\Psi \gamma^b (F\la a\ra)_X.
\end{align*}
It follows from Proposition \ref{prop:pfs}\ref{prop:pfs2} that $f_*$ is independent of the choice of 
the factorization \eqref{defn:pfs2} and it follows from the  commutativity of \eqref{para:cancel2.5} 
that it is independent of the  choice of $e$. 
\end{para}

\begin{rmk}
By \eqref{para:RSCtwist6} we have 
\[\gamma^b(F\la a\ra)=\begin{cases} F\la a-b\ra & a\ge b\\ \gamma^{b-a}F & a\le b. \end{cases}\]
Above we work with $\gamma^b(F\la a\ra)$ so that we don't have to distinguish the two cases.
\end{rmk}

\begin{thm}\label{prop:bc-pf}
Let $f:Y\to X$, $\Phi, \Psi$, $a,r, b$ be as in \ref{para:RSC-pf}.
\begin{enumerate}[label=(\arabic*)]
\item\label{prop:bc-pf0} 
$\id_{X*}: F_X\to F_X$ is the identity.
\item\label{prop:bc-pf0.5}
Let $g: Z\to Y$ be another quasi-projective morphism in $\Sm$ of relative dimension $s=\dim Z-\dim Y$ and
let $\Xi$ be a family of proper supports for $Z/Y$ with $\Xi\subset g^{-1}\Phi$, then 
for $a\ge 0$ with $a+r+s\ge 0$ and $a+r\ge 0$ and $b\ge 0$ we have 
\[
(f\circ g)_*= f_*\circ g_*: R(f\circ g)_* R\ul{\Gamma}_{\Xi}\gamma^b( F\la a+r+s\ra)_Z[r+s]
     \to R\ul{\Gamma}_\Psi \gamma^b(F\la a \ra)_X.
\]
\item\label{prop:bc-pf1}
Let $h: X'\to X$ be a smooth morphism. We form the base change diagram
\[\xymatrix{
Y'\ar[r]^{h'}\ar[d]^{f'} & Y\ar[d]^f\\
X'\ar[r]^{h} & X.
}\]
Set $\Psi'=h^{-1}\Psi$, $\Phi'= {h'}^{-1}\Phi$.
The following diagram commutes
\[\xymatrix{
R f_* R\ul{\Gamma}_{\Phi} \gamma^b(F\la a+r\ra)_{Y}[r] \ar[r]^-{f_*}\ar[d]^{{h'}^*}&
R\ul{\Gamma}_{\Psi} \gamma^b(F\la a\ra)_{X}\ar[d]^{h^*}\\
R (fh')_*R\ul{\Gamma}_{\Phi'} \gamma^b(F\la a+r\ra)_{Y'} [r]\ar[r]^-{f'_*} &
Rh_* R\ul{\Gamma}_{\Psi'} \gamma^b(F\la a\ra)_{X'}.
}\] 
\item\label{prop:bc-pf2}
Let $\Xi$ be some family of supports and $\alpha\in \CH^s_\Xi(X)$.
Then $f^*\alpha\in \CH^s_{f^{-1}\Xi}(Y)$ and the following diagram commutes
\[
\xymatrix{
Rf_* R\ul{\Gamma}_{\Phi} \gamma^b(F\la a+r\ra)_{Y}[r]\ar[r]^-{C_{f^*\alpha}}\ar[d]_-{f_*} &
Rf_* R\ul{\Gamma}_{\Phi\cap f^{-1}\Xi} \gamma^b(F\la a+r+s\ra) _Y[r+s]\ar[d]^-{f_*}\\
R\ul{\Gamma}_\Psi \gamma^b(F\la a\ra)_{X}\ar[r]^-{C_\alpha} &
R\ul{\Gamma}_{\Psi\cap \Xi} \gamma^b(F\la a+s\ra)_X [s].
}
\]
\item\label{prop:bc-pf3}
Let $\beta\in\CH^s_{\Phi}(Y)$.
Then $f_*\beta\in \CH^{s-r}_{\Psi}(X)$ and the following diagram commutes 
\[\xymatrix{
Rf_* F_Y\ar[r]^-{C_\beta} & Rf_*R\ul{\Gamma}_{\Phi} F\la s \ra_Y[s]\ar[d]^-{f_*}\\
F_X\ar[u]^-{f^*}\ar[r]^-{C_{f_*\beta}} & R\ul{\Gamma}_\Psi F\la s-r\ra_X[s-r].
}\]
\end{enumerate}
\end{thm}
\begin{proof}
\ref{prop:bc-pf0} follows from Lemma \ref{lem:traceP}\ref{lem:traceP3} and from \eqref{para:RSCtwist5}
being the identity. 
\ref{prop:bc-pf0.5}  follows from  {Proposition \ref{prop:pfs}\ref{prop:pfs3} and \eqref{para:cancel2.5}.}
\ref{prop:bc-pf1} follows from Proposition \ref{prop:gysin-bc} and Lemma \ref{lem:traceP}\ref{lem:traceP2}.
\ref{prop:bc-pf2} follows from Proposition \ref{prop:gysin-cup} and by the definition of $\tr_{\sU/\sX}$ 
from the following equality for a projective bundle $\pi: P=\P(V)\to X$
\[c_{\pi^*\alpha}\circ \lambda^i_V= c_{\pi^*\alpha}\circ c_{\xi^i}\circ \pi^*=\lambda_V^i\circ c_\alpha,\]
which follows from Lemma \ref{lem:cyclecup}. 
\ref{prop:bc-pf3} By  Corollary \ref{cor:pf-gysin} we are reduced to the case where 
$f=\pi: \P(V)\to X$ is a projective bundle with $V$ locally free of rank $r+1$ and $\Phi=Y$, 
$\Psi=X$. In this case we can write $\beta=\sum_{i=0}^{r}\xi^i\cdot \pi^*\alpha_i$, with
$\alpha_i\in \CH^i(X)$ and $\xi^i= c_1(\sO_\P(1))^{i}$. Thus $\pi_*\beta=\alpha_r$ and hence 
the commutativity of the diagram follows in this case from the definition of $\tr_{P/X}$ and the equality
\[c_\beta\circ \pi^*=\sum_{i=0}^{r} \lambda^i_V\circ c_{\alpha_i}, \]
which holds by Lemma \ref{lem:cyclecup}\ref{lem:cyclecup1}, \ref{lem:cyclecup2}, \ref{lem:cyclecup3}.
\end{proof}

\subsection{Proper correspondence action}
In this subsection we fix a scheme $S$ separated and of finite type over $k$.

\begin{para}\label{para:CS}
We denote by $C_S$ the category with objects the $S$-schemes $X\to S$ with the property that
the induced map $X\to \Spec k$ is smooth and quasi-projective; the morphisms are given by 
\[C_S(X, Y)= \bigoplus_i \CH^{\dim Y_i}_{\Phi^{\rm prop}_{X\times_S Y_i}}(X\times Y_i),\]
where for simplicity we write $X$ instead of $X\to S$, 
where $Y=\sqcup_i Y_i$ is the decomposition into connected components,
and where $\Phi^{\rm prop}_{X\times_S Y_i}$ is the family of supports on $X\times Y_i$ consisting of those
closed subsets which are contained in $X\times_S Y_i$ and are proper over $X$;
the composition is defined by 
\[C_S(X_1, X_2)\times C_S(X_2,X_3)\to C_S(X_1, X_3), \]
\[(\alpha,\beta)\mapsto \beta\circ \alpha:= p_{13*} (p_{12}^*\alpha \cdot p_{23}^*\beta),\]
where $p_{i,j}: X_1\times X_2\times X_3\to X_i\times X_j$ are the projections, the pullbacks
$p_{ij}^*$ are induced by flat pullback, 
the intersection product is given by \eqref{para:ext-prod8}, and the pushforward is well-defined
since $p_{13}$ is proper along 
$(\Phi^{\rm prop}_{X_1\times_S X_2}\times_k X_3)\cap (X_1\times_k \Phi^{\rm prop}_{X_2\times_S X_3})$
and maps this family of supports into $\Phi^{\rm prop}_{X_1\times_S X_3}$.
It follows from \cite[Prop 1.1.34, Prop 1.3.10]{CR11} that $C_S$ is a category
and the identity in $C_S(X,X)$ is induced by the diagonal $\Delta\subset X\times_S X$
(cf. also \cite[Prop 16.1.1]{Fu}).

Note that for $S=\Spec k$ and $X,Y\in C_S$ we have a natural map $\Cor(X,Y)\to C_S(X,Y)$
which is compatible with composition.
\end{para}

\begin{para}\label{para:cycle-action}
Let $F\in \RSC_\Nis$. 
Let $S$ be a $k$-scheme, let  $(f:X\to S)$, $(g:Y\to S)\in C_S$.
For $\alpha\in C_S(X, Y)$ we define a morphism  in $D^+(S_\Nis)$
\eq{para:cylce-action1}{\alpha^*: Rg_* F_Y\to Rf_* F_X }
as follows: set  $\tF:=\tau_!\omega^{\CI} F$; it suffices to consider the case that $Y$ is of pure dimension $d$; 
then \eqref{para:cylce-action1} is defined to be the composition
\begin{align*}
Rg_* F_Y &   \xr{p_Y^*} Rg_* Rp_{Y*} F_{X\times Y} \\
   & \xr{C_\alpha} 
   Rg_* Rp_{Y*}R\ul{\Gamma}_{\Phi^{\rm prop}_{X\times_S Y}}F\la d\ra_{X\times Y}[d]\\
   & \cong Rf_* Rp_{X*}R\ul{\Gamma}_{\Phi^{\rm prop}_{X\times_S Y}} F\la d\ra_{X\times Y}[d]\\
   &\xr{p_{X*}} Rf_*F_{X},
\end{align*}
where $C_\alpha$ is defined in \ref{para:RSCcyclecup}, $p_X, p_Y:X\times Y\to X, Y$ denote the projections,
$p_{X*}$ is the pushforward from \ref{para:RSC-pf}, and
the isomorphism in the third line follows from the equality
\[g_*p_{Y*}\ul{\Gamma}_{\Phi^{\rm prop}_{X\times_S Y}}(G)= 
f_*p_{X*}\ul{\Gamma}_{\Phi^{\rm prop}_{X\times_S Y}}(G),\]
for any sheaf $G$ on $X\times Y$.
\end{para}
 
\begin{prop}\label{prop:cycle-action}
Let the assumptions be as in \ref{para:cycle-action} above.
\begin{enumerate}[label=(\arabic*)]
\item\label{prop:cycle-action1}
Let $(h: Z\to S)\in C_S$ and $\beta\in C_S(Y, Z)$. Then
\[\alpha^*\circ \beta^* = (\beta\circ \alpha)^*: Rh_* F_Z\to Rf_* F_X.\]
\item\label{prop:cycle-action2}
Let $\nu: S\to T$ be a morphism of separated and finite type  $k$-schemes. 
Then $\nu$ induces a functor $\nu_*: C_S\to C_T$. Furthermore, for $\alpha\in C_S(X,Y)$
we have 
\[(\nu_*\alpha)^*= R\nu_*(\alpha^*): R\nu_*Rg_* F_Y\to R\nu_* Rf_* F_X.\]                 
\item\label{prop:cycle-action3}
Let $h: X\to Y$ be a $k$-morphism and denote by $[\Gamma_h]\in C_Y(X,Y)$
         the class induced by the graph of $h$. Then 
         \[[\Gamma_h]^* =h^*: F_Y\to Rh_* F_X.\]
\item\label{prop:cycle-action4}
Let $h: X\to Y$ be a proper $S$-morphism of relative dimension $0$, then the transpose of the graph of 
           $h$ defines a class $[\Gamma^t_h]\in C_Y(Y,X)$ and 
           \[[\Gamma^t_h]^*= h_*: Rh_* F_X\to F_Y,\]
           where $h_*$ is induced by the pushforward from Proposition \ref{prop:pfs0}.
\item\label{prop:cycle-action4.1}
Let $V\in \Cor(X,Y)$ be a finite correspondence and denote by $[V]$ its image in $C_{\Spec k}(X,Y)$.
Then
\[V^*=[V]^*: F(Y)\to F(X).\]
\end{enumerate}
 \end{prop}
\begin{proof}
For \ref{prop:cycle-action1} it  suffices to show that an equality 
of maps $Rh_* F_Z\to Rf_* F_X$ (with the obvious notation)
\[(p^{XY}_{X*}\circ C_\alpha \circ p^{XY*}_Y)\circ (p^{YZ}_{Y*}\circ C_\beta\circ p^{YZ*}_Z)
=p^{XZ}_{X*}\circ C_{\beta\circ \alpha}\circ p_Z^{XZ*}\]
This follows directly from Theorem \ref{prop:bc-pf}\ref{prop:bc-pf0} - \ref{prop:bc-pf3} and \ref{para:RSCcyclecup}
(cf. Lemma \ref{lem:cyclecup}\ref{lem:cyclecup2}, \ref{lem:cyclecup3}).
The first statement of \ref{prop:cycle-action2} follows from the fact that $X\times_S Y$ is closed in $X\times_T Y$;
the second statement is direct from the definition.
For \ref{prop:cycle-action3} we first observe that 
\eq{prop:cycle-action5}{\id=\Delta_X^*: F_X\to F_X,}
where $\Delta_X$ denotes the class of the diagonal in $\CH^{d_X}_{\Delta_X}(X\times X)$.
Indeed this follows from Theorem \ref{prop:bc-pf}\ref{prop:bc-pf3} and the fact that 
$C_X: F_X\to F_X$ is the identity, where we view $X=p_{X*}(\Delta_X)\in \CH^0_X(X)=\Z$.
Now \ref{prop:cycle-action3} holds by
\[ 
[\Gamma_h]^*  =p_{X*}\circ C_{(\id_X\times h)_*\Delta_X}\circ p_Y^*
                        =p_{X*} \circ C_{\Delta_X}\circ p_X^*\circ h^* 
                        = h^*,
\]              
where the second equality holds by \ref{prop:bc-pf}\ref{prop:bc-pf3} and the third by \eqref{prop:cycle-action5}.
The proof of \ref{prop:cycle-action4} is similar.
Finally \ref{prop:cycle-action4.1}. By the injectivity of the restriction map along a dense open immersion
(e.g., \cite[Thm 3.1]{S-purity}) we can shrink $X$ around its generic points and henceforth assume
that $X$ and $V$ are smooth and irreducible.
Denote by $h: V\to Y$ and $f: V\to X$ the maps induced by projection; note that $f$ is finite and surjective.
Denote by $\Gamma_h$ the graph of $h$ etc. We have 
$V=\Gamma_h\circ \Gamma_f^t$. By \ref{prop:cycle-action1},  \ref{prop:cycle-action3}, \ref{prop:cycle-action4}
we are reduced to show
\[(\Gamma_f^t)^*=H^0(f_*): F(V)\to F(X).\]
This follows from Proposition \ref{prop:pfs0}\ref{prop:pfs03}.
\end{proof}

\begin{para}\label{para:coracx}
We explain how to extend the cycle action to bounded below complexes in $\RSC_\Nis$.
Let $F^\bullet\in {\rm Comp}^+(\RSC_\Nis)$ be a bounded below complex of reciprocity sheaves.
Let $(f: X\to S)$, $(g: Y\to S)\in C_S$ and $\alpha\in C_S(X,Y)$. Then we define 
\eq{para:coracx1}{\alpha^*: Rg_*F^\bullet_Y\to Rf_*F^\bullet_X \qquad \text{in } D^+(S_\Nis)}
as follows:
Denote by $\Ij_S$ the category of injective Nisnevich sheaves on $S$.
By, e.g., \cite[\href{https://stacks.math.columbia.edu/tag/013V}{Tag 013V}]{stacks-project}, we have an equivalence
of categories 
\eq{para:coracx2}{K^+(\Ij_S)\xr{\simeq} D^+(S_\Nis),}
where $K^+$ denotes the homotopy category of bounded below complexes.
The inverse of this equivalence induces a resolution functor $j_S: C^+(S_\Nis)\to K^+(\Ij_S)$,
which for any bounded below complex $C^\bullet$ comes with a quasi-isomorphism 
of complexes $C^\bullet\to j_S(C^\bullet)$.
In fact we can choose such $j_S$ that  for a complex $C^\bullet$ we have 
\eq{para:coracx3}{ j_S(C^\bullet)= {\rm tot}( \ldots\to j_S(C^i)\to j_S(C^{i+1})\to \ldots),}
where ${\rm tot}(\text{double complex})$ denotes the associated total complex.
By construction $\alpha$ induces a commutative diagram in $D^+(S_\Nis)$ for all $i$
\[\xymatrix{
Rg_* F^i_Y\ar[r]^{\alpha^*}\ar[d]^d & Rf_*F^i_X\ar[d]^d\\
Rg_* F^{i+1}_Y\ar[r]^{\alpha^*} & Rf_*F^{i+1}_X,
}\]
where $d: F^i\to F^{i+1}$ is the differential in the complex.
Using the resolution functors on $X$, $Y$, and $S$ this translates into a commutative diagram in $K^+(\Ij_S)$
\[\xymatrix{
j_S(g_* j_Y(F^i_Y))\ar[r]^{\alpha^*}\ar[d]^d & j_S(f_*j_X(F^i_X))\ar[d]^d\\
j_S(g_* j_Y(F^{i+1}_Y))\ar[r]^{\alpha^*} & j_S(f_*j_X(F^{i+1}_X)).
}\]
Hence $\alpha^*$ induces a morphism from the total complex of the left column (running over all $i$)
to the total complex of the right column, using \eqref{para:coracx3} (and an argument using triple complexes) 
it is direct to check that the latter can be identified with a morphism
\[j_S(g_*j_Y(F_Y^\bullet))\to j_S(f_*j_X(F_X^\bullet)), \]
which under the equivalence \eqref{para:coracx2} induces the morphism \eqref{para:coracx1}.

It follows from Proposition \ref{prop:cycle-action} that the above construction in fact defines a functor
\eq{para.coracx4}{C_S\to D^+(S_\Nis), \quad (f: X\to S)\mapsto Rf_* F^\bullet_{X}.}
This functor is natural in $F^\bullet$ in the obvious sense.
\end{para}

\begin{lemma}\label{lem:cas}
Let $F^\bullet$, $(f: X\to S)$, $(g:Y\to S)$, and $\alpha$ be as in \ref{para:coracx} above.
Assume $g$ is projective,  $Y$ has pure dimension $d$,
 and $\alpha$ lies in the image of the natural map $\imath:\CH^d_{Z\times_S Y}(X\times Y)\to C_S(X,Y)$, for some closed subset $Z\subset X$.
Then \eqref{para:coracx1} factors as
\[Rg_* F^\bullet_Y\to Rf_*R\ul{\Gamma}_Z F^\bullet_X \to Rf_*F^\bullet_X,\]
where the second map is the forget-supports map.
\end{lemma}
\begin{proof}
We first consider the case of a sheaf $F\in \RSC_\Nis$.
Let $\alpha_0\in \CH^d_{Z\times_S Y}(X\times Y)$ with $\imath \alpha_0=\alpha$.
By Lemma \ref{lem:cyclecup}\ref{lem:cyclecup1.5} and the definition of $C_\alpha$ in \ref{para:RSCcyclecup}
we have $\imath C_{\alpha_0}= C_{ \alpha}$, where by abuse of notation we 
denote the enlarge support map $R\ul{\Gamma}_{Z\times_S Y}\to R\ul{\Gamma}_{\Phi^{\rm prop}_{X\times_S Y}}$ also
by $\imath$. The following diagram commutes by construction of the pushforward (see \ref{para:RSC-pf}) 
\[
\xymatrix{
Rf_* Rp_{X*}R\ul{\Gamma}_{Z\times_S Y}F\la d\ra_{X\times Y}[d]\ar[r]^{\imath}\ar[d]^{p_{X*}} & 
 Rf_* Rp_{X*}R\ul{\Gamma}_{\Phi^{\rm prop}_{X\times_S Y}}F\la d\ra_{X\times Y}[d]\ar[d]^{p_{X*}}\\
Rf_*R\ul{\Gamma}_Z F_X\ar[r] &
Rf_*F_X.
}
\]
Therefore, $p_{X*}\circ C_{\alpha_0}\circ p_Y^*: Rg_*F_Y\to Rf_*R\ul{\Gamma}_Z F_X$ induces the looked for factorization.
The case of a complex $F^\bullet$ follows directly from the sheaf case by construction of the correspondence action in
\ref{para:coracx}. 
\end{proof}

The proof of the following Proposition is inspired by \cite[Lem 8.1]{CL17}, where
a similar result is proven for the cohomology of the de Rham-Witt complex.
\begin{prop}\label{prop:cor-van-top}
Let $(f: X\to S)$  and $(g: Y\to S)\in C_S$. 
Let $V\subset X\times_S Y$ be an integral closed subscheme with $\dim V=\dim X$, which is proper over $X$.
Assume  the closure $V_Y$ of the image of $V$ in $Y$ has codimension $r$.
Let $F\in\RSC_\Nis$ and assume $F(\xi) =0$, for all points $\xi$ which are finite and separable over the generic point of $V_Y$.
Then 
\[0= [V]^*: Rg_* F_Y\to Rf_* F_X.\]
\end{prop}
\begin{proof}
Let $V$ and $F$ be as in the assumption, except that $V$ does not need to be proper over $X$.
We can assume $Y$ is of pure dimension $d$.
\begin{claim}\label{prop:cor-van-top-claim}
The following composition is zero
\begin{equation}\label{eq:proof:prop:cor-van-top}
    Rg_* F_Y \xr{p_2^*} R(g p_2)_* F_{X\times Y}\xr{C_{[V]}} R(g p_2)_*R\ul{\Gamma}_V F\la d\ra_{X\times Y}[d].\end{equation}
\end{claim}
The claim clearly implies the statement, since $[V]^*$ factors via \eqref{eq:proof:prop:cor-van-top} according to  \eqref{para:cylce-action1} and Lemma \ref{lem:cas}. 
By \cite[Cor 8.6(1)]{S-purity}  we have $R\ul{\Gamma}_V F\la d\ra_{X\times Y}$ is concentrated in degree $\ge d$
and hence $C_{[V]}: F_{X\times Y}\to R\ul{\Gamma}_V F\la d\ra_{X\times Y}[d]$ 
factors via the natural map $\sH_V^d(F\la d\ra)\to R\ul{\Gamma}_V F\la d\ra_{X\times Y}[d]$.
Thus it suffices to prove the claim with the complex
$R\ul{\Gamma}_V F\la d\ra_{X\times Y}[d]$ replaced by the sheaf $\sH_V^d(F\la d\ra)$.
For $\sU$ a Nisnevich cover of $Y$ and  $W\to Y$ \'etale denote by $\sU_W$ the induced cover of $W$ and 
by $C(\sU_W, F)$ the \v{C}ech complex of $F_W$.
Denote by $\sC(\sU, F)$ the complex of sheaves on $Y$,  given by $W\mapsto C(\sU_W, F)$.
The natural map $F_Y\to \sC(\sU, F)$ is a resolution (cf. \cite[II, Thm 5.2.1]{God}).
Let $\sC(F)=\varinjlim_{\sU} \sC(\sU, F)$ be the colimit over the filtered category of 
Nisnevich coverings of $Y$ with refinements as maps. Since $Y$ is noetherian $\sC(F)$ is still a complex
of sheaves and defines a resolution $F_Y\to \sC(F)$. It follows from
\cite[Thm 13.1]{Schroeer} that the natural map
$g_*\sC(F)\to Rg_*F_Y$ is an isomorphism in the derived category.
Note that similar as above we also have a natural map 
$(gp_2)_*\varinjlim_{\sU} \sC(X\times\sU, F)\to R(g p_2)_* F_{X\times Y}$
(which is in general not an isomorphism). 
For an \'etale map $U\to Y$  denote by $V_U$ the restriction of $V$ to $X\times U$.
Note that $C_{[V_U]}$ induces a map $F(X\times U)\to H^0(X\times U, \sH^d_{V_U}(F\la d\ra)_{X\times U})$, 
which is compatible with \'etale pullbacks (by Lemma \ref{lem:cyclecup}\ref{lem:cyclecup2}).
Therefore $C_{[V]}$ induces the bottom right map in the following diagram, which is commutative
(we set $q_2=g\circ p_2$, $XY=X\times Y$)
\[\xymatrix{
Rg_* F_Y\ar[r]^{p_2^*} &
R q_{2*} F_{XY}\ar[r]^{C_{[V]}} & 
R q_{2*} \sH^d_V(F\la d\ra_{XY})\\
g_*\sC(F)\ar[u]_{\simeq}\ar[r]^-{p_2^*}&
q_{2*}\varinjlim_{\sU} \sC(X\sU, F)\ar[u]\ar[r]^-{C_{[V]}} & 
q_{2*}\varinjlim_{\sU}\sC(X\sU,\sH^d_V(F\la d\ra_{XY})).\ar[u]
}\]
This reduces Claim \ref{prop:cor-van-top-claim} to the following claim.
\begin{claim}\label{prop:cor-van-top-claim2}
Let $U\to Y$ be \'etale. Then the composition 
\eq{prop:cor-van-top3}{F(U)\xr{p_2^*} F(X\times U)\xr{C_{[V_U]}} 
H^0(X\times U, \sH^d_{V_U}(F\la d\ra_{X\times U}))}
is zero.
\end{claim}
We prove this claim. We may assume $V_U$ is integral with generic point $\eta$. 
Since the natural restriction $\sH^d_{V_U}(F\la d\ra_{X\times U})\to \sH^d_{\eta}(F\la d\ra_{X\times U})$
is injective by \cite[Cor 8.6(1)]{S-purity} we may shrink $X\times U$ around $\eta$ and  $U$ around $\xi:=p_2(\eta)$.
The point $\xi\in U$ is finite and separable over the generic point of $p_2(V)$ and thus by assumption 
$\dim \sO_{U,\xi}=r$.
Note that $p_2^*: \sO_{U,\xi}\to \sO_{X\times U,\eta}$ is essentially smooth between regular local rings,
we therefore find a regular parameter sequence of $\sO_{X\times U,\eta}$ of the form
$p_2^*(s_1),\ldots, p_2^*(s_r), t_{r+1}, \ldots, t_{d}\in \sO_{X\times U, \eta}$, with
$s_1,\ldots, s_r$ a regular parameter sequence of $\sO_{U,\xi}$.
Thus up to shrinking $U$ we find a neighborhood $W\subset X\times U$ of $\eta$ such that the restriction
of the cycle $[V]\in \CH^d_V(X\times U)$ to $W$ can be written as
\[[V_W]= p_2^*\alpha\cdot \beta \quad \text{in } \CH^d_{V_W}(W),\]
with $\alpha\in \CH^r_{A}(U)$ and $\beta\in \CH^{d-r}_{B}(W)$,
where $A=V(s_1,\ldots, s_r)$, $B=V(t_{r+1},\ldots t_d)$.
By Lemma \ref{lem:cyclecup}\ref{lem:cyclecup2}, \ref{lem:cyclecup3} the composition of \eqref{prop:cor-van-top3} with
the injection 
\[H^0(X U, \sH^d_{V_U}(F\la d\ra_{X U}))\inj H^0(W, \sH^d_{V_W}(F\la d\ra_{W})))\]
factors as
\eq{prop:cor-van-top4}{F(U)\xr{C_\alpha} H^0(U,\sH^r_A(F\la r\ra_Y))
\xr{C_\beta\circ (p_{2|W})^*} H^0(W, \sH^d_{V_W}(F\la d\ra_{W}))).}
Thus it suffices to show that $F(U)\xr{C_\alpha} H^0(U,\sH^r_A(F\la r\ra_Y))$ is the zero map.
By \cite[Cor 8.6(1)]{S-purity} we have 
\[H^0(U,\sH^r_A(F\la r\ra_Y)= H^r_A(U, F\la r\ra_Y)\inj H^r_\xi(F\la r\ra_Y).\]
(Note that $\xi=p_2(\eta)$ is the generic point of $A$.)
Hence,   by Nisnevich excision and Lemma \ref{lem:loc-cl-sec}, we may assume that $A$ is smooth over $k$ and 
admits a map $U\to A$, of which the closed immersion $i: A\inj U$ is a section.
By Theorem \ref{prop:pf-gysin} and the definition of $C_\alpha$ (see \ref{para:RSCcyclecup})
it factors as
\[C_\alpha: F(U)\xr{i^*} F(A)\to H^r_A(U,F\la r\ra_U),\]
where the second map involves the local Gysin map.
Since $F(A)\subset F(\xi)$ by global injectivity, the vanishing of $C_\alpha$ and hence of 
\eqref{prop:cor-van-top4} follows from $F(\xi)=0$, which holds by assumption.
\end{proof}

\begin{prop}\label{prop:cor-van}
Let $(f: X\to S)$  and $(g: Y\to S)\in C_S$ and $F\in \RSC_\Nis$. 
Let $V\subset X\times_S Y$ be an integral closed subscheme with $\dim V=\dim X$, which is proper over $X$.
Assume:
\begin{enumerate}[label=(\arabic*)]
        \item\label{prop:cor-van1a} 
         there exists an integral  closed subscheme $Z_0\subset X$ of codimension $r\ge 1$ such that 
        $V\subset Z_0\times Y$;
        \item\label{prop:cor-van1c}
                there exists a projective alteration (i.e., a generically finite, surjective and projective morphism)
                 $Z\xr{h_0}Z_0$  with $Z\in \Sm$, such that $h_0\times \id_Y$ induces
                an alteration  $(h_0\times \id_Y)^{-1}(V)\to V$ of degree $N$ over $V$;
        \item\label{prop:cor-van1b} $(\gamma^r F)_{Z}=0$. 
\end{enumerate}
Then 
\[0= N\cdot [V]^*: Rg_* F_Y\to Rf_* F_X.\]
\end{prop}
\begin{proof}
Set $d:=\codim(V, X\times Y)=\dim Y$.
Denote by $h:= h_0\times \id_Y: W:=Z\times Y\to Z_0\times Y$ and $h_1: W\to X\times Y$ the maps induced
by  the alteration $h_0: Z\to Z_0$  from \ref{prop:cor-van1c}.
By assumption $h^{-1}(V)\to V$ is a projective alteration of degree $N$.
Let $\alpha= \sum_i m_i [A_i]\in \CH^{d-r}_{h^{-1}(V)}(W)$, where 
the $A_i$ are  those irreducible components of $h^{-1}(V)$ which are dominant and generically finite over $V$ 
and where $m_i$ is the multiplicity of $A_i$ in the cycle $[h^{-1}(V)]$.
We have $h_*\alpha =N\cdot  [V]\in \CH^d_V(X\times Y)$. 
Thus the following diagram commutes by Theorem \ref{prop:bc-pf}\ref{prop:bc-pf3}
\eq{prop:cor-van3}{\xymatrix{
F_{X\times Y} \ar[d]^{h_1^*}\ar[r]^-{N\cdot C_V} &
R\ul{\Gamma}_V F\la d\ra_{X\times Y}[d] \\
Rh_{1*} F_{W}\ar[r]^-{C_\alpha} &
Rh_{1*}R\ul{\Gamma}_{h^{-1}(V)} F\la d-r\ra_W[d-r],\ar[u]^{h_{1*}}
}}
see \ref{para:RSCcyclecup} for $C_\alpha$ and \ref{para:RSC-pf} for $h_{1*}$.
By definition of the cycle action in \ref{para:cycle-action} 
the map $N\cdot [V]^*: Rg_* F_Y\to Rf_* F_X$ therefore factors via
\eq{prop:cor-van4}{R(fp_Xh_1)_* R\ul{\Gamma}_{h^{-1}(V)} F\la d-r\ra_{W}[d-r]\xr{(p_Xh_1)_*} Rf_*F_X,}
where $p_X: X\times Y\to X$ is the projection.
Denote by $h_{01}: Z\to X$ the map induced by $h_0$ and by $p_Z: W\to Z$ the projection.
Then $\text{rel-dim}(h_{01})=-r$, $\text{rel-dim}(p_Z)=d$ and $p_Xh_1=h_{01}p_Z$.
By Theorem \ref{prop:bc-pf}\ref{prop:bc-pf0.5} and cancellation \eqref{para:RSCtwist6}
 we can rewrite \eqref{prop:cor-van4} as the composition
\begin{align*}
 R(fh_{01}p_Z)_* R\ul{\Gamma}_{h^{-1}(V)} \gamma^r F\la d\ra_{W}[d-r]
 &\xr{p_{Z*}} R(fh_{01})_* \gamma^r F_Z[-r]\\
& \xr{h_{01*}} Rf_*\gamma^r F\la r\ra_X\\
&\xr{\simeq} Rf_*F_X.
\end{align*}
Thus $N\cdot [V]^*$ factors via $R(fh_{01})_* \gamma^r F_Z[-r]$,
which is zero by \ref{prop:cor-van1b}.
\end{proof}

\begin{cor}\label{cor:cor-van}
Let $(f: X\to S)$  and $(g: Y\to S)\in C_S$ and $F\in \RSC_\Nis$. 
Let $V\subset X\times_S Y$ be an integral closed subscheme with  $\dim V=\dim X$, which is proper over $X$.
Denote by $V_X\subset X$ the closure of the image of $V$ in $X$.
Assume  $\codim(V_X, X)= r$ and $(\gamma^r F)_Z =0$, for all $Z\in \Sm$ with $\dim Z=\dim V_X$.
\begin{enumerate}[label=(\arabic*)]
\item\label{cor:cor-van1} If the singularities of $V_X$ can be resolved, then  
\[0=[V]^*: Rg_* F_Y\to Rf_* F_X.\]
\item\label{cor:cor-van2}
 If ${\rm char}(k)=p>0$, then there exists a number $n$ only depending on $V$ (not on $F$) such that 
                 \[0=p^n \cdot [V]^*: Rg_* F_Y\to Rf_* F_X.\]                                     
\end{enumerate}
\end{cor}
\begin{proof}
\ref{cor:cor-van1} follows directly from Proposition \ref{prop:cor-van}.
\ref{cor:cor-van2} follows from that proposition together with the Gabber-de-Jong alteration theorem,
see \cite[Thm 2.1]{IT-Gabber}.
\end{proof}

\begin{rmk}
It would be nice to have a resolution-free proof of \ref{cor:cor-van}\ref{cor:cor-van1}, in the spirit of
Proposition \ref{prop:cor-van-top}. In \cite[Prop 3.2.2(1)]{CR11} such a statement was proven
for the cohomology of the K{\"a}hler differentials. But the argument relies on the
K{\"u}nneth decomposition for differentials and it is not clear how to imitate this proof in the current setup.
\end{rmk}

\begin{lemma}\label{lem:azc}
Let $(f: X\to S)$, $(g:Y\to S)\in C_S$.
Let $S_1\subset Y$ be a closed integral subscheme which is finite and surjective over $S$
and let $\nu: \tilde{S}_1\to S_1$ be its normalization.
Assume $S$ and $\tilde{S}_1$ are smooth over $k$ and $f$ is flat.
Then the cycle associated to $X\times_S  S_1$ defines an element
$[X\times_S S_1]\in C_S(X,Y)$ and the following diagram commutes
\[\xymatrix{
Rg_* F_Y\ar[r]^{[X\times_S S_1]^*}\ar[d]^{\nu_1^*} & Rf_* F_X\\
(g\nu_1)_* F_{\tilde{S}_1}\ar[r]^-{(g\nu_1)_*} & F_S,\ar[u]_{f^*}
}\]
where  $\nu_1: \tilde{S}_1\to Y$ is induced by $\nu$.
(Note that by assumption $g\nu_1: \tilde{S}_1\to S$ is finite.)
\end{lemma}
\begin{proof}
Let $d=\dim Y$.
By Proposition \ref{prop:cycle-action}\ref{prop:cycle-action1}, \ref{prop:cycle-action3}, \ref{prop:cycle-action4} it suffices to show the following equality
 in $\CH^d_{X\times_S S_1}(X\times Y)$
\eq{prop:azc1}{[\Gamma_{\nu_1}]\circ [\Gamma^t_{g\nu_1}]\circ [\Gamma_f]= [X\times_S S_1],}
where $\Gamma_h$ denotes the graph of the map $h$ and $\Gamma_h^t$ its transpose.
As in \cite[Prop 16.1.1, (a), (c)]{Fu} the left hand side of \eqref{prop:azc1} is equal to
\[ (f\times \id_{S_1})^*(\id_S\times \nu_1 )_* [S\times_S \tilde{S}_1]
=(f\times \id_{S_1})^* [S\times_S S_1] =[X\times_S S_1],\]
where we use the flatness of $f$ for the second equality.
Hence the lemma.
\end{proof}

\section{General applications}\label{sec:gen-appl}

\subsection{Obstructions to the existence of zero-cycles of degree 1}\label{ssec:BMtype-obstructions}
 
We can use the existence of the proper correspondence action on the cohomology of an arbitrary reciprocity sheaf to construct new local-to-global obstruction for the existence of zero cycles of degree $1$. In general, this kind of obstructions are considered when the base is a classical global field, i.e., a number field or a function field in one variable over a finite field. 
Instead, we have the following general result, where there is no restriction on the dimension of the base scheme.

\begin{thm}\label{thm:dds}
Let $f\colon Y\to X$ be  a dominant quasi-projective morphism between connected smooth $k$-schemes.
Assume that there are integral subschemes $V_i\subset Y$ which are proper, surjective, and generically finite
over $X$ of degree $n_i$, $i=1,\ldots, s$. Set $N={\rm gcd}(n_1,\ldots, n_s)$.
Let $F^\bullet\in {\rm Comp}^+(\RSC_\Nis)$ be a bounded below complex of reciprocity sheaves. 
Then there exists a morphism $\sigma: Rf_* F_Y^\bullet\to F_X^\bullet$ in $D(X_\Nis)$
such that the composition
\[F_X^\bullet\xr{f^*} Rf_*F_Y^\bullet\xr{\sigma} F_X^\bullet \]
is multiplication with $N$. In particular if $N=1$, then $F_X^\bullet$ is a derived direct summand of $Rf_*F_Y^\bullet$.
\end{thm}
\begin{proof}
Set $r=\dim Y-\dim X$. Denote by $\Phi$ the family of supports on $Y$ generated by the $V_i$;
$\Phi$ therefore is a family of proper supports for $Y/X$.
Take $a_i\in \Z$ such that $N=\sum_i a_i n_i$ and set $\alpha:= \sum_i a_i [V_i]$,
which we can view as a cycle in $\CH^r_\Phi(Y)$. 
We define $\sigma$ as the composition
\[Rf_* F^\bullet_Y \xr{C_\alpha} Rf_* R\ul{\Gamma}_{\Phi} F^\bullet\la r\ra_Y[r]\xr{f_*} F^\bullet_X,\]
where $C_\alpha$ is defined as in \ref{para:RSCcyclecup} and $f_*$ as in \ref{para:RSC-pf},
extended to complexes as in \ref{para:coracx}.
The statement follows from Theorem \ref{prop:bc-pf}\ref{prop:bc-pf3}.
\end{proof}

We spell out as a Corollary the implication on the index of the generic fiber of $f$.

\begin{cor}\label{cor:dds} Let $f\colon X\to Y$ be a projective dominant morphism between connected smooth $k$-schemes. 
Let $N$ be the index of the generic fiber $X_K$ over $K=k(Y)$
(i.e. $N$ is the gcd of the residue field degrees $[K(x):K]$, where $x\in X_K$ is running through all closed points). 

Then  for any bounded below complex of reciprocity sheaves $F^\bullet$ and for any $i\ge 0$
 the kernel $\Ker (f^*: H^i(Y, F^\bullet_Y)\to H^i(X, F^\bullet_X))$ is $N$-torsion.
 In particular, if $f^*\colon H^i(Y, F^\bullet_Y)\to H^i(X, F^\bullet_X)$
 is not split injective for some $F^\bullet$ and some $i$, then the generic fiber of $f$ cannot have index $1$.
\end{cor}
Let us now discuss how  Theorem \ref{thm:dds} and Corollary \ref{cor:dds}  can be specialized to construct new local-to-global obstructions for the existence of zero-cycles of degree $1$, that give back the classical Brauer-Manin obstruction as a special case. For the reader's convenience, let us quickly review the construction of the Brauer-Manin pairing (see \cite[8]{Saito89}, \cite[1.1]{Wittenberg_fibrations}).

Let $K$ be a function field in one variable over a finite field $\F_q$ of characteristic $p>0$, and let $S$  be a proper smooth model of $K$. Let $f\colon X \to S$ be a projective dominant morphism, with $X$ smooth over $k$. Write $X_K$ for the base change $X\times_S K$.  For $v\in S_{(0)}$, let $S_v$ be the henselization of $S$ at $v$, and let $K_v = k(S_v)$. Write $X_{K_v}$ for $X\times_K K_v$ and $X_{S_v}$ for $X\times_S S_v$.

Let $\varepsilon\colon \Sm_{\et}\to \Sm_{\Nis}$ be the change of site functor, and let $\Q/\Z(1)$ be the \'etale motivic complex of weight $1$ with $\Q/\Z$ coefficients. By \ref{subsec:exaRSC}\ref{subsec:exaRSC4.3}, $R^2\varepsilon_* \Q/\Z(1)$ defines a Nisnevich reciprocity sheaf.

Since $X$ is projective over $S$, any cycle $\alpha_v \in \CH_0(X_{K_v})$ defines an element of $C_{{\Spec(K_v)}}( \Spec{K}_v, X_{K_v})$, and since $F=R^2\varepsilon_* \Q/\Z(1)$ is in $\RSC_{\Nis}$, we can apply the proper correspondence action \eqref{para:cylce-action1} to define a morphism
\begin{equation}\label{eq:BM3bis} \Br(X_{K_v})/\Br(X_{S_v}) \xrightarrow{\alpha^*} \Br(K_v) = \Br(K_v)/\Br(S_v)\end{equation}
where the last equality follows from the fact that $\Br(S_v) = 0$, since the residue field of $S_v$ is finite. 
By Proposition \ref{prop:cycle-action}\ref{prop:cycle-action4.1}, $\alpha^*$ agrees with the morphism induced by the transfer structure on the cohomology presheaves $H^2_{\et}(-, \G_m)$, which in turn is given by the classical norm map (see \cite[Ex.\ 2.4]{MVW}).

Taking (Nisnevich) cohomology with support, we can define a morphism by composition
\begin{equation}\label{eq:BM4}\bigoplus_{v\in S_{(0)}}  \Br(K_v) (= H^0(K_v, F)) \to\bigoplus_{v\in S_{(0)}} H^1_{v}(S_v, F)\cong \bigoplus_{v\in S_{(0)}} H^1_{v}(S, F) \to H^1(S, F) \end{equation}
where the last map is surjective, since $H^1(K, F)=0$ for dimension reasons. If we now compose \eqref{eq:BM4} with \eqref{eq:BM3bis} for varying $\alpha_v$, and we reassemble the maps for $v\in S_{(0)}$, we get
\[\Psi\colon \prod_{v\in S_{(0)}} \CH_0(X_{K_v}) \to \Hom(\bigoplus_{v\in S_{(0)}}  \Br(X_{K_v}) /\Br(X_{S_v}), H^1(S, F)),\]
and composing with the diagonal embedding $\Br(X_K)\xrightarrow{\iota}\bigoplus_{v\in S_{(0)}}  \Br(X_{K_v})$, we finally get
\[\iota^* \Psi \colon \prod_{v\in S_{(0)}} \CH_0(X_{K_v}) \to \Hom(\Br(X_K), H^1(S, F)) \]
that we can further compose with the diagonal morphism  from $\CH_0(X_K)$, giving
\begin{equation}\label{eq:BM5}\CH_0(X_K) \xrightarrow{} \prod_{v\in S_{(0)}} \CH_0(X_{K_v})\xrightarrow{\iota^* \Psi} \Hom(\Br(X_K), H^1(S, F)). \end{equation}
This is the Brauer-Manin sequence in disguise: in fact, the Brauer-Hasse-Noether Theorem (see  e.g., \cite[XIII]{WeilBNT} or \cite[Thm. 12.1.8]{CT-S_BGgroup_book}) implies that 
\[H^1(S,F)  \simeq \Coker\big(\Br(K) \to \underset{\Sd 0}{\bigoplus}\; \frac{\Br(K_v)}{\Br(S_v)}\big) \simeq \mathbb{Q}/\mathbb{Z}.\]
Conjecturally, the complex \eqref{eq:BM5} is exact. See \cite[Conjecture 4]{CT-survey-conj}, \cite{Saito89}. 
\medskip 



We now extend the construction of the complex \eqref{eq:BM5} replacing $\Br(-)$ with an arbitrary reciprocity sheaf. We begin with the following result. Note that $S$ doesn't have to be of dimension $1$, and that the ground field is an arbitrary perfect field. 
\begin{thm}\label{thm2:dds}
Let $f\colon X\to S$ be a projective and dominant  morphism between smooth connected $k$-schemes with $d=\dim(X)-\dim(S)$.
Let $K=k(S)$ be the function field of $S$ and $X_K=X\times_S \Spec K$.
Let
\[ \deg_K: \CH^d(X) \to \CH_0(X_K) \to \Z,\]
where the second map is the degree map. 
Then, for any $F\in \RSC_\Nis$, there exist homomorphisms 
\[\phi : \CH^d(X) \to \Hom_{D(S_\Nis)}(Rf_* F_X, F_S)\]
satisfying the following conditions:
\begin{itemize}
\item[(i)]
For any $\alpha\in \CH^d(X)$ with $N=\deg_K(\alpha)$, the composite 
\[ F_S\rmapo{f^*} Rf_* F_X\rmapo{\phi(\alpha)} F_S\]
is the multiplication by $N$.
\item[(ii)]
The map $f_* F_X \to F_S$ induced by $\phi(\alpha)$ depends only on 
the image $\alpha_K\in \CH_0(X_K)$ of $\alpha$.
\end{itemize}
\begin{proof}
For (i), it is enough to consider the case $\alpha =[x]$ for a closed point $x\in X_K$. Since the closure of $x$ in $X$ is projective, surjective and generically finite over $S$ of degree $[k(x):K]$, the statement follows  directly from Theorem \ref{thm:dds}. As for (ii), it is enough to show that if $\beta$ is a cycle in $ \CH^d(X)$ supported on $f^{-1}(T)$ for some proper closed subscheme $T\subset S$, then  the morphism $f_* F_X\to F_S$ induced by $\phi(\beta)$ is zero. But by  Lemma \ref{lem:cas} we have that $\phi(\beta)$ factors through $\ul{\Gamma_T} F_S \to F_S$, and since $\ul{\Gamma_T} F_S =0$ by \cite[Thm. 3.1]{S-purity}, the claim follows.
\end{proof}
\end{thm}

Let's go back to the case where $\dim(S)=1$.
For $F\in \RSC_\Nis$, we have a complex
\eq{eq1;thm2:dds}{ F(X_K) \rmapo{\iota} \underset{\Sd 0}{\bigoplus}\; \frac{F(X_{K_v})}{F(X_{S_v})} 
\rmapo{\delta} H^1(X,F),}
where the first map is the diagonal and the second is the composite
\begin{multline}\label{deltav}
 \delta_v: F(X_{K_v})=F(X_{S_v}-f^{-1}(v)) \to H^1_{f^{-1}(v)}(X_{S_v},F_X)\\
\simeq    H^1_{f^{-1}(v)}(X,F_X) \to H^1(X,F_X).\end{multline}
Here, we have used excision in the displayed isomorphism of \eqref{deltav}, and the fact that \eqref{eq1;thm2:dds} is a complex  follows at once from a diagram chase using the exact sequence of the cohomology with support. 

By Theorem \ref{thm2:dds}(ii), for $\alpha_v\in \CH_0(X_{K_v})$,
we have a map
\[\psi_v(\alpha_v) : \frac{F(X_{K_v})}{F(X_{S_v})}\rmapo{\phi(\tilde{\alpha}_v)} \frac{F(K_v)}{F(S_v)}\to H^1(S,F),\]
where $\tilde{\alpha}_v\in \CH^d(X_{S_v})$ is any lift of $\alpha_v$.
This gives homomorphisms
\[ \psi_v : \CH_0(X_{K_v})\to \Hom(\frac{F(X_{K_v})}{F(X_{S_v})}, H^1(S,F)),\]
which we can assemble for varying $v$ to get 
\[ \Psi : \underset{v\in \Sd 0}{\prod} \CH_0(X_{K_v}) \to 
\Hom(\underset{\Sd 0}{\bigoplus}\; \frac{F(X_{K_v})}{F(X_{S_v})}, H^1(S,F)).\]
Composing this with the diagonal morphism from \eqref{eq1;thm2:dds}, we get 
\[\iota^* \Psi : \underset{v\in \Sd 0}{\prod} \CH_0(X_{K_v}) \to 
\Hom(F(X_K), H^1(S,F)).\]
We have a commutative diagram
\eq{CD}{\xymatrix{
\CH^d(X) \ar[r]^-{H^1(\phi)}\ar[d] &  
\Hom(H^1(X,F_X), H^1(S,F)) \ar[d]^{\delta_v^*}\\
 \CH_0(X_{K_v}) \ar[r]^-{\psi_v} &
\Hom(\frac{F(X_{K_v})}{F(X_{S_v})}, H^1(S,F))\\}}
where $\delta_v$ comes from \eqref{deltav}.
Hence Theorem \ref{thm2:dds} implies the following.

\begin{cor}\label{cor2:dds}
Assume given $\xi=(\alpha_v)\in \underset{v\in \Sd 0}{\prod} \CH_0(X_{K_v})$.
If $\xi$ is in the diagonal image of $CH_0(X_K)$, there exists $s\in \Hom(H^1(X,F),H^1(S,F))$ such that $\Psi(\xi)=s\circ \delta$, in particular, 
we have $\iota^*\Psi(\xi)=0$.
If $\deg(\alpha_v)=1$, we can take $s$ to be a splitting of
$f^*: H^1(S,F)\to H^1(X,F)$.
\end{cor}
\begin{rmk} Note that thanks to  Theorem \ref{prop:bc-pf}(3) and the definition of $\sigma$ in \ref{thm:dds}, the splitting $s$ of Corollary \ref{cor2:dds} is functorial with respect to smooth base change $S'\to S$.
\end{rmk}


 
\begin{rmk} In Theorem \ref{thm2:dds} we have shown how it is possible to use sections of an arbitrary reciprocity sheaf to construct obstructions of Brauer-Manin type to the existence of zero cycles of degree one over non-classical global fields.
If one is interested in the (in general) finer question of finding obstructions to the existence of rational points over non-classical global fields $K$, there is a vast literature in which higher unramified cohomology groups $H^n_{nr}(K(X)/K, \Z/\ell)$ or $H^{n+1}_{nr}(K(X)/K, \Q/\Z(n))$ (indeed examples of global sections of reciprocity sheaves, see the list of examples  \ref{subsec:exaRSC}) have been used, starting from \cite{CT96}. The classical case, i.e., using the Brauer group, corresponds to $H^2_{nr}(K(X)/K, \Q/\Z(1))$. Here some examples of global fields, together with the invariant used.
\begin{enumerate}
\item Function field $K$ of a curve over the real field (or a real closed field) using $H^n_{nr}(k(X)/K, \Z/2))$. \cite{CT96}, \cite{DucrosCRAS}, \cite{DucrosCrelle}, \cite{Pal-Szabo}.
\item Function field $K$ of a curve over the complex field, using $H^1_{nr}(K(X)/K, \Z/n)$ \cite{CT-Gille} and the Appendix by O.~Wittenberg to \cite{ottem2020}.
\item Function field $K$ of a curve over a $p$-adic field, using $H^3_{nr}(K(X)/K, \Q/\Z(2))$. \cite{MR3352255}, \cite{MR3493592} and \cite{MR2984579}, \cite{MR3373241}.
\item Function field $K$ of a curve over $\mathbb{C}((t))$. \cite{MR3356812}.
\end{enumerate}
We thank J-L.~Colliot-Th\'el\`ene for providing us with a list of references on the subject. 
\end{rmk}
 \subsection{Birational invariants.}\label{sssec:BI}
 
As observed in \cite{CR11} cycle actions can be used to find birational invariants.
In the following $S$ is a finite type separated $k$-scheme.
We say that $(f: X\to S)$ and  $(g: Y\to S)\in C_S$, with $X$ and $Y$ integral, 
are 
\begin{enumerate}[label=(\arabic*)]
\item\label{sssec:BI1} 
{\em properly birational} over $S$, if there exists an integral scheme $Z$ over $S$ 
and two proper birational  $S$-morphisms $Z\to X$, $Z\to Y$;
in this case we call $Z$ a {\em proper birational correspondence between $X$ and $Y$}
(note that we don't assume that $f$, or $g$  is proper);
\item\label{sssec:BI2}
{\em stably properly birational over $S$}, if there exist locally free coherent $\sO$-modules $V$ and $W$
on $X$ and $Y$, respectively, such that the corresponding projective bundles $\P(V)$ and $\P(W)$ 
are properly birational over $S$.
\end{enumerate}
 
\begin{thm}\label{lem:BIH0}
A reciprocity sheaf $F\in \RSC_\Nis$ is a stably properly birational invariant over $S$, i.e.,
for $(f: X\to S)$, $(g: Y\to S)\in C_S$,  with $X, Y$ integral, any 
proper birational correspondence between projective bundles over $X$ and $Y$ induces an isomorphism
\[f_*F_X\cong g_*F_Y.\]
\end{thm}
\begin{proof}
Let $Z$ be a proper birational correspondence between
$P:=\P(V)$ and $Q:= \P(W)$, where $V$ and $W$ are locally free coherent sheaves on $X$ and $Y$, respectively.
Let $Z_0\subset P\times Q$ be the image of the induced map 
$Z\to P\times Q$ and denote by $Z^t_0$ its transpose. We obtain induced elements 
$[Z_0]\in C_S(P,Q)$, $[Z_0^t]\in C_S(Q,P)$. 
By assumption on $Z$ and the localization sequence for Chow groups 
the compositions $[Z_0]\circ [Z_0^t]$ and $[Z_0^t]\circ [Z_0]$ are equal
to the respective diagonal plus a cycle $E$ which maps to at least 1-codimensional subschemes in both $P$ and $Q$.
Since by \cite[Thm 3.1(2)]{S-purity} the restriction to the generic point 
$F_P\to j_*F_\eta$ is injective (and similar for $Q$),
$E$ acts as zero in both cases, by Lemma \ref{lem:cas}. 
By Proposition \ref{prop:cycle-action}\ref{prop:cycle-action1}, \ref{prop:cycle-action3}
the actions
$[Z_0]^*: g_*\pi_{Y*} F_Q\to f_*\pi_{X*} F_P$ and 
$[Z_0^t]^*:  f_*\pi_{X*} F_P\to g_*\pi_{Y*} F_Q$ are inverse to each other,
where $\pi_X: P\to X$ and $\pi_Y: Q\to Y$ denote the projections.
The statement follows from the projective bundle formula, Theorem \ref{thm:pbf}.
\end{proof}

\begin{rmk}
In case $f$ and $g$ are projective, the above Theorem also follows directly from purity and 
the projective bundle formula, see also \cite[Thm 8.5.1, 8.6.1]{CTHK}.
\end{rmk}

\begin{thm}\label{thm:BIu}
Let $p$ be the exponential characteristic of $k$.
Let $(f: X\to S)$, $(g: Y\to S)\in C_S$,  with $X, Y$ integral,  and let 
$Z$ be a proper birational correspondence between them.
Let $Z_0\subset X\times Y$ be the image of $Z\to X\times Y$.

Then there exists a natural number $n\ge 0$ such that for all $F\in \RSC_{\Nis}$ with $\gamma^1 F=0$
the composition
\[p^n \cdot [Z_0]^*\circ [Z_0^t]^*: Rf_*F_X\to Rg_* F_Y\to Rf_* F_X\]
is equal to the multiplication by $p^n$.
If $p=1$ or if singularities can be resolved in dimension $\dim X-1$, then we have an isomorphism
\[[Z_0]^*: Rg_*F_Y\xr{\simeq} Rf_* F_X.\]
\end{thm}
\begin{proof}
The proof is similar to the one of Theorem \ref{lem:BIH0},
where we use the extra assumption on $F$ and Corollary \ref{cor:cor-van}
to see that the cycle $p^n E$ acts as zero on $Rf_*F_X$ and $Rg_*F_Y$, respectively. 
\end{proof}

\begin{thm}\label{thm:BItop}
Let $(f: X\to S)$, $(g: Y\to S)\in C_S$,  with $X, Y$ integral.
Let $F\in \RSC_\Nis$ and assume that $F(\xi)=0$, for all points $\xi$ 
which are finite and separable over a  point of $X$ or $Y$ of codimension $\ge 1$.
Then any proper birational correspondence between $X$ and $Y$ induces an isomorphism
\[Rg_*F_Y\xr{\simeq} Rf_* F_X.\]
\end{thm}
\begin{proof}
The proof is similar to the one of Theorem \ref{lem:BIH0},
where we use the extra assumption on $F$ and Proposition \ref{prop:cor-van-top}
to see that the cycle $E$ acts as zero on $Rf_*F_X$ and $Rg_*F_Y$, respectively. 
\end{proof}

\begin{rmk}
\begin{enumerate}
\item Note that taking  $g=\id_Y: Y\to Y=S$ in Theorem \ref{thm:BItop}, yields the vanishing
         $R^i f_* F_X=0$, $i\ge 1$, for any projective birational morphism $f:X\to Y$ and 
             any $F$ as in the theorem.
\item The archetype of reciprocity sheaf which satisfies the condition $F(\xi)=0$ 
        is $\Omega^{\dim X}_{/k}$. See Corollary \ref{cor:BI-exa} below for this and more examples. 
        Also the next lemma shows that there is an ample supply of 
        non-trivial reciprocity sheaves satisfying this condition.
\end{enumerate}
\end{rmk}

\begin{lemma}\label{lem:top}
Let $p$ be the characteristic of the perfect base field $k$.
Let $\ell$ be a prime  number.
If $\ell\neq p$, we additionally assume that $\dim_\ell k<\infty$,  
where $\dim_\ell$ denotes the $\ell$-cohomological dimension.  
Let $F\in \RSC_\Nis$ be  $\ell$-primary torsion. 
Let $X\in\Sm$ be integral and set 
\[d=\begin{cases} \dim X, &\text{if } p=\ell\\ \dim_\ell k(X), &\text{if } p\neq \ell.\end{cases}\]
Then  for every point $\xi$ which is finite  over a  point of $X$ of codimension $\ge 1$  we have 
 \[ F\la d\ra(\xi)=0.\]
 If furthermore $X$ is quasi-projective and has a zero-cycle of degree prime to $\ell$ and
 $F(k)\neq 0$, then $F\la d\ra_X\neq 0$.
\end{lemma}
\begin{proof}
Note that the second statement is a direct consequence of the proof of Theorem \ref{thm:dds}.
Let $\tF=\uomega^{\CI}F\in \RSC_\Nis$.
By Corollary \ref{cor:gtwist} we have a surjection
\[\ul{a}_\Nis(\tF\otimes_{\uMPST} \uomega^*K^M_n)\surj \tF(n).\]
Since $\uomega_!$ is monoidal and exact  and  $\uomega_! \ul{a}_\Nis= a_\Nis\uomega_!$
(see \ref{para:MCor} and \ref{para:sheaf}) we obtain a surjection for  $n\ge 1$
\eq{lem:top1}{F\otimes_{\NST} K^M_n\to \uomega_!(\tF(n))\xr{\eqref{para:RSC:twist2}}  F\la n\ra.}
We have to show $F\la d\ra(\xi)=0$, for $\xi$ as in the statement. 
By a colimit argument we may assume that 
$F\la d\ra$ is $\ell^m$-torsion for some $m$.
Set $K=k(\xi)$.
By \eqref{lem:top1} it suffices to show $(F\otimes_{\PST} K^M_d/\ell^m)(K)=0$.
By \cite[5.1.3 Prop]{IvRu} we have a surjection
\[\bigoplus_{L/K} F(L)\otimes_{\Z} K^M_{d}(L)/\ell^m\surj (F\otimes_{\PST} K^M_{d}/\ell^m)(K),\]
where $L$ runs over all finite field extensions of $K$.
Therefore it suffices to show $K^M_{d}(L)/\ell^m=0$, for all fields $L$ finite over $K$.
If $\ell=p$,  we have $\td(L/k)=\td(K/k)\le d-1$ and the vanishing follows from 
the Bloch-Kato-Gabber Theorem (see \cite[Cor 2.8]{BK})
\[K^M_{d}(L)/p^m\cong W_m\Omega^{d}_{L,\log}\subset W_m\Omega^{d}_{L},\]
where $W_m\Omega^{d}_L$ is the de Rham-Witt sheaf in degree $d$, and the fact
that the latter group is zero, as follows from \cite[I, Prop 3.11]{Il}.
If $\ell\neq p$ we have  an isomorphism
\[K^M_{d}(L)/\ell \cong H^{d}_{\et}(L, \mu^{\otimes d}_\ell),\]
by the Milnor-Bloch-Kato conjecture proven by Voevodsky  (see \cite[Thm 6.16]{V11}).
In this case the vanishing follows from $\dim_{\ell} L=\dim_\ell k(\xi)\le  d-1$, which holds by assumption, and
\cite[II, \S4, Prop 11]{Serre-CG}.
\end{proof}

\subsection{Decomposition of the diagonal.}\label{para:dec-diag}
In the following we will investigate the implications of the cycle action in case
we have a decomposition of the diagonal, a method which was first employed in \cite{BlSr}.
 
\begin{thm}\label{thm:comp-diag}
Let $F^\bullet$ be a bounded below complex of reciprocity sheaves.
Let $A$ be an integral excellent $k$-algebra of dimension $\le 1$, which is a directed limit 
$A=\varinjlim_\nu A_\nu$ such that the $A_\nu$ are smooth and of finite type over $k$ and the
transition maps $A_\nu\to A_{\nu'}$, for $\nu\le \nu'$, are flat.
Let $f :X\to S=\Spec A$ be a smooth projective  morphism of relative dimension $d$.
Let $\eta\in S$ be the generic point and $X_\eta$ the generic fiber of $f$.
Assume there exists an integer $N$, a zero cycle $\xi\in \CH_0(X_\eta)$ of degree $N$  and 
a cycle $\beta\in \CH_d(Z \times_\eta X_\eta)$, where
$i: Z\inj X_\eta$ is a closed immersion of codimension $\ge 1$, such that 
\eq{thm:comp-diag1}{N\cdot [\Delta_{X_\eta}]= p_2^*\xi + (i\times \id)_*\beta 
\qquad \text{in } \CH^d(X_\eta\times_\eta X_\eta),}
where $\Delta_{X_\eta}$ denotes the diagonal.

Then there exists a strict closed subset $S_0\subset S$ such that, for all $i\ge 1$ the cokernel
\[\Coker (H^i(S_\Nis, F^\bullet)\oplus H^i_{\ol{Z}\cup X_{S_0}}(X_\Nis, F^\bullet)
\xr{f^*+ {\rm nat.}} H^i(X_\Nis, F^\bullet))\]
is $N$-torsion, where $X_{S_0}=X\times_S S_0$ and $\ol{Z}\subset X$ is the closure of $Z$.
Furthermore, if $N=1$ and $F^\bullet$ sits in degrees $\ge 0$, then 
\[H^0(S, F^\bullet)= H^0(X, F^\bullet).\]
\end{thm}
\begin{proof}
We assume $\dim A=1$. The proof for $\dim A=0$ is similar (and easier).
For $T$ a regular $k$-scheme and $Y$ a quasi-projective  $T$-scheme we
denote by $\CH_n(Y/T)$ the subgroup of $\CH_*(Y)$ formed by those cycles of relative dimension 
$n$ over $T$, see \cite[20.1]{Fu}. 
We find a dense open subset $U\subset S$, such that  the decomposition \eqref{thm:comp-diag1} extends 
to a decomposition in $\CH_d(X_U\times_U X_U/U)$, where $X_U=X\times_S U$.
Using the localization sequence for Chow groups  we find cycles
$\bar{\xi}\in \CH_0(X/S)$ and $\bar{\beta}\in \CH_{d}(\ol{Z}\times_S X/S)$ which lift $\xi$ and $\beta$,
respectively, and a cycle $\alpha\in \CH_d(X_{S_0}\times_{S_0} X_{S_0}/S)$, $S_0=S\setminus U$, 
such that the following equality holds in $\CH_d(X\times_S X/S)$
\eq{thm:comp-diag2}{N\cdot[\Delta_X]= p_2^*\bar{\xi}+ (i\times\id)_* \bar{\beta} +i_{0*}\alpha,}
where $i_{0}: X_{S_0}\times_{S_0} X_{S_0}\inj X\times_S X$ denotes the closed immersion.
Furthermore, since $X\to S$ is projective and $\dim S=1$, we can write
$\bar{\xi}=\sum_i m_i [T_i]$, where the $T_i$ are integral closed subschemes of $X$ which are finite over $S$.
(The $T_i$ are quasi-finite over $S$ by the dimension formula  \cite[(5.6.5.1)]{EGAIV2}.)

By assumption we find a projective system of smooth projective maps $(f_\nu: X_\nu\to S_\nu)$ between
smooth $k$-schemes, such that for $\nu'\ge \nu$ the transition maps $S_{\nu'} \to S_\nu$ are affine and flat
and we have $X_{\nu'}=X_{\nu}\times_{S_{\nu}} S_{\nu'}$,  and 
such that $f =\varprojlim_\nu f_\nu$. 
Hence
\[\CH_d(X\times_S X/S)=\varinjlim_\nu \CH_d(X_\nu\times_{S_\nu} X_\nu/S_\nu),\]
where the transition maps on the right are induced by flat pullback.
It follows that the decomposition \eqref{thm:comp-diag2} extends for $\nu$ large enough
to the following decomposition with the obvious notation
\[N\cdot [\Delta_{X_{\nu}}]= p_2^*\xi_\nu +(i_\nu\times \id)_*\beta_\nu +i_{0\nu*}\alpha_\nu
\quad \text{in } \CH_d(X_{\nu}\times_{S_\nu} X_{\nu}/S_\nu).\]
Additionally we can assume that $\xi_\nu= \sum_i m_i [T_{i,\nu}]$, 
where $T_{i,\nu}\subset X_{\nu}$ are finite and surjective over $S_\nu$ and such that
$T_{i,\nu}\times_{S_\nu} S=T_i$. 
Since $S$ is excellent and of dimension 1, so is $T_i$ and thus the normalization $\tilde{T}_i\to T_i$
is finite and $\tilde{T}_i$ is regular. Therefore, for $\nu$ large enough we can assume 
that the normalization $\tilde{T}_{i,\nu}$ of $T_{i,\nu}$ is smooth over $k$.
In the limit the action of $N\cdot [\Delta_{X_\nu}]$ is equal to $N\cdot \id$ on $R\Gamma(X, F^\bullet)$,
and the action of $(i_\nu\times \id)_*\beta_\nu +i_{0\nu*}\alpha_\nu$ on $R\Gamma(X, F^\bullet)$ 
factors by Lemma \ref{lem:cas} via $R\Gamma_{\ol{Z}\cup X_{S_0}}(X, F^\bullet)$. Furthermore, by Lemma \ref{lem:azc},
the action of  $p_2^*\xi_\nu$ factors in the limit via
\ml{thm:comp-diag3}{R\Gamma(X, F^\bullet)
\xr{\bigoplus_i m_i\mu_{i,1}^*} \bigoplus_i R\Gamma(\tilde{T}_i, F^\bullet) 
\\ \xr{\sum_i (f\mu_{i,1})_*} R\Gamma(S, F^\bullet) \xr{f^*} R\Gamma(X, F^\bullet),}
where $\mu_{i}: \tilde{T}_i\to T_i$ is the normalization and $\mu_{i,1}: \tilde{T}_i\to X$ is the induced map.
This yields the first statement. If $N=1$,  $\xi$ is a zero-cycle of degree $1$, and hence
$\sum_i m_i (f\mu_{i,1})_* (f\mu_{i,1})^*=\id$, by  Prop \ref{prop:pfs0}\ref{prop:pfs02}.
Thus the precomposition of \eqref{thm:comp-diag3} with $f^*$ is equal to $f^*$. 
For the second statement we observe, that if $F^\bullet$ is concentrated in degree $\ge 0$,
then $H^0_T(X, F^\bullet)=H^0_T(X, \sH^0(F^\bullet))$, for $T\subset X$, 
where $\sH^0(F^\bullet)$ is the zeroth cohomology sheaf of $F^\bullet$.
Since $\sH^0(F^\bullet)\in \RSC_\Nis$ the restriction 
$\sH^0(F^\bullet)\to j_* \sH^0(F^\bullet)_{|X\setminus T}$ is injective, for $T\subset X$ of positive codimension,
by \cite[Thm 3.1(2)]{S-purity}. Hence    $H^0_T(X, F^\bullet)=0$, for any such $T$, 
and the second statement follows from the above.
\end{proof}

\begin{rmk}\label{rem:dec_diagonal}
We recall some classical examples when the diagonal decomposes.
Let $K$ be a function field over $k$ and $X$ a smooth projective and geometrically connected $K$-scheme.
For a field extension $E/K$ we set $X_E= X\otimes_K E$. 
Then the following implications hold:
\begin{align*}
X_{\ol{K}} \text{ is } & \text{ rationally chain connected  over } \ol{K} \\
& \Rightarrow \deg:\CH_0(X_E)_{\Q}\xr{\simeq} \Q, \text{ for all } E/K\\
& \Rightarrow \eqref{thm:comp-diag1} \text{ holds for some } N\ge 1,
\end{align*}
for the first implication see, e.g., \cite[IV, 3.13 Thm]{Kollar}, for the second see \cite[Prop 1]{BlSr};
\begin{align*}
X \text{ is retract rational over }K & \Rightarrow \deg: \CH_0(X_E)\xr{\simeq} \Z, \text{ for all } E/K\\
& \Rightarrow \eqref{thm:comp-diag1} \text{ holds for } N=1,
\end{align*}
see \cite[Prop 1.4, Lem 1.5]{CTP}. (Recall that $X$ is \emph{retract rational} if
there exists a dense open $U\subset X$ and a dense open $V\subset \P^n_K$
and a morphism $V\to U$ which admits a section). More generally, following \cite{CTP} one says that a smooth projective $K$-scheme is \emph{universally $\CH_0$-trivial} if the degree map $\CH_0(X_E)\to \Z$ is an isomorphism for every field extension $E/K$.
\end{rmk}

\begin{thm}\label{thm:comp-diag-u}
Let $f:X\to S=\Spec A$ be as in Theorem \ref{thm:comp-diag}.
We assume the diagonal of the generic fiber of $f$ decomposes as in \eqref{thm:comp-diag1}.
Let $p$ denote the exponential characteristic of $k$.

Then there exists a number $n\ge 0$, such that for all
 bounded below complexes of reciprocity sheaves $F^\bullet$ with 
$\gamma F^i=0$, for all $i$, the quotient $H^j(X, F^\bullet)/f^*H^j(S, F^\bullet)$ is $p^n N$-torsion.
If furthermore $p^n N=1$ (which for example happens if $N=1$ and we can resolve singularities 
of all strict closed subschemes of $X$), then the pulback is an isomorphism 
\[f^* : R\Gamma(S, F^\bullet)\xr{\simeq} R\Gamma(X, F^\bullet).\]
\end{thm}
\begin{proof}
The proof works as the one of Theorem \ref{thm:comp-diag}, only that under the extra assumption on $F^\bullet$
we additionally find by Corollary \ref{cor:cor-van} some $n\ge 0$, such that the correspondence 
$p^n \cdot ((i_\nu\times \id)_*\beta_\nu +i_{0\nu*}\alpha_\nu)$ 
acts as zero. Note that we use Hironaka and Gabber-de-Jong to find $p^n$.
For the claim in the brackets of the second statement one has to employ the finer
Proposition \ref{prop:cor-van}. 
\end{proof}

\begin{thm}\label{thm:comp-diag-top}
Let $S$ be a smooth, connected, and separated $k$-scheme with generic point $\eta$.
Let $X$ be a smooth and quasi-projective $k$-scheme and $f:X\to S$  a flat, projective $k$-morphism 
of relative dimension $d$.
We assume the diagonal of the generic fiber of $f$ decomposes as in \eqref{thm:comp-diag1}.
Let $F^\bullet$ be a bounded below complex of reciprocity sheaves.
Assume $F^i(\xi)=0$, for all $\xi$ which are finite and separable over a point of codimension $\ge 1$ of  $X$ 
and for all $i$.

Then the cohomology sheaves of the cone of 
\eq{thm:comp-diag-top0}{f_*:Rf_* F^\bullet _X\to \gamma^d(F^\bullet)_S[-d]}
are annihilated by $N^2$. If $N=1$, then \eqref{thm:comp-diag-top0} is an isomorphism.
\end{thm}
\begin{proof}
The proof is similar to the one of Theorem \ref{thm:comp-diag}.
The transposition induces an automorphism of $\CH_d(X_\eta\times_\eta X_\eta)$.
Applying it to a decomposition $\eqref{thm:comp-diag1}$ yields a decomposition of the form
\[N\cdot [\Delta_{X_\eta}]= p_1^*\xi+ (\id\times i)_* \beta,\]
with $\xi\in \CH_0(X_\eta)$, $\beta\in\CH_d(X_\eta\times_\eta Z)$, and $i:Z\inj X_\eta$ 
a closed immersion of codimension $\ge 1$. Set $e:=\dim S$.
Using the localization sequence for Chow groups
we obtain a decomposition in $\CH_{d+e}(X\times_S X)$
\[N\cdot[\Delta_X]= p_1^*\bar{\xi}+ \underbrace{(\id\times i)_* \bar{\beta} + i_{0*}\alpha}_{= (*)},\]
with $\bar{\xi}\in \CH_e(X)$, $\bar{\beta}\in \CH_{d+e}(X\times_S\bar{Z})$, 
$\alpha\in \CH_{d+e}(X_{S_0}\times_{S_0} X_{S_0})$, with 
$\bar{Z}$ (resp. $S_0$) strictly closed in $X$ (resp. $S$).
By the assumption on $F^\bullet$ and Proposition \ref{prop:cor-van-top}
the cycle $(*)$ acts as zero on $Rf_*F^\bullet_X$.
To understand the action of $p_1^*\bar{\xi}$ consider the  following cartesian diagram
\[\xymatrix{
X\times_S X\ar[d]_{p_1}\ar@{^(->}[r] & X\times X\ar[d]^{\id_X\times f}\\
X\ar@{^(->}[r]^-{(\id_X, f)} & X\times S.
}\]
It implies that the image of $p_1^*\bar{\xi}$ in $\CH^{d+e}_{X\times_S X}(X\times X)$
is equal to  $(\id_X\times f)^*\xi_1$, with $\xi_1:=(\id_X,f)_*\bar{\xi}\in \CH^{d+e}_{X\times_S S}(X\times S)$.
Let $\Gamma^t_f\in \CH^e_{S\times_S X}(S\times X)$ be induced by the transpose of the graph of $f$.
As in \cite[Prop 16.1.1]{Fu} we have 
\[(\id_X\times f)^*\xi_1=(((\id_X\times f)^*\xi_1)^t)^t
= ((f\times \id_X)^*(\xi_1^t))^t= (\xi_1^t\circ \Gamma_f)^t= \Gamma_f^t\circ \xi_1.\]
Thus by (the same argument as in) Proposition \ref{prop:cycle-action} 
the action of the cycle $p_1^*\bar{\xi}=(\id_X\times f)^*\xi_1$
on $Rf_* F^\bullet_X$ is equal to the following composition 
\eq{thm:comp-diag-top1}{Rf_*F^\bullet_X\xr{f_*} \gamma^d (F^\bullet)_S[-d]\xr{\xi_1^*} Rf_*F^\bullet_X, }
where $\xi_1^*$ is defined as in \ref{para:cycle-action} by 
\mlnl{\gamma^d (F^\bullet)_S[-d]\xr{p_2^*} Rp_{2*}\gamma^d (F^\bullet)_{X\times S}[-d]
\xr{\kappa_e} Rp_{2*}\gamma^{d+e} (\tF^\bullet(e))_{(X\times S,\emptyset)}[-d]\\
\xr{c_{\xi_1}} Rf_*R p_{1*} R\ul{\Gamma}_{X\times_S S} \tF^{\bullet}(e)_{(X\times S, \emptyset)}[e]\\
\xr{\eqref{para:RSC:twist2}}Rf_*R p_{1*} R\ul{\Gamma}_{X\times_S S} F^{\bullet}\la e \ra_{X\times S}[e]
\xr{p_{1*}} Rf_* F^\bullet_X.
}
Since $[\Delta_X]^*$ acts as the identity on $Rf_* F^\bullet_X$ the above yields altogether  
that \eqref{thm:comp-diag-top1} is equal to multiplication with $N$. 

Furthermore,  we claim 
\eq{thm:comp-diag-top1.1}{N\cdot= f_*\circ \xi_1^*: \gamma^d(F^\bullet)_S[-d]\to \gamma^d(F^\bullet)_S[-d].}
Similar as above this comes down to show
$\xi_1\circ \Gamma^t_f=N\cdot  [\Delta_S]$, which by \cite[Prop 16.1.1]{Fu}
is equivalent to 
\eq{thm:comp-diag-top2}{(f\times\id_S)_* \xi_1=N\cdot [\Delta_S].}
By definition of $\xi_1$ we have 
\[(f\times\id_S)_* \xi_1=(f\times f)_* \bar{\xi}= \delta_{S*} f_*\bar{\xi},\]
where $\delta_S : S\inj S\times S$ is the diagonal. Since $\bar{\xi}$ is a lift of the degree $N$  zero-cycle $\xi$
over $\eta$ we find $f_*\bar{\xi}=N\cdot [S]\in \CH^0(S)$, which yields \eqref{thm:comp-diag-top2}.

Altogether the compositions  $\xi_1^*\circ f_*$ and $f_*\circ \xi_1^*$ are multiplication by $N$. A simple diagram chase shows  that if $\sH^i(C)$ are the cohomology sheaves of the cone $C$ of $f_*$, then each section of $\sH^i(C)$ is annihilated by $N^2$.
\end{proof}

\section{Examples}\label{sec:examples}
For $F\in \RSC_\Nis$ we set $\tilde{F}=\uomega^{\CI}F\in \CItspNis$ (see \eqref{para:tut2}).
We spell out some of the results for specific examples.
We will use without further mentioning the fact that the category of reciprocity sheaves $\RSC_\Nis$ 
is an abelian category, see \cite[Thm 0.1]{S-purity}.

\subsection{Examples of Reciprocity sheaves}\label{subsec:exaRSC}
Here we list some basic examples of reciprocity sheaves and morphisms between them. 
Note that a morphism of reciprocity sheaves $F\to G$ is the same as natural transformation of the underlying 
functors $\Cor^{\rm op}\to \Ab$; since $\RSC_\Nis$ is an abelian category we obtain many more interesting examples by 
taking kernels and quotients. More examples can be fabricated by using the \emph{lax} symmetric monoidal structure $(-,-)_{\RSC_\Nis}$ from \cite[4]{RSY} (denoted $\otimes$ in \emph{loc.cit.}, see also \cite[1]{MS} for the notation).

\begin{enumerate}[label = (\arabic*)]
\item\label{subsec:exaRSC1} The category of homotopy invariant Nisnevich sheaves with transfers is an 
              abelian subcategory of  $\RSC_\Nis$, see,  \cite[Cor 2.3.4]{KSY2} and \cite[Thm 0.1]{S-purity}.
              For $H\in\HI_\Nis$ we have $\widetilde{H}=\uomega^*H$.
\item\label{subsec:exaRSC2} Every  smooth commutative $k$-group scheme  is a reciprocity sheaf;
                                   every $k$-morphism between such group schemes is a morphism of reciprocity sheaves,
                                   see \cite[Cor 3.2.5(1)]{KSY2}.                                   
In case ${\rm char}(k)=0$ we have for $(X,D)\in \MCorls$
\[\widetilde{\G_a}(X,D)= H^0(X, \sO_X(D-|D|)),\]
see \cite[Cor 6.8]{RS};
if ${\rm char}(k)=p>0$, then we have for $(X,D)\in \MCorls$ with $U=X\setminus |D|$
\[\widetilde{W_n}(X,D)=\left\{a\in W_n(U) \,\middle|\, 
\begin{minipage}{0.5\textwidth}
  $\rho^*a\in {\rm fil}^F_{v_L(D)}W_n(L), \forall \rho\in U(L)$,
  for all henselian discrete valuation rings $L$ of geometric type over $k$
\end{minipage}\right\},\]                                   
where  $v_L(D)$ denotes the multiplicity of the pullback of $D$ to $\Spec \sO_L$ and
${\rm fil}^F_jW_n(L)=\sum_{s\ge 0} F^s({\rm fil}_jW_n(L))$, where for $j\ge 1$
\[{\rm fil}_j W_n(L)= {\rm fil}^{\log}_{j-1}W_n(L)+ V^{n-r}({\rm fil}^{\log}_j W_r),\]
with $r:= \min\{n, {\rm ord}_p(j)\}$, and 
\[{\rm fil}^{\log}_j W_n(L)=\{(a_0,\ldots, a_{n-1})\in W_n(L)\mid p^{n-1-i} v_L(a_i)\ge -j, \text{all i}\},\]
see \cite[Thm 7.20]{RS}.
\item\label{subsec:exaRSC3}
Assume $k$ has characteristic zero. Then the absolute K{\"a}hler differentials $\Omega^n=\Omega^n_{/\Z}$ and the 
relative  ones $\Omega^n_{/k}$  form  reciprocity sheaves.  
This follows from \cite[Cor 3.2.2]{KSY2} and \cite[Thm A 6.2]{KSY1}. Note that the proof
in {\em loc. cit.} relies on duality theory. However since we assume ${\rm char}(k)=0$,
the action of finite correspondences can be constructed in an ad hoc manner, see \cite[Thm 1.1]{LW}, and  to show
that the differentials (absolute or relative) have reciprocity can be shown by using residues on curves 
and the trace for finite field extensions (i.e. classical duality theory for smooth curves over a field of characteristic 0).
We note that 
\[d: \Omega^n\to \Omega^{n+1},\quad \dlog: K^M_n\to \Omega^n\]
are morphisms in $\RSC_\Nis$, as follows  from \cite[Lem 1.1]{MS}.
\item\label{subsec:exaRSC4} Assume char$(k)=p>0$. The ($p$-typical) de Rham-Witt sheaves $W_n\Omega^q$,
$q\ge 0$, $n\ge 1$, of Bloch-Deligne-Illusie are reciprocity sheaves, see \cite[Cor 3.2.5(3)]{KSY2}. 
As observed in \cite{CR12} it follows from Grothendieck-Ekedahl duality theory that
the structure maps of the de Rham-Witt complex $F$ (Frobenius), $V$ (Verschiebung), $R$ (restriction),
and $d$ (differential) are morphisms of reciprocity sheaves, see  also \cite[Lem 7.7]{RS}.
Since $W\Omega^q=\varprojlim W_n\Omega^q$ has no $p$-torsion  and 
\eq{subsec:exaRSC4.1}{\dlog: K^M_q\to W_n\Omega^q}
factors via $W\Omega^q\to W_n\Omega^q$, it follows from \cite[Lem 1.1]{MS}, that \eqref{subsec:exaRSC4.1}
is a morphism in $ \RSC_\Nis$. (Note however that $W\Omega^q\not\in\RSC_\Nis$.)
\item\label{subsec:exaRSC4.2}
Assume char$(k)=p>0$. Denote by $W_n\Omega^r_{\log}$ the subsheaf of $W_n\Omega^r$ \'etale locally generated by 
log forms (by e.g. \cite{Il}) and by $\Z/p^n(r)$ the motivic complex 
of weight $r$ with $\Z/p^n$-coefficients, viewed as a complex of \'etale sheaves.  
By \cite{GL} we have $W_n\Omega^r_{\log}[-r]\cong \Z/p^n(r)$ on $\Sm_\et$.
There is  an exact sequence on $\Sm_\et$ (see  \cite{CTSS})
\[0\to W_n\Omega^r_{\log}\to W_n\Omega^r \xr{F-1} W_n\Omega^r/dV^{n-1}\Omega^{r-1}\to 0.\]
Since the two sheaves on the right of this complex admit a structure of coherent modules on $W_n(X)$, they are
$\epsilon_*$-acyclic, where $\epsilon: \Sm_{\et}\to \Sm_{\Nis}$ denotes the morphism of sites.
Hence on $\Sm_{\Nis}$
\eq{subsec:exaRSC4.2.1}{
R\epsilon_*\Z/p^n(r) \cong\left( W_n\Omega^r \xr{F-1} (W_n\Omega^r/dV^{n-1}\Omega^{r-1})\right)[-r],}
which is obviously a complex of reciprocity sheaves.
\item\label{subsec:exaRSC4.3}
Denote by $\Q/\Z(n)$ the \'etale motivic complex of weight $n$ with $\Q/\Z$-coefficients.
Then 
\eq{subsec:exaRSC4.3.1}{R^i \epsilon_* \Q/\Z(n)\in \RSC_\Nis, \quad \text{all } i,}
where $\epsilon: \Sm_{\et}\to \Sm_{\Nis}$ denotes the morphism of sites.
Indeed, let $p$ be the characteristic exponent of $k$.
We can decompose $R^i \epsilon_* \Q/\Z(n)$ into the prime-to-$p$ torsion part, which is $\A^1$-invariant by 
\cite[Cor 5.29]{VoPST}, and the $p$-primary torsion part which is a reciprocity sheaf by \ref{subsec:exaRSC4.2} above.
In particular, taking $n=1$ and $i=2$ we see that the Brauer group defines a reciprocity sheaf,  ${\rm Br}\in\RSC_{\Nis}$.
\item\label{subsec:exaRSC5} Assume char$(k)=p>0$. Let $G$ be a finite commutative $k$-group scheme.
Denote by $H^1(G)$ the presheaf on $\Sm$ given by $X\mapsto H^1(G)(X):= H^1(X_{\rm fppf}, G)$.
Then $H^1(G)\in \RSC_{\Nis}$, see \cite[Thm 9.12]{RS}.
\end{enumerate}

\subsection{Results with modulus}

\begin{thm}\label{lem:dR0mod}
Let $k$ be a field of characteristic zero. Then there is a canonical isomorphism in $\CItspNis$
\[\widetilde{\G}_a(n)\xr{\simeq} \widetilde{\Omega^n},\]
where $(n)$ denotes the twist from Definition \ref{def:gtwist}.
Furthermore, if $(X,D)\in \uMCorls$ we have 
\eq{lem:dR0mod1}{
\widetilde{\G}_a(n)_{(X,D)}\cong \widetilde{\Omega^n}_{(X,D)}= \Omega^n_{X/\Z}(\log D)(D- |D|).
}
Furthermore,
\eq{lem:dR0mod1.1}{
(\widetilde{\Omega^n_{/k}})_{(X,D)}= \Omega^n_{X/k}(\log D)(D- |D|).
}
\end{thm}
\begin{proof}
The equalities in \eqref{lem:dR0mod1} and \eqref{lem:dR0mod1.1} follow from \cite{RS}.
Indeed, let $U\subset X$ be an open subscheme, and write $D_U$ for $D\cap U$. Let $\sY= (Y, Y_\infty)$
be a log-smooth modulus compactification
of $(U, D_U)$, i.e., $Y$ is smooth and proper and  $Y_\infty= \bar{D}_{U}+\Sigma$, 
where $\bar{D}_U$ and $\Sigma$ are effective Cartier divisors
on $Y$, such that $|\bar{D}_{U}+\Sigma|$ is a SNCD, $U= Y\setminus| \Sigma|$, and 
the restriction of $\bar{D}_U$ to $U$ is equal to $D_U$.
Set $Y_{N,\infty}=\bar{D}_U+ N\cdot \Sigma$.
Then  for $R\in\{\Z, k\}$
\begin{align*}
H^0(U, \Omega^n_{X/R}(\log D)(D-|D|))
  &= \varinjlim_N H^0(Y, \Omega^n_{Y/R}(\log Y_{\infty})(Y_{N,\infty}- |Y_{\infty}|))\\
  &= \varinjlim_N H^0(Y, (\widetilde{\Omega^n_{/R}})_{(Y, Y_{N,\infty})})\\
   &= H^0(U, (\widetilde{\Omega^n_{/R}})_{(X,D)}),
\end{align*}
where the second equality is \cite[Cor 6.8(1)]{RS} and the third equality 
holds by  $M$-reciprocity, see \ref{para:CI}.

By \cite[Thm 5.20]{RSY} we have a canonical isomorphism
\eq{lem:dR0mod2}{\uomega_!(\widetilde{\G_a}(n))\xr{\simeq} \Omega^n,}
which is defined in such a way that the composition 
with the natural map $\G_a\otimes_\Z K^M_n\to \uomega_!(\widetilde{\G_a}(n))$ 
(see \eqref{para:Tmap1}, \eqref{cor:gtwist1.1}) is given by
\[\sO_X\otimes_\Z K^M_{n,X}\to \Omega^n_X, \quad a\otimes \beta \mapsto a\dlog \beta.\]
By adjunction (see \eqref{para:tut2})  we obtain the canonical map  from the statement;
it is injective by semi-purity. To prove surjectivity it suffices by Lemma \ref{lem:crit-surj}
and resolution of singularities to prove the surjectivity of 
$\widetilde{\G}_a(n)_{\sX}\to \widetilde{\Omega^n}_{\sX}$,
for any $\sX=(X,D)\in \uMCorls$.
We may assume $X$ is affine. By  Lemma \ref{lem:rootMpair} 
we find a finite map $\pi_1: Y_1\to X$ and an effective divisor $E_1$ on $Y_1$, such that
$(Y_1,E_1)\in \uMCor$ and $\pi_1^*|D|= 2 E_1$.
Using resolution of singularities we find an isomorphism in $(Y, E)\cong (Y_1, E_1)$ in $\uMCor$
with $(Y,E)\in \uMCorls$ which is induced by a birational projective map $f: Y\to Y_1$ and $E=f^*E_1$.
Set $\pi:=f\circ \pi_1 :Y\to X$.
Thus $\pi^* |D|=2 E$ and we find  an effective Cartier divisor $E'$ on $Y$ 
with $\pi^*D= 2  E'$ and $|E'|=|E|=:E_0$.
We have
\eq{lem:dR0mod3}{\widetilde{\G_a}_{(X,D)}= \sO_X(D-|D|)\xr{\pi^*} \pi_*\sO_Y(2 E'- 2E)
\subset \pi_* \widetilde{\G_a}_{(Y, \pi^*D- E_0)},}
where the equality and the inclusion follow from \eqref{lem:dR0mod1} with $n=0$.
Set $e=\deg \pi$ and $\sY=(Y, \pi^*D)$. Consider the following diagram
\[\xymatrix{
\pi_* \left(\widetilde{\G_a}_{(Y, \pi^*D-E_0)}\otimes_\Z \widetilde{K^M_{n}}_{(Y, E_0)}\right)
\ar[r]^-{\eqref{para:Tmap1}} &
 \pi_* \left(\widetilde{\G_a}\otimes_{\uMPST} \widetilde{K^M_n}\right)_{\sY}\ar[r]^-{(*)} &
 \pi_*(\widetilde{\Omega^n}_{\sY})\ar[d]^{(\Gamma^t_{\pi})^* }\\
 \widetilde{\G_a}_{\sX}\otimes_\Z \widetilde{K^M_{n}}_{\sX}\ar[rr]^{e(\id\otimes\dlog)}\ar[u]^{\pi^*} & &
 \widetilde{\Omega^n}_{\sX},
}\]
where the map $(*)$ is induced by \eqref{lem:dR0mod2}, the left vertical map exists by \eqref{lem:dR0mod3} and the fact that $(\widetilde{\sK^M_n})_{\sX} = (\uomega^* \sK_n^M)_{\sX} = j_*(\sK^M_n)_{X-D}$, where $j$ is the open immersion $X-D \hookrightarrow X$, and
the right vertical map exists  since $\pi$ is projective and finite over $X\setminus|D|$.
The diagram commutes by the explication of \eqref{lem:dR0mod2} above and the formula 
$(\Gamma^t_{\pi})^*\circ \pi^*=e$ on $\widetilde{\Omega^n}_{\sX}$. 
By the description of $ \widetilde{\Omega^n}_{\sX}$ above we see that the bottom horizontal map
is a surjective morphisms of sheaves. We can factor the composition 
$(\Gamma^t_{\pi})^*\circ (*)$ also as 
\[\pi_* \left(\widetilde{\G_a}\otimes_{\uMPST} \widetilde{K^M_n}\right)_{\sY}\xr{(\Gamma^t_{\pi})^*}
 (\widetilde{\G_a}\otimes_{\uMPST} \widetilde{K^M_n})_{\sX}\to \widetilde{\Omega^n}_{\sX}.\]
Hence $\widetilde{\G_a}(n)_{\sX}\to \widetilde{\Omega^n}_{\sX}$ is surjective.
\end{proof}
\begin{cor}\label{cor:gammaOmega}Let $k$ be a field of characteristic zero. Then 
 \[\gamma^j\widetilde{\Omega^n} \cong \widetilde{\Omega^{n-j}},\quad 
 \gamma^j\widetilde{\Omega^n_{/k}} \cong \widetilde{\Omega^{n-j}_{/k}}\] for every $n,j\geq0$.
\end{cor}
\begin{proof} For the absolute differentials and $j\le n$, this follows immediately from Theorem \ref{lem:dR0mod} 
and the weak cancellation theorem \cite[Cor 3.6]{MS}, see \eqref{para:tutRSC3}.
For $j>n$ this is a vanishing statement, which reduces to show
\eq{cor:gammaOmega0}{\gamma(\widetilde{\G_a})=0.}
By Lemma \ref{lem:P1bf} we have $(\gamma\G_a)_X= R^1\pi_* \sO_{\P^1_X}=0$, where $\pi: \P^1_X\to X$
is the projection. Then semipurity and \eqref{para:RSCtwist4} together imply \eqref{cor:gammaOmega0}.

Now the relative case. Note that  $\Omega^i_{k/\Z}\otimes_k \Omega^{n-i}$ is a reciprocity sheaf,
since the choice of a basis of $\Omega^i_{k/\Z}$ yields an identification with a direct sum
(indexed by the basis) $\oplus \Omega^{n-i}$; similar with the relative differentials.
It follows from this and  \cite[Thm 6.4 and Thm 4.15(4)]{RS} that the natural map
\eq{cor:gammaOmega1}{\widetilde{\Omega^i_{k/\Z}\otimes_k \Omega^{n-i}}\to 
\widetilde{\Omega^i_{k/\Z}\otimes_k \Omega^{n-i}_{/k}}}
is surjective in $\CItspNis$. Set 
\[{\rm Fil}^{i,n}:= \Im(\Omega^i_{k/\Z}\otimes_k \Omega^{n-i}\to \Omega^n),\quad i\in [0,n],\]
and ${\rm Fil}^{i,n}:=0$, for $i>n$. 
As is well-known we have  an isomorphism 
${\rm Fil}^{i,n}/{\rm Fil}^{i+1,n}\cong \Omega^i_{k/\Z}\otimes_k \Omega^{n-i}_{/k}$.
Consider the following diagram 
\eq{cor:gammaOmega3}{
\xymatrix{
0\ar[r] & \widetilde{{\rm Fil}^{i+1,n-1}}\ar[r]\ar[d] &
\widetilde{{\rm Fil}^{i,n-1}}\ar[r]\ar[d] &
\widetilde{\Omega^i_{k/\Z}\otimes_k \Omega^{n-1-i}_{/k}}\ar[r]\ar[d] & 0\\
0\ar[r] & \gamma(\widetilde{{\rm Fil}^{i+1,n}})\ar[r] &
\gamma(\widetilde{{\rm Fil}^{i,n}})\ar[r] &
\gamma(\widetilde{\Omega^i_{k/\Z}\otimes_k \Omega^{n-i}_{/k}})\ar[r] & 0.
}}
The top row is exact for all $n, i$ by the left exactness of $\uomega^{\CI}$ and the surjectivity of
\eqref{cor:gammaOmega1}. Since we are in characteristic zero the exactness of the bottom
sequence follows from this and  Lemma \ref{lem:gamma-exact}. The two vertical maps on the left are induced by 
the statement of the corollary for the absolute differentials, and the vertical map on the right is the
induced morphism between the cokernels, in particular the diagram is commutative. 
Note that by definition and \eqref{cor:gammaOmega0} we have
\eq{cor:gammaOmega4}{\gamma (\widetilde{{\rm Fil}^{i,n}}) =0,\quad \text{for }i\ge n.}
Therefore for $n=1$, the  case of the relative differentials follows from the one for the absolute differentials and the
diagram \eqref{cor:gammaOmega3} with $i=0$.
Now assume we know $\gamma(\widetilde{\Omega^m_{/k}})\cong \widetilde{\Omega^{m-1}_{/k}}$ for all $m<n$.
Then by descending induction over $i\ge 1$, \eqref{cor:gammaOmega4}, and  
diagram \eqref{cor:gammaOmega3}  we have 
an isomorphism $\gamma(\widetilde{{\rm Fil}^{i,n}})\cong \widetilde{{\rm Fil}^{i,n-1}}$, for all $i\ge 1$, and 
by the absolute case also for $i=0$. Thus taking $i=0$ in diagram \eqref{cor:gammaOmega3}
also implies the statement in the relative case.
\end{proof}
\begin{cor}\label{cor:exa-mod-bu}
Assume char$(k)=0$. Let $(X,D)\in \uMCorls$ and let $i:Z\inj X$ be a smooth closed subscheme of codimension
$c$ intersecting $D$ transversally (see Definition \ref{defn:ti}).
Let  $\rho: \tilde{X}\to X$ be the blow-up of $X$ in $Z$ and set  $\tilde{\sX}:=(X, \rho^*D)$. 
There is a canonical isomorphism in $D(X_\Nis)$ 
\mlnl{
R\rho_* \Omega^q_{\tilde{X}}(\log \rho^*D)(\rho^*D- |\rho^*D|)\\
\cong
\Omega^q_{X}(\log D)(D-|D|)\oplus \bigoplus_{r=1}^{c-1} i_*\Omega^{q-r}_Z(\log i^*D)(i^*D-|i^*D|)[-r]. }
Same with $\Omega^q$ replaced by $\Omega^q_{/k}$.
\end{cor}
\begin{proof}
This follows from Corollary \ref{cor:bud}, Corollary \ref{cor:gammaOmega}, and \eqref{lem:dR0mod1}, resp.
\eqref{lem:dR0mod1.1}. 
\end{proof}

\begin{cor}\label{cor:exa-mod-g}
Let the assumption be as in Corollary \ref{cor:exa-mod-bu}.
There is a distinguished triangle in $D(X_\Nis)$
\mlnl{i_*\Omega^{q-c}_Z(\log i^*D)(i^*D-|i^*D|)[-c]\xr{g_{\sZ/\sX}} \Omega^q_X(\log D)(D-|D|)\\
\xr{\rho^*} R\rho_* \Omega^q(\log \rho^*D+E)(\rho^*D- |\rho^*D|)\xr{\partial},}
where $E=\rho^{-1}(Z)$. Same with $\Omega^q$ replaced by $\Omega^q_{/k}$.
\end{cor}
\begin{proof}
This follows from Theorem \ref{thm:gysin-tri}, Corollary \ref{cor:gammaOmega}, and \eqref{lem:dR0mod1}, resp.
\eqref{lem:dR0mod1.1}. 
\end{proof}

\begin{rmk}
\begin{enumerate}
\item We can extend the statements from the Corollaries \ref{cor:exa-mod-bu} and \ref{cor:exa-mod-g}
to complexes as in \ref{para:coracx} to obtain similar formulas
for $\Omega^\bullet$, $\Omega^{\ge n}$, 
$\tau_{\le n}\Omega^{\bullet}$; same with  $\Omega^\bullet_{/k}$.
\item One can check that in case $c=1$ the distinguished triangle in Corollary \ref{cor:exa-mod-g}
is up to shift and sign induced by the exact sequence 
\mlnl{0\to\Omega^q_X(\log D)(D-|D|)\to \Omega^q_X(\log D+Z)(D-|D|)\\
\xr{{\rm Res}_Z}\Omega^{q-1}_Z(\log i^*D)(i^*D-|i^*D|)\to 0.}
\end{enumerate}
\end{rmk}

\begin{cor}\label{cor:CL1}
Let $k$ be a perfect field, $(X,D)\in \MCorls$ and let $i: Z\inj X$ be a smooth closed subscheme of codimension $c$ intersecting
$D$ transversally. Denote by $\rho: \tilde{X}\to X$ the blow-up of $X$ in $Z$.
\begin{enumerate}[label =(\arabic*)]
\item\label{cor:CL1.1}
Assume that char$(k)=0$. Denote by ${\rm Conn}^1$ the reciprocity sheaf whose
sections over $X$ are rank 1 connections on $X$. Recall from \cite[Thm 6.11]{RS}
that the group $\widetilde{{\rm Conn}^1}(X,D)$ consists of the rank 1 connections on $X\setminus|D|$ 
whose  non-log-irregularity is bounded by $D$.
If $c=1$, then there is an exact sequence
\mlnl{0\to \widetilde{{\rm Conn}^1}(X,D)\to \widetilde{{\rm Conn}^1}(X,D+Z)\to 
H^0(Z, \sO_{Z}(i^*D-|i^*D|))/\Z\\
\to H^1\left(X, \tfrac{\Omega^1_{X/k}(\log D)(D-|D|)}{\dlog j_*\sO_{X\setminus |D|}^\times}\right)
\to H^1\left(X, \tfrac{\Omega^1_{X/k}(\log D+Z)(D-|D|)}{\dlog j_*\sO_{X\setminus |D+Z|}^\times}\right).}
If $c\ge 2$, then 
\[\widetilde{{\rm Conn}^1}(X,D)\cong \widetilde{{\rm Conn}^1}(\tilde{X},\rho^*D+E).\]
\item\label{cor:CL1.2} Assume that char$(k)=p>0$ and fix  a prime $\ell\neq p$. 
Denote by ${\rm Lisse}^1$ the presheaf whose sections over $X$ are
the lisse $\bar{\Q}_\ell$ sheaves of rank 1. By \cite[Cor 8.10, Thm 8.8]{RS}
we have ${\rm Lisse}^1\in \RSC_\Nis$  and $\widetilde{{\rm Lisse}^1}(X,D)$
is the group of lisse $\bar{\Q}_\ell$-sheaves of rank 1 on $X\setminus |D|$
whose Artin conductor is bounded by $D$. If $c\ge 2$, then
\[\widetilde{{\rm Lisse}^1}(X,D)\cong \widetilde{{\rm Lisse}^1}(\tilde{X},\rho^*D+E).\]
\end{enumerate}
\end{cor}
\begin{proof}
For $c\ge 2$ both \ref{cor:CL1.1} and \ref{cor:CL1.2} follow directly from the Gysin sequence, 
Theorem \ref{thm:gysin-tri}.
We consider the case $c=1$ in \ref{cor:CL1.1}. 
We have an isomorphism of reciprocity sheaves
 $(\Omega^1_{/k}/\dlog\G_m)_\Nis\cong {\rm Conn}^1$ (cf. \cite[6.10]{RS}), whence an 
isomorphism in $\CItspNis$
\eq{cor:CL2}{\uomega^{\CI}(\Omega^1_{/k}/\dlog\G_m)_{\Nis}\xr{\simeq} \widetilde{{\rm Conn}^1}.}
We claim that the induced composite map  
\eq{cor:CL3}{\widetilde{\Omega^1_{/k}}\to \uomega^{\CI}(\Omega^1_{/k}/\dlog\G_m)_{\Nis} \xr{\simeq}\widetilde{{\rm Conn}^1}}
is surjective. Indeed, by Lemma \ref{lem:crit-surj} and resolution of singularities it 
suffices to show that its restriction to any $(X,D)\in \uMCor_{ls}$ is surjective. 
The latter is a local question. Let $A=\sO_{X,x}^h$ be the local ring at some $x\in X$ and $f\in A$ 
and equation for $D$ at $x$. Since $X\in \Sm$, the local ring $A$ is regular and hence the localization $A_f$ is factorial.
By the exact sequence
\[\Omega^1_{A_f/k}\to {\rm Conn}^1(A_f)\to \Pic(A_f)=0,\]
(induced by taking cohomology of the complex $[\mathbb{G}_m \xrightarrow{\dlog}\Omega^1_{/k}]$) we can lift any rank 1 connection  $E$ to a differential $\omega_E\in \Omega^1_{A_f}$.
It is direct to check from \cite[Thm 6.11]{RS} that $E\in \widetilde{{\rm Conn}^1}(A,f)$ if and only if
$\omega_E\in \widetilde{\Omega^1}(A,f)$ (with the obvious abuse of notation), proving the claim.

By  \eqref{cor:CL2}, the surjectivity of \eqref{cor:CL3} and the left exactness of $\uomega^{\CI}$
we have an exact sequence
\[0\to \uomega^*(\G_m/(k^{\rm alg})^{\times})\xr{\dlog} \widetilde{\Omega^1_{/k}}\to 
\widetilde{{\rm Conn}^1}\to 0,\]
where $(k^{\rm alg})^\times=\Ker (\dlog: \G_m\to \Omega^1)$.
Since $\gamma(\uomega^*(k^{\rm alg})^{\times})=0$ (e.g. by the projective bundle formula)
and $\gamma(\uomega^*\G_m)=\gamma(\Z(1))=\Z$ by the weak cancellation theorem,
Corollary \ref{cor:gammaOmega} and Lemma \ref{lem:gamma-exact} yield
\[ \gamma (\widetilde{{\rm Conn}^1})= \widetilde{\G_a}/\Z.\]
The statement follows from this and  the Gysin sequence.
\end{proof}



\subsection{Results without modulus}

\begin{para}\label{para:gdRW}
Assume char$(k)=p>0$.
For $F\in \RSC_\Nis$ denote by $h^0_{\A^1}(F)$ the maximal $ \A^1$-invariant subsheaf of $F$.
In particular $\dlog: K^M_r\to W_n\Omega^r$ factors via the inclusion $h^0_{\A^1}(W_n\Omega^r)\inj W_n\Omega^r$.
We obtain an induced map in $\CItspNis$
\[\dlog: \uomega^*K^M_r\to \uomega^*h^0_{\A^1}(W_n\Omega^r)= 
\uomega^{\CI}h^0_{\A^1}(W_n\Omega^r)\inj \widetilde{W_n\Omega^r}.\]
Thus, for $q\ge  r\ge 0$, $n\ge 1$, $X\in \Sm$ and $\sY\in \uMCor$  we can define
\[\phi^r_{X,\sY}: W_n\Omega^{q-r}(X)\to \Hom(\uomega^*K^M_r(\sY), \widetilde{W_n\Omega^q}(X\otimes \sY))\]
 by
\[\phi^r_{X,\sY}(\alpha)(\beta):= p_X^*\alpha \cdot \dlog p_{\sY}^*\beta,\] 
where $\alpha\in W_n\Omega^{q-r}(X)$, $\beta\in \uomega^*K^M_r(\sY)$ and $p_X: X\otimes \sY\to X$ 
and $p_{\sY}: X\otimes \sY\to Y$ denote the projections.

\end{para}
The following result is essentially a corollary of the projective bundle formula for reciprocity sheaves
and  the computation of the cohomology of the K{\"a}hler differentials of the projective line.
The case ${\rm char}(k)=0$ was proved  in \cite[Thm 6.1]{MS} by a slightly different  method.
Recall from Proposition \ref{prop:twistMilnor} that we have 
$\uHom_{\PST}(K^M_n, F)=\gamma^n F$,  $F\in \RSC_\Nis$. 
\begin{thm}\label{thm:gdRW}
\begin{enumerate}[label =(\arabic*)]
\item\label{thm:gdRW01} Assume char$(k)=p>0$. For $q, r\ge 0$, $n\ge 1$ the collection 
$\{\phi^r_{X,Y}\}_{X, \sY}$ from \ref{para:gdRW} induces an isomorphism
\[\phi^r: W_n\Omega^{q-r}\xr{\simeq} \gamma^r W_n\Omega^q \quad \text{in }\RSC_\Nis,\]
where we set $W_n\Omega^{q-r}:=0$, if $q<r$. 
We have 
\eq{thm:gdRW1}{\phi^{r+s}= \gamma^s(\phi^r)\circ \phi^s, \quad \text{for } r,s\ge 0.}
Furthermore  $\phi^r$ commutes with $R$, $F$, $V$, $d$ (see \ref{subsec:exaRSC}\ref{subsec:exaRSC4}
for notation), i.e.,
\eq{thm:gdRW2}{\phi^r\circ f=\gamma^r(f)\circ \phi^r, \quad \text{for }f\in \{R, F, V, d\}.}
\item\label{thm:gdRW02}
Assume char$(k)=0$. Then similarly as in \ref{thm:gdRW01}, we have an isomorphism
$\phi^r: \Omega^{q-r}\xr{\simeq} \gamma^r\Omega^q$, $q,r\ge 0$, 
 which satisfies \eqref{thm:gdRW1}; also for relative differentials $\Omega ^\bullet_{/k}$.
\end{enumerate}
\end{thm}
\begin{proof}
We will consider the situation in \ref{thm:gdRW01} and make a remark on \ref{thm:gdRW02} later. 
First assume $q\ge r$. 
Fix $\alpha\in W_n\Omega^{q-r}(X)$. We claim that $\phi^r_X(\alpha):=\{\phi^r_{X,\sY}(\alpha)\}$ with varying $\sY$
defines a morphism $\uomega^*K^M_r\to \widetilde{W_n\Omega^q}(X\otimes -)$ in $\uMPST$, i.e.,
we have to show for $\Gamma\in \uMCor(\sY', \sY)$ 
\[(\id_X\otimes\Gamma)^*\phi^r_{X, \sY}(\alpha)(\beta)= \phi^r_{X,\sY'}(\alpha)(\Gamma^*\beta) \quad \text{in }
\widetilde{W_n\Omega^q}(X\otimes\sY').\]
Since  restriction to open subsets is injective on $W_n\Omega^q$ it suffices to check this for 
$X$ affine and $\sY=(Y,\emptyset)$, $\sY'=(Y',\emptyset)$.
In this case we can lift $\alpha$ to $\tilde{\alpha}\in W\Omega^{q-r}(X)$
and we can use this to lift $\phi^r_X(\alpha)$ to 
$\phi^r_X(\tilde{\alpha}): K^M_r\to W\Omega^q(X\times -)$
via $\beta\mapsto p_X^*\tilde{\alpha}\cdot  \dlog p_{\sY}^*\beta$. Since $\phi^r_X(\tilde{\alpha})$ is obviously a map
of sheaves (without transfers) and since $W\Omega^q(X\times- )$ is $p$-torsion free (see \cite[I, Cor 3.5]{Il})
it follows from \cite[Lem 1.1]{MS} that $\phi_X(\tilde{\alpha})$ is compatible with transfers and hence so is
$\phi^r_X(\alpha)$. Thus 
\begin{align*}
\phi^r_X(\alpha) \in& \Hom_{\uMPST}(\uomega^*K^M_r, \uHom_{\uMPST}(\Ztr(X), \widetilde{W_n\Omega^{q}}))\\
                        &=(\uomega_!\gamma^r\widetilde{W_n\Omega^q})(X)= (\gamma^r W_n\Omega^q)(X),
\end{align*}
where the last equality holds by \eqref{para:RSCtwist4}.
Next we claim, that
\[\phi^r: W_n\Omega^{q-r}\to \gamma^r W_n\Omega^q, \quad \alpha \text{ on }X\mapsto \phi^r_X(\alpha)\]
is a morphism in $\PST$ (hence also in $\RSC_\Nis$). 
Indeed, this can be checked similarly as above.
We show \eqref{thm:gdRW1}.
Let $\alpha\in W_n\Omega^{q-r-s}(X)$. 
Then $\phi^s(\alpha)$ is determined by $\phi^s_{X,\sY}(\alpha)(\beta_s)$ for $\beta_s\in \uomega^*K^M_s(\sY)$
and $\gamma^s(\phi^r)(\phi^s(\alpha))$ is determined by
\[\phi^r_{X\otimes\sY, \sZ}(\phi^s_{X,\sY}(\alpha)(\beta_s))(\beta_r)
= p_X^*\alpha\cdot\dlog (p_{\sY}^*\beta_s\cdot p_{\sZ}^*\beta_r)\]
for $\beta_r \in \uomega^* K^M_r(\sZ)$, $\sZ\in \uMCor$. Here $p_X, p_{\sY}$ and $p_{\sZ}$ denote the obvious projections. 

For $\alpha\in W_n\Omega^{q-r-s}(X)$ the map
 $(\beta_s,\beta_r)\mapsto p_X^*\alpha\cdot \dlog (p_{\sY}^*\beta_s\cdot p_{\sZ}^*\beta_r)$
also induces an element in
\[\uHom_{\uMPST}(\uomega^*K^M_s\otimes_{\uMPST} \uomega^*K^M_r, \widetilde{W_n\Omega^q})(X)
\cong (\gamma^{s+r}W_n\Omega^q)(X),\]
where the isomorphism follows from \eqref{cor:gtwist1} and \eqref{para:RSCtwist4}.
This yields \eqref{thm:gdRW1}. 
Furthermore,  \eqref{thm:gdRW2} follows from the following general formulas in $W_\bullet\Omega^*$
\[R(\alpha\cdot\dlog\beta)=R(\alpha)\cdot\dlog\beta, \quad d(\alpha\cdot\dlog\beta)= d(\alpha)\cdot\dlog\beta,\]
\[ F(\alpha\cdot\dlog\beta)= F(\alpha)\cdot \dlog \beta, \quad V(\alpha\cdot\dlog\beta)=V(\alpha)\cdot\dlog\beta.\]
Next we prove that
\eq{thm:gdRW3}{\phi^1: \Omega^{q-1}\to \gamma^1\Omega^q}
is an isomorphism.
Indeed by Lemma \ref{lem:P1bf} the map 
\eq{thm:gdRW4}{H^1(\lambda^1_{\sO^{\oplus 2}}): (\gamma^1 \Omega^{q})_X
\xr{\simeq} R^1\pi_*(\Omega^q_{\P^1_X}),}
is an isomorphism, where $\pi: \P^1_X\to X$ is the projection.
On the other hand by definition of $\lambda^1_{\sO^{\oplus 2}}$ and $\phi^1$, the precomposition of
\eqref{thm:gdRW4} with $\phi^1_X$ is exactly the cup-product with 
\[-\cup c_1(\sO_{\P^1_X}(1)): \Omega^{q-1}_X\to R^1\pi_*\Omega^q_{\P^1_X},\]
which is well-known to be an isomorphism. Hence \eqref{thm:gdRW3} is an isomorphism as well.
Iterating and \eqref{thm:gdRW1} yields the isomorphism
\[\phi^r: \Omega^{q-r}\xr{\simeq} \gamma^r\Omega^q, \quad \text{for } q\ge r.\]
Note that a similar argument also works in characteristic zero for $\Omega^\bullet$ and $\Omega^\bullet_{/k}$;
hence \ref{thm:gdRW02} holds.
Since $\gamma^r$ is exact (see Lemma \ref{lem:gamma-exact}) 
and $\phi^r$ is compatible with $d$ we obtain isomorphisms (which we will  denote by $\phi^r$ again)
\[Z^{q-r}\xr{\simeq}\gamma^r Z^r,  \quad B^{q-r}\xr{\simeq} \gamma^r B^q, \quad
Z^{q-r}/B^{q-r}\xr{\simeq} \gamma^r (Z^q/B^q), \]
where $Z^q=\Ker(d: \Omega^q\to \Omega^{q+1})$ and $B^q= d\Omega^{q-1}$. 
Denote by $C^{-1}: \Omega^{q}\xr{\simeq} Z^q/B^q$ the inverse Cartier operator.
One easily checks that the following diagram commutes
\[\xymatrix{
\Omega^{q-r}\ar[r]^{\phi^r}\ar[d]^{C^{-1}} & \gamma^r \Omega^q\ar[d]^{\gamma^r(C^{-1})}\\
Z^{q-r}/B^{q-r}\ar[r]^{\phi^r} &                \gamma^r (Z^q/B^q).
}\]
(In fact  this follows from the compatibility of $\phi^r$ with $F$ and \cite[I, Prop 3.3]{Il}.)
It is direct from this to check that we also have isomorphisms 
\[B_n^{q-r}\xr{\simeq} \gamma^r B_n^q, \quad Z_n^{q-r}\xr{\simeq} \gamma^r Z^q_n,\]
where $B_n^q$ and $Z_n^q$ are defined as in \cite[0, (2.2.2)]{Il}. 
Set ${\rm gr}^{q,n}:=\Ker(W_{n+1}\Omega^q\xr{R} W_n\Omega^q)$.
By compatibility of $\phi^r$ with $R$, we see that $\phi^r$ on $W_{n+1}\Omega^{q-r}$ restricts to
\eq{thm:gdRW5}{\phi^r: {\rm gr}^{q-r,n}\to \gamma^r {\rm gr}^{q,n}.}
By \cite[I, Cor3.9 ]{Il} we have an exact sequence
\[0\to \Omega^q/B_n^q\xr{V^n} {\rm gr}^{q,n}\xr{\beta} \Omega^{q-1}/Z_n^{q-1}\to 0,\]
where $\beta$ is determined by $\beta(V^n(a)+dV^n(b))=b$.
It follows from this and the above that \eqref{thm:gdRW5} is an isomorphism.
Hence $\phi^r: W_n\Omega^{q-r}\to \gamma^r W_n\Omega^q$ is an isomorphism
by induction on $n$.

It remains to show that $\gamma^r W_n\Omega^q=0$, for $q<r$.
By the above it suffices to show $\gamma^1 W_n=0$. By the exactness of $\gamma^1$
it suffices to show $\gamma^1\G_a=0$. 
By Lemma \ref{lem:P1bf}, we have 
$(\gamma^1\G_a)_X\cong R^1\pi_* \sO_{\P^1_X}=0$. 
This completes the proof of the theorem.
\end{proof}

\begin{para}\label{para:gros}
Given the fact that the de Rham-Witt sheaves are reciprocity sheaves, we 
obtain as a corollary of the above and the abstract results for reciprocity sheaves 
new (motivic) proofs of the following results of Gros:
projective bundle formula (by Theorem \ref{thm:pbf} with empty modulus, cf. \cite[I, Thm 4.1.11]{gros}),
blow-up formula (by \eqref{cor:bud1} with empty modulus, cf. \cite[IV, Cor 1.1.11]{gros}),
and a proper pushforward and Gysin morphisms for smooth quasi-projective schemes 
(for $r\ge 0$ in \ref{para:RSC-pf} take $a=0$ and $b=r$ and precompose $f_*$ with  the map induced by
$W_n\Omega^q_Y \to \gamma^r (W_n \Omega^q\la r\ra )_Y$;
for $r<0$ take in \ref{para:RSC-pf} $a=-r$ and $b=-r$ and post compose with the isomorphism from the weak Cancellation
Theorem; cf. \cite[II]{gros}).
We also obtain the (opposed) action of properly supported Chow correspondences 
constructed in \cite{CR12} (by \ref{para:cycle-action}, \ref{prop:cycle-action}).
However let us remind the reader that at the moment the finite transfers on the 
de Rham-Witt complex are defined by restricting the action of  properly supported Chow correspondences
from \cite{CR12}(the construction of which uses all the results above) 
to the case of finite correspondences. It is therefore an interesting problem 
to have a more direct construction of the transfers structure for the de Rham-Witt complex (cf. the discussion
in the characteristic zero  case in \ref{subsec:exaRSC}\ref{subsec:exaRSC3}).
The Gysin sequence however is to our knowledge new:
\end{para}

\begin{cor}\label{cor:dRW-Gysin}
Let $k$ be a perfect field. Let $X\in \Sm$ and let $i: Z\inj X$ be a smooth closed subscheme of codimension $c$.
Denote by $\rho: \tilde{X}\to X$ the blow-up of $X$ in $Z$ and set $E:=\rho^{-1}(Z)$.
\begin{enumerate}[label =(\arabic*)]
\item\label{cor:dRW-Gysin1}
There is a distinguished triangle in $D(X_\Nis)$
\[i_*\Omega^{q-c}_Z[-c]\xr{g_{Z/X}} \Omega^q_X\to R\rho_* \widetilde{\Omega^q}_{(\tilde{X}, E)}\xr{\partial} 
i_*\Omega^{q-c}_Z[-c+1].\]
Similarly with $\Omega^\bullet_{/k}$.
\item\label{cor:dRW-Gysin2}
Assume char$(k)=p>0$. Then we have a distinguished  triangle (for all $n$)
\[i_*W_n\Omega^{q-c}_Z[-c]\xr{g_{Z/X}} W_n\Omega^q_X\to R\rho_* \widetilde{W_n\Omega^q}_{(\tilde{X}, E)}
\xr{\partial} i_*W_n\Omega^{q-c}_Z[-c+1].\]
\item\label{cor:dRW-Gysin3}
Let $(X/W_n)_{\rm crys}$ be the Nisnevich-crystalline site of $X$ relative to $\Spec W_n(k)$ and 
$u_X: (X/W_n)_{\rm crys}\to X_{\Nis}$ the map of sites.
There is distinguished triangle in $D(X_\Nis)$
\[i_* R u_{Z*}\sO_{Z/W_n}[-2c]\xr{g_{Z/X}} R u_{X*}\sO_{X/W_n}\to 
R\rho_* \widetilde{W_n\Omega^\bullet}_{(\tilde{X}, E)}\xr{\partial} \ldots \]
\end{enumerate}
\end{cor} 
\begin{proof}
\ref{cor:dRW-Gysin1} and \ref{cor:dRW-Gysin2} follow directly from Theorem \ref{thm:gysin-tri}
and Theorem \ref{thm:gdRW}.
\ref{cor:dRW-Gysin3} follows from Illusie's isomorphism $R u_{X*}\sO_{X/W_n}\cong W_n\Omega^\bullet_X$
and the above  (cf. \ref{para:coracx}).
\end{proof}

\begin{rmk}\label{rmk:dRW-Gysin}
\begin{enumerate}
\item Note that in characteristic zero we have $\widetilde{\Omega^q}_{(\tilde{X}, E)}=\Omega^q_{\tilde{X}}(\log E)$.
In positive characteristic this is expected to hold but  not yet known (also for the de Rham-Witt sheaves).
\item If char$(k)=0$, then 
$R\rho_* \Omega^\bullet_{\tilde{X}/k}(\log E)\cong Rj_* \Omega^\bullet_{U/k}$,
with $j: U:=X\setminus Z\inj X$ the open immersion,
and the Gysin sequence becomes the classical one for de Rham cohomology.
\end{enumerate}
\end{rmk}

\begin{cor}\label{cor:refined-trace}
Assume char$(k)=p> 0$. Let $f\colon Y\to X$ be a morphism of relative dimension $r\ge 0$
between smooth projective $k$-schemes. Assume that $X$ is ordinary in the sense of \cite[Def (7.2)]{BK}.
Then the Ekedahl-Grothendieck pushfoward (see \cite[II, 1.]{gros})  factors via
\eq{cor:refined-trace1}{R\Gamma(Y, W_n\Omega^q_Y)[r]\to R\Gamma(Y,  W_n\Omega^q_Y/B_{n,\infty}^q)[r]
\xr{f_*} R\Gamma(X, W_n\Omega^{q-r}_X),}
where $B_{n,\infty}^q= \bigcup_s F^{s-1}d W_{n+s-1}\Omega^{q-1}$ (see \cite[IV, (4.11.2)]{IR}) and  
$f_*$ is induced by the pushforward from \ref{para:RSC-pf}.
\end{cor}
\begin{proof}
The exactness of $\gamma$ and Theorem \ref{thm:gdRW} imply 
$\gamma^r B_{n,\infty}^q= B_{n, \infty}^{q-r}$. Thus by functoriality the pushforward 
induce a morphism of triangles
\begin{equation}\label{eq:cor:refined-trace}\xymatrix@R-5pt@C-5pt{
R\Gamma(Y, B^q_{n,\infty})[r]\ar[r]\ar[d]^{f_*} &
R\Gamma(Y, W_n\Omega^q)[r]\ar[r]\ar[d]^{f_*} &
R\Gamma(Y,  W_n\Omega^q/B_{n,\infty}^q)[r]\ar[d]^{f_*}\ar[r] &\\
R\Gamma(X, B^{q-r}_{n,\infty})\ar[r] &
R\Gamma(X, W_n\Omega^{q-r})\ar[r] &
R\Gamma(X,  W_n\Omega^{q-r}/B_{n,\infty}^{q-r})\ar[r] &.\\
}\end{equation}
Since $X$ is ordinary we have $R\Gamma(X, B^{q-r}_{n,\infty})=0$, by \cite[IV, Thm 4.13]{IR}.
This yields the factorization. That the pushforward  coincides (up to sign) with the Ekedahl-Grothendieck
pushforward, follows from the construction of the pushforward in \ref{para:RSC-pf}
and the explicit description of the projective trace (see Definition \ref{defn:traceP}) 
and the Gysin map (see Theorem \ref{prop:pf-gysin}) as well as the corresponding description
for the Ekedahl-Grothendieck pushforward (see \cite[II, 2.6, 3.3]{gros}, also \cite[Prop 2.4.1, Cor 2.4.3]{CR12}).
\end{proof}

\begin{rmk} Let $f\colon Y\to X$ be as in Corollary \ref{cor:refined-trace}, but without assuming that $X$ is ordinary. Then the above proof shows that there is a factorization
    \[R\Gamma(Y, W_n\Omega^r_Y)[r]\to R\Gamma(Y,  W_n\Omega^r_Y/B_{n,\infty}^r)[r]
\xr{f_*} R\Gamma(X, W_n)
    \]
which simply follows from the diagram \eqref{eq:cor:refined-trace} and the fact that $B^0_{n, \infty}=0$. 
\end{rmk}

\begin{cor}\label{cor:ordinary}  Assume char$(k)=p> 0$.
Let $f\colon X\to S$ be a surjective morphism between smooth projective connected $k$-schemes.
Assume that the generic fiber has index prime to $p$.
Then
\[ X \text{ is ordinary } \Longrightarrow S \text{ is ordinary.}\]
\end{cor}
\begin{proof}
Let $B^q=d\Omega^{q-1}$. This is a reciprocity sheaf. By \cite[Def (7.2)]{BK}
$X$ is ordinary if and only if $H^i(X, B^q)=0$, for $i,q\ge 0$.
Thus the statement follows from Corollary \ref{cor:dds}. 
\end{proof}
\begin{rmk}\label{rmk:HWcrystorsionfree} Let $f\colon X\to S$ be as in Corollary \ref{cor:ordinary}, and assume moreover that the generic fiber has a zero cycle of degree prime to $p$.
Then it is possible, with a similar argument, to prove the implication
\[ X \text{ is Hodge-Witt } \Longrightarrow S \text{ is Hodge-Witt.}\]
See \cite[IV, 4.12]{IR}. Similarly, if the crystalline cohomology $H^*(X/W)$ of $X$ is torsion free, the existence of the splitting in  Corollary \ref{cor:dds} implies that the crystalline cohomology of $S$ is torsion free as well. 
\end{rmk}

\begin{cor}\label{cor:BI-exa}
Let $S$ be a separated $k$-scheme of finite type, let $X$ and $Y$ be integral smooth quasi-projective $k$-schemes
both of dimension $N$,
and let $f:X\to S$, $g: Y\to S$  be morphism of $k$-schemes.  Assume that $X$ and $Y$ are properly birational over $S$
(see \ref{sssec:BI}). Then any proper birational correspondence over $S$ between $X$ and $Y$ induces an isomorphism
\eq{cor:BI-exa0}{Rf_* F_X\xr{\simeq} Rg_*F_Y,}
where $F$ is one of the following sheaves (resp. complexes)
\begin{enumerate}[label=(\arabic*)]
\item\label{cor:BI-exa1} any complex of reciprocity sheaves whose terms are subquotients of $\Omega^N_{/k}$, e.g.,
\[\Omega^N_{/k}, \quad \Omega^N_{/k}/\dlog K^M_N, \quad \Omega^N_{/k}/h^0_{\A^1}(\Omega^N_{/k}),\]
 where $h^0_{\A^1}(\Omega^N_{/k})$ is the maximal $\A^1$-invariant subsheaf of $\Omega^N_{/k}$;
\end{enumerate}
and in case  char$(k)=p>0:$
\begin{enumerate}[resume]
\item\label{cor:BI-exa2} 
any complex of reciprocity sheaves whose terms are  subquotient of  $W_n\Omega^N$ ($n\ge 1$),
e.g., 
\[W_n\Omega^N_{\log},\quad W_n\Omega^N/F^{r-1}dW_{n+r-1}\Omega^{N-1}, 
\quad F^{r-1}dW_{n+r-1}\Omega^{N-1}, \quad (r\ge 1),\]
\[R\epsilon_*(\Z/p^n(N)),\quad  R^i\epsilon_*(\Z/p^n(N)), \text{ all }i 
\quad \text{(see \ref{subsec:exaRSC} \ref{subsec:exaRSC4.2})};\]
\item\label{cor:BI-exa4} 
\[G\la N\ra, \]
where $G$ is a smooth commutative  unipotent $k$-group;
 \item\label{cor:BI-exa5}
 \[H^1(G)\la N\ra,\]
where $G$ is a finite commutative $p$-group scheme over $k$ and $H^1(G)(X)=H^1(X_{\rm fppf}, G)$.
\end{enumerate}
\end{cor}
\begin{proof}
First note that all the listed examples are (resp. complexes of) reciprocity sheaves.
This follows directly from the examples discussed in \ref{subsec:exaRSC}.
Therefore the assertions follow directly from Theorem  \ref{thm:BItop} and  Lemma \ref{lem:top}.
\end{proof}

\begin{rmk}
\begin{enumerate}
\item The birational invariance of the cohomology of $\Omega^N_{/k}$ in characteristic zero  is classical
 (using resolution of singularities), 
in positive characteristic this was proven in \cite{CR11} by a similar method as here.
However the statement in {\em loc. cit.} was only for the cohomology sheaves not the complexes in the derived category.
The whole statement can in this case be also deduced from \cite{Kovacs}.
To our knowledge the other statements in Corollary \ref{cor:BI-exa} are new.
\item We get an isomorphism \eqref{cor:BI-exa0} up to bounded torsion
for all $F\in \RSC_\Nis$ which are successive extensions of subquotients  of $\G_a$. 
This follows from Theorem \ref{thm:BIu}, the vanishing $\gamma(\G_a)=0$ (see Theorem \ref{thm:gdRW}),
and the fact that $\gamma$ is exact. In this case the statement can however be also deduced
without the 'up-to-bounded-torsion' assumption from \cite[Thm 1.4]{Kovacs}, the existence of Macaulayfication,
and the fact  that smooth schemes have pseudo-rational singularities (Theorem of Lipman-Teissier).
\end{enumerate}
\end{rmk}

\begin{para}\label{para:cohdim}
Let $K$ be a field of characteristic $p$. 
For $\ell \neq p$, the cohomological $\ell$-dimesnion $\dim_{\ell}(K)$ is defined e.g. in \cite{Serre-CG}.
For $\ell=p$ the cohomological $p$-dimension of $K$ is defined in \cite[Def 1]{KK86} by
\[\dim_p(K):={\rm inf}\{i\in\N\mid \Omega^{i+1}_K= 0 \text{ and } 
H^1_{\et}(K', \Omega^i_{\log})=0, \text{ for all } K'/K \text{ finite} \},\]
where $K'/K$ ranges over all finite field extensions. 
\end{para}

The corollary below generalizes results of Pirutka \cite{Pirutka} and Colliot-Th\'{e}l\`{e}ne-Voisin \cite{CV}, see 
Remark \ref{rmk:BI-exa2}. We thank Colliot-Th\'{e}l\`{e}ne for pointing to these results.
\begin{cor}\label{cor:BI-exa20}
Let $S$ be a separated $k$-scheme of finite type, let $X$ and $Y$ be integral smooth quasi-projective $k$-schemes
both of dimension $d$, and let $f:X\to S$, $g: Y\to S$  be morphism of $k$-schemes.  
Assume that $X$ and $Y$ are properly birational over $S$
(see \ref{sssec:BI}). Let $\ell$ be a prime  and assume $e:=\dim_\ell(k)<\infty$.
(Note $e=0$, if $\ell=p$.)
Then any proper birational correspondence over $S$ between $X$ and $Y$ induces an isomorphism 
\eq{cor:BI-exa020}{Rf_* (R^{d+e}\epsilon_*\Z/\ell^r(j))_X\xr{\simeq} Rg_* (R^{d+e}\epsilon_*\Z/\ell^r(j))_Y,\quad
\text{for all } j\ge 0, r\ge 1.}
\end{cor}
\begin{proof}
Set $F:=R^{d+e}\epsilon_*\Z/\ell^r(j)$. We know $F\in\RSC_\Nis$.
By Theorem \ref{thm:BItop} it suffices to show $F(K)=0$, for any finitely generated field 
$K/k$ with $\td(K/k)<d$. Assume $\ell\neq p$.
By \cite[Thm 1.5]{GL01} we have 
$F(K)= H^{d+e}_{\et}(K, \mu_{\ell^r}^{\otimes j})$, which vanishes 
since $\dim_{\ell}(K)\le d-1+e$, by \cite[II, \S4, Prop 11]{Serre-CG}.
Now assume $\ell=p$ (hence $e=0$). By \cite{GL} (see \ref{subsec:exaRSC}\ref{subsec:exaRSC4.2}) we have 
\[F(K)=\begin{cases} W_n\Omega^d_K, & \text{if }j=d\\ 
                      H^1_{\et}(K,W_n\Omega^{d-1}_{\log}), &\text{if }j=d-1\\
                      0, &\text{else}. \end{cases}\]
Thus in this case $F(K)=0$, since $\dim_{p}(K)\le d-1$ by \cite[3., Cor 2]{KK86}.
\end{proof}

\begin{rmk}\label{rmk:BI-exa2}
If we assume $\ell\neq p$, $f$ and $g$ projective and generically smooth, and we take in 
\eqref{cor:BI-exa020} the stalks in the generic point of $S$ and cohomology we get back
the first statement of \cite[Thm 3.3]{Pirutka} (which generalizes \cite[Prop 3.4]{CV}),
at least in the case where the base field is a finitely generated field over a perfect field.

To our knowledge  the  global statement for $S$ of arbitrary dimension is new even for $\ell\neq p$.
The case $\ell=p$ and $j=d-1$ seems to be completely new. 
\end{rmk}

The following corollary holds for all reciprocity sheaves. We just spell it out for one special example.
\begin{cor}\label{cor:dd0}
Let $S$ be the henselization of a smooth $k$-scheme  in a 1-codimensional point or
a regular connected affine scheme of dimension $\leq 1$ and of finite type over a function field $K$ over $k$.
Let $f:X \to S$ be a smooth projective morphism of relative dimension $d$ and let $\eta\in S$ be the generic point.
Assume the diagonal of the generic fiber $X_\eta$ decomposes as
\eq{cor:dd01}{[\Delta_{X_\eta}]= p_2^*\xi + (i\times\id)_*\beta \quad \text{in } \CH_0(X_\eta\times_\eta X_\eta),}
where $\xi\in \CH_0(X_\eta)$ and $\beta\in \CH_d(Z\times_\eta X_\eta)$ with $i: Z\inj X_\eta$ a 
closed immersion of codimension $\ge 1$.

Then for all $n, i\ge 0$
\[H^0(X,R^i \epsilon_* \Q/\Z(n))= H^0(S, R^i \epsilon_* \Q/\Z(n))\]
see \ref{subsec:exaRSC}, \ref{subsec:exaRSC4.3} for notation.
In particular, taking $n=1$ and $i=2$ we obtain ${\rm Br}(X)={\rm Br}(S)$.
\end{cor}
\begin{proof}
This follows from Theorem \ref{thm:comp-diag}. (Note, in case $S$ is a regular affine curve over $K$ it follows, e.g.,
from \cite[Lem 2.4]{RS} that $S$ satisfies the conditions of that theorem.)
\end{proof}

\begin{cor}\label{cor:ddtop}
Let $f: X\to S$ be a flat projective morphism of relative dimension $d$ 
between smooth integral and quasi-projective $k$-schemes. Let $\dim X=N$.
Assume the diagonal of the generic fiber of $f$ decomposes
as in \eqref{cor:dd01}. 
For $F=F^N$ as in Corollary  \ref{cor:BI-exa}\ref{cor:BI-exa1}-(\ref{cor:BI-exa5})
(in \ref{cor:BI-exa1} and (\ref{cor:BI-exa2}) we only consider
the explicitly listed examples with the exception of $F =\Omega^N_{/k}/h^0_{\A^1}(\Omega^N_{/k})$)
 the pushforward 
\[Rf_* F^N_X\xr{f_*}  F^{N-d}_S[-d]\]
is an isomorphism.
\end{cor}
\begin{proof}
The proof is similar to the proof of Corollary \ref{cor:BI-exa} except that we have to replace the reference to
Theorem \ref{thm:BItop} by a reference to Theorem \ref{thm:comp-diag-top}.
Furthermore we have to observe that in the cases considered we have $\gamma^d F^N\cong F^{N-d}$,
as follows directly from the exactness of $\gamma$ (see Lemma \ref{lem:gamma-exact}),
Theorem \ref{thm:gdRW}, and the weak Cancellation Theorem from \cite{MS}, see \eqref{para:RSCtwist6}.
\end{proof}

\begin{rmk}
\begin{enumerate}
\item If the diagonal only decomposes rationally, then we obtain a similar statement as in Corollary \ref{cor:ddtop}, up to 
bounded torsion.
\item If $F$ is a successive extensions of subquotients  of $\G_a$, then 
$\gamma F=0$ (as follows from Theorem \ref{thm:gdRW} and the exactness of $\gamma$) and 
we can similarly apply Theorem \ref{thm:comp-diag-u}.
\end{enumerate}
\end{rmk}

The following statement seems to be new if $S$ is not an algebraically closed field.

\begin{cor}\label{cor-pictor}
Let $f:X\to S$ and $g: Y\to S$ be two flat, geometrically integral, and projective morphisms 
between smooth connected $k$-schemes.
We furthermore assume that  the generic fiber of $f$ and the generic fiber of $g$ have index 1  
over the function field $k(S)$. Denote by $\Pic_{X/S}$ the relative Picard functor
(which is representable, e.g., \cite[8.2, Thm 1]{BLR}) and by $\Pic_{X/S}[n]$ its $n$-th torsion subfunctor.

If $X$ and $Y$ are stably properly birational over $S$ (see \ref{sssec:BI}\ref{sssec:BI2}),
then any  proper birational correspondence between projective bundles over $X$ and $Y$  
induces an isomorphism of sheaves
on $S_\Nis$ 
\[\Pic_{X/S}[n]\cong \Pic_{Y/S}[n], \quad \text{for all }n.\]
\end{cor}
\begin{proof}
By definition, we have that for any morphism of schemes $T\to S$, $\Pic_{X/S}(T) = H^0(T, R^1 f_{*} \sO_{X_{\rm fppf}}^\times)$. The assumptions on $f$ ensure that for any such $T$ we have 
$f_{T*}\sO_{X_T}=\sO_T$, where $f_T :X_T\to T$ denotes the base change of $f$, (and similarly with $g$). 
Hence $f_{*} \sO_{X_{\rm fppf}}^\times= \sO_{S_{\rm fppf}}^\times$ and  therefore also
\eq{cor-pictor1}{f_*\mu_{n, X_{\rm fppf}}= \mu_{n, S_{\rm fppf}},}
where $\mu_{n, X_{\rm fppf}}$ denotes the sheaf of $n$-th roots of unity in the fppf-topology on $X$.
This yields an exact sequence on $S_{\rm fppf}$
\[0\to R^1f_*\mu_{n, X_{fppf}} \to R^1 f_* \sO_{X_{\rm fppf}}^\times \xr{n\cdot} 
R^1 f_* \sO_{X_{\rm fppf}}^\times.\]
Thus on $S_\et$
\eq{cor-pictor2}{\Pic_{X/S}[n]= v_*(R^1f_*\mu_{n, X_{\rm fppf}}),}
where $v: S_{\rm fppf}\to S_{\rm \et}$ denotes the morphisms of sites (we denote the corresponding morphism 
on $X$ by the same letter.) The spectral sequence for the composition
$R v_* \circ Rf_*$ yields an exact sequence
\ml{cor-pictor3}{
0\to R^1v_*( f_*\mu_{n, X_{\rm fppf}})\to R^1(v\circ f)_*\mu_{n, X_{\rm fppf}}\\
\to v_* R^1f_*\mu_{n, X_{\rm fppf}} \to R^2v_* (f_*\mu_{n, X_{\rm fppf}}).}
Write $n=m p^r$, where $p$ is the exponential characteristic of $k$, $(m,p)=1$ and $r\ge 0$.
By \cite[Thm 11.7]{G-BrauerIII} we have 
\eq{cor-pictor4}{Rv_*\mu_{m, X_{\rm fppf}}\cong \mu_{m, X_{\et}}, \quad Rv_* \G_{m, X_{\rm fppf}} \cong \G_{m, X_{\et}}}
in particular, the second isomorphism of \eqref{cor-pictor4} gives
\eq{cor-pictor5}{Rv_*\mu_{p^r, X_{\rm fppf}}\cong \sO^\times_{X_{\et}}/(\sO^\times_{X_{\et}})^{p^r}[-1]
\cong W_r\Omega^1_{X_{\et}, \log}[-1]}
where  the second isomorphism holds by \cite[I, Prop 3.23.2]{Il} and we use that $X$ is smooth 
(also for the first isomorphism).
Putting \eqref{cor-pictor1} - \eqref{cor-pictor5} together and using $R(v\circ f)_*= Rf_* Rv_*$ and $R^2v_* ( f_*\mu_{n, X_{\rm fppf}})= R^2v_*( \mu_{n, S_{\rm fppf}})= 0$ (note that $S$ is smooth too), 
we obtain an exact sequence on $S_\et$
\eq{cor-pictor6}{0\to W_r\Omega^1_{S_{\et}, \log}\to 
R^1f_* \mu_{m, X_{\et}} \oplus f_*W_r\Omega^1_{X_{\et},\log}
\to \Pic_{X/S}[n] \to 0.}
Let $\epsilon: \Sm_\et\to \Sm_\Nis$ be the morphism  of sites.
We saw in \eqref{subsec:exaRSC4.2.1} that $F^i:=R^i\epsilon_*W_r\Omega^1_{ \log}$
is a reciprocity sheaf, for all $i$. The assumption on the index of the general fiber
of $f$ is equivalent to the existence of a zero-cycle of degree $1$ on the generic fiber $X_K\to K=k(S)$, 
so that we can apply Corollary \ref{cor:dds} to find that 
\[f^*: F^1_S\to f_* F^1_X\]
is split injective. 
We can factor this map as the following composition
\eq{cor-pictor6.5}{
F^1_S\xr{R^1{\epsilon_*}(f^*)} R^1\epsilon_* (f_* W_r\Omega^1_{X_{\et},\log})\xr{d^1_{\epsilon, f}}
R^1(\epsilon\circ f)_*W_r\Omega^1_{X_{\et},\log}
\xr{e^1_{f, \epsilon}}  f_*F^1_X,}
where $d^1_{\epsilon, f}$ (resp. $e^1_{f, \epsilon}$ ) is an edge map of the spectral sequence 
associated to $R\epsilon_* Rf_*$ (resp.  to $Rf_* R\epsilon_*$).
Since by the above the composition \eqref{cor-pictor6.5} is injective, so is the first map in that composition;
it follows that applying $\epsilon_*$ to \eqref{cor-pictor6} yields an exact sequence on
$S_\Nis$
\eq{cor-pictor7}{0\to F^0_S\to 
\epsilon_*R^1f_{*}\mu_{m, X_{\et}}\oplus f_*F^0_X\to \epsilon_*\Pic_{X/S}[n]\to 0.
}
Furthermore, the spectral sequence for $R\epsilon_*\circ R f_*$ and the restriction 
of \eqref{cor-pictor1} to $S_{\et}$ yield an exact sequence
\ml{cor-pictor7.5}{0\to R^1\epsilon_* \mu_{m, S_\et}\to R^1(\epsilon \circ f)_*\mu_{m, X_{\et}}
\to  \epsilon_*R^1f_{*}\mu_{m, X_{\et}}\\
\to R^2\epsilon_* \mu_{m, S_{\et}} \xr{d^2_{\epsilon,f}} R^2(\epsilon\circ f)_*\mu_{m, X_{\et}}.}
We claim that $d^2_{\epsilon,f}$ is injective.
Indeed, set $M^j:= R^j\epsilon_*\mu_{m}\in \HI_\Nis$. 
We can factor $f^*: M^2_S\to f_*M^2_X$ as 
\[f^*: M^2_S\xr{d^2_{\epsilon,f}} R^2(\epsilon\circ f)_*\mu_{m, X_{\et}}\xr{e^2_{f,\epsilon}} f_*M^2_X\]
where the maps are the edge morphisms of the sprectral sequence to  $R\epsilon_* Rf_*$ 
(resp. $Rf_*R\epsilon_*$). By the assumption on the index of the general fiber
of $f$ and Corollary \ref{cor:dds} we find that 
$f^*:  M^2_S\to f_* M^2_X$ is injective; hence so is $d^2_{\epsilon, f}$.
Furthermore since the Nisnevich cohomology of a constant sheaf is trivial,
we have $R^jf_* (\epsilon_*\mu_{m,X_{\et}})=0$, for all $j\ge 1$.
This yields 
\[R^1(f\circ\epsilon)_*\mu_{m,X_{\et}}\cong f_*R^1\epsilon_*\mu_{m,X_{\et}}=f_*M^1_X.\]
Thus \eqref{cor-pictor7.5} yields an exact sequence
\[0\to M^1_S\to f_*M^1_X\to \epsilon_*R^1f_{*}\mu_{m,X_{\et}}\to 0.\]
Together with  \eqref{cor-pictor7} we obtain
\[ \epsilon_*\Pic_{X/S}[n]= \Coker(F^0_S\oplus M^1_S\xr{f^*} f_*(F^0_X\oplus M^1_X))\]
Since we get a similar description for $g: Y\to S$ the statement follows from Theorem \ref{lem:BIH0}.
\end{proof}

\bibliographystyle{amsalpha}
\bibliography{MCorActCoh}

\end{document}